\documentclass[openany]{amsbook}
\usepackage{amsmath,amssymb,hyperref}
\usepackage[all]{xy}

\newcommand{\IN}{\mathbb N}
\newcommand{\II}{\mathbb I}
\newcommand{\IR}{\mathbb R}

\newcommand{\w}{\omega}

\newcommand{\F}{\mathcal F}
\newcommand{\U}{\mathcal U}
\newcommand{\V}{\mathcal V}
\newcommand{\W}{\mathcal W}
\newcommand{\K}{\mathcal K}
\newcommand{\A}{\mathcal A}
\newcommand{\C}{\mathcal C}

\newcommand{\e}{\varepsilon}

\newcommand{\cl}{\mathrm{cl}}
\newcommand{\pr}{\mathrm{pr}}
\newcommand{\Ra}{\Rightarrow}

\newcommand{\supp}{\mathrm{supp}}
\newcommand{\Supp}{\mathrm{Supp}}
\newcommand{\cs}{\mathrm{cs}}

\newcommand{\id}{\mathrm{id}}
\newcommand{\Id}{\mathrm{Id}}
\newcommand{\Top}{\mathbf{Top}}
\newcommand{\TEA}{\mathbf{T}\!E\!\mathbf{A}}
\newcommand{\Tych}{\mathbf{Tych}}

\newcommand{\T}{\mathbf{T}}

\newcommand{\cf}{\mathrm{cf}}

\newcommand{\HM}{\mathsf{HM}}

\newcommand{\Cld}{\mathsf{Cld}}
\newcommand{\Fin}{\mathsf{F\kern-0.3pt in}}
\newcommand{\Clop}{\mathsf{Clop}}
\newcommand{\Comp}{\mathsf{Comp}}
\newcommand{\LCS}{\mathbf{LCS}}
\newcommand{\diam}{\mathrm{diam}}
\newcommand{\Law}{\mathsf{Law}}
\newcommand{\Sem}{\mathsf{Sem}}
\newcommand{\Sl}{\mathsf{SL}}
\newcommand{\IS}{\mathsf{IS}}
\newcommand{\CS}{\mathsf{CS}}
\newcommand{\ICS}{\mathsf{ICS}}
\newcommand{\IAS}{\mathsf{IAS}}
\newcommand{\FA}{\mathsf{AG}}
\newcommand{\FG}{\mathsf{G}}
\newcommand{\PA}{\mathsf{PA}}
\newcommand{\PG}{\mathsf{PG}}
\newcommand{\Lin}{\mathsf{Lin}}
\newcommand{\Lc}{\mathsf{Lc}}

\newtheorem{itheorem}{Theorem}[chapter]
\newtheorem{iproblem}[itheorem]{Problem}
\newtheorem{iprop}[itheorem]{Proposition}

\newtheorem{theorem}{Theorem}
\numberwithin{theorem}{section}
\newtheorem{proposition}[theorem]{Proposition}
\newtheorem{corollary}[theorem]{Corollary}
\newtheorem{lemma}[theorem]{Lemma}
\newtheorem{claim}[theorem]{Claim}
\newtheorem{example}[theorem]{Example}

\theoremstyle{definition}
\newtheorem{remark}[theorem]{Remark}
\newtheorem{definition}[theorem]{Definition}

\newtheorem{idef}[itheorem]{Definition}
\newtheorem{problem}[theorem]{Problem}
\makeindex

\begin{document}

\keywords{$k$-space, Ascoli space, fan, strong fan, function space, Banach space, weak topology, rectifiable space, functor, category, free topological algebra, free topological group.}
\subjclass[2010]{Primary 54D50; 54E20; 54H10; 46B20;\\ Secondary  08B20; 22A05; 22A15; 22A30; 54B30; 54C35; 54D55; 54E18; 54G10; 54G12; 54H11}


\title[Fans and their generalization]{Fans and their applications in\\ General Topology, Functional Analysis\\ and Topological Algebra}
\author{Taras Banakh}
\address{Ivan Franko National University of Lviv (Ukraine)}
 \author{}
 \address{Jan Kochanowski University in Kielce (Poland)}
\email{t.o.banakh@gmail.com}

\maketitle
\tableofcontents

\begin{abstract}
A family $(F_\alpha)_{\alpha\in\lambda}$ of closed subsets of a topological space $X$ is called a (strict) $\Cld$-fan in $X$ if this family is (strictly) compact-finite but not locally finite in $X$. Applications of (strict) $\Cld$-fans are based on a simple observation that $k$-spaces contain no $\Cld$-fan and Ascoli spaces contain no strict $\Cld$-fan. In this paper we develop the machinery of (strict) fans and apply it to detecting the $k$-space and Ascoli properties in spaces that naturally appear in General Topology, Functional Analysis, and Topological Algebra. In particular, we detect (generalized) metric spaces $X$ whose functor-spaces, functions spaces, free (para)topological (abelian) groups, free (locally convex) linear topological spaces, free (Lawson) topological semilattices, and free
(para)topological (Clifford, Abelian) inverse semigroups are $k$-spaces or Ascoli spaces.
\end{abstract}

\chapter{Introduction and survey of main results}


In this paper we isolate one property of topological spaces (namely, the presence of a strong $\Fin$-fan) which is responsible for the failure of the $k$-space (more precisely, the Ascoli) property in many natural spaces such as (free) topological groups, function spaces, or Banach spaces with the weak topology. Let us recall that a topological space $X$ is a \index{topological space!$k$-space}{\em $k$-space} if each $k$-closed subset of $X$ is closed. A subset $A\subset X$ is called \index{subset!$k$-closed}{\em $k$-closed} if for every  compact subset $K\subset X$ the intersection $A\cap K$ is closed in $K$.

\begin{idef}\label{d:fans}
Let $\lambda$ be a cardinal. An indexed family $(X_\alpha)_{\alpha\in \lambda}$ of subsets of a topological space $X$ is called
\begin{itemize}
\item \index{family of sets!locally finite}{\em locally finite} if any point $x\in X$ has a neighborhood $O_x\subset X$ such that the set $\{\alpha\in\lambda:O_x\cap X_\alpha\ne\emptyset\}$ is finite;
\item \index{family of sets!compact-finite} {\em compact-finite} (resp. {\em compact-countable}) in $X$ if for each compact subset $K\subset X$ the set $\{\alpha\in\lambda:K\cap X_\alpha\ne\emptyset\}$ is finite (resp. at most countable);
\item  \index{family of sets!strongly compact-finite}{\em strongly compact-finite} in $X$ if each set $X_\alpha$ has an $\IR$-open neighborhood $U_\alpha\subset X$ such that the family $(U_\alpha)_{\alpha\in A}$ is compact-finite in $X$;
\item  \index{family of sets!strictly compact-finite}{\em strictly compact-finite} in $X$ if each set $X_\alpha$ has a functional neighborhood $U_\alpha\subset X$ such that the family $(U_\alpha)_{\alpha\in A}$ is compact-finite in $X$;
\end{itemize}
\end{idef}

Now we explain some undefined notions appearing in this definition.

A subset $U$ of a topological space $X$ is called \index{subset!$\IR$-open} {\em $\IR$-open} if for each point $x\in U$ there is a continuous function $f:X\to [0,1]$ such that $f(x)=1$ and $f(X\setminus U)\subset\{0\}$. It is clear that each $\IR$-open set is open. The converse is true for open subsets of Tychonoff spaces.

A subset $U$ of a topological space $X$ is called a \index{functional neighborhood}{\em functional neighborhood} of a set $A\subset X$ if there is a continuous function $f:X\to[0,1]$ such that $f(A)\subset \{0\}$ and $f(X\setminus U)\subset\{1\}$. It is clear that each functional neighborhood $U$ of a set $A$ in a topological space $X$ contains an $\IR$-open neighborhood of the closure $\bar A$ of $A$ in $X$. The converse is true if $\bar A$ is compact. In a normal space $X$ each neighborhood of a closed subset $A\subset X$ is functional.

Now we are able to introduce the notion of a fan, the principal notion of this paper.

\begin{idef}\label{d:fan}
Let $X$ be a topological space and $\lambda$ be a cardinal. An indexed family $(F_\alpha)_{\alpha\in\lambda}$ of subsets of $X$ is called a \index{fan}{\em fan}  (more precisely, a \index{fan!$\lambda$-fan}{\em $\lambda$-fan}) in $X$ if this family is compact-finite but not locally finite in $X$. A fan $(F_\alpha)_{\alpha\in\lambda}$ is called \index{fan!strong}\index{fan!strict}{\em strong} (resp. {\em strict}) if each set $F_\alpha$ has a (functional) $\IR$-open neighborhood $U_\alpha\subset X$ such that the family $(U_\alpha)_{\alpha\in\lambda}$ is compact-finite.
\end{idef}
Since each strictly compact-finite family is strongly compact-finite, we get the following implications:
$$\mbox{strict fan $\Ra$ strong fan $\Ra$ fan}.$$

If all sets $F_\alpha$ of a $\lambda$-fan $(F_\alpha)_{\alpha\in\lambda}$ belong to some fixed family $\F$ of subsets of $X$, then the fan will be called an\index{fan!$\F$-fan} {\em $\F$-fan} (more precisely, an \index{fan!$\F^\lambda$-fan}{\em $\F^\lambda$-fan}) in $X$.

Two instances of the family $\F$ will be especially important for our purposes: the family $\Fin$ of  finite subsets of $X$ and the family $\Cld$ of closed subsets of $X$. In these two cases {\em $\F$-fans} will be called \index{fan!$\Fin$-fan}\index{fan!$\Cld$-fan}{\em $\Fin$-fans} and {\em $\Cld$-fans}, respectively. Since each finite subset of a $T_1$-space is closed, each $\Fin$-fan in a $T_1$-space is a $\Cld$-fan. Taking into account that each $\IR$-open neighborhood of a finite subset $F\subset X$ is a functional neighborhood of $F$ in $X$, we conclude that an $\Fin$-fan is strong if and only it is strict.
So, in $T_1$-spaces we get the implications:
$$\xymatrix{
\mbox{strict $\Fin$-fan}\ar@{<=>}[r]\ar@{=>}[d]&
\mbox{strong $\Fin$-fan}\ar@{=>}[r]\ar@{=>}[d]&
\mbox{$\Fin$-fan}\ar@{=>}[d]\\
\mbox{strict $\Cld$-fan}\ar@{=>}[r]&
\mbox{strong $\Cld$-fan}\ar@{=>}[r]&
\mbox{$\Cld$-fan}.
}
$$

Applications of $\Cld$-fans to $k$-spaces are based on the following simple observation.

\begin{iprop}\label{k-no-Cld-fan} A $k$-space contains no $\Cld$-fans and no $\Fin$-fans.
\end{iprop}

\begin{proof} It suffices to check that each compact-finite family $(F_\alpha)_{\alpha\in\lambda}$ of closed subsets of a $k$-space $X$ is locally finite at each point $x\in X$. Since the singleton $\{x\}$ is compact, the set $\Lambda=\{\alpha\in\lambda:x\in F_\alpha\}$ is finite. We claim that the union $F=\bigcup\{F_\alpha:\alpha\in \lambda\setminus \Lambda\}$ is closed in $X$. This follows from the fact that for every  compact subset $K\subset X$ the intersection $K\cap F=\bigcup\{K\cap F_\alpha:\alpha\in\lambda\setminus\Lambda,\;K\cap F_\alpha\ne\emptyset\}$ is closed in $K$  (as the finite union of closed sets). Then $X\setminus F$ is an open neighborhood of the point $x$ witnessing that the family $(F_\alpha)_{\alpha\in \lambda}$ is locally finite at $x$.
\end{proof}

This proposition suggests that topological spaces containing no $\Cld$-fans of various sorts can be considered as generalizations of $k$-spaces.

It turns out that spaces containing no strict $\Cld$-fan spaces are related to another two generalizations of $k$-spaces: $k_\IR$-spaces and Ascoli spaces.

A topological space $X$ is called a \index{topological space!$k_\IR$-space}{\em $k_\IR$-space} if each $k$-continuous real-valued function on $X$ is continuous. A function $f:X\to\IR$ is \index{map!$k$-continuous}{\em $k$-continuous} if for every compact subset $K\subset X$ the restriction $f|K$ is continuous. It is clear that each $k$-space is a $k_\IR$-space. The converse is not true even in the class of $\aleph_0$-spaces (see \cite{Mi73}, \cite{Bor81}).

A Tychonoff space $X$ is \index{topological space!Ascoli}{\em Ascoli} if each compact subset $K\subset C_k(X)$ is \index{subset!evenly continuous}{\em evenly continuous} in the sense that the evaluation map $K\times X\to\IR$, $(f,x)\mapsto f(x)$ is continuous.
Here by $C_k(X)$ we denote the space of continuous real-valued functions on $X$, endowed with the compact-open topology. Ascoli spaces were introduced in \cite{BG} and studied in \cite{Gab2} and \cite{GKP}. The classical Ascoli Theorem \cite[3.4.20]{En} implies that each $k$-space is Ascoli. By \cite{Nob69}, each Tychonoff $k_\IR$-space is Ascoli.
In Corollary~\ref{c:A->noFan} we shall prove that Ascoli spaces contain no strict $\Cld$-fans.

The following diagram describes the relations of various classes of $T_1$-spaces related to $k$-spaces. In this diagram $\lambda$ stands for an infinite cardinal.
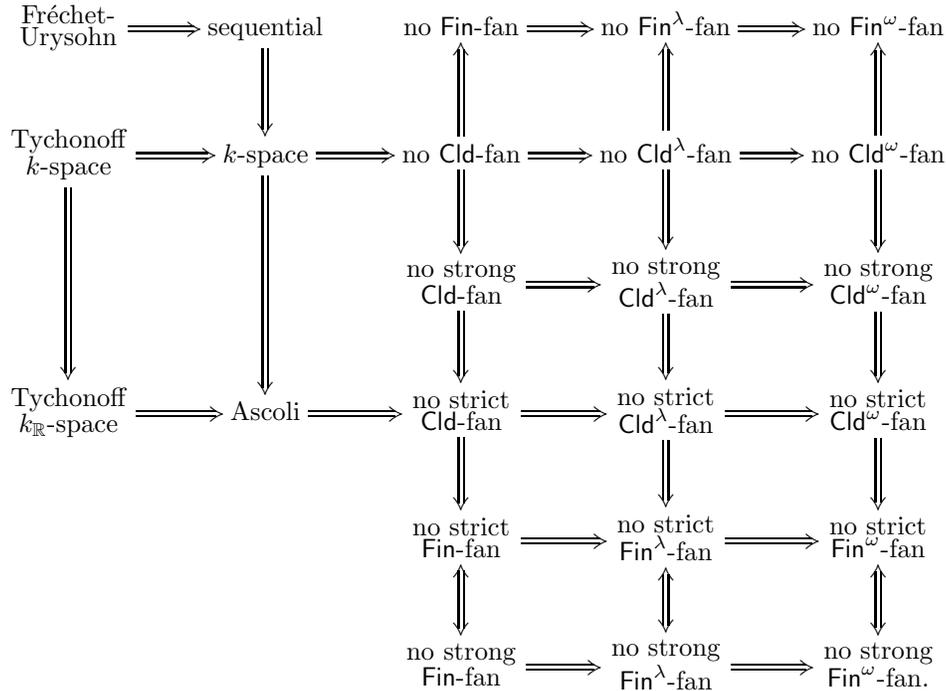
{
$$
\xymatrix{
\genfrac{}{}{0pt}{}{\mbox{Fr\'echet-}}{\mbox{Urysohn}}\ar@{=>}[r]&\mbox{sequential}\ar@{=>}[d]&\mbox{no $\Fin$-fan}\ar@{=>}[r]
&\mbox{no $\Fin^\lambda$-fan}\ar@{=>}[r]&
\mbox{no $\Fin^\w$-fan}\\
\genfrac{}{}{0pt}{}{\mbox{Tychonoff}}{\mbox{$k$-space}}
\ar@{=>}[r]\ar@{=>}[dd]&\mbox{$k$-space}\ar@{=>}[dd]\ar@{=>}[r]
&\mbox{no $\Cld$-fan}\ar@{=>}[r]\ar@{=>}[d]\ar@{=>}[u]
&\mbox{no $\Cld^\lambda$-fan}\ar@{=>}[r]\ar@{=>}[d]\ar@{=>}[u]&
\mbox{no $\Cld^\w$-fan}\ar@{=>}[d]\ar@{=>}[u]\\
&&
\genfrac{}{}{0pt}{}{\mbox{no strong}}{\mbox{$\Cld$-fan}}
\ar@{=>}[d]\ar@{=>}[r]&
\genfrac{}{}{0pt}{}{\mbox{no strong}}{\mbox{$\Cld^\lambda$-fan}}
\ar@{=>}[r]\ar@{=>}[d]&
\genfrac{}{}{0pt}{}{\mbox{no strong}}{\mbox{$\Cld^\w$-fan}}
\ar@{=>}[d]\\
\genfrac{}{}{0pt}{}{\mbox{Tychonoff}}{\mbox{$k_\IR$-space}}
\ar@{=>}[r]&\mbox{Ascoli}\ar@{=>}[r]&
\genfrac{}{}{0pt}{}{\mbox{no strict}}{\mbox{$\Cld$-fan}}
\ar@{=>}[r]\ar@{=>}[d]&
\genfrac{}{}{0pt}{}{\mbox{no strict}}{\mbox{$\Cld^\lambda$-fan}}
\ar@{=>}[r]\ar@{=>}[d]&
\genfrac{}{}{0pt}{}{\mbox{no strict}}{\mbox{$\Cld^\w$-fan}}
\ar@{=>}[d]\\
&&
\genfrac{}{}{0pt}{}{\mbox{no strict}}{\mbox{$\Fin$-fan}}
\ar@{=>}[r]\ar@{<=>}[d]&
\genfrac{}{}{0pt}{}{\mbox{no strict}}{\mbox{$\Fin^\lambda$-fan}}
\ar@{=>}[r]\ar@{<=>}[d]&
\genfrac{}{}{0pt}{}{\mbox{no strict}}{\mbox{$\Fin^\w$-fan}}
\ar@{<=>}[d]\\
&&
\genfrac{}{}{0pt}{}{\mbox{no strong}}{\mbox{$\Fin$-fan}}
\ar@{=>}[r]&
\genfrac{}{}{0pt}{}{\mbox{no strong}}{\mbox{$\Fin^\lambda$-fan}}
\ar@{=>}[r]&
\genfrac{}{}{0pt}{}{\mbox{no strong}}{\mbox{$\Fin^\w$-fan.}}
}
$$
}
We recall that a topological space $X$ is
\begin{itemize}\itemsep=0pt\parskip=0pt
\item  \index{topological space!sequential}{\em sequential} if for each non-closed subset $A\subset X$ some sequence of points of $A$ converges to a point $x\in X\setminus A$;
\item \index{topological space!Fr\'echet-Urysohn}{\em Fr\'echet-Urysohn} if for each subset $A\subset X$ and point $a\in \bar A$ there exists a sequence $\{a_n\}_{n\in\w}\subset A$ that converges to $a$.
\end{itemize}

Looking at the above diagram, we can see that among the considered classes of spaces the largest one is the class of spaces containing no strict (=strong) $\Fin^\w$-fans. In Proposition~\ref{p:Fin-fan-char} we observe that a topological space $X$ contains a (strong) $\Fin^\lambda$-fan if and only if $X$ contains a non-closed subset $A\subset X$ of cardinality $|A|\le\lambda$, which is (strongly) compact-finite in the sense that for each compact subset $K\subset X$ the intersection $K\cap A$ is finite (resp. each point $a\in A$ has an $\IR$-open neighborhood $U_a\subset X$ such that the family $(U_a)_{a\in A}$ is compact-finite in $X$).

Looking for possible applications of $\Fin$-fans, we discovered that many known results concerning sequential, Ascoli or $k$-spaces can be generalized to spaces possessing no $\Cld$-fans or no strong $\Fin^\w$-fans.
\smallskip

In Chapter~\ref{ch:GT} we find some examples of this phenomenon in General Topology. It is easy to see that a $k$-space $X$ is discrete if and only if it is {\em $k$-discrete} in the sense that each compact subset of $X$ is finite. It turns out that the $k$-space property in this  characterization can be replaced by the absence of strict $\Cld$-fans (see Theorem~\ref{t:disc-char}).

\begin{itheorem}
 For a Tychonoff space $X$ the following conditions are equivalent:
 \begin{enumerate}
 \item[\textup{1)}] $X$ is discrete;
 \item[\textup{2)}] $X$ is zero-dimensional, $k$-discrete, countably-tight and contains no $\Clop^\w$-fans;
 \item[\textup{3)}] $X$ is zero-dimensional, $k$-discrete, and contains no $\Clop$-fans;
 \item[\textup{4)}] $X$ is $k$-discrete and contains no strict $\Cld^\w$-fans and no $\Clop$-fans.
 \end{enumerate}
  \end{itheorem}

By a $\Clop$-fan we understand a fan consisting of closed-and-open sets.
\smallskip

To formulate our next results, we need to recall some information on generalized metric spaces, in particular, those defined via $k$-networks.

A family $\mathcal N$ of subsets of a topological space $X$ is called a \index{network}{\em network} if for each point $x\in X$ and neighborhood $O_x\subset X$ of $x$ there is a set $N\in\mathcal N$ with $x\in N\subset O_x$.

A regular space $X$ is called
\begin{itemize}
\item \index{topological space!cosmic}{\em cosmic} if $X$ has a countable network;
\item a \index{topological space!$\sigma$-space}{\em $\sigma$-space} if $X$ has a network that can be written as a countable union of locally finite families;
\item \index{topological space!submetrizable}{\em submetrizable} if $X$ admits a continuous metric.
\end{itemize}
By \cite[4.9]{Grue}, cosmic spaces can be equivalently defined as  continuous images of separable metrizable spaces.

A family $\mathcal N$ of (closed) subsets of a topological space $X$ is called a \index{network!closed $k$-network}({\em closed}) {\em $k$-network} if for each open set $U\subset X$ and a compact subset $K\subset U$ there exists a finite subfamily $\F\subset\mathcal N$ such that $K\subset\bigcup\F\subset U$.
It is clear that each base of the topology is a $k$-network and each $k$-network is a network.

A regular space $X$ is called
\begin{itemize}
\item an \index{topological space!$\aleph_0$-space}{\em $\aleph_0$-space} if $X$ has a countable $k$-network;
\item an \index{topological space!$\aleph$-space}{\em $\aleph$-space} if $X$ has a $k$-network that can be written as a countable union of locally finite families;
\item \index{topological space!$k^*$-metrizable}{\em $k^*$-metrizable} if $X$ has a $k$-network that can be written as a countable union of compact-finite families;
\item an \index{topological space!$\aleph_k$-space}{\em $\aleph_k$-space} if $X$ has a compact-countable $k$-network;
\item an \index{topological space!$\bar\aleph_k$-space}{\em $\bar\aleph_k$-space} if $X$ has a compact-countable closed $k$-network.\end{itemize}
$\aleph_0$-Spaces were introduced by Michael \cite{Mi66} and $\aleph$-spaces by O'Meara \cite{O'M}. By \cite[\S11]{Grue}, $\aleph_0$-spaces can be equivalently defined as images of metrizable separable spaces under continuous compact-covering maps. $k^*$-Metrizable spaces were introduced and studied in \cite{BBK}. By Theorem~6.4 \cite{BBK}, a regular space $X$ is $k^*$-metrizable if and only if $X$ is the image of a metrizable space $Z$ under a continuous map $f:X\to Z$ that has a (not necessarily continuous) section $s:X\to Z$ which preserves precompact sets in the sense that for each compact set $K\subset X$ the set $s(K)$ has compact closure in $Z$. The classes of $\aleph_k$-spaces and $\bar\aleph_k$ seem to be new.
These classes of generalized metric spaces relate as follows:
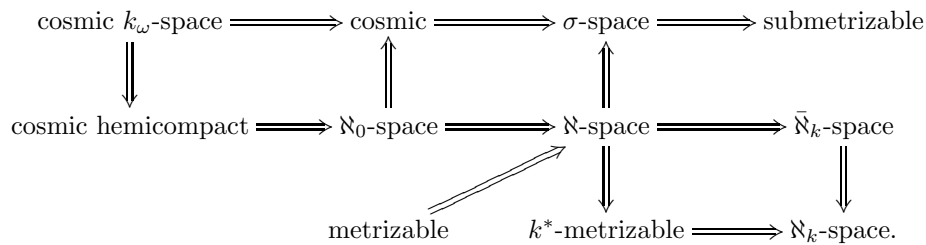
{
$$\xymatrix{
\mbox{cosmic $k_\w$-space}\ar@{=>}[d]\ar@{=>}[r]&\mbox{cosmic}\ar@{=>}[r]&\mbox{$\sigma$-space}\ar@{=>}[r]&\mbox{submetrizable}\\
\mbox{cosmic hemicompact}\ar@{=>}[r]&\mbox{$\aleph_0$-space}\ar@{=>}[r]\ar@{=>}[u]&\mbox{$\aleph$-space}\ar@{=>}[r]\ar@{=>}[d]\ar@{=>}[u]
&\mbox{$\bar\aleph_k$-space}\ar@{=>}[d]\\
&\mbox{metrizable}\ar@{=>}[ru]
&\mbox{$k^*$-metrizable}\ar@{=>}[r]&\mbox{$\aleph_k$-space}.
}
$$
}
We recall that a topological space $X$ is a \index{topological space!$k_\w$-space}{\em $k_\w$-space} if there exists a countable cover $\K$ of $X$ by compact subsets such that a subset $F\subset X$ is closed in $X$ if and only if for every $K\in\K$ the intersection $F\cap K$ is closed in $K$. Each $k_\w$-space is \index{topological space!hemicompact}{\em hemicompact}. The latter means that $X$ admits a countable cover $\K$ by compact subsets such that each compact subset $C\subset X$ is contained in some set $K\in\K$.

A topological space $X$ is called \index{topological space!$\mu$-complete}{\em $\mu$-complete} if each bounded closed subset of $X$ is compact. A subset $B$ of a topological space $X$ is called \index{subset!bounded}{\em bounded} if every locally finite family of open sets in $X$ that meet $B$ is finite. It can be shown that a subset $B$ of a Tychonoff space $X$ is bounded if and only if  for each continuous function $f:X\to\IR$ the set $f(B)$ is bounded in $\IR$. The class of $\mu$-complete spaces contains all Diuedonn\'e complete spaces and consequently all paracompact spaces and all submetrizable Tychonoff spaces (see \cite[6.9.7, 6.10.8]{AT}).
\smallskip

In Section~\ref{s:prod} we reprove and generalize an old result of Tanaka \cite{Tan76} on the $k$-space property of products and exploiting $\Fin^\w$-fans obtain the following characterization (in Corollary~\ref{c:tanaka-gen}).

\begin{itheorem} For $k^*$-metrizable spaces $X,Y$ the product $X\times Y$ is a $k$-space if and only if one of the following conditions holds:
\begin{enumerate}
\item[\textup{1)}] $X$ and $Y$ are metrizable;
\item[\textup{2)}] $X$ and $Y$ are topological sums of cosmic $k_\w$-spaces;
\item[\textup{3)}] $X$ and $Y$ are $k$-spaces and one of them is metrizable and locally compact.
\end{enumerate}
\end{itheorem}
\smallskip

In Section~\ref{s:rectif} we study the topological structure of rectifiable spaces containing no (strong) $\Fin^\w$-fans. We recall that a topological space  $X$ is \index{topological space!rectifiable}{\em rectifiable} if there exists a homeomorphism $h:X\times X\to X\times X$ such that $h(\Delta_X)=X\times\{e\}$ for some point $e\in X$ and $h(\{x\}\times X)=\{x\}\times X$ for all $x\in X$.  Here by $\Delta_X=\{(x,y)\in X\times X:x=y\}$ we denote the diagonal of $X\times X$. Rectifiable spaces were introduced by Arhangel'skii as non-associative generalizations of topological groups (see \cite{Ar02}, \cite{Gul}). Corollary~\ref{c:rect-aleph} implies the following theorem.

\begin{itheorem}\label{it:rec1} If a rectifiable Tychonoff $\aleph$-$k_\IR$-space $X$ contains no strong $\Fin^\w$-fan, then $X$ is either metrizable or is a topological sum of cosmic $k_\w$-spaces.
\end{itheorem}

This result can be compared with a recent result of Banakh and Repov\v s \cite{BR} on the structure of sequential rectifiable spaces with countable $\cs^*$-character. We recall \cite{BZ04} that a topological space $X$ has {\em countable $\cs^*$-character} if for each point $x\in X$ there exists a countable family $\mathcal N_x$ of subsets of $X$ such that for any neighborhood $O_x\subset X$ and any convergent sequence $S\subset X$ with $x=\lim S$ there is a set $N\in\mathcal N_x$ such that $N\subset O_x$ and the intersection $S\cap N$ is infinite.

It is easy to see that each $\bar\aleph_k$-space has countable $\cs^*$-character. On the other hand, the sequence fan $V_{\w_1}=\big((\w{+}1)\times\aleph_1\big)/\big(\{\w\}\times\aleph_1\big)$ with $\w_1$ many spikes is an example of a sequential $\aleph_k$-space with uncountable $\cs^*$-character.  For topological groups the following theorem \cite{BR} was proved in \cite{BZ04}.

\begin{itheorem}[Banakh, Repov\v s]\label{it:rec2} Any sequential rectifiable space $X$ with countable $\cs^*$-character is either metrizable or a topological sum of cosmic $k_\w$-spaces.
\end{itheorem}

We do not know if Theorems~\ref{it:rec1} and \ref{it:rec2} can be unified (at least for topological groups).

\begin{iproblem} Assume that a rectifiable space (or topological group) $X$ is a $k_\IR$-space with countable $\cs^*$-character contains no (strong) $\Fin^\w$-fans. Is $X$ metrizable or a topological sum of cosmic $k_\w$-spaces?
\end{iproblem}
\smallskip

In Chapter~\ref{ch:Funct} we detect spaces without strong $\Fin^\w$-fans among spaces $C_k(X)$ of continuous real-valued functions, endowed with the compact-open topology.
 According to a result of Pol \cite{Pol74}, for a metrizable separable space $X$ the function space $C_k(X)$ is a $k$-space if and only if the space $X$ is locally compact. This characterization was generalized by Gabriyelyan, K\c akol, and Plebanek \cite{GKP} who proved that a metrizable space $X$ is locally compact if and only if the function space $C_k(X)$ is Ascoli. We shall extend this characterization in two directions and prove the following characterization (see, Corollary~\ref{c:aleph0Ck}).

\begin{itheorem} For an $\aleph_0$-space $X$ the following conditions are equivalent:
\begin{enumerate}
\item[\textup{1)}] $C_k(X)$ is metrizable;
\item[\textup{2)}] $C_k(X)$ contains no strong $\Fin^\w$-fan.
\item[\textup{3)}] $X$ is hemicompact;
\end{enumerate}
\end{itheorem}

Another result describing the structure of the function spaces $C_k(X)$ is proved in Theorem~\ref{t:Ck-char}:

\begin{itheorem} For a Tychonoff $\mu$-complete $\bar\aleph_k$-$k_\IR$-space $X$ the following conditions are equivalent:
\begin{enumerate}
\item[\textup{1)}] $C_k(X)$ is a $k_\IR$-space.
\item[\textup{2)}] $C_k(X)$ is Ascoli.
\item[\textup{3)}] $C_k(X)$ contains no strong $\Fin^\w$-fan.
\item[\textup{4)}] $C_k(X)$ is homeomorphic to $\IR^\kappa$ for some cardinal $\kappa$.
\item[\textup{5)}] $X$ is a topological sum of cosmic $k_\w$-spaces.
\end{enumerate}
\end{itheorem}
\smallskip

Our next result illustrating the importance of strong $\Fin^\w$-fan generalizes a recent result of Gruenhage, Tsaban and Zdomskyy \cite{GTZ} who proved that for a metrizable separable space $X$ the function space $C_k(X,2)=\big\{f\in C_k(X):f(X)\subset\{0,1\}\big\}$ is sequential if and only if $X$ is locally compact or the set $X'$ of non-isolated points of $X$ is compact. The sequentiality of $C_k(X,2)$ in this characterization was replaced by the Ascoli property by Pol \cite{Pol2} and Gabriyelyan \cite{Gab}. We generalize this characterization in two directions:

\begin{itheorem} For a zero-dimensional $\aleph_0$-space $X$ the following conditions are equivalent:
\begin{enumerate}
\item[\textup{1)}] $C_k(X,2)$ is metrizable or a cosmic $k_\w$-space;
\item[\textup{2)}] $C_k(X,2)$ is sequential;
\item[\textup{3)}] $C_k(X,2)$ is contains no strong $\Fin^\w$-fan;
\item[\textup{4)}] $X$ is either hemicompact or metrizable with compact set $X'$ of non-isolated points.
\end{enumerate}
\end{itheorem}

Another theorem describing the topological structure of the function spaces $C_k(X,2)$ unifies Theorems~\ref{t:C2-kR}--\ref{t:C2-FU} proved in Section~\ref{s:Funct}. For metrizable spaces this theorem was proved by Gabriyelyan  in \cite{Gab}.

\begin{itheorem} For a $\mu$-complete $\bar\aleph_k$-$k_\IR$-space $X$ its function space $C_k(X,2)$ is:
\begin{enumerate}
\item[\textup{1)}] metrizable iff $C_k(X,2)$ is Fr\'echet-Urysohn iff $X$ is a cosmic $k_\w$-space;
\item[\textup{2)}] sequential iff $X$ is either a cosmic $k_\w$-space or a Polish space with compact set $X'$ of non-isolated points;
\item[\textup{3)}] a $k$-space iff $X$ is either a topological sum $K\oplus D$ of a cosmic $k_\w$-space $K$ and a discrete space $D$ or $X$ is metrizable and has compact set $X'$ of non-isolated points;
\item[\textup{4)}] a $k_\IR$-space iff $C_k(X,2)$ is Ascoli iff $C_k(X,2)$ contains no strong $\Fin^\w$-fans iff $X$ is either a topological sum of cosmic $k_\w$-spaces or $X$ is metrizable and has compact set $X'$ of non-isolated points.
\end{enumerate}
\end{itheorem}
\smallskip

Next, we detect strong $\Fin^\w$-fans in locally convex spaces endowed with the weak topology.
For a subset $B$ of a locally convex space $X$ by $B_w=(B,w)$ we denote the set $B$ endowed with the weak topology. The closed unit ball of the Banach space $X$ endowed with the weak topology will be called the {\em weak unit ball} of $X$.

By \cite[1.5]{GKP}, for a normed space $X$ the space $(X,w)$ is Ascoli  if and only if $X$ is finite-dimensional. The following theorem proved in Section~\ref{s:weak-lc} generalizes this result in two directions and answers a question posed in \cite{GKP}.

\begin{itheorem} For a locally convex linear metric space $X$ the following conditions are equivalent:
\begin{enumerate}
\item[\textup{1)}] $(X,w)$ is metrizable.
\item[\textup{2)}] $(X,w)$ is Ascoli.
\item[\textup{3)}] $(X,w)$ contains no strong $\Fin^\w$-fan.
\item[\textup{4)}] The weak topology on $X$ coincides with the original topology of $X$.
\item[\textup{5)}] The dual space $X^*$ has at most countable Hamel basis.
\end{enumerate}
\end{itheorem}

For bounded subsets of Banach spaces the situation changes dramatically: according to \cite[1.9]{GKP}, the weak unit ball $B_w$ of a Banach space $X$ is Ascoli if and only if $X$ contains no isomorphic copy of the Banach space $\ell_1$. In the following theorem we generalize this result to spaces containing no strong $\Fin^\w$-fan.

\begin{itheorem} For a Banach space $X$ the following conditions are equivalent:
\begin{enumerate}
\item[\textup{1)}] The weak unit ball $B_w$ of $X$ is Fr\'echet-Urysohn.
\item[\textup{2)}] $B_w$ contains no strong $\Fin^\w$-fan.
\item[\textup{3)}] $X$ contains no isomorphic copy of the Banach space $\ell_1$.
\end{enumerate}
\end{itheorem}

This theorem will be proved in Section~\ref{s:Banach} (using different ideas than those in \cite{GKP}).
\smallskip

Now we survey some applications of $\Fin$-fans in Topological Algebra.
According to an old result of Arhangelskii, Okunev and Pestov \cite{AOP}, the free topological group $F(X)$ over a metrizable space $X$ is a $k$-space if and only if $X$ either discrete or a $k_\w$-space. We shall generalize this result of Arhangelskii, Okunev and Pestov in three directions.
First, instead of free topological groups we consider free universal algebras in general varieties of topologized universal algebras (for more details, see Chapter~\ref{ch:magma}). Second, we weaken the $k$-space property of the free universal algebra to the absence of (strong) $\Fin^\lambda$-fans. And third, our characterizations concern not only metrizable spaces, but some generalized metric spaces (including all $\bar\aleph_k$-spaces).
\smallskip

Applying the general results (proved in Chapters~\ref{ch:algebra} and \ref{ch:free}) to some concrete free constructions, we obtain the following ten characterizations. The first of them generalizes the mentioned result of Arhangelskii, Okunev and Pestov \cite{AOP}.

\begin{itheorem}\label{it:FG} For a $\mu$-complete Tychonoff $k_\IR$-space $X$ the following conditions are equivalent:
\begin{enumerate}
\item[\textup{1)}] $X$ is either discrete of a cosmic $k_\w$-space.
\item[\textup{2)}] The free topological group $\FG(X)$ of $X$ is either discrete or a cosmic (countable) $k_\w$-space.
\item[\textup{3)}] $X$ is an $\aleph_k$-space and $\FG(X)$ is a $k$-space.
\item[\textup{4)}] $X$ is a $\bar \aleph_k$-space and $\FG(X)$ is Ascoli.
\item[\textup{5)}] $X$ is a $\bar\aleph_k$-space and $\FG(X)$ is contains no strong $\Fin^{\w_1}$-fan.
\item[\textup{6)}] $X$ is an $\aleph_k$-space and $\FG(X)$ contains  no strong $\Fin^{\w}$-fan and no $\Fin^{\w\!{}_1}$-fan.
\end{enumerate}
\end{itheorem}

A bit different characterization holds for free Abelian topological groups.
For metrizable spaces the equivalence of the conditions (1)--(3) was proved by Arhangelskii, Okunev and Pestov \cite{AOP}.

\begin{itheorem}\label{it:FA} For a Tychonoff $\mu$-complete $k_\IR$-space $X$ the following conditions are equivalent:
\begin{enumerate}
\item[\textup{1)}] $X$ is a topological sum $D\oplus K$ of a discrete space $D$ and a cosmic $k_\w$-space $K$.
\item[\textup{2)}] The free topological Abelian group $\FA(X)$ of $X$ is the product $D\times K$ of a discrete space $D$ and a cosmic $k_\w$-space $K$.
\item[\textup{3)}] $X$ is an $\aleph_k$-space and $\FA(X)$ is a $k$-space.
\item[\textup{4)}] $X$ is an $\bar\aleph_k$-space and $\FA(X)$ is Ascoli.
\item[\textup{5)}] $X$ is an $\bar\aleph_k$-space and $\FA(X)$ contains no strict $\Fin^{\w_1}$-fans.
\item[\textup{6)}] $X$ is an $\aleph_k$-space and $\FA(X)$ contains no strict $\Fin^\w$-fan and no $\Fin^{\w_1}$-fan.
\end{enumerate}
\end{itheorem}

Theorems~\ref{it:FG} and \ref{it:FA} characterize metrizable spaces with Ascoli free topological (Abelian) groups and thus answer Question 6.10 of \cite{GKP}.
\smallskip

The general methods elaborated in this paper work quite well also for free (Abelian) paratopological groups. The following two characterizations answer Problems~7.4.3 and 7.4.4 posed by Arhangelskii and Tkachenko in \cite{AT}.

\begin{itheorem} For a $\mu$-complete Tychonoff $k_\IR$-space $X$ following conditions are equivalent:
\begin{enumerate}
\item[\textup{1)}] $X$ is either discrete or a countable $k_\w$-space.
\item[\textup{2)}] The free paratopological group $\PG(X)$ of $X$ is either discrete or a countable $k_\w$-space.
\item[\textup{3)}] $X$ is an $\aleph_k$-space and $\PG(X)$ contains no strong $\Fin^\w$-fan and no  $\Fin^{\w_1}$-fan.
\item[\textup{4)}] $X$ is an $\aleph_k$-space and $\PG(X)$ is a $k$-space.
\end{enumerate}
\end{itheorem}

\begin{itheorem} For a $\mu$-complete Tychonoff $k_\IR$-space $X$ the following conditions are equivalent:
\begin{enumerate}
\item[\textup{1)}] $X$ is a topological sum $K\oplus D$ of a countable $k_\w$-space $K$ and a discrete space $D$.
\item[\textup{2)}] The free paratopological Abelian group $\PA(X)$ of $X$ is a product $K\times D$ of a countable $k_\w$-space $K$ and a discrete space $D$.
\item[\textup{3)}] $X$ is an $\aleph_k$-space and $\PA(X)$ contains no strong $\Fin^\w$-fan and no $\Fin^{\w_1}$-fan.
\item[\textup{4)}] $X$ is an $\aleph_k$-space and $\PA(X)$ is a $k$-space.
\end{enumerate}
\end{itheorem}

\smallskip

Besides free (para)topological groups, in Chapter~\ref{ch:magma} we consider free constructions in many other varieties of topological semigroups and $I$-semigroups. By a \index{topological semigroup}{\em topological semigroup} we understand a topological space $X$ endowed with a continuous associative binary operation $\cdot\colon X\times X\to X$.
The following characterization is proved in Theorem~\ref{t:Sem}.

\begin{itheorem} For every Tychonoff space $X$ the following conditions are equivalent:
\begin{enumerate}
\item[\textup{1)}] $X$ either is metrizable or is a topological sum of cosmic $k_\w$-spaces.
\item[\textup{2)}] The free topological semigroup $\Sem(X)$ of $X$ either is metrizable or is a topological sum of cosmic $k_\w$-spaces.
\item[\textup{3)}] $X$ is $k^*$-metrizable and $\Sem(X)$ is a $k$-space;
\item[\textup{4)}] $X$ is an $\aleph$-space and $\Sem(X)$ contains no strong $\Fin^\w$-fan.
\item[\textup{5)}] $X$ is $k^*$-metrizable and $\Sem(X)$ contains no strong $\Fin^\w$-fan and no $\Fin^{\w_1}$-fan.
\end{enumerate}
\end{itheorem}
\smallskip

A topological semigroup $X$ is called a \index{topological semilattice}{\em topological semilattice} if the binary operation on $X$ is commutative and idempotent (in the sense that $xx=x$ for all $x\in X$). The following characterization of the $k$-space property in free topological semilattices is proved in Theorem~\ref{t:SL}.

\begin{itheorem}
 For every $\mu$-complete Tychonoff $k_\IR$-space $X$ the following conditions are equivalent:
\begin{enumerate}
\item[\textup{1)}] The space $X$ is a topological sum of cosmic $k_\w$-spaces;
\item[\textup{2)}] The free topological semilattice $\Sl(X)$ of $X$ is a topological sum of cosmic $k_\w$-spaces;
\item[\textup{3)}] $X$ is a $\bar\aleph_k$-space and $\Sl(X)$ is a $k$-space;
\item[\textup{4)}] $X$ is a $\bar\aleph_k$-space and $\Sl(X)$ contains no strong $\Fin^\w$-fan.
\end{enumerate}
\end{itheorem}

A topological semilattice $X$ is called \index{topological semilattice!Lawson}{\em Lawson} if open subsemilattices form a base of the topology in $X$. The free Lawson semilattice $\Law(X)$ over a topological space $X$ can be identified with the hyperspace of all non-empty finite subsets of $X$, endowed with the Vietoris topology and the semilattice operation of union. The following characterization of the $k$-space property in free Lawson semilattices is proved in Theorem~\ref{t:Law}.

\begin{itheorem} For a functionally Hausdorff space $X$ the following conditions are equivalent:
\begin{enumerate}
\item[\textup{1)}] $X$ is metrizable.
\item[\textup{2)}] $\Law(X)$ is metrizable.
\item[\textup{3)}] $X$ is a $k^*$-metrizable $k$-space and $\Law(X)$ contains no $\Fin^\w$-fan.
\item[\textup{4)}] $X$ is a Tychonoff sequential $\aleph$-space and $\Law(X)$ contains no strong $\Fin^\w$-fan.
\item[\textup{5)}] $X$ is $k^*$-metrizable and $\Law(X)$ is a $k$-space.
\end{enumerate}
\end{itheorem}

Topological groups and topological semilattices belong to the class of topological $I$-semigroups. Those are topological semigroups $X$ endowed with a continuous unary operation $*:X\to X$ such that $(x^*)^*=x$, $xx^*x=x$ for all $x\in X$. \index{topological $I$-semigroup}A topological $I$-semigroup is called a \index{topological inverse semigroup}\index{topological Clifford semigroup}{\em topological inverse semigroup} (resp. a {\em topological Clifford semigroup}) if $(xy)^*=y^*x^*$ (resp. $x^*x=xx^*$) for every $x,y\in X$. The following characterization generalizes some results proved in \cite{BGG}.

\begin{itheorem} Let $X$ be a $\mu$-complete Tychonoff $k_\IR$-space and $FS(X)$ be the free topological inverse semigroup (resp. free topological Clifford semigroup, free topological inverse Clifford semigroup, free topological inverse commutative semigroup) of $X$. Then the following conditions are equivalent:
\begin{enumerate}
\item[\textup{1)}] $X$ is a topological sum of cosmic $k_\w$-spaces.
\item[\textup{2)}] $FS(X)$ is a topological sum of cosmic $k_\w$-spaces.
\item[\textup{3)}] $X$ is a $\bar\aleph_k$-space and $FS(X)$ is a $k$-space.
\item[\textup{4)}] $X$ is a $\bar\aleph_k$-space and $FS(X)$ contains no strong $\Fin^\w$-fan.
\end{enumerate}
\end{itheorem}
\smallskip

The $k$-space property in free topological linear spaces is studied in Section~\ref{s:Lin} where the following characterization is proved.

\begin{itheorem} For a Tychonoff $\mu$-complete $k_\IR$-space $X$ the following conditions are equivalent:
\begin{enumerate}
\item[\textup{1)}] $X$ is a cosmic $k_\w$-space.
\item[\textup{2)}] The free linear topological space $\Lin (X)$ of $X$ is a cosmic $k_\w$-space.
\item[\textup{3)}]  $X$ is an $\aleph_k$-space and $\Lin (X)$ is a $k$-space.
\item[\textup{4)}]  $X$ is an $\bar\aleph_k$-space and $\Lin (X)$ is Ascoli.
\item[\textup{5)}] $X$ is a $\bar\aleph_k$-space and $\Lin (X)$ contains no strong $\Fin^{\w_1}$-fan.
\item[\textup{6)}] $X$ is an $\aleph_k$-space and $\Lin (X)$ contains no $\Fin^{\w_1}$-fan.
\end{enumerate}
\end{itheorem}

Finally, we discuss a characterization of topological spaces $X$ whose free locally convex space $\Lc(X)$ is a $k$-space. In \cite{Gab2} Gabriyelyan  proved that a Tychonoff space $X$ is discrete and countable if and only if its free locally convex space $L(X)$ is a $k$-space. The following theorem generalizes this characterization of Gabriyelyan and partly answers Question 3.6 in\cite{Gab14} and Question 6.9 in \cite{GKP}.

\begin{itheorem} For a Tychonoff space $X$ the following conditions are equivalent:
\begin{enumerate}
\item[\textup{1)}] $X$ is discrete and countable.
\item[\textup{2)}] The free locally convex space $\Lc(X)$ is a cosmic $k_\w$-space.
\item[\textup{3)}] $\Lc(X)$ is $k$-space.
\item[\textup{4)}] $\Lc(X)$ and $X$ are Ascoli.
\item[\textup{5)}] $X$ is $\mu$-complete and $\Lc(X)$ contains no strict $\Cld$-fan.
\item[\textup{6)}] $X$ is $\mu$-complete and contains no $\Clop$-fan and $\Lc(X)$ contains no strict $\Cld^\w$-fan.
\end{enumerate}
\end{itheorem}
\medskip

In fact, the results on the $k$-space property in free objects are corollaries of more general results of four various levels of generality. The most general (fourth) level is considered in Chapters~\ref{ch:functor} and \ref{ch:algebra} where we study functors with finite supports in some categories of topological spaces. Culmination results are obtained in Section~\ref{s:monad} where we construct (strong) $\Fin$-fans in functor-spaces $FX$ of monadic functors $F$. A typical result is the following theorem unifying some results proved in Theorems~\ref{t:F-xy}, \ref{t:Functor2} and \ref{t:Functor3}.

\begin{itheorem} Assume that a functor $F:\Top_{3\frac12}\to\Top_{3\frac12}$ in the category of Tychonoff spaces is monomorphic, bounded, $\II$-regular, has finite supports and can be completed to a monad $(F,\delta,\mu)$ that preserves disjoint supports. Let $X$ be a $\mu$-complete Tychonoff $\aleph$-$k_\IR$-space whose functor-space $FX$ contains no strong $\Fin^\w$-fans.
\begin{enumerate}
\item[\textup{1)}] If $F(n)\ne F_1(n)$ for some $n\in\w$, then $X$ is either metrizable or a topological sum of cosmic $k_\w$-spaces.
\item[\textup{2)}] If the functor $F$ does not preserve preimages, then $X$ is either metrizable with compact set $X'$ of non-isolated points or $X$ is a topological sum of a cosmic $k_\w$-space and a discrete space.
\item[\textup{3)}] If $F$ strongly fails to preserve preimages, then $X$ is either discrete or a cosmic $k_\w$-space.
\end{enumerate}
\end{itheorem}

 These general results on functors are applied in Chapter~\ref{ch:free} devoted to studying free universal algebras in varieties of topologized universal algebras. These general results (of the third level) are then applied in the final Chapter~\ref{ch:magma} to construct $\Fin$-fans in free objects in varieties of topological magmas (= topological spaces with a continuous binary operation) and paratopological  $*$-magmas (= topological spaces with an unary operation and a  continuous binary operation). Finally, the general results (of second generality level) on free objects in varieties of topological magmas and paratopological $*$-magmas and applied to free universal algebras in concrete varieties (topological semigroups, [para]topological (abelian) groups, topological semilattices, (locally convex) linear topological spaces, etc.).

\chapter{Some notation and general tools}

In this chapter we discuss some notation and general tools that will be used throughout the paper.

\section{Some standard notations} By $\w$ we denote the set of finite ordinals and by $\IN=\w\setminus\{0\}$ the set of natural numbers; $\IR$ denotes the real line, $\IR_+$ the closed half-line $[0,\infty)$ and $\II$ denotes the closed unit interval $[0,1]$. Each finite ordinal $n\in\w$ is identified with the set $\{0,\dots,n-1\}$ of smaller ordinals.

For a set $X$ by $[X]^{<\w}$ we denote the family of all finite subsets of $X$.  On the other hand, $X^{<\w}$ denotes the set $\bigcup_{n\in\w}X^n$ of finite sequences of points of the set $X$. 

For a subset $A$ of a topological space $X$ its closure in $X$ is denoted by $\bar A$. A subset $A$ of a topological space is called \index{subset!clopen}{\em clopen} if it is closed and open in $X$. A topological space is \index{topological space!zero-dimensional}{\em zero-dimensional} if clopen subsets form a base of the topology of $X$.
A topological space $X$ is {\em Tychonoff} if $X$ is homeomorphic to a subspace of a Tychonoff cube $\II^\kappa$.

For a topological space $X$ by $C_k(X)$ we denote the space of continuous real-valued functions on $X$, endowed with the compact-open topology, see \cite[\S3.4]{En} for more details.

A countable subset $S$ of a topological space $X$ is called \index{convergent sequence} {\em a convergent sequence} if its closure $\bar S$ in $X$ is compact and has a unique non-isolated point $x$, which is called the {\em limit} of $S$ and is denoted by $\lim S$. A standard example of a convergent sequence is the space $\w+1=\w\cup\{\w\}$ endowed with the order topology.

We shall say that a sequence $(A_n)_{n\in\w}$ of subsets of a topological space $X$ {\em converges} to a point $x\in X$ if each neighborhood $O_x\subset X$ of $x$ contains all but finite many sets $A_n$, $n\in\w$. In this case the point $x$ is called the {\em limit} of the sequence $(A_n)_{n\in\w}$.

We say that a subset $A$ of a topological space $X$ is \index{subset!sequentially compact}{\em sequentially compact} in $X$ if every sequence $\{a_n\}_{n\in\w}\subset A$ contains a subsequence $(a_{n_k})_{k\in\w}$ that converges in $X$.

\section{Almost disjoint families of sets}

An indexed family of sets $(A_\alpha)_{\alpha\in\lambda}$ is called \index{family of sets!almost disjoint}{\em almost disjoint} if for any distinct indices $\alpha,\beta$ the intersection $A_\alpha\cap A_\beta$ is finite.
It is well-known \cite[II.1.3]{Kunen} that for any  infinite set $X$ there exists an almost disjoint family $(A_\alpha)_{\alpha\in\mathfrak c}$ consisting of continuum many infinite subsets of $X$.

We shall often exploit the following property of almost disjoint families on countable sets.

\begin{lemma}\label{l:ad} For any almost disjoint family $(A_\alpha)_{\alpha\in\w_1}$ of infinite subsets of a countable set $X$ and any finite subset $F\subset X$ the set $\Lambda=\{(\alpha,\beta)\in\w_1\times\w_1:\alpha\ne\beta$ and $A_\alpha\cap A_\beta\not\subset F\}$ is uncountable.
 \end{lemma}

\begin{proof} Assuming that $\Lambda$ is countable, we can find a countable subset $C\subset\w_1$ such that $\Lambda\subset C\times C$ and conclude that $(A_\alpha\setminus F)_{\w_1\setminus C}$ is an uncountable disjoint family of infinite subsets of $X$, which is not possible.
\end{proof}

\section{Functionally Hausdorff spaces}

A topological space $X$ is called \index{topological space!functionally Hausdorff}{\em functionally Hausdorff} if continuous functions  into the unit interval $\II=[0,1]$ separate points of $X$.

Given a topological space $X$ let $C(X,\II)$ denote the family of all continuous maps from $X$ to the unit interval $\II$. Consider the Tychonoff cube $\II^{C(X,\II)}$ and the map $\delta:X\to \II^{C(X,\II)}$ assigning to each point $x\in X$ the functional  $\delta_x:C(X,\II)\to\II$, $\delta_x:f\mapsto f(x)$. The closure $\beta X$ of the set $\delta(X)$ in $\II^{C(X,\II)}$ is called the \index{Stone-\v Cech compactification}{\em Stone-\v Cech compactification} of $X$ and the map $\delta:X\to \beta X$ is called the {\em canonical map} of $X$ into its Stone-\v Cech compactification. Observe that a topological space $X$ is functionally Hausdorff if and only if the canonical map $\delta:X\to\beta X$ is injective.

The Stone-\v Cech compactification can be used to prove the following extension property of functionally Hausdorff spaces.

\begin{lemma}\label{l:fH-ext} Any continuous map $f:K\to\IR$ defined on a compact subspace $K$ of a functionally Hausdorff space $X$ can be extended to a continuous map $\bar f :X\to\IR$.
\end{lemma}

\begin{proof} Consider the injective canonical map $\delta:X\to\beta X$ and observe that $e=\delta|K$ is a homeomorphism of the compact set $K$ onto the closed subset $\delta(K)$ of the compact Hausdorff space $\beta(X)$. By the normality of $\beta(X)$, the continuous map $g:\delta(K)\to\IR$, $g:y\mapsto f(e^{-1}(y))$, extends to a continuous map $\bar g:\beta(X)\to\IR$. Then the map $\bar f=\bar g\circ \delta:X\to\IR$ is  the required continuous extension of $f$.
\end{proof}

Sometimes we shall need the following lemma.

\begin{lemma}\label{l:fH-inj} For any continuous map $f:X\to \IR$ on a functionally Hausdorff space $X$ and any countable set $A\subset X$ there is a map $g:X\to [0,1]$ such that for the map $h=f+g:X\to\IR$ the restriction $h|A$ is injective.
\end{lemma}

\begin{proof} Since $X$ is functionally Hausdorff, there exists a continuous map $\xi:X\to\II^\w$ whose restriction $\xi|A$ is injective.  Now consider the function space $C_k(\II^\w,\II)$ endowed with the compact-open topology. The space $C_k(\II^\w,\II)$ is  completely metrizable by the metric $d(g,h)=\max_{x\in\II^\w}|g(x)-h(x)|$ and hence $C_k(\II^\w,\II)$ is a Baire space. For every distinct points $a,b\in A$ consider the open dense subspace
$U_{a,b}=\{g\in C_k(\II^\w,\II):f(a)+g\circ \xi(a)\ne f(b)+g\circ\xi(b)\}$. Since the function space $C_k(\II^\w,\II)$ is Baire, the intersection $G=\bigcap\{U_{a,b}:a,b\in A,\;a\ne b\}$ is dense in $C_k(\II^\w,\II)$ and hence it contains some function $g$ such that for the function $h=f+g\circ\xi$ the restriction $h|A$ is injective.
\end{proof}

\section{Bounded sets in topological spaces}

Given a cardinal $\kappa$, we say that a subset $B$ of a topological space $X$ is \index{subset!$\kappa$-bounded}{\em $\kappa$-bounded} in $X$ if each locally finite family $\U$ of open subsets of $X$ meeting $B$ has cardinality $|\U|<\w_1$. $\w$-Bounded sets are called \index{subset!bounded}{\em bounded}. It can be shown that a subset $B$ of a Tychonoff space $X$ is bounded in $X$ if and only if for any continuous function $f:X\to \IR_+$ the image $f(B)$ is bounded in $\IR_+:=[0,\infty)$. A Tychonoff space $X$ is called \index{topological space!pseudocompact}{\em pseudocompact} if $X$ is bounded in $X$.

We shall often exploit the following known fact (see \cite[7.5.1]{AT}).

\begin{lemma}\label{l:Runbound} If a countable subset $B\subset X$ of a Tychonoff space $X$ is unbounded in $X$, then there exists a continuous map $f:X\to\IR_+$ such that the restriction $f|B$ is injective and $f(B)$ is unbounded in $\IR_+$.
\end{lemma}

\begin{proof} Since $B$ is unbounded, there exists an infinite locally finite family $\U$ of open subsets of $X$ that meet $B$. We can assume that the family $\U$ is countable and hence can be enumerated as $\U=\{U_n\}_{n\in\w}$. For every $n\in\w$ choose a point $b_n\in B\cap U_n$ and using the Tychonoff property of $X$, choose a continuous function  $f_n:X\to [0,n]$ such that $f_n(b_n)=n$ and $f_n(X\setminus U_n)\subset \{0\}$. Taking into account that the family $\U$ is locally finite, we conclude that the function $f=\sum_{n\in\w}f_n:X\to\IR_+$ is well-defined, continuous, and unbounded on the set $B$.  By Lemma~\ref{l:fH-inj}, there exists a continuous function $g:X\to[0,1]$ such that the function $h=g+f:X\to\IR_+$ has injective restriction $h|B$. Since $h\ge f$ the set $h(B)$ is unbounded in $\IR_+$.
\end{proof}

A topological space $X$ is called
\index{topological space!$\mu$-complete}\index{topological space!$\mu_s$-complete}{\em $\mu$-complete} (resp. {\em $\mu_s$-complete}) if each closed bounded subset of $X$ is compact (resp. sequentially compact).
We recall that a topological space $X$ is \index{topological space!sequentially compact}{\em sequentially compact} if each sequence in $X$ contains a convergent subsequence.
It is clear that a $\mu$-complete space is $\mu_s$-complete if each compact subset of $X$ is sequentially compact.

By \cite[6.9.7]{AT}, the class of $\mu$-complete spaces includes all Dieudonn\'e complete spaces and hence all paracompact spaces \cite[8.5.13]{En} and all submetrizable Tychonoff spaces (see \cite[6.10.8]{AT}).

\section{$k_\w$-Spaces}

An important role in our results is due to $k_\w$-spaces. So, here we recall the necessary information on such spaces.

By definition, a topological space $X$ is a \index{topological space!$k_\w$-space}{\em $k_\w$-space} if $X$ has a countable cover $\K$ by compact sets such that a subset $U\subset X$ is open if and only if for every $K\in\K$ the intersection $U\cap K$ is open in $K$.
Enumerating the countable family $\K$ as $\K=\{K_n\}_{n\in\w}$ and replacing each compact set $K_n$ by the union $\bigcup_{i\le n}K_i$, we see that a topological space $X$ is a $k_\w$-space if and only if there exists an increasing sequence $(K_n)_{n\in\w}$ of compact sets in $X$ such that $X=\bigcup_{n\in\w}K_n$ and a subset $A\subset X$ is closed in $X$ if and only if $A\cap K_n$ is closed in $K_n$. In this case we shall say that the sequence $(K_n)_{n\in\w}$ {\em generates the $k_\w$-topology} of $X$ and call this sequence a \index{$k_\w$-sequence}{\em $k_\w$-sequence} in $X$.

Each increasing sequence $(K_n)_{n\in\w}$ of compact sets in a  topological space $X$ generates a $k_\w$-topology $\tau$ on its union $Y=\bigcup_{k\in\w}K_n$: the topology $\tau$ consists of open sets $U\subset Y$ such that for every $n\in\w$ the intersection $K_n\cap U$ is open in $K_n$. If all compact spaces $K_n$ are Hausdorff, then the $k_\w$-topology on $Y$ is Hausdorff and moreover normal (since $Y$ is $\sigma$-compact and hence Lindel\"of).

It is known that each $k_\w$-sequence $(K_n)_{n\in\w}$ in a topological space $X$ {\em swallows compact subsets} of the space in the sense that each compact subset $K\subset X$ is contained in some set $K_n$. More generally, this fact holds for bounded subsets of $X$. 

\begin{lemma}\label{l:bound-kw} Let $(K_n)_{n\in\w}$ be a $k_\w$-sequence in a Hausdorff $k_\w$-space $X$. Then each bounded subset $B\subset X$ is contained in some set $K_n$.
\end{lemma}

\begin{proof} Assuming that for every $n\in\w$ a set $B\subset X$ contains a point $x_n\in B\setminus K_n$, we shall prove that $B$ is not bounded in $X$. Using the Hausdorff property of $X$, we can choose an open neighborhood $U_n\subset X$ of $x_n$ whose closure $\overline{U}_n$ does not intersect the compact subset $K_n$. We claim that the family $(U_n)_{n\in\w}$ locally finite in $X$. Given any point $x\in X$, find $n\in\w$ with $x\in K_n$ and consider the set $F=\bigcup_{k\ge n} \overline U_n$. It follows that for every $m\in\w$ the intersection $F\cap K_m=\bigcup_{n\le k\le m}K_m\cap \overline{U}_n$ is closed in $K_m$, which implies that $F$ is closed in $X$ and $X\setminus F$ is an open neighborhood of $x$ meeting only finitely many sets $U_k$, $k\in\w$. Therefore the family $(U_n)_{n\in\w}$ is locally finite in $X$. On the other hand, each set $U_n$ meets the set $B$, which implies that $B$ is not bounded in $X$.
\end{proof}

Another important product property of $k_\w$-spaces is described by the following (known folklore) lemma.

\begin{lemma}\label{l:kw-product} The product $X\times Y$ of any Hausdorff $k_\w$-spaces $X,Y$ is a $k_\w$-space. Moreover, for any $k_\w$-sequences $(X_n)_{n\in\w}$ and $(Y_n)_{n\in\w}$ in the spaces $X,Y$, respectively,  $(X_n\times Y_n)_{n\in\w}$ is a $k_\w$-sequence for the $k_\w$-spaces $X\times Y$.
\end{lemma}

\begin{proof} We need to prove that the $k_\w$-topology $\tau_\w$ generated by the sequence $(X_n\times Y_n)_{n\in\w}$ coincides with the product topology $\tau$ on $X\times Y$. To see that $\tau\subset\tau_\w$, observe that for every open set $U\subset X$ the set $U\times Y$ is $\tau_\w$-open since for every $n\in\w$ the intersection $(U\times Y)\times (X_n\times Y_n)=(U\cap X_n)\times Y_n$ is open in $X_n\times Y_n$. By analogy we can prove that for every open set $V\subset Y$ the set $X\times V$ belongs to the topology $\tau_\w$. Now the inclusion $\tau\subset\tau_\w$ follows from the fact that the topology $\tau$ has a sub-base in $\tau_\w$.

The proof of the inclusion $\tau_\w\subset\tau$ is more complicated. Fix a point $(x,y)\in X\times Y$ and an open neighborhood $W\in\tau_\w$ of $(x,y)$. Find $n\in\w$ such that $(x,y)\in X_n\times Y_n$. Since $W\cap (X_n\times Y_n)$ is an open neighborhood of $(x,y)$ in the compact Hausdorff space $X_n\times Y_n$, there are open sets $U_n\subset X_n$ and $V_n\subset Y_n$  such that
$(x,y)\in U_n\times V_n\subset\bar U_n\times \bar V_n\subset W\cap (X_n\times Y_n)$. Further we proceed by induction. Namely, for every $m>n$ we inductively construct open subsets  $U_m\subset X_m$ and $V_m\subset Y_m$ such that $\bar U_{m-1}\times \bar V_{m-1}\subset U_m\times V_m\subset W\cap (X_m\times Y_m)$.

Then $U=\bigcup_{m\ge n}U_n$ and $V=\bigcup_{m\ge n}$ are open sets in the $k_\w$-spaces $X$ and $Y$ such that $(x,y)\subset U\times V\subset W$, which means that $\tau_\w\subset\tau$ and implies the desired equality $\tau_\w=\tau$.
\end{proof}

\section{$k$-Spaces, $k$-homeomorphisms and $k$-metrizable spaces}

A topological space $X$ is a \index{topological space!$k$-space}{\em $k$-space} if a subset $U\subset X$ is open if and only if for every compact subset $K\subset X$ the intersection $U\cap K$ is open in $K$. It is clear that each $k_\w$-space is a $k$-space. In contrast to $k_\w$-spaces, the product of two  $k$-spaces need not be a $k$-space (see \cite{Tan76}).

A function $f:X\to Y$ between topological spaces is called \index{map!$k$-continuous}{\em $k$-continuous} if for every compact subset $K\subset X$ the restriction $f|K$ is continuous. It follows that each $k$-continuous map defined on a $k$-space is continuous.

A map $f:X\to Y$ between two topological spaces is called a \index{$k$-homeomorphism}{\em $k$-homeomorphism} if $f$ is bijective and both maps $f$ and $f^{-1}$ are $k$-continuous. Two topological spaces $X,Y$ are called \index{topological spaces!$k$-homeomorphic}{\em $k$-homeomorphic} if there exists a $k$-homeomorphism $f:X\to Y$.

By the \index{topological space!$k$-modification of}{\em $k$-modification} $kX$ of a topological space $X$ we understand the set $X$ endowed with the topology consisting of all sets $U\subset X$ such that for every compact subset $K\subset X$ the intersection $U\cap K$ is relatively open in $U$. It follows that $kX$ is a $k$-space and the identity map $\id:kX\to X$ is a continuous $k$-homeomorphism.

A topological space $X$ is called \index{topological space!$k$-metrizable}{\em $k$-metrizable} \index{topological space!$k$-discrete}({\em $k$-discrete}) if it is $k$-homeomorphic to a metrizable (discrete) space. It is clear that a topological space is metrizable if and only if it is a $k$-metrizable $k$-space. It is easy to see that a space is $k$-metrizable (resp. $k$-discrete) if and only if its $k$-modification is metrizable (resp. discrete). Observe also that a space $X$ is $k$-discrete if and only if it contains no infinite compact subset.

In general the $k$-modification of a Tychonoff space is not regular \cite{Mi73}.

\begin{lemma}\label{l:kmod=normal} If Hausdorff topological space $X$ contains a compact subset $K\subset X$ with $k$-discrete complement $X\setminus K$, then the  $k$-modification $kX$ of $X$ is normal and has compact set of non-isolated points.
\end{lemma}

The normality of $kX$ follows from the following simple lemma.

\begin{lemma}\label{l:iso-norm} Each Hausdorff space $X$ with compact set $X'$ of non-isolated points is normal.
\end{lemma}

\begin{proof} The proof repeats the well-known proof of the normality of compact Hausdorff spaces. First we show that the space $X$ is regular. Fix a closed set $F\subset X$ and a point $x\in X\setminus F$. For every point $y\in F\cap X'$ the Hausdorff property of $X$ yields two open disjoint sets $U_y,V_y\subset X$ such that $x\in U_y$ and $y\in V_y$. By the compactness of the set $F\cap X'$ the open cover $\{V_y:y\in F\cap X'\}$ has a finite subcover $\{V_y:y\in E\}$. Here $E$ is a finite subset of $F\cap X'$. Then $O_F=F\cup\bigcup_{y\in E}V_y$ is an open neighborhood of $F$, disjoint with the open neighborhood $U_x=\bigcap_{y\in E}U_y\setminus F$ of $x$.

Now we can prove that $X$ is normal. Fix any two closed disjoint sets $A,B\subset X$. For any point $a\in A\cap X'$ by the regularity of $X$, choose two disjoint open sets $V_a,U_a\subset X$ such that $a\in V_a$ and $B\subset U_a$. By the compactness of $A\cap X'$ the open cover $\{V_a:a\in A\cap X'\}$ of  $A\cap X'$ has a finite subcover $\{V_a:a\in E\}$ (here $E$ is a finite subset of $A\cap X'$). Then $V_A=A\cup\bigcup_{a\in E}V_a$ is an open neighborhood of the set $A$, which is disjoint with the open neighborhood $\bigcap_{a\in E}U_a\setminus A$ of $B$.
\end{proof}

\section{Hemicompact spaces and $k$-sums of hemicompact spaces}

\index{topological space!hemicompact} We recall that a topological space $X$ is {\em hemicompact} if there exists a countable family $\K$ of compact subsets of $X$ such that each compact subset of $X$ is contained in some set $K\in\K$. If the space $X$ is cosmic, then each compact subset $K\in\K$ has countable network and hence has a countable base
$\mathcal B_K$. Then $\bigcup_{K\in\K}\mathcal B_K$ is a countable $k$-network in $X$, which means that $X$ is an $\aleph_0$-space.
Therefore a hemicompact space $X$ is an $\aleph_0$-space if and only if $X$ is cosmic. Observe that a topological space $X$ is hemicompact if and only if $X$ is $k$-homeo\-morphic to a $k_\w$-space.

We shall say that a topological space $X$ is a \index{topological space!$k$-sum of hemicompact spaces}{\em $k$-sum of (cosmic) hemicompact spaces} if $X$ can be written as the union $\bigcup_{\alpha\in\lambda}X_\alpha$ of a compact-clopen disjoint family of (cosmic) hemicompact subspaces of $X$. A family \index{family of sets!compact-clopen} $(X_\alpha)_{\alpha\in\lambda}$ of subsets of $X$ is called {\em compact-clopen} if for each compact set $K\subset X$ and each $\alpha\in\lambda$ the intersection $K\cap X_\alpha$ is clopen in $K$.

\begin{lemma}\label{l:ts-kw} A Hausdorff space $X$ is a $k$-sum of (cosmic) hemicompact spaces if and only if $X$ is $k$-homeomorphic to a topological sum of (cosmic) $k_\w$-spaces.
\end{lemma}

\begin{proof} To prove the ``only if'' part, assume that $X=\bigcup_{\alpha\in\lambda}X_\alpha$ is a $k$-sum of (cosmic) hemicompact spaces $X_\alpha$. For every $\alpha\in\lambda$ consider the topology $\tau_\alpha$ on $X_\alpha$ consisting of all subsets $U\subset X_\alpha$ such that for every compact subset $K\subset X_\alpha$ the set $U\cap K$ is open in $K$. It can be shown that the space $(X_\alpha,\tau_\alpha)$ is a $k_\w$-space and the identity map $X\to \bigoplus_{\alpha\in\lambda}(X_\alpha,\tau_\alpha)$ is a $k$-homeomorphism.

To prove the ``if'' part, assume that $h:X\to Y$ is a $k$-homeomorphism of $X$ onto a topological sum $Y=\oplus_{\alpha\in\lambda}Y_\alpha$ of $k_\omega$-spaces. Then $X=\bigcup_{\alpha\in\lambda}h^{-1}(Y_\alpha)$ is a $k$-sum of the hemicompact subspaces $X_\alpha=h^{-1}(Y_\alpha)$
 of $X$.
 \end{proof}

\section{Compact-finite sets in topological spaces}

A subset $A$ of a topological space $X$ is called \index{subset!compact-finite}{\em compact-finite in} $X$ if the family of singletons $(\{a\})_{a\in A}$ is compact-finite in $X$. If this family is (strictly) strongly compact-finite, then the set $A$ with be called \index{subset!strongly compact-finite}\index{subset!strictly compact-finite} ({\em strictly}) {\em strongly compact-finite}. In a $k$-space each compact-finite subset is closed and discrete.

\begin{proposition}\label{p:unbound-str-c-f} Each unbounded subset $A$ of a Tychonoff space $X$ contains a strictly compact-finite closed discrete infinite subset $D\subset A$.
\end{proposition}

\begin{proof} Choose a continuous function $h:X\to \IR_+$ such that $h(A)$ is an unbounded subset of $\IR_+$. For every $n\in\w$ choose a point $a_n\in A$ such that $h(a_n)>n$ and consider the functional neighborhood $U_n=h^{-1}\big((n,\infty)\big)$ of $a_n$. The family $(U_n)_{n\in\w}$ is locally finite and hence compact-finite in $X$, witnessing that the set $D=\{a_n\}_{n\in\w}$ is strictly compact-finite. Since $\lim_{n\to\infty}h(a_n)=\infty$ the set $D$ has no accumulation point in $X$ and hence is closed and discrete in $X$.
\end{proof}

The presence or absence (strict or strong) $\Fin^\lambda$-fans can be detected using (strongly or strictly) compact-finite sets.

\begin{proposition}\label{p:Fin-fan-char} For an infinite cardinal $\lambda$, a topological $T_1$-space $X$ contains a (strict, strong) $\Fin^\lambda$-fan if and only if $X$ contains a non-closed (strictly, strongly) compact-finite subset $D\subset X$ of cardinality $|D|\le\lambda$.
\end{proposition}

\begin{proof} To prove the ``only if'' part, assume that $(F_\alpha)_{\alpha\in\lambda}$ is a (strong, strict) $\Fin^\lambda$-fan in $X$. Let $x\in X$ be a point at which the fan is not locally finite and let $\Lambda=\{\alpha\in\lambda:x\in F_\alpha\}$.

Since each set $F_\alpha$ is finite, the union $F=\bigcup_{\alpha\in\lambda\setminus\Lambda}F_\alpha$ has cardinality $|F|\le\lambda$ and is not closed in $X$ (as $x\in \bar F\setminus F$).
If the fan $(F_\alpha)_{\alpha\in\lambda}$ is strictly or strongly compact-finite, then so is the set $F$.
\smallskip

To prove the ``if'' part, assume that $F$ is a (strongly, strictly) closed non-closed subsets of cardinality $|F|\le\lambda$. Fix any bijective map $f:|F|\to F$ and for every ordinal $\alpha<\lambda$ put $F_\alpha=\{f(\alpha)\}$ if $\alpha<|F|$ and $F_\alpha=\emptyset$ otherwise. It can be shown that $(F_\alpha)_{\alpha\in\lambda}$ is a (strong, strict) $\Fin^\lambda$-fan in $X$.
\end{proof}

\section{(Strongly) compact-finite families in $\aleph$-spaces}

We recall that a regular $T_1$-space $X$ is an \index{topological space!$\aleph_k$-space}{\em $\aleph$-space} if $X$ has a $\sigma$-locally finite $k$-network $\mathcal N$. The latter means that the $k$-network $\mathcal N$ can be written as the countable union $\mathcal N=\bigcup_{i\in\w}\mathcal N_i$ of locally finite families $\mathcal N_i$, $i\in\IN$.

\begin{proposition}\label{p:aleph-nice-k-network} Each $\aleph$-space $X$ has a $\sigma$-locally finite $k$-network, which is closed under finite intersections and consists of closed subsets of $X$.
\end{proposition}

\begin{proof}  Fix a $\sigma$-locally finite $k$-network $\mathcal A=\bigcup_{i\in\w}\A_i$ in $X$. Replacing each set $A\in\A$ by its closure, we can assume that the family $\A$ consists of closed subsets of $X$.

Replacing each locally finite subfamily $\A_i$ by the locally finite family $\{\bigcap\F:\F\in[\A_i]^{<\w}\}$, we can assume that each $\A_i$ is closed under finite intersections. Observe that for any finite subset $F\subset \w$ the family $\A_F=\{\bigcap_{i\in F}A_i:(A_i)_{i\in F}\in\prod_{i\in F}\A_i\}$ is locally finite. Consequently, the family $\bigcup_{F\in[\w]^{<\w}}\A_F$ is a $\sigma$-locally finite $k$-network, closed under finite intersections and consisting of closed subsets of $X$.
\end{proof}

\begin{proposition}\label{p:aleph-strong-fan} Each compact-finite family $(X_n)_{n\in\w}$ of subsets of a Tychonoff $\aleph$-space $X$ is strongly compact-finite.
\end{proposition}

\begin{proof} By Proposition~\ref{p:aleph-nice-k-network}, the $\aleph$-space $X$ has a $\sigma$-locally finite network $\mathcal A=\bigcup_{i\in\w}\mathcal A_i$ which is closed under finite intersections and consists of closed subsets of $X$.

Let $(X_n)_{n\in \w}$ be a compact-finite family of subsets of $X$. For every $n\in\w$ consider the set $F_n=\bigcup\{A\in\bigcup_{i\le n}\A_i:A\cap X_n=\emptyset\}$ of $X$. The set $F_n$ is closed in $X$, being the union of a locally finite family of closed subsets of $X$.

Since the set $F_n$ is disjoint with $X_n$, its complement $V_n=X\setminus F_n$ is an open neighborhood of $X_n$.
Since the space $X$ is Tychonoff, the open set $V_n$ is $\IR$-open in $X$.

 We claim that the family $(V_n)_{n\in\w}$ is compact-finite in $X$. Given any compact subset $K\subset X$, consider the subfamily $\mathcal A_K=\{A\in\mathcal A:K\cap A\ne\emptyset\}$ and observe that it is at most countable (because $\A$ is $\sigma$-locally finite). Let $[\A_K]^{<\w}$ be the family of finite subsets of $\A_K$ and $\mathcal B_K=\{\F\in[\A_K]^{<\w}:K\subset\bigcup\F\}$.

The family $\mathcal B_K$, being countable, can be enumerated as $\mathcal B_K=\{\F_n\}_{n\in\w}$. For every $n\in\w$ consider the set $B_n=\bigcap_{i\le n}\bigcup\F_i$. It follows that $B_n$ can be written as the finite union of finite intersections of sets of the family $\A_K$. Since the family $\mathcal A$ is closed under finite intersections, we conclude that $B_n=\bigcup\F_m$ for some $m\in\w$. We claim that the sequence $(B_n)_{n\in\w}$ tends to $K$ in the sense that each neighborhood $O_K\subset X$ of $X$ contains all but finitely many sets $B_n$. Indeed, since $\mathcal A$ is a $k$-network, there is a finite subfamily $\F\subset\mathcal A_K$ such that $K\subset\bigcup\F\subset O_K$ and hence $\F=\F_n$ for some $n\in\w$. Then for every $m\ge n$, we get $B_m\subset \bigcup\F_n=\bigcup\F\subset O_K$. So, the sequence $(B_n)_{n\in\w}$ converges to $K$.

Next, we show that for some $k\in\w$ the set $\{n\in\w:B_k\cap X_n\ne\emptyset\}$ is finite. Assuming the opposite, we could construct an increasing number sequence $(n_k)_{k\in\w}$ and a sequence of points $b_k\in B_k\cap X_{n_k}$, $k\in\w$. The convergence of the sequence $(B_k)_{k\in\w}$ to $K$ implies that the set $B=K\cup\{b_k\}_{k\in\w}$ is compact. This compact set meets each set $X_{n_k}$, $k\in\w$, which is not possible as the family $(X_n)_{n\in\w}$ is compact-finite in $X$. This contradiction implies that for some $k\in\w$ the set $\Omega=\{n\in\w:B_k\cap X_n\ne\emptyset\}$ is finite.
It follows that $B_k=\bigcup\F_l$ for some $l\in\w$ and $\F_l\subset \bigcup_{i\le m}\A_i$ for some $m\in\w$. Then the set $\{n\in\w:K\cap V_n\ne\emptyset\}\subset \{n\in \w:B_k\not\subset F_n\}\subset \{0,\dots,m\}\cup\Omega$ is finite, which means that the family $(V_n)_{n\in\w}$ is compact-finite, and the sequence $(X_n)_{n\in\w}$ is strongly compact-finite in $X$.
\end{proof}

\begin{corollary}\label{c:fan-in-aleph-space} Each countable fan $(X_n)_{n\in\w}$ in a Tychonoff $\aleph$-space $X$ is strong.
\end{corollary}

For a topological space $X$ its \index{tightness}{\em tightness} $t(X)$ is defined as the smallest cardinal $\lambda$ such that for each subset $A\subset X$ and point $a\in\bar A$ there is a subset $B\subset X$ of cardinality $|B|\le\lambda$ such that $x\in \bar B$. Spaces with countable tightness are called \index{topological space!countably tight}{\em countably tight}.

\begin{lemma} For any $\Cld$-fan $(F_\alpha)_{\alpha\in\lambda}$ in $X$ there is a subset $\Lambda\subset\lambda$ of cardinality $|\Lambda|\le t(X)$ such that $(F_\alpha)_{\alpha\in\Lambda}$ is a $\Cld$-fan in $X$.
\end{lemma}

\begin{proof} By Definition~\ref{d:fan}, the $\Cld$-fan $(F_\alpha)_{\alpha\in\lambda}$ is not locally finite at some point $x\in X$. Since this family is compact-finite, the set $\Lambda=\{\alpha\in\lambda:x\in F_\alpha\}$ is finite.
Consider the union $F=\bigcup_{\alpha\in\lambda\setminus\Lambda}F_\alpha$ and observe that $x\in\bar F\setminus F$. Choose a subset $E\subset F$ of cardinality $|E|\le t(X)$ such that $x\in\bar E$ and find a subset $B\subset\lambda\setminus\Lambda$ of cardinality $|B|\le |E|\le t(X)$ such that $E\subset\bigcup_{\alpha\in B}F_\alpha$. It follows that the family $(F_\alpha)_{\alpha\in B}$ is not locally finite at $x$ and hence is a $\Cld$-fan in $X$.
\end{proof}

\begin{corollary}\label{c:tight} Let $X$ be a topological space and $\F$ be a family of closed subsets of $X$. The space $X$ contains a (strong, strict) $\F$-fan if and only if $X$ contains a (strong, strict) $\F^{t(X)}$-fan.
\end{corollary}

Corollary~\ref{c:tight} combined with Proposition~\ref{p:aleph-strong-fan} implies the following characterization.

\begin{corollary}\label{c:ct-reduct} Let $X$ be a countably tight Tychonoff $\aleph$-space, $\F$ be a family of closed subsets of $X$. The following conditions are equivalent:
\begin{enumerate}
\item[\textup{1)}] $X$ contains an $\F$-fan;
\item[\textup{2)}] $X$ contains an $\F^\w$-fan;
\item[\textup{3)}] $X$ contains a strong $\F$-fan;
\item[\textup{4)}] $X$ contains a strong $\F^\w$-fan.
\end{enumerate}
If the space $X$ is normal, then the conditions \textup{(1)--(4)} are equivalent to:
\begin{enumerate}
\item[\textup{5)}] $X$ contains a strict $\F$-fan;
\item[\textup{6)}] $X$ contains a strict $\F^\w$-fan.
\end{enumerate}
\end{corollary}

\begin{proof} The equivalences $(1)\Leftrightarrow(2)$, $(3)\Leftrightarrow(4)$, and $(5)\Leftrightarrow(6)$ are proved in Corollary~\ref{c:tight} and the implications $(5)\Ra(3)\Ra(1)$ are trivial.
The implication $(2)\Ra(4)$ follows from Proposition~\ref{p:aleph-strong-fan}.
The implication $(4)\Ra(6)$ follows from the fact that each open neighborhood of a closed subset $F$ of a normal space $X$ is a functional neighborhood of $F$ in $X$.
\end{proof}

\begin{remark} An example of a paracompact $\aleph$-space, which is not countably tight can be found in \textup{\cite[3.9]{GK15}}.
\end{remark}

\section{Extension operators and $C_k$-embedded subsets}

For a subspace $Z$ of a topological space $X$ a map $E:C_k(Z)\to C_k(X)$ will be called an \index{extension operator}{\em extension operator} if for any function $f\in C_k(Z)$ the function
$Ef\in C_k(X)$ restricted to $Z$ equals $f$. Such operator will be called \index{extension operator!$0$-continuous} {\em $0$-continuous} if for every compact subset $K\subset X$ there is a compact set $B\subset Z$ such that for any function $f\in C_k(Z)$ with $f|B=0$ we get $Ef|K=0$.

\begin{proposition}\label{p:lin->0} Let $Z$ be a subspace of a topological space $X$. Each continuous linear extension operator $E:C_k(Z)\to C_k(X)$ is $0$-continuous.
\end{proposition}

\begin{proof} For every compact set $K\subset X$ the set $[K;1]=\{f\in C_k(X):$ $\sup_{x\in K}|f(x)|<1\}$ is a neighborhood of zero in $C_k(X)$. The continuity of $E$ yields a neighborhood of zero $U\subset C_k(Z)$ such that $E(U)\subset [K,1]$. By the definition of the compact-open topology on $C_k(Z)$, there exist a compact set $B\subset Z$ and a number $\e>0$ such that the set $[B;\e]=\{f\in C_k(Z):\sup_{x\in B}|f(x)|<\e\}$ is contained in $U$. We claim that $Ef|K=0$ for any function $f\in C_k(X)$ with $f|B=0$. Assuming that $Ef(x)\ne 0$ for some $x\in K$, consider the real number $\lambda=1/Ef(x)$ and observe that the function $\lambda\cdot f$ belongs to $[B,\e]$ and hence $\lambda\cdot E(f)=E(\lambda f)\in [K;1]$, which is not possible as $\lambda\cdot f(x)=1$.
\end{proof}

We shall say that a subspace $Z\subset X$ of a topological space $X$ is \index{subset!$C_k$-embedded}{\em $C_k$-embedded} if there exists a $0$-continuous extension operator $E:C_k(Z)\to C_k(X)$. An Extension Theorem of Borges \cite{Bor66} implies that every closed subspace of a stratifiable space is $C_k$-embedded.

Let us recall that a regular space $X$ is \index{topological space!stratifiable}{\em stratifiable} if each open set $U\subset X$ can be represented as the union $U=\bigcup_{n\in\w}U_n$ of a sequence $(U_n)_{n\in\w}$ of open sets $U_n$ such that $\overline{U}_n\subset U_{n+1}$ for every $n\in\w$ and $V_n\subset U_n$ for any open sets $V\subset U$ in $X$ and any $n\in\w$. Stratifiable spaces were introduced by Ceder \cite{Ceder} and named by Borges \cite{Bor66}. They form an important class of generalized metric spaces, see \cite[\S5]{Grue}. It is known \cite[5.7]{Grue} that stratifiable spaces are perfectly paracompact and hence are normal. The following deep fact was proved by Borges in \cite{Bor66} (see also Corollary~\ref{c:s->C_k-emb}).

 \begin{proposition}\label{p:Borges} For any closed subspace $Z$ of a stratifiable space $X$ there exists a linear continuous extension operator $E:C_k(Z)\to C_k(X)$.
Consequently, $Z$ is $C_k$-embedded in $X$.
\end{proposition}

For our purposes $C_k$-embedded subspaces are interesting because of the following fact.

\begin{proposition}\label{p:Ck-embed1}  If $Z$ is an $C_k$-embedded subspace in a topological space $X$, then every strictly compact-finite family of sets $(F_\alpha)_{\alpha\in\lambda}$ in $Z$ remains strictly compact-finite in $X$.
\end{proposition}

\begin{proof} Let $E:C_k(Z)\to C_k(X)$ be a $0$-continuous extension operator. Since the family $(F_\alpha)_{\alpha\in\lambda}$ is strictly compact-finite, each set $F_\alpha$ has a functional neighborhood $U_\alpha\subset Z$ such that the family $(U_\alpha)_{\alpha\in\lambda}$ is compact-finite in $Z$.

For every $\alpha\in\lambda$ choose  a continuous function $f_\alpha:Z\to[0,1]$ such that $f_\alpha(F_\alpha)\subset\{1\}$ and $f_{\alpha}(Z\setminus U_\alpha)\subset \{0\}$. Consider the extended function $Ef_\alpha\in C_k(X)$ and observe that $V_\alpha=(Ef_\alpha)^{-1}((0,1])$ is a functional neighborhood of the set $F_\alpha$ in $X$.

We claim that the family $(V_\alpha)_{\alpha\in\lambda}$ is compact-finite in $X$.
Given any compact subset $K\subset X$, use the $0$-continuity of the operator $E$ and find a compact set $B\subset Z$ such that $Ef|K=0$ for any  $f\in C_k(Z)$ with $f|B=0$. Since the family $(U_\alpha)_{\alpha\in\lambda}$ is compact-finite, the set $\Lambda=\{\alpha\in\lambda:U_\alpha\cap B\ne \emptyset\}$ if finite. Then for every $\alpha\in\lambda\setminus \Lambda$ we get $f_\alpha|B=0$ and thus $Ef_\alpha|K=0$ and $V_\alpha\cap K=\emptyset$, which means that the family $(V_\alpha)_{\alpha\in\lambda}$ is compact-finite in $X$ and $(F_\alpha)_{\alpha\in\lambda}$ is strictly compact-finite in $X$.
\end{proof}

A similar result holds also for strongly compact-finite families of Lindel\"of sets.
We recall that a topological space $X$ is \index{topological space!Lindel\"of}{\em Lindel\"of} if each open cover of $X$ has a countable subcover. A subset $U$ of a topological space $X$ is \index{subset!functionally open}{\em functionally open} if $U=f^{-1}((0,1])$ for some continuous function $f:X\to[0,1]$.

\begin{proposition}\label{p:Ck-embed2}  Let $Z$ be a $C_k$-embedded subspace in a topological space $X$ and $(F_\alpha)_{\alpha\in\lambda}$ be a strongly compact-finite family of set in the space $Z$. If each $\IR$-open subset of $X$ is functionally open or every set $F_\alpha$, $\alpha\in A$, is Lindel\"of, then the family $(F_\alpha)_{\alpha\in\lambda}$ is strongly compact-finite in $X$.
\end{proposition}

\begin{proof} Since the family $(F_\alpha)_{\alpha\in\lambda}$ is strongly compact-finite in $Z$, each set $F_\alpha$ has an $\IR$-open neighborhood $U_\alpha\subset Z$ such that the family $(U_\alpha)_{\alpha\in\lambda}$ is compact-finite in $Z$. For every $\alpha\in\lambda$ we shall construct a continuous function $f_\alpha:Z\to [0,1]$ such that $f_{\alpha}(F_\alpha)\subset(0,1]$ and $f_\alpha(X\setminus U_\alpha)\subset\{0\}$.

If the set $U_\alpha$ is functionally open, then the function $f_\alpha$ exists by the definition of a functionally open set. So, now we assume that $F_\alpha$ is Lindel\"of. Since the set $U_\alpha$ is $\IR$-open, for every $x\in F_\alpha$ there is a continuous function $g_x:Z\to[0,1]$ such that $g_x(x)=1$ and $g_x(X\setminus U_\alpha)\subset\{0\}$. Since the space $F_\alpha$ is Lindel\"of, the open cover $\{g_x^{-1}((\frac12,1]):x\in F_\alpha\}$ of $F_\alpha$ has a countable subcover $\{g_x^{-1}((\frac12,1]):x\in F'_\alpha\}$. Here $F'_\alpha$ is a suitable countable subset of $F_\alpha$. Let $F_\alpha'=\{x_n\}_{n\in\w}$ be an enumeration of the set $F'_\alpha$. Then the function $f_\alpha=\sum_{n\in\w}\frac1{2^{n+1}}g_{x_n}$ is continuous and has the required property: $f_\alpha(F_\alpha)\subset(0,1]$ and $f_\alpha(X\setminus U_\alpha)\subset\{0\}$.

Let $E:C_k(A)\to C_k(X)$ be a 0-continuous extension operator. For every $\alpha\in\lambda$ consider the extended function $Ef_\alpha\in C_k(X)$ and put $V_\alpha=(Ef_\alpha)^{-1}((0,1])$. Proceeding as in the proof of Proposition~\ref{p:Ck-embed1}, we can show that the family $(V_\alpha)_{\alpha\in\lambda}$ is compact-finite in $X$, which implies that $(F_\alpha)_{\alpha\in \lambda}$ is a strongly compact-finite family in $X$.
\end{proof}

For retracts, Proposition~\ref{p:Ck-embed2} can be proved without any conditions on sets of the family. We recall that a subset $A$ of a topological space $X$ is called a \index{subset!retract}{\em retract} of $X$ if there exists a continuous map $r:X\to A$ such that $r(a)=a$ for all $a\in A$.

\begin{proposition} Let $Z$ be a retract of a topological space $X$. Each strongly (strictly) compact-finite family of sets $(F_\alpha)_{\alpha\in\lambda}$ in $Z$ remains a strongly (strictly) compact-finite in $X$.
\end{proposition}

\begin{proof} Let $r:X\to Z$ be a retraction. Since $(F_\alpha)_{\alpha\in\lambda}$ is strongly (strictly) compact-finite in $Z$, each set $F_\alpha$ has a (functional) $\IR$-open neighborhood $U_\alpha\subset Z$ such that the family $(U_\alpha)_{\alpha\in\lambda}$ is compact-finite. It can be shown that for every $\alpha\in\lambda$ the set $W_\alpha=r^{-1}(U_\alpha)$ is a (functional) $\IR$-neighborhood of the set $F_\alpha$ in $X$, witnessing that the family $(F_\alpha)_{\alpha\in\lambda}$ is (strictly) strongly compact-finite in $X$.
\end{proof}

Combining Propositions~\ref{p:Borges}--\ref{p:Ck-embed2} with Proposition~\ref{p:aleph-strong-fan}, we get the following corollary.

\begin{corollary}\label{c:s-sfan1}  Let $Z$ be a closed subspace of a stratifiable space $X$.
\begin{enumerate}
\item[\textup{1)}] Each strongly (strictly) compact-finite family of sets $(F_\alpha)_{\alpha\in\lambda}$ in $Z$
remains strongly (strictly) compact-finite in $X$.
\item[\textup{2)}] If $Z$ is an $\aleph$-space, then each countable compact-finite family of sets $(F_n)_{n\in\w}$ in $Z$ is strongly compact-finite in $X$.
\end{enumerate}
\end{corollary}

\begin{proof} The space $X$, being stratifiable, is perfectly paracompact, which implies that all open subsets of $X$ are functionally open. By Proposition~\ref{p:Borges}, $Z$ is $C_k$-embedded into $X$. Now the first statement can be derived from Propositions~\ref{p:Ck-embed1} and \ref{p:Ck-embed2}. If $Z$ is an $\aleph$-space, then by Proposition~\ref{p:aleph-strong-fan}, each compact-finite family of sets $(F_n)_{n\in\w}$ in $Z$ is strongly compact-finite in $Z$ and also in $X$ (by the first statement).
\end{proof}

\chapter{Applications of strict $\Cld$-fans in General Topology}\label{ch:GT}.

In this section we discuss some applications of (strict) $\Cld$-fans and $\Clop$-fans in General Topology. We recall that a \index{fan!$\Cld$-fan}\index{fan!strict $\Cld$-fan}({\em strict}) {\em $\Cld$-fan} in a topological space $X$ is a (strictly) compact-finite family $(F_\alpha)_{\alpha\in\lambda}$ of closed subsets of $X$, which is not locally finite in $X$.

 \index{fan!$\Clop$-fan}By a {\em $\Clop$-fan} in a topological space $X$ we shall understand an $\F$-fan for the family $\F=\Clop$ of clopen subsets of $X$. It is clear that each $\Clop$-fan is strict.

\section{Characterizing $P$-spaces and discrete spaces}

In this section we apply strict $\Cld$-fans to characterizing discrete spaces and $P$-spaces.
Let us recall that a topological space $X$ is called a \index{topological space!$P$-space}{\em $P$-space} if each $G_\delta$-set in $X$ is open. A space $X$ is {\em $k$-discrete} if each compact subset of $X$ is finite.

\begin{proposition}\label{P-space} For a Tychonoff space $X$ the following conditions are equivalent:
\begin{enumerate}
\item[\textup{1)}] $X$ is a $P$-space;
\item[\textup{2)}] $X$ is $k$-discrete and contains no $\Cld^\w$-fans;
\item[\textup{3)}] $X$ is $k$-discrete and contains no strict $\Cld^\w$-fans;
\item[\textup{4)}] $X$ is zero-dimensional, $k$-discrete and contains no $\Clop^\w$-fans.
\end{enumerate}
\end{proposition}

\begin{proof} $(1)\Ra(2)$. Assume that $X$ is a $P$-space. Then each $F_\sigma$-set in $X$ is closed. In particular, each countable subset of $X$ is closed. This implies that all (countably) compact subspaces of $X$ are finite and hence $X$ is $k$-discrete. To see that $X$ contains no $\Cld^\w$-fan,  we need to check that each compact-finite family $(F_n)_{n\in\w}$ of closed subsets of $X$ is locally finite at each point $x\in X$. Since the singleton $\{x\}$ is compact, the set $\Lambda=\{n\in\w:x\in F_n\}$ is finite. Since $X$ is a $P$-space, the countable intersection $U=\bigcap_{n\in\w\setminus \Lambda}X\setminus F_n$ is an open neighborhood of $x$ meeting only finitely many sets $F_n$, $n\in\w$, and witnessing that the family $(F_n)_{n\in\w}$ is locally finite at $x$.
\smallskip

The implication $(2)\Ra(3)$ is trivial.
\smallskip

$(3)\Ra(1)$ Assume that $X$ is $k$-discrete and contains no strict $\Cld^\w$-fans. To derive a contradiction, assume that $X$ is not a $P$-space. Then there is an $F_\sigma$-set $A\subset X$ which is not closed in $X$. Fix any point $x\in\bar A\setminus A$.

Write the $F_\sigma$-set $A$ as the union $\bigcup_{n\in\IN}A_n$ of an increasing sequence $(A_n)_{n\in\IN}$ of closed sets $A_n\subset A_{n+1}\subset X$. Since the space $X$ is Tychonoff, for every $n\in\w$ we can find a continuous function $f_n:X\to [0,2^{-n}]$ such that $f(x)=0$ and $f(A_n)\subset\{2^{-n}\}$. Then $f=\sum_{n=1}^\infty f_n:X\to [0,1]$ is a continuous function such that $f(x)=0$ and $f(A)\subset(0,1]$. For every $n\in\IN$ consider the closed subset $F_n=f_n^{-1}\big([\frac1{2^{n+1}},\frac1{2^n}]\big)$ and its functional neighborhood $U_n=f_n^{-1}\big((\frac1{2^{n+2}},\frac1{2^{n-1}})\big)$. Since all compact subsets of $X$ are finite, the family $(U_n)_{n\in\IN}$ is compact-finite and hence the family $(F_n)_{n\in\IN}$ is strictly compact-finite in $X$. Since $X$ contains no strict $\Cld^\w$-fan,  this family is locally finite. Then $x$ has an open neighborhood $O_x\subset X$ such that the set $E=\{n\in\IN:O_x\cap F_n\ne\emptyset\}$ is finite. Replacing $O_x$ by the open neighborhood $O_x\setminus \bigcup_{n\in E}F_n$ we can assume that $O_x\cap F_n=\emptyset$ for all $n\in\IN$, which implies that $O_x\subset f^{-1}(0)$. On the other hand, the choice of the sequence $(A_n)_{n\in\IN}$ guarantees that the neighborhood $O_x$ has a common point $a\in O_x\cap A$ with the set $A$ and hence $a\in A_n$ for some $n\in\IN$. Then $f(a)\ge f_n(a)=\frac1{2^n}>0$, which contradicts the inclusion $a\in O_x\subset f^{-1}(0)$.
 This contradiction shows that $X$ is a $P$-space.
 \smallskip

 It is clear that the equivalent conditions (1) and (2) imply (4).
 To prove that $(4)\Ra(1)$, assume that $X$ is zero-dimensional, $k$-discrete and contains no  $\Clop^\w$-fans.
Assuming that $X$ is not a $P$-space, fix a non-open $G_\delta$-set $G\subset X$ and find a point $x\in G$ such that $G$ is not a neighborhood of $x$ in $X$. Using the zero-dimensionality of $X$, choose a decreasing sequence $(U_n)_{n\in\w}$ of clopen neighborhoods of $x$ such that $\bigcap_{n\in\w}U_n\subset G$ and the set $V_n=U_n\setminus U_{n+1}$ is not empty for all $n\in\w$. Then the family $(V_n)_{n\in\w}$ is compact-finite but locally finite in $X$, which means that $(V_n)_{n\in\w}$ is a $\Clop^\w$-fan in $X$.
\end{proof}

 Proposition~\ref{P-space} will be used to prove the following characterization of discrete spaces.

 \begin{theorem}\label{t:disc-char}
 For a Tychonoff space $X$ the following conditions are equivalent:
 \begin{enumerate}
 \item[\textup{1)}] $X$ is discrete;
 \item[\textup{2)}] $X$ is zero-dimensional, $k$-discrete, countably-tight and contains no $\Clop^\w$-fan;
 \item[\textup{3)}] $X$ is zero-dimensional, $k$-discrete, and contains no $\Clop$-fan;
 \item[\textup{4)}] $X$ is $k$-discrete and contains no strict $\Cld^\w$-fans and no $\Clop$-fans.
 \end{enumerate}
  \end{theorem}

 \begin{proof} The implication $(1)\Ra(2)$ is trivial and $(2)\Ra(3)$ follows from Corollary~\ref{c:tight}.

 To prove that $(3)\Ra(1)$, assume that $X$ is zero-dimensional, $k$-discrete and contains no $\Clop$-fans.  To derive a contradiction, assume that the space $X$ is not discrete and fix a non-isolated point $x\in X$. By Zorn's Lemma, there exists a maximal disjoint family $(U_\alpha)_{\alpha\in\lambda}$ of non-empty clopen subsets of $X$ which do not contain the point $x$.

By the maximality of $\U$, the point $x$ belongs to the closure of the union $\bigcup_{\alpha\in\lambda}U_\alpha$ in $X$, which implies that the family $(U_\alpha)_{\alpha\in\lambda}$ is not locally finite. On the other hand, this family is compact-finite (because all compact subsets of $X$ are finite). So,  $(U_\alpha)_{\alpha\in\lambda}$ is a $\Clop$-fan, which is a desired contradiction.
\smallskip

The implication $(1)\Ra(4)$ is trivial. To prove $(4)\Ra(3)$, assume that $X$ is $k$-discrete and contains no strict $\Cld^\w$-fans and no $\Clop$-fans. By Proposition~\ref{P-space}, $X$ is a $P$-space. Observe that each functionally open subset of the space $X$ is of type $F_\sigma$ and hence is closed in $X$. Then $X$ is a zero-dimensional space.
\end{proof}

\section{Characterizing $k_\IR$-spaces and scattered sequential spaces}

In this section we shall apply strict $\Cld^\w$-fans to characterize $k_\IR$-spaces. We recall that a topological space $X$ is called a \index{topological space!$k_\IR$-space}{\em $k_\IR$-space} if each $k$-continuous function $f:X\to\IR$ is continuous.

\begin{proposition}\label{kR-no-Cld-fan}A $k_\IR$-space $X$ contains no strict $\Cld$-fan.
\end{proposition}

\begin{proof} To derive a contradiction, assume that the $k_\IR$-space $X$ contains a strict $\Cld$-fan $(F_\alpha)_{\alpha\in \lambda}$. This fan is not locally finite at some point $x\in X$.
 Since the family $(F_\alpha)_{\alpha\in \lambda}$ is strictly compact-finite, each set $F_\alpha$ has a functional neighborhood $U_\alpha\subset X$ such that the family $(U_\alpha)_{\alpha\in \alpha}$ is compact-finite.
Then the set $\Lambda=\{\alpha\in \lambda:x\in U_\alpha\}$ is finite and hence the family $(F_\alpha)_{\alpha\in \lambda\setminus\Lambda}$ is not locally finite at $x$.

 For every $\alpha\in \lambda\setminus\Lambda$ fix a continuous function $f_\alpha:X\to[0,1]$ such that $f_{\alpha}(F_\alpha)\subset \{1\}$ and $f_\alpha(X\setminus U_\alpha)\subset\{0\}$. Consider the function $f=\sum_{\alpha\in \alpha\setminus\Lambda}f_\alpha:X\to\IR$ and observe that it is well-defined and $k$-continuous since for each compact set $K\subset X$ all but finitely many functions $f_\alpha|K$ are zero. Since $X$ is a $k_\IR$-space, the function $f$ is continuous. In particular, it continuous at $x$. So, we can find a neighborhood $O_x\subset X$ such that $f(O_x)\subset [0,1)$. Since the family $(F_\alpha)_{\alpha\in \alpha\setminus\Lambda}$ is not locally finite at $x$, the neighborhood $O_x$ has a common point $z$ with some set $F_\alpha$, $\alpha\in \alpha\setminus\Lambda$. At this point we get $f(z)\ge f_\alpha(z)=1$, which contradicts the choice of the neighborhood $O_x$.
 \end{proof}

In scattered perfectly normal spaces the $k_\IR$-space property is  equivalent to the absence of strict $\Cld^\w$-fans. We recall that a space $X$ is called \index{topological space!prefectly normal}{\em perfectly normal} if $X$ is normal and each closed subset $F$ of $X$ is of type $G_\delta$.

We define a topological space $X$ to be \index{topological space!scatteredly-$k_\IR$}{\em scatteredly-$k_\IR$} if for any $k$-continuous function $f:X\to \IR$ each non-empty closed subset $F\subset X$ contains a non-empty relatively open subset $U\subset F$ such that the restriction $f|U$ is continuous. It is easy to see that each scattered space is scatteredly-$k_\IR$. We recall that a topological space $X$ is \index{topological space!scattered}{\em scattered} if each non-empty subspace $A\subset X$ contains an isolated point.

\begin{proposition} A perfectly normal space $X$ is a $k_\IR$-space if and only if $X$ is scatteredly-$k_\IR$ and $X$ contains no strict $\Cld^\w$-fan.
\end{proposition}

\begin{proof} The ``only if'' part follows from Proposition~\ref{kR-no-Cld-fan}. To prove the ``if'' part, assume that
$X$ is scatteredly-$k_\IR$  and $X$ contains no strict $\Cld^\w$-fan. Assuming that $X$ is not $k_\IR$, find a bounded discontinuous $k$-continuous function $f:X\to \IR$. Let $U\subset X$ be the maximal open set such that $f|U$ is continuous. The scattered-$k_\IR$ property of $X$ guarantees that $U$ is dense in $X$. Since $f$ is discontinuous, the closed set $F=X\setminus U$ is not empty and hence contains a relatively open dense subset $V\subset F$ such that the restriction $f|V$ is continuous. Fix any open set $\widetilde V\subset X$ such that $\widetilde V \cap F=V$. By the maximality of the set $U$, the set $\widetilde V$ contains a discontinuity point $a\in X$ of the function $f$ and this point belongs to $\tilde V\cap F=V$. Using the complete regularity of $X$, we can find a continuous function
$\lambda:X\to [0,1]$ such that $\lambda^{-1}(1)$ is a neighborhood of $a$ in $X$ and $\cl_X\big(\lambda^{-1}\big((0,1]\big)\big)\subset \widetilde V$. Then the function $\lambda f:X\to \IR$, $\lambda f:x\mapsto \lambda(x)\cdot f(x)$, is discontinuous at $a$ and $\lambda f|F$ is continuous.

By Tietze-Urysohn Theorem \cite[2.1.8]{En}, there exists a bounded continuous function $g:X\to \IR$ such that $g|F=\lambda f|F$. Then the function $h=f-g$ is $k$-continuous, discontinuous at $a$, continuous on the set $U$ and is equal to zero on $F$. Since $h$ is discontinuous at $a$, we can find a positive $\e>0$ such that the set $A=\{x\in X:|h(x)|\ge\e\}$ contains the point $a$ in its closure.

Using the perfect normality of $X$, choose a sequence of open sets $(W_n)_{n\in\w}$ in $X$ such that $W_{-1}=W_0=X$, $F=\bigcap_{n\in\w}W_n$ and $\overline{W}_{n}\subset W_{n-1}$ for all $n\in\IN$. For every $n\in\IN$ consider the set $A_n=A\cap (W_n\setminus W_{n+1})$ and observe that the set $B_n=\{x\in W_{n-1}\setminus W_{n+2}:h(x)>\frac\e2\}$ is a functional neighborhood of $A_n$ in $X$. Since $a\in \bar A$, the family $(A_n)_{n\in\w}$ is not locally finite at $a$. On the other hand, for any compact subset $K\subset X$, the continuity of the function $h|K$, the equality $h(a)=0$, and the convergence $W_n\cap K\to K\cap F\subset h^{-1}(0)$ imply that the compact set $K$ meets at most finitely many sets $B_n$. This means that the family $(B_n)_{n\in\w}$ is compact-finite and hence $(A_n)_{n\in\w}$ is a strict $\Cld^\w$-fan in $X$. This is a desired contradiction witnessing that $X$ is a $k_\IR$-space.
\end{proof}

Next, we characterize the sequentiality of scattered spaces using  strict $\Fin$-fans.

\begin{proposition} A scattered regular space  $X$ is sequential if and only if each closed subspace of $X$ contains no strict $\Fin$-fan.
\end{proposition}

\begin{proof} The ``only if'' part follows from Proposition~\ref{k-no-Cld-fan}.
To prove the ``if'' part, assume that a scattered space $X$ is not sequential. Then $X$ contains a non-closed subset $A\subset X$ such that for every compact countable set $K\subset X$ the intersection $K\cap A$ is compact.
Since $X$ is scattered, the remainder $\bar A\setminus A$ contains an isolated point $a$. This point has a neighborhood $O_a\subset X$ such that $O_a\cap(\bar A\setminus A)=\{a\}$. This means that $a$ is a unique non-isolated point of the intersection $\bar A\cap O_a$. By the regularity of $X$, the point $x$ has a closed neighborhood $U_a\subset X$ such that $U_a\subset O_a$. Then $B=U_a\cap \bar A=\{a\}\cup(U_a\cap A)$ is a closed subset of $X$ and $a$ is a unique non-isolated point of $B$.

We claim that each compact subset $K\subset B$ is finite. Assume conversely that $B$ contains an infinite compact set $K$. Since the space $B\setminus\{a\}$ is discrete, the point $a$ belongs to $K$ and is non-isolated in $K$. Replacing $K$ by any countable subset containing $a$, we can assume that $K$ is countable. Then by our assumption, $K\cap A=K\setminus\{a\}$ is closed in $K$ and hence $a$ is an isolated point of $K$. This contradiction shows that all compact subsets of $B$ are finite. It follows that $(\{b\})_{b\in B\setminus\{a\}}$ is a $\Fin$-fan in $\bar B$. Since each point $b\in B\setminus\{a\}$ is isolated in $\bar B$ this $\Fin$-fan is strict, which contradicts our assumption.
\end{proof}

\section{Two characterizations of spaces containing no strict $\Cld$-fans}

In this section we present two characterizations of spaces containing no strict $\Cld$-fans. The first of them implies that Ascoli spaces contain no strict $\Cld$-fan. We recall \cite{BG} that a topological space $X$ is \index{topological space!Ascoli}{\em Ascoli} if each compact subset $K\subset C_k(X)$ is evenly continuous, which means that the evaluation function $K\times X\to Y$, $(f,x)\mapsto f(x)$, is continuous.

Let $\lambda$ be an infinite cardinal. A subset $S$ of a topological space $X$ will be called a \index{subset!convergent $\lambda$-sequence}{\em convergent $\lambda$-sequence} if $|S|\le\lambda$ and the closure $\bar S$ in $X$ is compact and has a unique non-isolated point.

\begin{theorem}\label{Cld-fan<->sequence} Let $\lambda$ be an infinite cardinal. A topological space $X$ contains no strict $\Cld^\lambda$-fan if and only if every convergent $\lambda$-sequence in the function space $C_k(X)$ is evenly continuous.
\end{theorem}

\begin{proof} To prove the ``if'' part, assume that $X$ contains a strict $\Cld^\lambda$-fan $(X_\alpha)_{\alpha\in \lambda}$. For every $\alpha\in \lambda$ choose a functional neighborhood $U_\alpha$ of $X_\alpha$ in $X$ such that the family $(U_\alpha)_{\alpha\in \lambda}$ is compact-finite. Since $U_\alpha$ is a functional neighborhood of $X_\alpha$, there is a continuous function $f_\alpha:X\to [0,1]$ such that $f_\alpha(X_\alpha)\subset\{1\}$ and $f_\alpha(X\setminus U_\alpha)\subset\{0\}$. Denote by $\zeta:X\to\{0\}\subset\IR$ the zero function and observe that the set $K=\{\zeta\}\cup\{f_\alpha\}_{\alpha\in \lambda}$ is compact in the function space $C_k(X)$. Moreover, the zero function is a unique accumulation point of $K$. So, $K$ is a convergent $\lambda$-sequence in $C_k(X)$. Assuming that $K$ is evenly continuous, for every point $x\in X$ we would find a neighborhood $O_x\subset X$ and a neighborhood $W\subset K$ of $\zeta$ such that $\{f(z):f\in W,\;z\in O_x\}\subset(-1,1)$. Since $\zeta$ is the unique accumulation point of $K$, the set $\Lambda=\{\alpha\in \lambda:f_\alpha\notin W\}$ is finite. Then the set $\{\alpha\in \lambda:O_x\cap X_\alpha\ne\emptyset\}\subset \Lambda$ is finite. This means that the $\Cld$-fan $(X_\alpha)_{\alpha\in \lambda}$ is locally finite, which is a contradiction.
\smallskip

To prove the ``only if'' part, assume that the function space $C_k(X)$ contains a convergent $\lambda$-sequence $K=\{f_\alpha\}_{\alpha\in\lambda}$ which is not evenly continuous at some point $x\in X$.
Let $\zeta$ be the unique non-isolated point of the compact set $\bar K$. Replacing each function $f_\alpha$ by $f_\alpha-\zeta$, we can assume that $\zeta$ coincides with the zero function.

Since $K$ is not evenly continuous at $x$, there is a positive $\e$ such that for every neighborhood $O_x\subset X$ of $x$ and every neighborhood $O_\zeta\subset K$ of $\zeta$ we get $O_\zeta(O_x)\not\subset (-\e,\e)$. For every ordinal $\alpha\in\lambda$ consider the sets $X_\alpha=f_\alpha^{-1}(\IR\setminus (-\e,\e))$ and $W_\alpha=f_\alpha^{-1}(\IR\setminus [-\frac\e2,\frac\e2])$, and observe that $W_\alpha$ is a functional neighborhood of $X_\alpha$. The convergence $(f_\alpha)_{\alpha\in\lambda}\to\zeta=0$ in $C_k(X)$ implies that the family $(W_\alpha)_{\alpha\in\lambda}$ is compact-finite in $X$. On the other hand, the choice of $\e$ guarantees that the family $(X_\alpha)_{\alpha\in \lambda}$ is not
locally finite, which means that this family is a strict $\Cld^\lambda$-fan in $X$.
\end{proof}

\begin{corollary}\label{c:A->noFan} Ascoli spaces contain no strict $\Cld$-fans.
\end{corollary}

Another characterization of spaces without $\Cld^\lambda$-fans is analogous to a characterization of Ascoli spaces proved in \cite[5.10]{BG}.

\begin{proposition} Let $\lambda$ be a cardinal. A topological space $X$ contains no strict $\Cld^\lambda$-fan if and only if each point $x\in X$ is contained in a dense subspace of $X$ containing no strict $\Cld^\lambda$-fans.
\end{proposition}

\begin{proof}
The ``only if'' part is trivial. To prove the ``if'' part, assume that each point $x\in X$ is contained in a dense subspace of $X$ containing no strict $\Cld^\lambda$-fans. To prove that $X$ contains no strict $\Cld^\lambda$-fan, we shall apply the characterization Theorem~\ref{Cld-fan<->sequence}. Given a convergent $\lambda$-sequence $K\subset C_k(X,Y)$ we need to prove that $K$ is evenly continuous. Replacing $K$ by its closure, we can assume that $K$ is compact.

To prove that $K$ is evenly continuous, fix a function $f\in K$, a point $x\in X$, and a neighborhood $O_{f(x)}\subset \IR$ of $f(x)$. By the regularity of $\IR$, there is a neighborhood $W_{f(x)}\subset \IR$ of $f(x)$ such that $\overline{W_{f(x)}}\subset O_{f(x)}$. By our assumption, the point $x$ is contained in a dense subspace $Z\subset X$ containing no strict $\Cld^\lambda$-fan. By Theorem~\ref{Cld-fan<->sequence}, each convergent $\lambda$-sequence in the function space $C_k(Z)$ is evenly continuous. The density of $Z$ in $X$ implies that the restriction operator
$$
R:C_k(X)\to C_k(Z), \quad R:g\mapsto g|Z,
$$
is injective and hence the restriction $R|K:K\to R(K)\subset C_k(Z)$ is a homeomorphism. By the choice of $Z$,  the convergent $\lambda$-sequence $R(K)\subset C_k(Z)$ is evenly continuous. Consequently, for the function $h:=R(f)=f|Z$ and the neighborhood $W_{f(x)}$ of $h(x)=f(x)$ there are neighborhoods $U_h\subset R(K)$ of $h$ and $W_x\subset Z$ of $x$  such that $U_h(W_x)\subset W_{f(x)}$. It follows that $U_f :=\{g\in {K}:R(g)\in U_h\}$ is a neighborhood of $f$ in $K$ and the closure $\overline{W}_x$ of $W_x$ in $X$ is a (closed) neighborhood of $x$ in $X$ such that $U_f(\overline{W}_x)\subset \overline{W}_{f(x)}\subset O_{f(x)}$. Thus the compact set $K$ is evenly continuous and by Theorem~\ref{Cld-fan<->sequence}, the space $X$ contains no strict $\Cld^\lambda$-fans.
\end{proof}

A similar characterization of Ascoli spaces was proved in \cite{BG}.

\begin{theorem} A topological space $X$ is Ascoli if and only if each point $x\in X$ is contained in a dense Ascoli subspace of $X$.
\end{theorem}

An analogous characterization of $k_\IR$-spaces looks a bit differently.

\begin{theorem}\label{t:Ascoli-kY} A topological space $X$ is a $k_\IR$-space if and only if $X$ admits a cover $\A$ by $k_\IR$-subspaces such that for any sets $A,B\in\A$ the intersection $A\cap B$ is dense in $X$.
\end{theorem}

\begin{proof} The ``only if'' part is trivial. To prove the ``if'' part, assume that $\A$ is a cover of a topological space $X$ by $k_\IR$-subspaces such that for any sets $A,B\in\A$ the intersection $A\cap B$ is dense in $X$. To show that $X$ is a $k_\IR$-space, fix a $k$-continuous function $f:X\to \IR$. Then for every $A\in\A$ the restriction $f|A$ is continuous. To show that $f$ is continuous, fix any point $a\in X$ and a neighborhood $U_{f(a)}\subset Y$ of $f(a)$. By the regularity of $\IR$, there is an open neighborhood $V_{f(a)}\subset Y$ of $f(a)$ such that $\overline{V_{f(a)}}\subset U_{f(a)}$. Find a set $A\in\A$ containing the point $a$, and observe that the set $V_a :=A\cap f^{-1}(V_{f(a)})$ is an open neighborhood of $a$ in $A$. Choose any open set $\widetilde V_a\subset X$ such that $A\cap\widetilde V_a=V_a$. We claim that $f(\widetilde V_a)\subset \cl_Y(V_{f(a)})\subset U_{f(a)}$. Indeed, for any point $b\in\widetilde V_a$ we can find a set $B\in\A$ containing $b$. By our assumption the intersection $A\cap B$ is dense in $X$ and hence the set $C:=A\cap B\cap \widetilde V_a$ is dense in $B\cap\widetilde V_a$. The continuity of $f|B$ implies that
$$
f(b)\in f(B\cap\widetilde V_a)  \subset f(\overline{C})\subset \cl(f(C)) \subset \cl(f(A\cap\widetilde V_a))
 = \cl(f(A\cap V_a))\subset\cl(V_{f(a)})\subset U_{f(a)}.
$$
So, $f(\widetilde V_a)\subset U_{f(a)}$ and $X$ is a $k_\IR$-space.
\end{proof}

\section{Preservation of (strong) fans by some operations}

In this subsection we study the problem of preservation of (strong) fans by some operations over topological spaces. First observe that a fan in a subspace $Z$ of a topological space $X$ needs not remain a fan in $X$. On the other hand, any fan in a {\em closed} subspace  of a topological space $X$ remains a fan in $X$. For strong fans this preservation property does not hold.

\begin{example}\label{Cld-not-hereditary} There exists a pseudocompact space $Y$ and a closed subspace $X\subset Y$ such that
\begin{enumerate}
\item[\textup{1)}] $X$ contains a strict $\Fin^\w$-fan;
\item[\textup{2)}] $Y$ contains no strong $\Cld$-fans.
\end{enumerate}
\end{example}

\begin{proof} Take any Tychonoff space $X$ containing a strict $\Fin^\w$-fan (for example, take any countable $k$-discrete space $X$ with a unique non-isolated point). Embed $X$ into the Tychonoff cube $[0,\frac12]^\lambda$ for some uncountable cardinal $\lambda$. By standard methods (see \cite[3.12.21]{En}), it can be shown that the complement $P=\II^\lambda\setminus[0,\frac12]^\lambda$ is pseudocompact. Then the space $Y=\II^\lambda\setminus (\bar X\setminus X)$  is pseudocompact and contains $X$ as a closed subspace. It remains to prove that $Y$ contains no strong $\Cld$-fans. For this we shall prove that each strongly compact-finite family $(Y_\alpha)_{\alpha\in A}$ of non-empty subsets of $Y$ is finite. Each set $Y_\alpha$ has an open neighborhood $U_\alpha\subset Y$ such that the family $(U_\alpha)_{\alpha\in A}$ is compact-finite in $Y$. Since the sets $\bar X\subset[0,\frac12]^\lambda$ are nowhere dense in $\II^\lambda$,  each open set $U_\alpha$ contains a non-empty open set $U_\alpha\subset\II^\lambda$ such that $\bar V_{\alpha}\cap[0,\frac12]^\lambda=\emptyset$. It follows that the compact-finite family $(V_\alpha)_{\alpha\in A}$ of open subsets of the locally compact space $P=\II^\lambda\setminus[0,\frac12]^\lambda$
is locally finite and by the pseudocompactness of $P$, is finite. Therefore, the index set $A$ is finite and so is the family $(X_\alpha)_{\alpha\in A}$
\end{proof}

\begin{remark} Under some conditions (in particular, those given in Corollary~\ref{c:s-sfan1}) strong fans in closed subspaces remain strong fans in ambient spaces.
\end{remark}

Next, we study the problem of preservation of (strong) fans by quotient maps.
 Let us recall that a map $f:X\to Y$ between topological spaces is \index{map!quotient}{\em quotient} if a subset $U\subset Y$ is open if and only if its preimage $f^{-1}(U)$ is open in $X$.

\begin{proposition}\label{p:quot-fan} Let $f:X\to Y$ be a quotient map between topological spaces. If $(F_\alpha)_{\alpha\in\lambda}$ is a (strict, strong) $\Cld$-fan in $Y$, then $(f^{-1}(F_\alpha))_{\alpha\in\lambda}$ is a (strict, strong) $\Cld$-fan in $X$.
\end{proposition}

\begin{proof} The continuity of $f$ implies that the family $(f^{-1}(F_\alpha))_{\alpha\in\lambda}$ is (strictly, strongly) compact-finite in $X$. It remains to prove that it is not locally finite in $X$.

To derive a contradiction, assume that the family $(f^{-1}(F_\alpha))_{\alpha\in\lambda}$ is locally finite in $X$. The family $(F_\alpha)_{\alpha\in\lambda}$, being a fan in $Y$, is not locally finite at some point $y\in Y$. By the compactness of the singleton $\{y\}$, the set $\Lambda=\{\alpha\in\lambda:y\in F_\alpha\}$ is finite. For every $\alpha\in\lambda\setminus\Lambda$ the set $F_\alpha$ does not contain the point $y$ and consequently, the closed set $f^{-1}(F_\alpha)$ is disjoint with $f^{-1}(y)$. The union $F=\bigcup_{\alpha\setminus\Lambda}f^{-1}(Y_\alpha)$ of the locally finite family $(f^{-1}(Y_\alpha))_{\alpha\in\lambda}$ of closed subsets of $X$ is closed in $X$, which implies that $X\setminus F$ is an open neighborhood of $f^{-1}(y)$. Since $X\setminus F=f^{-1}(X\setminus\bigcup_{\alpha\in\lambda\setminus\Lambda}Y_\alpha)=f^{-1}(f(X\setminus Z))$, the set $f(X\setminus F)$ is an open neighborhood of $y$, disjoint with the sets $Y_\alpha$, $\alpha\in\lambda\setminus\Lambda$, and witnessing that the family $(Y_\alpha)_{\alpha\in\lambda}$ is locally finite at the point $y$. But this contradicts the choice of the point $y$.
 \end{proof}

\begin{corollary} Let $f:X\to Y$ be a quotient map. If for some cardinal $\lambda$ the space $X$ contains no $\Cld^\lambda$-fan, then the space $Y$ also contains no $\Cld^\lambda$-fan.
\end{corollary}

For strict fans quotient maps in Proposition~\ref{p:quot-fan} can be replaced by  pseudo-clopen maps.

A map $f:X\to Y$ between topological spaces is called
\begin{itemize}
\item \index{map!pseudo-open}{\em pseudo-open} if for each point $y\in Y$ and a neighborhood $U\subset X$ of $f^{-1}(y)$ the image $f(U)$ is a neighborhood of $y$;
\item \index{map!pseudo-clopen}{\em pseudo-clopen} if a closed subset $F\subset Y$ is a neighborhood of a point $y\in F$ in $Y$ if $f^{-1}(F)$ is a neighborhood of $f^{-1}(y)$ in $X$.
\end{itemize}
It is clear that each open map and each closed map is pseudo-open and each pseudo-open map is pseudo-clopen and quotient. Pseudo-open maps are well-known in General Topology \cite[2.4.F]{En}. The notion of a pseudo-clopen map seems to be new.

\begin{proposition} Let $f:X\to Y$ be a pseudo-clopen map between topological spaces. If for some cardinal $\lambda$ the space $X$ contains no strict $\Cld^\lambda$-fan, then the space $Y$ also contains no strict $\Cld^\lambda$-fans.
\end{proposition}

\begin{proof} To derive a contradiction, assume that $(F_\alpha)_{\alpha\in\lambda}$ is a strict $\Cld^\lambda$-fan in $Y$. Then each set $F_\alpha$ has a functional neighborhood $U_\alpha\subset Y$ such that the family $(U_\alpha)_{\alpha\in\lambda}$ is compact-finite in $Y$.
Let $y\in Y$ be a point in which the fan $(F_\alpha)_{\alpha\in\lambda}$ is not locally finite. It follows that the set $\Lambda=\{\alpha\in \lambda:x\in U_\alpha\}$ is finite. Replacing the fan $(F_\alpha)_{\alpha\in\lambda}$ by the fan $(F_\alpha)_{\alpha\setminus \Lambda}$, we can assume that $\Lambda=\emptyset$ and hence $y\notin\bigcup_{\alpha\in\lambda}U_\alpha$.

For every $\alpha\in\lambda$ choose a continuous function $\xi_\alpha:Y\to[0,1]$ such that $\xi_\alpha(F_\alpha)\subset\{1\}$ and $\xi_\alpha(Y\setminus U_\alpha)\subset \{0\}$, and consider the open set  $V_\alpha=\xi^{-1}_\alpha((\frac12,1])$ and its closure $\bar V_\alpha$ in $Y$. The inclusion $\bar V_\alpha\subset\xi_\alpha^{-1}([\frac12,1])$ implies that $U_\alpha$ is a functional neighborhood of $\bar V_\alpha$ in $Y$.
Then the family $(\bar V_\alpha)_{\alpha\in\lambda}$ is strictly compact-finite in $Y$ and the family $(f^{-1}(\bar V_\alpha))_{\alpha\in\lambda}$ is strictly compact-finite in $X$. Since $X$ contains no strict $\Cld^\lambda$-fans, the latter family is locally finite in $X$, which implies that the set $\bigcup_{\alpha\in\lambda}f^{-1}(\bar V_\alpha)$ is closed in $X$ and $W=X\setminus \bigcup_{\alpha\in\lambda}f^{-1}(\bar V_\alpha)$ is a neighborhood of $f^{-1}(y)$ in $X$. Now consider the closed set $E=Y\setminus\bigcup_{\alpha\in\lambda}V_\alpha$ in $X$ and observe that
$f^{-1}(E)=X\setminus\bigcup_{\alpha\in\lambda}f^{-1}(V_\alpha) \supset W$ is a neighborhood of $f^{-1}(y)$. Since the map $f$ is pseudo-clopen, the closed set $E$ is a neighborhood of $y$, disjoint with $\bigcup_{\alpha\in\lambda}F_\alpha$, which contradicts the choice of $y$.
\end{proof}

An analogous result holds for Ascoli spaces, too.

\begin{proposition}\label{p:As-pso} The image $Y$ of an Ascoli space $X$ under a continuous pseudo-clopen map $\varphi:X\to Y$ is Ascoli.
\end{proposition}

\begin{proof} To prove that the space $Y$ is Ascoli, we need to check that each compact subset $K\subset C_k(Y)$ is evenly continuous. Fix a function $f\in K$, a point $y\in Y$ and an open  neighborhood $O_{f(y)}\subset \IR$ of $f(y)$. Choose two neighborhoods $V_{f(y)}\subset U_{f(y)}$ of $f(y)$ such that $\overline{V}_{f(y)}\subset U_{f(y)}\subset\overline{U}_{f(y)}\subset O_{f(y)}$. It follows that the set $K_f=\{g\in K:g(y)\in \overline{V}_{f(y)}\}$ is a closed neighborhood of $y$ in $K$. Consider the closed set $N_y=\bigcap_{g\in K_f}g^{-1}(\overline{U}_{f(y)})$ in $Y$ and observe that $K_f(N_y)\subset \overline{U}_{f(y)}\subset O_{f(y)}$. We claim that $N_y$ is a closed neighborhood of $y$ in $Y$.

For this consider the preimage $\varphi^{-1}(N_y)$ and observe that $$\varphi^{-1}(N_y)=\bigcap_{g\in K_f}(g\circ\varphi)^{-1}(\overline{U}_{f(y)}).$$ The continuity of the dual map $\varphi^*:C_k(Y)\to C_k(X)$, $\varphi^*:g\mapsto g\circ\varphi$, implies that the set $K_f^*=\varphi^*(K_f)$ is compact. Observe that for every point $x\in\varphi^{-1}(y)$ we get $K^*_f(x)\subset K_f(y)\subset \overline{V}_{f(y)}\subset U_{f(y)}$. Since $X$ Ascoli, each point $x\in  \varphi^{-1}(x)$ has a neighborhood $O_x\subset X$ such that $K_f^*(O_x)\subset U_{f(x)}$. Then $O_x\subset \varphi^{-1}(N_y)$, which implies that $\varphi^{-1}(N_y)$ is a neighborhood of $\varphi^{-1}(y)$ in $X$. Since the map $\varphi$ is pseudo-clopen, $N_y$ is a neighborhood of $y$ in $Y$.
\end{proof}

\begin{remark} Proposition~\ref{p:As-pso} improves Proposition~3.3 \cite{GKP}, which says that the image of an Ascoli space under a surjective open continuous map is Ascoli.
\end{remark}

\begin{problem}\label{prob:Ascoli-quotient} Is the quotient image of an Ascoli space Ascoli?
\end{problem}

\section{Some counterexamples}

As we already know, for any topological space and
each infinite cardinal $\lambda$
the following implications hold:
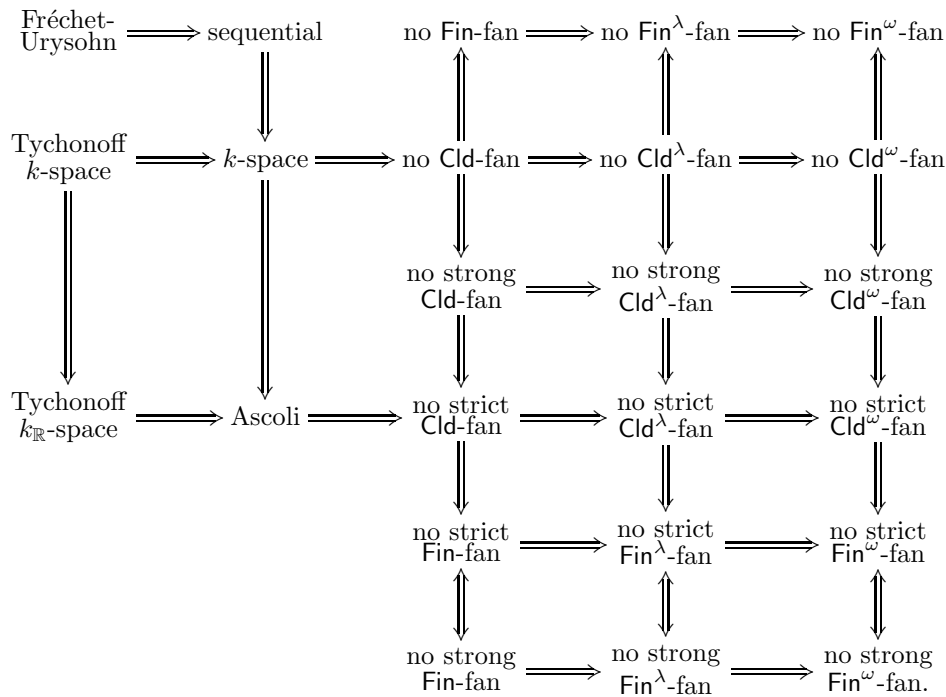
{
$$
\xymatrix{
\genfrac{}{}{0pt}{}{\mbox{Fr\'echet-}}{\mbox{Urysohn}}\ar@{=>}[r]&\mbox{sequential}\ar@{=>}[d]&\mbox{no $\Fin$-fan}\ar@{=>}[r]
&\mbox{no $\Fin^\lambda$-fan}\ar@{=>}[r]&
\mbox{no $\Fin^\w$-fan}\\
\genfrac{}{}{0pt}{}{\mbox{Tychonoff}}{\mbox{$k$-space}}
\ar@{=>}[r]\ar@{=>}[dd]&\mbox{$k$-space}\ar@{=>}[dd]\ar@{=>}[r]
&\mbox{no $\Cld$-fan}\ar@{=>}[r]\ar@{=>}[d]\ar@{=>}[u]
&\mbox{no $\Cld^\lambda$-fan}\ar@{=>}[r]\ar@{=>}[d]\ar@{=>}[u]&
\mbox{no $\Cld^\w$-fan}\ar@{=>}[d]\ar@{=>}[u]\\
&&
\genfrac{}{}{0pt}{}{\mbox{no strong}}{\mbox{$\Cld$-fan}}
\ar@{=>}[d]\ar@{=>}[r]&
\genfrac{}{}{0pt}{}{\mbox{no strong}}{\mbox{$\Cld^\lambda$-fan}}
\ar@{=>}[r]\ar@{=>}[d]&
\genfrac{}{}{0pt}{}{\mbox{no strong}}{\mbox{$\Cld^\w$-fan}}
\ar@{=>}[d]\\
\genfrac{}{}{0pt}{}{\mbox{Tychonoff}}{\mbox{$k_\IR$-space}}
\ar@{=>}[r]&\mbox{Ascoli}\ar@{=>}[r]&
\genfrac{}{}{0pt}{}{\mbox{no strict}}{\mbox{$\Cld$-fan}}
\ar@{=>}[r]\ar@{=>}[d]&
\genfrac{}{}{0pt}{}{\mbox{no strict}}{\mbox{$\Cld^\lambda$-fan}}
\ar@{=>}[r]\ar@{=>}[d]&
\genfrac{}{}{0pt}{}{\mbox{no strict}}{\mbox{$\Cld^\w$-fan}}
\ar@{=>}[d]\\
&&
\genfrac{}{}{0pt}{}{\mbox{no strict}}{\mbox{$\Fin$-fan}}
\ar@{=>}[r]\ar@{<=>}[d]&
\genfrac{}{}{0pt}{}{\mbox{no strict}}{\mbox{$\Fin^\lambda$-fan}}
\ar@{=>}[r]\ar@{<=>}[d]&
\genfrac{}{}{0pt}{}{\mbox{no strict}}{\mbox{$\Fin^\w$-fan}}
\ar@{<=>}[d]\\
&&
\genfrac{}{}{0pt}{}{\mbox{no strong}}{\mbox{$\Fin$-fan}}
\ar@{=>}[r]&
\genfrac{}{}{0pt}{}{\mbox{no strong}}{\mbox{$\Fin^\lambda$-fan}}
\ar@{=>}[r]&
\genfrac{}{}{0pt}{}{\mbox{no strong}}{\mbox{$\Fin^\w$-fan.}}
}
$$
}

In this section we present some examples showing that some of the implications in this diagram cannot be reversed.

\begin{example} For any uncountable cardinal $\lambda$ the power $\IR^\lambda$ is a $k_\IR$-space but not a $k$-space  \textup{(see \cite[Theorem 5.6]{Nob70} and \cite[Problem 7.J(b)]{Kelley})}.
\end{example}

An example of an Ascoli space which fails to be a $k_\IR$-space was constructed in \cite{BG}:

\begin{example} \label{exa:Ascoli-Group-non-kR} For infinite cardinals $\lambda<\kappa$ and a first countable regular space $Y$ the subspace $X=\{f\in Y^\kappa:\exists y\in Y\;|f^{-1}(Y\setminus\{y\})|<\lambda\}\subset Y^\kappa$ is Ascoli but not a $k_\IR$-space.
\end{example}

\begin{example} For any uncountable cardinal $\lambda$ there is a topological space $X$ of cardinality $|X|=\lambda$ with a unique non-isolated point such that
\begin{enumerate}
\item[1)] $X$ has no $\Cld^\kappa$-fans for all cardinals $\kappa<\cf(\lambda)$;
\item[2)] $X$ contains a strict $\Fin^\lambda$-fan.
\end{enumerate}
\end{example}

\begin{proof} Take any point $\infty\notin\lambda$ and on the space $X=\{\infty\}\cup\lambda$ consider the regular topology in which points $x\in\lambda$ are isolated and a subset $U\subset X$ is a neighborhood of $\infty$ if and only if $X\setminus U$ is a subset of cardinality $<\kappa$ in $\lambda$. It can be shown that the space $X$ is $k$-discrete, which implies that the family of singletons $(\{x\})_{x\in\lambda}$ is a strict $\Fin^\lambda$-fan in $X$.

Next, we show that the space $X$ has no $\Cld^\kappa$-fan for every cardinal $\kappa<\cf(\lambda)$. It suffices to check that  compact-finite family $(F_\alpha)_{\alpha\in\kappa}$ of closed subsets of $X$ is locally finite at the unique non-isolated point $\infty$ of $X$. Since $(F_\alpha)_{\alpha\in\kappa}$ is compact-finite, the set $\Lambda=\{\alpha\in\kappa:\infty\in F_\alpha\}$ is finite. For every $\alpha\in\kappa\setminus \Lambda$ the definition of the topology of $X$ implies that $|F_\alpha|<\lambda$ and then the union $F=\bigcup_{\alpha\in\kappa\setminus\Lambda}F_\alpha$ has cardinality $|F|<\lambda$ and hence is closed in $X$. Then $X\setminus F$ is a neighborhood of $\infty$ meeting finitely many sets $F_\alpha$, $\alpha\in\kappa$.
\end{proof}

By a $\Comp$-fan in a topological space $X$ we understand a fan $(F_\alpha)_{\alpha\in\lambda}$ consisting of compact subsets $F_\alpha$ of $X$.

\begin{example} There exists a cosmic hemicompact space, which contains a strict $\Comp^\w$-fan but does not contain $\Fin$-fans.
\end{example}

\begin{proof} Let $\II$ denote the unit interval $[0,1]$. Take any point $\infty\notin\II\times\IN$ and on the space $X=\{\infty\}\cup(\II\times \IN)$ consider the topology which coincides with the product topology at each point of the set $\II\times\IN$ and at the point $\infty$ is generated by the neighborhood base consisting of the sets $U\subset X$ such that the intersection $U\cap (\II\times\IN)$ is open in $\II\times\IN$ and $\lim_{n\to\infty}\lambda_n(U\cap (\II\times\{n\})=1$ where $\lambda_n$ denotes the Lebesgue measure on the interval $\II\times\{n\}$. It can be shown that the space $X$ is cosmic and contains no sequences convergent to $\infty$. This implies that each compact subset of $X$ is contained in the compact set $\{\infty\}\cup(\II\cup\{1,\dots,n\})$ for some $n\in\IN$. So, the space $X$ is hemicompact.

Observe that the family $(\II\times\{n\})_{n\in\IN}$ is a strict $\Comp^\w$-fan in $X$.

It remains to prove that $X$ contains no $\Fin$-fans. Since $X$ has countable tightness, it suffices to check that $X$ contains no $\Fin^\w$-fans. To derive a contradiction, assume that $(F_n)_{n\in\w}$ is a $\Fin^\w$-fan in $X$. We loss no generality assuming that the union $F=\bigcup_{n\in\w}F_n$ does not contain the point $\infty$. Taking into account that the fan is compact-finite, we conclude that for every $n\in\IN$ the set $(\II\times\{1,\dots,n\})\cap F$ is finite and hence $X\setminus F$ is an open neighborhood of the point $\infty$ by the definition of the topology at $\infty$. So the fan $(F_n)_{n\in\w}$ is locally finite at $\infty$. The local compactness of the space $X\setminus\{\infty\}$ implies that the fan is locally finite at each point $x\ne\infty$. So, the fan $(F_n)_{n\in\w}$ is locally finite in $X$, which is a desired contradiction.
\end{proof}

Two questions related to the above diagram remain open.

\begin{problem} Assume that a Tychonoff space contains no $\Cld$-fan. Is $X$ a $k$-space?
\end{problem}

\begin{problem} Assume that a Tychonoff space contains no strict $\Cld$-fan. Is $X$ Ascoli?
\end{problem}

\chapter{$S$-Fans, $S$-semifans and $D$-cofans in topological spaces}

In this chapter we shall describe three tools (called $S$-fans, $S$-semifans and $D$-cofans) which will be used for constructing (strong) $\Fin$-fans in topological spaces and topological groups.

\section{$S$-fans}

We recall that a \index{convergent sequence}\index{subset!convergent sequence} {\em convergent sequence} in a topological space $X$ is a countable subset $S\subset X$ whose closure $\bar S$ is compact and has a unique non-isolated point $\lim S$ called the \index{convergent sequence!limit of}{\em limit} of $S$. 

An \index{fan!$S$-fan}\index{$S$-fan}{\em $S$-fan} (more precisely, an \index{fan!$S^\lambda$-fan}\index{$S$-fan!$S^\lambda$-fan}{\em $S^\lambda$-fan}) in a topological space $X$ is a fan $(S_\alpha)_{\alpha\in\lambda}$ consisting of convergent sequences. So, the family $(S_\alpha)_{\alpha\in\lambda}$ is compact-finite but not locally finite in $X$. If this family is strongly compact-finite, then the $S$-fan is called \index{fan!strong $S$-fan}\index{$S$-fan!strong $S$-fan}{\em strong}.

An $S$-fan $(S_\alpha)_{\alpha\in\lambda}$ will be called a \index{fan!$\bar S$-fan}\index{$S$-fan!$\bar S$-fan}{\em $\bar S$-fan} if the set $\{\lim S_\alpha\}_{\alpha\in\lambda}$ has compact closure $K$ in $X$. If this closure is metrizable (resp.  a singleton), then the $S$-fan will be called a \index{fan!$\ddot S$-fan}\index{$S$-fan!$\ddot S$-fan}{\em $\ddot S$-fan} (resp. a \index{fan!$\dot S$-fan}\index{$S$-fan!$\dot S$-fan}{\em $\dot S$-fan}). So, a $\dot S$-fan consists of sequences converging to the same point, and $\ddot S$-fan consists of sequences that converge to points of some compact metrizable set. For each
$\ddot S$-fan $(S_\alpha)_{\alpha\in\lambda}$ there exists a countable  subset $\Lambda\subset\lambda$ such that the set $\{\lim S_\alpha\}_{\alpha\in\Lambda}$ is either a singleton or a convergent sequence.


It is clear that
$$
\xymatrix
{
\mbox{$\dot S$-fan}\ar@{=>}[r]&\mbox{$\ddot S$-fan}\ar@{=>}[r]&\mbox{$\bar S$-fan}\ar@{=>}[r]&\mbox{ $S$-fan}\\
\mbox{strong $\dot S$-fan}\ar@{=>}[r]\ar@{=>}[u]&\mbox{strong $\ddot S$-fan}\ar@{=>}[r]\ar@{=>}[u]&\mbox{strong $\bar S$-fan}\ar@{=>}[r]\ar@{=>}[u]&\mbox{strong $S$-fan.}\ar@{=>}[u]
}$$


By Corollary~\ref{c:fan-in-aleph-space}, each $S^\w$-fan in a Tychonoff $\aleph$-space $X$ is a strong $S^\w$-fan in $X$. If each compact subset of a topological space $X$ is metrizable, then each $\bar S$-fan in $X$ is a $\ddot S$-fan.

Now we define two test spaces containing $\ddot S^\w$-fans. One of them is the well-known \index{fan!Fr\'echet-Urysohn fan}{\em Frech\'et-Urysohn fan}
$$V=\{(0,0)\}\cup\{(\tfrac1n,\tfrac1{nm}):n,m\in\IN\}$$endowed with the strongest topology inducing the Euclidean topology on each convergent sequence $V_n=\{(0,0)\}\cup \{(\tfrac1m,\tfrac1{nm})\}_{m\in\IN}$, $n\in\IN$.
It can be shown that $V$ is a $k_\w$-space and $(V_n)_{n\in\IN}$ is a strong $\dot S^\w$-fan in $V$.

The other test space is the \index{fan!Arens' fan}{\em Arens' fan}
$$A=\{(0,0)\}\cup\{(\tfrac1n,0):n\in\IN\}\cup\{(\tfrac1n,\tfrac1{nm}):n,m\in\IN\}$$endowed with the strongest topology inducing the Euclidean topology on the convergent sequences $A_0=\{(0,0)\}\cup\{(\frac1n,0)\}_{n\in\w}$ and $A_n=\{(\frac1n,0)\}\cup \{(\tfrac1n,\tfrac1{nm})\}_{m\in\IN}$. It is well-known that $A$ is a (sequential) $k_\w$-space, which is not Fr\'echet-Urysohn. It can be shown that the family $(A_n)_{n\in\w}$ is a strong $\ddot S^\w$-fan in the space $A$.
The Fr\'echet-Urysohn and Arens' fans are important spaces that have many applications in the theory of generalized metric spaces, see \cite{GTZ}, \cite{Lin}, \cite{NT}, \cite{NST}, \cite{Tan89}.

\begin{proposition}\label{p:FU-fan} A topological space $X$ contains a $\dot S^\w$-fan if and only if $X$ contains a $k$-closed subset $F\subset X$ which is $k$-homeomorphic to the Fr\'echet-Urysohn fan $V$.
\end{proposition}

\begin{proof} To prove the ``if'' part, assume that $F\subset X$ is a $k$-closed subset of $X$, $k$-homeo\-morphic to the Fr\'echet-Urysohn fan $V$. We recall that $V=\{(0,0)\}\cup\bigcup_{n\in\IN}V_n$ where $V_n=\{(\frac1m,\frac1{nm})\}_{m\in\IN}$, $n\in\IN$, are sequences convergent to the point $v=(0,0)$ in $V$. Fix a $k$-homeomorphism $h:V\to F$ and observe that $(h(V_n))_{n\in\IN}$ is an $\dot S^\w$-fan in $X$. Indeed, the $k$-continuity of $h$ ensures that each set $h(S_n)$ is a convergent sequence in $X$ accumulating at the point $h(v)$. It remains to check that the family $(h(V_n))_{n\in\IN}$ is compact-finite in $X$. As $F$ is $k$-closed in $X$, for every compact subset $K\subset X$ the intersection $K\cap F$ is compact and by the $k$-continuity of $h^{-1}$ the preimage $h^{-1}(K\cap F)$ is a compact subset of the Fr\'echet-Urysohn fan $V$. The choice of the topology on $V$ ensures that $h^{-1}(K\cap F)$ meets at most finitely many convergent sequences $V_n$, which implies that $K$ meets only finitely many convergent sequences $h(V_n)$. But this means that the family $(h(V_n))_{n\in\IN}$ is compact-finite and hence is an $\dot S^\w$-fan in $X$.
\smallskip

To prove the ``only if'' part, assume that the space $X$ contains an $\dot S^\w$-fan $(S_n)_{n\in\w}$.  Replacing $(S_n)_{n\in\w}$ by a suitable subsequence we can assume that the family $(S_n)_{n\in\IN}$ is disjoint. Let $x$ be the common limit point of the convergent sequences $S_n$, $n\in\IN$, and for every $n\in\IN$ let $\{x_{n,m}\}_{m\in\IN}$ be an injective enumeration of the set $S_n\setminus \{x\}$. Define a continuous bijective map $h:V\to X$ letting $h(0,0)=x$ and $h(\frac1m,\frac1{nm})=x_{n,m}$ for all $n,m\in\IN$.

We claim that the set $F=h(V)$ is $k$-closed in $X$ and the map $h:V\to F$ is a $k$-homeomorphism. Taking into account that the family $(S_n)_{n\in\w}$ is compact-finite, for every compact subset $K\subset X$ we can find a number $k\in\IN$ such that $K\cap F=\bigcup_{n=1}^kK\cap \bar S_n$. Since $S_n$ is a convergent sequence, the set $K\cap\bar S_n$ is compact and so is the set $K\cap F$. Moreover, if $K\cap\bar S_n$ is infinite, then it is a sequence convergent to $x$, which implies that $h^{-1}|F\cap K$ is continuous. Hence  $h^{-1}:F\to V$ is $k$-continuous and $h:V\to F$ is a $k$-homeomorphism of the Fr\'echet-Urysohn fan $V$ onto the $k$-closed set $F$ in $X$.
\end{proof}


\begin{proposition}\label{p:FUAr} A topological space $X$ contains a $\ddot S^\w$-fan if and only if $X$ contains a $k$-closed subset $F\subset X$ which is $k$-homeomorphic to the Fr\'echet-Urysohn fan $V$ or to the Arens fan $A$.
\end{proposition}

\begin{proof} The ``if'' part can be proved by analogy with the proof of the ``if'' part in Proposition~\ref{p:FU-fan}. To prove the ``only if'' part, assume that $(S_n)_{n\in\w}$ is a $\ddot S^\w$-fan in $X$. Taking into account that the set $\{\lim S_n\}_{n\in\w}$ has compact metrizable closure in $X$, we can replace the family $(S_n)_{n\in\w}$ by a suitable subfamily and assume that the closure of the set $\{\lim S_n\}_{n\in\w}$ is a singleton or a convergent sequence; moreover, we can assume that the family $(S_n)_{n\in\w}$ is disjoint and the union $\bigcup_{n\in\w}S_n$ is disjoint with the closure of the set $\{\lim S_n\}_{n\in\w}$. If the set $\{\lim S_n\}_{n\in\w}$ is a singleton,  then the sequence $(S_n)_{n\in\w}$ is a $\dot S^\w$-fan. In this case we can apply Proposition~\ref{p:FU-fan} and conclude that $X$ contains a $k$-closed subspace $F\subset X$, which is $k$-homeomorphic to the Fr\'echet-Urysohn fan $V$.

If the closure of the set $\{\lim S_n\}_{n\in\w}$ is a convergent sequence, then we can replace the family $(S_n)_{n\in\w}$ by a suitable infinite subfamily, and additionally assume that the sequence $(\lim S_n)_{n\in\w}$ is convergent to some point $x$ in $X$ and consists of pairwise distinct points, which differ from the limit point $x$.

For every $n\in\IN$, fix an injective enumeration $(x_{n,m})_{m\in\IN}$ of the convergent sequence $S_n$ and observe that the sequence $(x_{n,m})_{m\in\IN}$ converges to $\lim S_n$ in $X$.

Now consider the injective map $h:A\to X$ defined by $h(0,0)=x$, $h(\frac1n,0)=\lim S_n$ and $h(\frac1n,\frac1{nm})=x_{n,m}$ for $n,m\in\IN$. Repeating the argument of the proof of Proposition~\ref{p:FU-fan}, we can show that the set $F=h(A)$ is $k$-closed in $X$ and the map $h:A\to F$ is a $k$-homeomorphism.
\end{proof}

We recall \cite{Ar92} that a topological space $X$ has \index{topological space!of countable fan-tightness}\index{countable fan-tightness}{\em countable fan-tightness} if for any sequences of sets $(A_n)_{n\in\w}$ and any $a\in\bigcap_{n\in\w}\bar A_n$ there are finite sets $F_n\subset A_n$, $n\in\w$, such that $a$ is a cluster point of the union $\bigcup_{n\in\w}F_n$.

\begin{proposition}\label{p:sfan-tight} If a topological space $X$ contains a (strong) $\bar S^\w$-fan and has countable fan-tightness, then $X$ contains a (strong) $\Fin^\w$-fan.
\end{proposition}

\begin{proof} Let $(S_n)_{n\in\w}$ be a  $\bar S_\w$-fan in $X$. Replacing $(S_n)_{n\in\w}$ by a suitable subfamily, we can assume that this family is disjoint and its union $\bigcup_{n\in\w}S_n$ does not intersect the (compact) closure $L$ of the set $\{\lim S_n\}_{n\in\w}$. Since $L$ is compact, there exists a point $x\in L$ such that each neighborhood $O_x\subset X$ of $x$ contains infinitely many points $\lim S_n$, $n\in\w$. For every $n\in\w$ consider the set $A_n=\bigcup_{m\ge n}S_m$ and observe that $x\in\bigcap_{n\in\w}\bar A_n$. Since $X$ has countable fan-tightness, each set $A_n$ contains a finite subset $F_n\subset A_n$ such that $x$ is contained in $\bigcup_{n\in\w}F_n$. This implies that the family $(F_n)_{n\in\w}$ is not locally finite in $X$. We claim that this family is compact-finite in $X$. Given any compact set $K\subset X$, find $n\in\w$ such that $K\cap\bigcup_{m\ge n}S_n=\emptyset$ and hence $K\cap F_m=\emptyset$ for all $m\ge n$. So, $(F_n)_{n\in\w}$ is compact-finite and hence is a $\Fin^\w$-fan in $X$. If the $\bar S^\w$-fan $(S_n)_{n\in\w}$ is strong, then each set $S_n$ has an $\IR$-open neighborhood $V_n\subset X$ such that the family $(V_n)_{n\in\w}$ is compact-finite. Then for every $n\in\w$ the set $U_n=\bigcup_{m\ge n}V_m$ is an $\IR$-open neighborhood of $F_n$ and the family $(U_n)_{n\in\w}$ is compact-finite in $X$, which implies that the $\Fin^\w$-fan $(F_n)_{n\in\w}$ is strong.
\end{proof}

We recall that a topological space $X$ is \index{topological space!$k$-metrizable}{\em $k$-metrizable} if $X$ is $k$-homeomorphic to a metrizable space.

\begin{proposition}\label{p:k-metr} A $k^*$-metrizable space $X$ is $k$-metrizable if and only if $X$ contains no $\ddot S^\w$-fan.
\end{proposition}

\begin{proof} To prove the ``if'' part, assume that a $k^*$-metrizable space $K$ contains no $\ddot S^\w$-fan. Since  By Proposition~\ref{p:FUAr}, $X$ contains no $k$-closed subspace $k$-homeomorphic to the Fr\'echet-Urysohn or Arens' fans. Applying Theorem 7.1 of \cite{BBK}, we conclude that $X$ is $k$-metrizable.

Now we prove the ``only if'' part. Assume that $X$ is $k$-metrizable and let $h:X\to M$ be a $k$-homeomorphism of $X$ onto a metrizable space.
Assuming that $X$ contains a $\ddot S^\w$-fan $(S_n)_{n\in\w}$, we can show that $(h(S_n))_{n\in\w}$ is an $\ddot S^\w$-fan in $M$. On the other hand, the space $M$, being metrizable, contains no $\ddot S^\w$-fan. This contradiction shows that $X$ contains no $\ddot S^\w$-fan.
\end{proof}

\begin{theorem}\label{t:Sfan} A topological space $X$ is $k$-metrizable if one of four conditions are satisfied:
\begin{enumerate}
\item[\textup{1)}] $X$ is a $k^*$-metrizable space containing no $\ddot S^\w$-fans;
\item[\textup{2)}] $X$ is a $k^*$-space of countable fan-tightness, containing no $\Fin^\w$-fans;
\item[\textup{3)}] $X$ is a Tychonoff $\aleph$-space containing no strong $\ddot S^\w$-fans;
\item[\textup{4)}] $X$ is a Tychonoff $\aleph$-space of countable fan-tightness, containing no strong $\Fin^\w$-fans.
\end{enumerate}
\end{theorem}

\begin{proof} The first statement follows from Proposition~\ref{p:k-metr} and the second statement follows from the first one and Proposition~\ref{p:sfan-tight}.
The third and fourth statements follow from the first and the second statements and  Corollary~\ref{c:fan-in-aleph-space}, respectively.
\end{proof}

Since $k$-spaces (resp. Tychonoff $k_\IR$-spaces) contain no (strong) $\Fin^\w$-fans,  Theorem~\ref{t:Sfan} implies the following metrizability criterion.

\begin{corollary}\label{c:fan-metrization} For a topological space $X$ the following conditions are equivalent:
\begin{enumerate}
\item[\textup{1)}] $X$ is metrizable.
\item[\textup{2)}] $X$ is a $k^*$-metrizable $k$-space with countable fan-tightness.
\item[\textup{3)}] $X$ is a Tychonoff $\aleph$-$k_\IR$-space with countable fan-tightness.
\end{enumerate}
\end{corollary}

Corollary~\ref{c:fan-metrization} can be compared with metrizability criterion proved in Theorem~4.5 \cite{GK15}: a topological space $X$ is metrizable if and only if $X$ is a $\mathfrak P$-space of countable fan-tightness.

 \section{$S$-semifans}

 A family $(S_\alpha)_{\alpha\in\lambda}$ of convergent sequences in a topological space $X$ is called an \index{$S$-semifan}\index{$S$-fan!$S$-semifan}{\em $S$-semifan} (more precisely, an \index{$S$-semifan!$S^\lambda$-semifan}{\em $S^\lambda$-semifan}) if this family is compact-finite in $X$. If this family is strongly (strictly)  compact-finite in $X$, then it is called a \index{$S$-semifan!strong $S$-semifan}\index{$S$-semifan!strict $S$-semifan}{\em strong $S$-semifan} (resp. a {\em strict $S$-semifan}).

 It is clear that each (strong) $S$-fan is a (strong) $S$-semifan. An $S$-semifan
 $(S_\alpha)_{\alpha\in\lambda}$ is an $S$-fan if and only if it is not locally finite in $X$.

 Recall that a subset $B$ of a topological space $X$ is called \index{subset!$\lambda$-bounded}{\em $\lambda$-bounded} for a cardinal $\lambda$ if $X$ contains no locally finite family $(U_\alpha)_{\alpha\in\lambda}$ of open sets of $X$ that meet $B$.

\begin{lemma}\label{l:cs-compact} If for some cardinal $\lambda$ a topological space $X$ contains no strong $S^\lambda$-semifan, then $X$ contains a closed $\lambda$-bounded set $K\subset X$ whose complement $X\setminus K$ contains no infinite metrizable compact subspaces.
\end{lemma}

\begin{proof} Let $\V$ be the family of all open sets that contain no infinite metrizable subsets. Then the union $\bigcup\V$ is open in $X$ and its complement $B=X\setminus\bigcup\V$ is closed in $X$.  We claim that the set $B$ is $\lambda$-bounded in $X$. Assuming the opposite, find a locally finite family $(U_\alpha)_{\alpha\in\lambda}$ of open sets that meet the set $B$.
Observe that for every $\alpha\in\lambda$ the set $U_\alpha$ does not belong to the family $\V$ and hence contains a convergent sequence $S_\alpha$. Then $(S_\alpha)_{\alpha\in\lambda}$ is a strong $S^\lambda$-semifan, whose existence is  forbidden by our assumption.

It is clear that the complement $X\setminus B$ contains no infinite metrizable compact subset.
\end{proof}

\begin{theorem}\label{t:kR-noSemifan} For a topological space $X$ the following conditions are equivalent:
\begin{enumerate}
\item[\textup{1)}] $X$ metrizable and has compact set of non-isolated points.
\item[\textup{2)}] $X$ is a $\mu$-complete $\bar\aleph_k$-$k_\IR$-space containing no strong $S^\w$-semifans.
\end{enumerate}
\end{theorem}

\begin{proof} The implication $(1)\Ra(2)$ is trivial. To prove that $(2)\Ra(1)$, assume that $X$ is a $\mu$-complete $\bar\aleph_k$-$k_\IR$-space containing no strong $S^\w$-semifans. By Lemma~\ref{l:cs-compact}, the space $X$ contains a closed bounded subset $B\subset X$ whose complement $X\setminus B$ contains no infinite compact metrizable subsets. Since $X$ is a $\bar\aleph_k$-space, each compact subset of $X$ is metrizable, which implies that $X\setminus B$ contains no infinite compact subsets and hence is $k$-discrete. Since $X$ is a $k_\IR$-space, the open $k$-discrete subspace $X\setminus B$ of $X$ is discrete. By the $\mu$-completeness of $X$, the closed bounded subset $B$ of $X$ is compact. So, $B$ is a compact metrizable set with discrete complement $X\setminus B$ in $X$.

Let $\mathcal N$ be a compact-countable closed $k$-network for the $\bar\aleph_k$-space $X$. It follows that the subfamily $\mathcal N_B=\{N\in\mathcal N:N\cap B\ne \emptyset\}$ is countable.

\begin{claim}\label{cl:int-fin} For every open neighborhood $U\subset X$ of $B$ there exists a finite subfamily $\F\subset \mathcal N_B$ and an open set $V\subset X$ such that $B\subset V\subset \bigcup\F\subset U$.
\end{claim}

\begin{proof} By the regularity of the space $X$, the compact set $B$ has an open neighborhood $W\subset X$ such that $\overline{W}\subset U$. Let $\{N_k\}_{k\in\w}$ be an enumeration of the countable subfamily $N_B(U)=\{N\in\mathcal N_B:N\subset U\}$. Since $\mathcal N$ is a $k$-network for $X$, there exists a number $m\in\w$ such that $\bigcup_{k\le m}N_k$ contains the compact set $B$. We claim that there exists a number $i\in\w$ such that the union $U_i=\bigcup_{k\le m+i}N_k$ is a neighborhood of $B$ in $X$. To derive a contradiction, assume that for every $i\in\w$ the set $U_i$ is not a neighborhood of $B$. Let $f_i:X\to\{0,1\}$ be the function defined by the equality $f_i^{-1}(0)=U_i$. Since $U_i$ is not a neighborhood of $B$, the function $f_i$ is discontinuous at some point $x\in B$. Since $X$ is a $k_\IR$-space, for some compact set $K_i\subset X$ the restriction $f_i|K_i$ is discontinuous at some point $x_i\in K_i$. It follows that the point $x_i$ is not isolated and hence belongs to the set $B$. The metrizability of the compact set $K_i$ yields a sequence $S_i\subset K\setminus U_i$ that converges to the point $x_i$. Replacing $S_i$ by a suitable subsequence, we can assume that $S_i\subset W$. Since $\mathcal N$ is a $k$-network in $X$, for every $i\in\w$ there exists a finite subfamily $\F_i\subset\mathcal N$ such that $\bar S_i\subset\bigcup\F\subset U$. Find a number $p(i)\in\w$ such that $\F_i\cap\mathcal N_B\subset\{N_k\}_{k\le m+p(i)}$. Since the space $X\setminus B$ is discrete,  each set $F\in\F_i\setminus \mathcal N_B$ is closed and discrete in $X$. Consequently, $F\cap\bar S_i$ is finite and the set $S_i\setminus \bigcup(\F_i\cap\mathcal N_B)\supset S_i\setminus U_{p(i)}$ is finite. Replacing $S_i$ by $S_i\cap U_{p(i)}$, we can assume that $S_i\subset U_{p(i)}$. Choose any increasing function $f:\w\to\w$ such that $f(i+1)\ge p(f(i))$ for all $i\ge m$ and observe that for any $i<j$ we get $S_{f(i)}\subset U_{p(f(i))}$ and $$S_{f(j)}\subset X\setminus U_{f(j)}\subset X\setminus U_{f(i+1)}\subset X\setminus U_{p(f(i))}\subset X\setminus S_{f(i)},$$ which implies that the family $(S_{f(i)})_{i\ge \w}$ is disjoint.
We claim that this family is compact-finite in $X$. In the opposite case we could find a compact subset
$K\subset X$ meeting infinitely many sequences $S_{f(i)}$, $i\in\w$. Since $\bigcup_{i\in\w}S_i\subset W$, the compact set $K\cap\overline{W}$ also meets infinitely many sequences $S_{f(i)}$, $i\in\w$. Replacing $K$ by its closed subset $K\cap\overline{W}$, we can assume that $K\subset\overline{W}\subset U$. The set $K$ intersects infinitely many (pairwise disjoint) sequences $S_{f(i)}$, $i\in\w$, and hence is
infinite. Since $X\setminus B$ is discrete, $K\cap B\ne\emptyset$ and hence $K\in\mathcal N_B(U)$ and $K=N_k$ for some $k\in\w$. Then $K\cap S_{f(i)}=\emptyset$ for all $i\in\w$ with $f(i)>m+k$, which contradicts the choice of $K$. This contradiction shows that for some $i\in\w$ the set $U_i$ is a neighborhood of $B$ in $X$. Then the finite subfamily $\F=\{N_k\}_{k\le m+i}\subset\mathcal N_B$ and the interior $V$ of the union $\bigcup\F=U_i$ in $X$ have the desired properties: $B\subset V\subset\bigcup\F\subset U$.
\end{proof}

\begin{claim}\label{cl:G-delta} The compact set $B$ is a $G_\delta$-set in $X$.
\end{claim}

\begin{proof} Let $\mathfrak F$ be the (countable) family of finite subfamilies $\F\subset\mathcal N_B$ whose union $\bigcup\F$ contains the set $B$ in its interior $(\bigcup\F)^\circ$. By Claim~\ref{cl:int-fin}, $B=\bigcap_{\F\in\mathfrak F}(\bigcup\F)^\circ$, which means that $B$ is a $G_\delta$-set in $X$.
\end{proof}

\begin{claim} The space $X$ is an $\aleph$-space.
\end{claim}

\begin{proof} By Claim~\ref{cl:G-delta}, there is a decreasing sequence $(W_n)_{n\in\w}$ of open sets such that $B=\bigcap_{n\in\w}W_n$. Since the space $X\setminus B$ is discrete, for every $n\in\w$ the family of singletons $\mathcal N_n=\big\{\{x\}:x\in X\setminus W_n\big\}$ is locally finite in $X$. Then the family $\mathcal N'=\mathcal N_B\cup\bigcup_{n\in\w}\mathcal N_n$ is $\sigma$-discrete. We claim that $\mathcal N'$ is a $k$-network for $X$. Fix any open set $U\subset X$ and a compact subset $K\subset U$. Since $\mathcal N$ is a $k$-network for $X$, there is finite subfamily $\F\subset\mathcal N$ such that $K\subset\bigcup\F\subset U$. Let $\F_B=\F\cap\mathcal N_B$. Observe that every set $F\in\F\setminus\mathcal N_B$ is closed and discrete in $X$. Consequently, the set $K\cap F$ is finite and so is the union $E=\bigcup_{F\in\F\setminus\F_B}K\cap B$. Then the finite subfamily  $\F'=\F_B\cup\big\{\{x\}\big\}_{x\in E}$ of $\mathcal N'$ has the desired property: $K\subset\bigcup\F'\subset U$. Now we see that $\mathcal N'$ is a $\sigma$-discrete $k$-network for $X$ and hence $X$ is an $\aleph$-space.
\end{proof}

By Lemma~\ref{l:iso-norm}, the space $X$ is normal and hence Tychonoff and by Theorem~\ref{t:Sfan}, the space $X$ is $k$-metrizable and being a $k_\IR$-space, is metrizable.
\end{proof}

In some cases the $k_\IR$-space property in Theorem~\ref{t:kR-noSemifan} can be replaced by the absence of $\Clop^\w$-fans.

\begin{proposition}\label{p:cosmic-sequential} If a cosmic space $X$ contains no strict $S^\w$-semifan and no $\Clop^\w$-fan, then $X$ is sequential and contains a compact subset $K\subset X$ with discrete complement.
\end{proposition}

\begin{proof} By Lemma~\ref{l:cs-compact}, $X$ contains a closed bounded subset $K$ whose complement $X\setminus K$ contains no infinite compact metrizable spaces.
Being Lindel\"of, the cosmic space $X$ is $\mu$-complete, which implies that the closed bounded set $K$ is compact.
  Being cosmic, the space $X\setminus K$ is the image of a metrizable separable space $Z$ under a continuous bijective map $f:Z\to X\setminus K$. Taking into account that the image of any convergent sequence in $Z$ is a convergent sequence in $X\setminus K$, we conclude that the metric space $Z$ contains no convergent sequence and hence is discrete and countable. Then its image $X\setminus K$ is countable and hence zero-dimensional. By Theorem~\ref{t:disc-char}, the space $X\setminus K$ is discrete.

Next, we show that $X$ is a $k$-space. To derive a contradiction, assume that $X$ contains a $k$-closed subset $A\subset X$ which is not closed. Fix any point $a\in \bar A\setminus A$ and observe that $a\in K$. Since $A$ is $k$-closed, the intersection $A\cap K$ is compact and does not contain the point $a$. So, we can find a neighborhood $O_a\subset X$ of $a$ such that $\bar{O}_a\cap (A\cap K)=\emptyset$. Since the set $A$ is $k$-closed, for every compact subset $C\subset X$ the intersection $C\cap (\bar{O}_a\cap A)$ is compact. We claim that this compact set is finite. Otherwise it contains a convergent sequence and its limit $z$ belongs to the set $K$. Then also $z\in K\cap (\bar O_a\cap A)=\bar O_a\cap(K\cap A)=\emptyset$, which is a contradiction showing that the set $\bar O_a\cap A$ is compact-finite. Then the family of singletons $\F=(\{x\})_{x\in \bar O_a\cap A}$ is compact-finite but not locally finite at the point $a$. So, $\F$ is a $\Fin^\w$-fan. Since each point $x\in \bar O_a\cap A$ is isolated in $X$, the singleton $\{x\}$ is a clopen subset of $X$. So the fan $\F$ a $\Clop^\w$-fan, which is not possible. This contradiction shows that $X$ is a $k$-space. Since compact subsets of $X$ are metrizable, the $k$-space $X$ is sequential.
\end{proof}

A similar result holds also for zero-dimensional $\mu$-complete $k^*$-metrizable spaces.

\begin{proposition}\label{p:k*-withoutSw} Any zero-dimensional $\mu$-complete $k^*$-metrizable space $X$ containing no strong $S^\w$-semifan and no $\Clop$-fan is metrizable and the set $X'$ of non-isolated points of $X$ is compact.
\end{proposition}

\begin{proof} By Lemma~\ref{l:cs-compact}, the $\mu$-complete space $X$ contains a compact  subset $K$ whose complement $X\setminus K$ contains no infinite compact metrizable spaces. Since each compact subset of the $k^*$-metrizable space $X$ is metrizable, the complement $X\setminus K$ contains no infinite compact set.

By Theorem~\ref{t:disc-char}(3), the space $X\setminus K$ is discrete. We claim that the space $X$ is $k$-metrizable. To obtain a contradiction, assume that $X$ is not $k$-metrizable. Applying Proposition~\ref{p:k-metr}, we can find a $\ddot S^\w$-fan  $(S_n)_{n\in\w}$ in $X$. Replacing the compact-finite family $(S_n)_{n\in\w}$ by a suitable infinite subfamily, we can assume that $S_n\cap K=\emptyset$ for all $n\in\w$. Since the space $X\setminus K\supset\bigcup_{n\in\w}S_n$ is discrete, the compact-finite family $(S_n)_{n\in\w}$ is strongly compact-finite in $X$, which means that $(S_n)_{n\in\w}$ is a strong $\ddot S^\w$-fan in $X$. But this is forbidden by our assumption. This contradiction shows that $X$ is a $k$-metrizable.

Repeating a corresponding piece of the proof of Proposition~\ref{p:cosmic-sequential}, we can show that the $k$-metrizable space $X$ is a $k$-space and hence is metrizable.
\end{proof}

Proposition~\ref{p:k*-withoutSw} and Corollary~\ref{c:tight} imply:

\begin{corollary}\label{c:k*-withoutSw} Any countably tight zero-dimensional $\mu$-complete $k^*$-metrizable space $X$ containing no strong $S^\w$-semifan and no $\Clop^\w$-fan is metrizable and the set $X'$ of non-isolated points of $X$ is compact.
\end{corollary}

Finally we prove some properties of topological spaces that contain no $S^{\w_1}$-semifan.

\begin{proposition}\label{p:k*->a} If a $k^*$-metrizable space $X$ contains no $S^{\w_1}$-semifan, then $X$ contains a $k$-closed $\aleph_0$-subspace $A\subset X$ with  $k$-discrete complement $X\setminus A$. Moreover, if $X$ contains no uncountable compact-finite set, then $X$ is an $\aleph_0$-space.
\end{proposition}

\begin{proof} By Theorem~6.4 \cite{BBK}, the $k^*$-metrizable space $X$ is the image of a metric space $(M,d)$ under a continuous map $f:M\to X$ admitting a (not necessarily continuous) section $s:X\to M$ that preserves precompact sets in the sense that for any compact set $K\subset K$ the image $s(K)$ has compact closure in $M$. Replacing the space $M$ by the closure of $s(X)$ in $M$, we can assume that $s(X)$ is dense in $M$. We claim that the set $M'$ of non-isolated points of $M$ is not separable.
To derive a contradiction, assume that $M'$ is separable. For every $\e>0$ choose a maximal subset $M_\e\subset M'$ which is $\e$-separated in the sense that $d(x,y)\ge \e$ for any distinct points $x,y\in M_\e$. Since the union $\bigcup_{n\in\IN}M_{1/n}$ is dense in the non-separable space $M'$, for some $n\in\IN$ the set $M_{1/n}$ is uncountable. Put $\e=1/n$ and for every point $x\in M_\e$ choose a sequence $S_x\subset s(X)$ of diameter $<\e/3$ that converges to $x$.
By the continuity and bijectivity of $f|s(X)$, the set $f(S_x)$ is a convergent  sequence converging to $f(x)$ in $X$. We claim that the family $\big(f(S_x)\big)_{x\in M_{\e}}$ is compact-finite. Indeed, for any compact set $K\subset X$ the set $s(K)$ has compact closure $\overline{s(K)}$ in $M$ and hence it meets at most finitely many sets $S_x$, $x\in M_\e$ (since the distance between them is $>\e/3$). Since $f|s(X)$ is bijective and $\bigcup_{x\in M_\e}S_x\subset s(X)$, the set $K=f(\overline{s(K)})$ meets only finitely many sets $f(S_x)$, $x\in M_\e$. Therefore the family $\big(f(S_x)\big)_{x\in M_\e}$, being compact-finite in $X$, is a $S^{|M_\e|}$-semifan, which is forbidden by our assumption.

This contradiction shows that the set $M'$ is separable. We claim that the image $A=f(M')$ is a $k$-closed $\aleph_0$-subspace of $X$. Assuming the opposite, we could find a compact subset $K\subset X$ such that $A\cap K$ is not closed in $K$ and hence contains a sequence $(x_n)_{n\in\w}$ converging to a point $x\notin A$. Then the sequence $\{s(x_n)\}_{n\in\w}$ is contained in some compact subset of $M$ and hence contains a subsequence $(s(x_{n_k}))_{k\in\w}$ that converges to some point $z\in M'$. The continuity of $f$ guarantees that $f(z)=\lim_{k\to\infty}f\circ s(x_{n_k})=\lim_{k\to\infty}x_{n_k}=x$, so $x\in f(M')=A$, which is a desired contradiction. Being a continuous image of the metrizable separable space $M'$ the $k^*$-metrizable space $A$ is cosmic. By Theorem 7.2 \cite{BBK}, $A$ is  an $\aleph_0$-space.

  It remains to prove that the complement $X\setminus A$ is $k$-discrete. Assuming  the opposite, we could find an infinite compact subset $K\subset X\setminus A$ and choose a sequence $(x_n)_{n\in\w}$ of pairwise distinct points in $K$. Since the set $s(K)$ has compact closure in $M$, the sequence $s(x_n)_{n\in\w}$ has an accumulation point $z\in \overline{s(K)}\cap M'$. Then $f(z)\in K\cap f(M')=K\cap A$, which contradicts the choice of $K$.
 \smallskip

 Now assuming that $X$ contains no uncountable compact-finite subset, we shall prove that the metrizable space $M$ is separable and hence $X$ is an $\aleph_0$-space. It suffices to prove that for every $\e>0$ the closed set $B_\e=\{z\in M:d(z,M')\ge\e\}$ is countable. To derive a contradiction, assume that for some $\e>0$ the set $B_\e$ is uncountable. Taking into account that the set $s(X)$ is dense in $M$, conclude that it contains all isolated points of $M$, so $B_\e\subset s(X)$ and the image $f(B_\e)$ is uncountable. We claim that this image is compact-finite. To derive a contradiction, assume that $f(B_\e)$ has infinite intersection with some compact set $K\subset X$ then  the infinite set $B_\e\cap s(K)$ is contained in the compact set $\overline{s(K)}$ and hence contains a non-trivial convergent sequence $(x_n)_{n\in\w}$, which is not possible as $B_\e$ is a closed set in $M$, consisting of isolated points of $M$.
 \end{proof}


Some results proved in this section are summed up in the following theorem.

\begin{theorem}\label{t:Ssemifan} Let $X$ be a topological space.
\begin{enumerate}
\item[\textup{1)}] If for some cardinal $\lambda$ the space $X$ contains no strong $S^\lambda$-semifans, then $X$ contains a $\lambda$-bounded closed subset $B$ whose complement $X\setminus B$ contains no infinite compact metrizable sets.
\item[\textup{2)}] If $X$ is $\aleph_k$-space containing no strong $S^\w$-semifans, then $X$ contains a bounded subset $B\subset X$ with $k$-discrete complement $X\setminus B$.
\item[\textup{3)}] If $X$ is $\mu$-complete $\bar\aleph_k$-$k_\IR$-space containing no strong $S^\w$-semifans, then $X$ is metrizable and has compact set of non-isolated points.
\item[\textup{4)}] If $X$ is a zero-dimensional $\mu$-complete $k^*$-metrizable space containing no strong $S^\w$-semifans and no $\Clop$-fans, then $X$ is metrizable and has compact set of non-isolated points.
\item[\textup{5)}] If $X$ is a cosmic space containing no strong $S^\w$-semifans, then $X$ contains a compact subset with countable complement and hence $X$ is $\sigma$-compact.
\item[\textup{6)}] If $X$ is a Tychonoff $k^*$-metrizable space containing no $\ddot S^\w$-fans and no strong $S^\w$-semifans, then  its $k$-modifi\-cation $kX$ is a metrizable space with compact set of non-isolated points.
\item[\textup{7)}] If $X$ is a Tychonoff $\aleph$-space containing no strong $S^\w$-semifans, then  its $k$-modifi\-cation $kX$ is a metrizable space with compact set of non-isolated points.
\item[\textup{8)}] If $X$ a $k^*$-metrizable space contains no $S^{\w_1}$-semifans, then $X$
contains a $k$-closed  $\aleph_0$-subspace $A\subset X$ with $k$-discrete complement $X\setminus A$.
\item[\textup{9)}] If $X$ is a $k^*$-metrizable and $X$ contains no uncountable compact-finite sets, then $X$ is an $\aleph_0$-space.
\end{enumerate}
\end{theorem}

\begin{proof} The first statement was proved in Lemma~\ref{l:cs-compact}. The second statement follows from the first one and the observation that compact sets in $\aleph_k$-spaces are metrizable.
The third and fourth statements were proved in Theorem~\ref{t:kR-noSemifan} and Proposition~\ref{p:k*-withoutSw}. The fifth
statement follows from the first one and the observation that cosmic $k$-discrete spaces are countable.  The sixth statement follows from Proposition~\ref{p:k-metr} and Lemma~\ref{l:cs-compact}. The seventh statement follows from the fifth one and Corollary~\ref{c:fan-in-aleph-space}. The eighths statement is proved in Proposition~\ref{p:k*->a}. The final ninth statement follows from Theorems 5.3 and  7.2 in \cite{BBK}.
\end{proof}

\section{$D_\w$-Cofans}

In this section we shall discuss the notion of a $D_\w$-cofan, which is dual in some sense to the notion of an $S^\w$-fan.

We recall that a subset $A$ of a topological space $X$ is \index{subset!compact-finite}{\em compact-finite} if the family of singletons $(\{a\})_{a\in A}$ is compact-finite in $X$. If this family is strongly compact-finite, then the subset $A$ is called \index{subset!strongly compact-finite}{\em strongly compact-finite}.

\begin{definition}  A \index{$D_\w$-cofan}{\em $D_\w$-cofan} in a topological space $X$ is a sequence $(D_n)_{n\in\w}$ of compact-finite subsets $D_n\subset X$ of cardinality $|D_n|=\w$, which converge to some point $x\in X$ in the sense that each neighborhood $O_x\subset X$ of $x$ contains all but finitely many sets $D_n$. The point $x$ is called \index{$D_\w$-cofan!limit point of}{\em the limit point} of the $D_\w$-cofan $(D_n)_{n\in\w}$. If all sets $D_n$, $n\in\w$, are strongly compact-finite, then the $D_\w$-cofan $(D_n)_{n\in\w}$ will be called \index{$D_\w$-cofan!strong}{\em strong}.
\end{definition}

Corollary~\ref{c:fan-in-aleph-space} implies that each  $D_\w$-cofan in a Tychonoff $\aleph$-space is strong. A simple application of Tietze-Urysohn Theorem yields another useful fact.

\begin{proposition}\label{p:cofan-normal} A $D_\w$-cofan $(D_n)_{n\in\w}$ in a normal space $X$ is strict if each set $D_n$ is closed and discrete in $X$. \end{proposition}

A typical space containing a $D_\w$-cofan is the \index{fan!metric fan}{\em metric cofan}
$$M=\{(0,0)\}\cup\big\{(\tfrac1n,\tfrac1{nm}):n,m\in\IN\big\}$$in the Euclidean plane. The space $M$ is not locally compact. By \cite[8.3]{vD}, a metrizable space $X$ is not locally compact if and only if $X$ contains a closed subspace homeomorphic to the space $M$.

\begin{proposition} A topological space $X$ contains a $D_\w$-cofan if and only if $X$ contains a $k$-closed subset $F\subset X$, which is $k$-homeomorphic to the metric cofan $M$.
\end{proposition}

\begin{proof} To prove the ``if'' part, assume that $X$ contains a $k$-closed subset $F\subset X$ $k$-homeomorphic to the space $M$.
Take any $k$-homeomorphism $h:M\to F$ and observe that for every $n\in\IN$ the set $D_n=\{h(\frac1n,\frac1{nm}):m\in\IN\}$ is compact-finite in $X$ and the sequence $(D_n)_{n\in\IN}$ converges to $h(0,0)$ by the continuity of $h$. So, $(D_n)_{n\in\IN}$ is a $D_\w$-cofan in $X$.

Now we prove the ``only if'' part. Assume that the space $X$ contains a $D_\w$-cofan $(D_n)_{n\in\w}$, which converges to a point $x\in X$. Fix a well-order $\preceq$ on the set $\IN\times\IN$ such that for every $(n,m)\in\w\times\w$ the initial segment $\{(p,q)\in\IN\times\IN:(p,q)\preceq (n,m)\}$ is finite. By induction, for every pair $(n,m)\in\IN\times\IN$ choose a point $x_{n,m}\in D_n\setminus(\{x\}\cup \{x_{p,q}:(p,q)\preceq (n,m)\})$. Then $x_{n,m}$, $n,m\in\IN$, are pairwise distinct points of the set $X\setminus\{x\}$ and hence the map $h:M\to X$ defined by $h(0,0)=x$ and $h(\frac1n,\frac1{nm})=x_{n,m}$ is bijective. The convergence of the sequence $(D_n)_{n\in\w}$ to $x$ implies that the map $h$ is continuous.

We claim that the set $F=h(M)$ is $k$-closed in $X$ and the map $h:M\to F$ is a $k$-homeomorphism. Take any compact subset $K\subset X$ and observe that for every $n\in\IN$ the intersection $K\cap D_n$ is finite. Taking into account that  the sequence $(D_n)_{n\in\w}$ converges to $x$, we conclude that the set $K\cap F\subset \{x\}\cup\bigcup_{n\in\w}K\cap D_n$ is either finite or a convergent sequence to $x$. In both cases the intersection $K\cap F$ is compact. Moreover, the definition of the map $h$ ensures that $h^{-1}|K\cap F$ is continuous, which means that $h:M\to F$ is a $k$-homeomorphism of the space $M$ onto the $k$-closed subset $F$ of $X$.
\end{proof}

\begin{lemma}\label{l:D-cof} Let $f:X\to Y$ be a continuous map of a first-countable space $X$ onto a (Tychonoff) topological space $Y$. If the space $X$ contains a point $x\in X$ such that for each neighborhood $O_x\subset X$ of $x$ the set $f(O_x)$ is not bounded in $Y$, then the space $Y$ contains a (strong) $D_\w$-cofan.
\end{lemma}

\begin{proof} Fix a decreasing neighborhood base $(U_n)_{n\in\w}$ of the point $x$ in the first-countable space $X$. The continuity of the map $f$ ensures that the sequence $(f(U_n))_{n\in\w}$ converges to the point $f(x)$. By our assumption, for every $n\in\w$ the set $f(U_n)$ is not bounded in $X$ and hence contains a (strongly) compact-finite subset $D_n$ of cardinality $|D_n|=\w$ (see Proposition~\ref{p:unbound-str-c-f}). Then $(D_n)_{n\in\w}$ is a (strong) $D_\w$-cofan in $X$.
\end{proof}

\begin{proposition}\label{p:XD->s+h+kw} Assume that a Tychonoff space $X$ contains no strong $D_\w$-cofan.
\begin{enumerate}
\item[\textup{1)}] If $X$ is cosmic, then $X$ is $\sigma$-compact;
\item[\textup{2)}] If $X$ is an $\aleph_0$-space, then $X$ is hemicompact;
\item[\textup{3)}] If $X$ is a sequential $\aleph_0$-space, then $X$ is a $k_\w$-space.
\end{enumerate}
\end{proposition}

\begin{proof} If $X$ is cosmic, then $X$ is the image of a separable metrizable space $Z$ under a continuous map $f:Z\to X$. By Lemma~\ref{l:D-cof}, each point $z\in Z$ has a neighborhood $O_z\subset Z$ whose image $f(O_z)$ is bounded in $X$. The space $X$, being Lindel\"of, is $\mu$-complete. So, the set $f(O_z)$ has compact closure. By the Linde\"of property of the metrizable separable space $Z$, the open cover $\{O_z:z\in Z\}$ has a countable subcover $\{O_z:z\in C\}$. Then $\K=\big\{\cl\big(f(O_z)\big):z\in C\big\}$ is a countable cover of $X$ by compact subsets, and hence $X$ is $\sigma$-compact.

If $X$ is an $\aleph_0$-space, then we can additionally assume that the map $f$ is compact-covering (see \cite[\S11]{Grue}). In this case the countable cover $\mathcal K$ witnesses that the space $X$ is hemicompact.

If $X$ is a sequential $\aleph_0$-space, then the family $\mathcal K$ generates the topology of $X$ in the sense that a subset $A\subset X$ is closed in $X$ if and only if for every compact set $K\in\K$ the intersection $K\cap A$ is closed in $K$. This means that $X$ is a $k_\w$-space.
\end{proof}

\begin{lemma}\label{l:b-nw} If a Hausdorff (Tychonoff) space $X$ contains no (strong) $D_\w$-cofan, then for every compact-countable $k$-network
$\mathcal N$ for $X$ the subfamily $$\mathcal B=\big\{\textstyle{\bigcap\F}:\mbox{$\F\in[\mathcal N]^{<\w}$ and $\bigcap\F$ is bounded in $X$}\big\}$$ is a $k$-network for $X$.
\end{lemma}

\begin{proof} To prove that $\mathcal B$ is a $k$-network for $X$, fix an open set $U\subset X$ and a compact subset $K\subset U$. Since the family $\mathcal N$ is compact-countable, the subfamily $\mathcal N_K=\{N\in\mathcal N:N\cap K\ne\emptyset\}$ is countable.

Let $\mathfrak F$ be the family of finite subfamilies $\F\subset\mathcal N_K$ such that $K\subset\bigcup\F\subset U$. It is clear that the family $\mathfrak F$ is countable and hence it can be enumerated as $\mathfrak F=\{\F_n\}_{n\in\IN}$. For every $n\in\IN$ consider the finite family $\F_1\wedge\dots\wedge\F_n=\{\bigcap_{i=1}^n F_i:(F_i)_{i=1}^n\in \prod_{i=1}^n\F_i\}$ and its subfamily $\mathcal E_n=\{A\in\F_1\wedge\dots\wedge \F_n:A\cap K\ne\emptyset\}$.
Let also $\mathcal E_0=\{X\}$.

Consider the countable set $$T=\bigcup_{n\in\w}\{n\}\times\mathcal E_n$$partially ordered by the order $(n,A)\le (m,B)$ if $n\le m$ and $B\subset A$. It follows that $T$ is a tree and each vertex $(n,A)$ of $T$ has finite degree. Each branch $t\in T$ (i.e., a maximal linearly ordered subset) of the tree $T$ can be identified with a decreasing sequence $(E_n)_{n\in\w}\in\prod_{n\in\w}\mathcal E_n$. By the compactness of $K$, the intersection $K_t=\bigcap_{n\in\w}\cl(K\cap E_n)$ is not empty.

\begin{claim}\label{cl:branch-open} For every point $x\in K_t$ and open neighborhood $O_x\subset X$ there is $n\in\w$ such that $E_n\subset O_x$.
\end{claim}

\begin{proof} Using the Hausdorff property of $X$, we can find an open neighborhood $U_x\subset O_x\cap U$ of $x$ and an open neighborhood $V\subset U$ of the compact set $K\setminus O_x$ such that $U_x\cap V=\emptyset$. It follows that $K=(K\setminus V)\cup\overline{K\cap V}$ and $K\setminus V\subset O_x$. Since $x$ does not belong to the compact set $\overline{K\cap V}$, there are disjoint open sets $W_x\subset U_x$ and $W\subset U$ such that $x\in W_x$ and $\overline{K\cap V}\subset W$.

Since $\mathcal N$ is a $k$-network, there are finite subfamilies $\F_x,\F'\subset\mathcal N$ such that $K\setminus V\subset \bigcup\F_x\subset O_x\subset U$ and $\overline{K\cap V}\subset\bigcup\F'\subset W\subset U$. We lose no generality assuming that each set $F\in \F_x\cup\F'$ meets the set $K$. Consequently, $\F_x\cup\F'\in \mathcal N_K$ and hence $\F_x\cup\F'=\F_n$ for some $n\in\IN$. The definition of the family $\mathcal E_n$ implies that $E_n\subset F$ for some  $F\in\F_n=\F_x\cup\F'$. Taking into account that $x\in\overline{E_n\cap K}\subset \overline{F}_n$ and $\bigcup\F'\subset W\subset\overline{W}\subset X\setminus W_x$, we conclude that $F\in\F_x$ and hence $E_n\subset F\subset\bigcup\F_x\subset O_x$.
\end{proof}

Claim~\ref{cl:branch-open} and the Hausdorff property of $X$ imply that the compact set $K_t$ is a singleton and the sequence $(E_n)_{n\in\w}$ converges to $K_t$.

\begin{claim}\label{cl:tree-bound} For any branch $(E_n)_{n\in\w}\in\prod_{n\in\w}\mathcal E_n$ of the tree $T$ there is a number $n\in\w$ such that the set $E_n$ is bounded.
\end{claim}

\begin{proof} To derive a contradiction, assume that each set $E_n$, $n\in\w$ is not bounded in $X$. Then $E_n$ contains an infinite compact-finite closed discrete subset $D_n\subset E_n$, which is strongly compact-finite if the space $X$ is Tychonoff (see Proposition~\ref{p:unbound-str-c-f}). By Claim~\ref{cl:branch-open}, the sequence $(D_n)_{n\in\w}$ converges to the unique point $x\in\bigcap_{k\in\w}\overline{K\cap E_k}$. This means that  the family $(D_n)_{n\in\w}$ is a (strong) $D_\w$-fan in $X$, which contradicts our assumption.
\end{proof}

In the tree $T$ consider the subtree $T_u$ consisting of pairs $(n,A)$ such that the set $A$ is unbounded. Claim~\ref{cl:tree-bound} implies that the subtree $T_u$ has no infinite branch. Since each point of $t$ has finite degree, by the K\"onig's Lemma, the subtree $T_u$ is finite. Consequently for some $n\in\w$ the $n$th level $\{n\}\times\mathcal E_n$ of the tree $T$ does not intersect the subtree $T_u$, which means that each set $E\in\mathcal E_n$ is bounded. This implies that $\mathcal E_n\subset\mathcal B$. Since $K\subset\bigcup\mathcal E_n\subset U$, the family $\mathcal B$ is a $k$-network for $X$.
\end{proof}

A topological space $X$ is called a \index{topological space!$k$-sum of hemicompact spaces}{\em $k$-sum of hemicompact spaces} if $X$ is the union of a compact-clopen disjoint family $\F$ of hemicompact subspaces of $X$.
A family $\F$ of subsets of $X$ is \index{family of sets!compact-clopen}{\em compact-clopen} if for any compact set $K\subset X$ and any set $F\in\F$ the intersection $F\cap K$ is clopen in $K$.
If a topological space $X$ is a $k$-sum of hemicompact spaces, then its $k$-modification $kX$ is a topological sum of $k_\w$-spaces.

\begin{proposition}\label{p:decomp1} Each $\mu$-complete (Tychonoff) $\bar\aleph_k$-space $X$ containing no (strong) $D_\w$-cofan is a $k$-sum of cosmic hemicompact spaces and the $k$-modification of $X$ is a topological sum of cosmic $k_\w$-spaces.
\end{proposition}

\begin{proof} By Lemma~\ref{l:b-nw}, the $\bar\aleph_k$-space $X$ has a compact-countable $k$-network $\mathcal K$ consisting of closed bounded subsets of $X$. Since $X$ is a $\mu$-complete $\bar\aleph_k$-space, all elements of $\mathcal K$ are metrizable compact subsets of $X$. For two sets $A,B\in\K$ we write that $A\sim B$ if there are sets $A_1,\dots,A_n\in\K$ such that $A=A_1$, $B=A_n$ and $A_i\cap A_{i+1}\ne\emptyset$ for every $i<n$. It is clear that $\sim$ is an equivalence relation. For every set $A\in\K$ by $[A]_{\sim}$ we denote the equivalence class of $A$. The compact-countability of the family $\K$ implies that each equivalence class $[A]_\sim$ is countable.

We claim that the union $\bigcup[A]_\sim$ is a hemicompact subspace of $X$. It suffices to prove that each compact subset $K\subset \bigcup[A]_\sim$ is contained in the union $\bigcup\F$ of some finite subfamily $\F\subset[A]_\sim$. Since $\K$ is a $k$-network, there is a finite subfamily $\F\subset\K$ such that $K\subset\bigcup\F$ and each $F\in\F$ meets $K$. Taking into account that $K\subset\bigcup[A]_\sim$, we conclude that $\F\subset[A]_\sim$.

So, $X$ decomposes into the union of the disjoint family $\U=\{\bigcup[A]_\sim:A\in\K\}$ of hemicompact subspaces. It remains to prove that this family is compact-clopen. Given a compact subset $K\subset X$, find a finite subfamily $\F\subset\K$ such that $K\subset\bigcup\F$ and observe that for any $A\in\K$ the intersection $K\cap (\bigcup[A]_\sim)$ is clopen in $K$ because both sets
$K\cap (\bigcup[A]_\sim)=\bigcup\{K\cap F:F\in\F\cap [A]_\sim\}$ and
$K\setminus (\bigcup[A]_\sim)=\bigcup\{K\cap F\in\F:F\in\F\setminus [A]_\sim\}$ are compact.
\end{proof}

\begin{proposition}\label{p:a+s->ba} If a (Tychonoff) $\mu$-complete $\aleph_k$-space $X$ contains  no (strong) $D_\w$-cofan and no $\ddot S^{\w_1}$-fan, then $X$ is an $\bar\aleph_k$-space and hence $X$ is a $k$-sum of hemicompact spaces and its $k$-modification $kX$ is a topological sum of cosmic $k_\w$-spaces.
\end{proposition}

\begin{proof} By Lemma~\ref{l:b-nw}, the $\aleph_k$-space $X$ has a compact-countable $k$-network $\mathcal N$ consisting of bounded subsets of $X$. By the $\mu$-completeness of $X$, each set $N\in\mathcal N$ has compact closure. The regularity of the space $X$ implies that the family $\overline{\mathcal N}=\{\bar N:N\in\mathcal N\}$ is a $k$-network for $X$. We claim that this $k$-network is compact-countable. Assuming the opposite, we could find a compact subset $K\subset X$ such that the family $\mathcal N_K=\{N\in\mathcal N:K\cap\overline{N}\ne\emptyset\}$ is uncountable. Since the family $\mathcal N$ is compact-countable, the family $\mathcal N_K'=\{N\in\mathcal N_K:K\cap N\ne\emptyset\}$ is a countable. Then the compact space $K$ has a countable network and hence $K$ is metrizable.

For two sets $A,B\in\mathcal N$ we write $A\sim B$ if there are sets $N_0,\dots,N_k\in\mathcal N$ such that $A=N_0$, $B\in N_k$ and $N_i\cap N_{i+1}\ne\emptyset$ for all $i<k$. The compact-countability of the family $\mathcal N$ implies that each equivalence class $[A]_\sim$ is countable. We claim that the family $\{\bigcup[A]_\sim:A\in\mathcal N\}$ is compact-finite. Indeed, for any compact set $C\subset X$ we can find a finite subfamily $\mathcal N_C\subset\mathcal N$ whose union contains $C$ and observe that family
$\{[A]_\sim:A\in\mathcal N,\;C\cap(\bigcup[A]_\sim)\ne\emptyset\}$
is contained in the finite family $\{[A]_\sim:A\in\mathcal N_C\}$.

By transfinite induction choose an indexed family $(N_\alpha)_{\alpha\in\w_1}$ of sets in the uncountable family $\mathcal N_K\setminus \mathcal N_K'$ such that $N_\alpha\notin\bigcup_{\beta<\alpha}[N_\beta]_\sim$. Such a choice ensures that the family $(N_\alpha)_{\alpha\in\w_1}$ is compact-finite.

For every $\alpha\in\w_1$ choose a sequence $S_\alpha\subset N_\alpha$ convergent to a point $x_\alpha\in K\cap\overline{N}_\alpha$. Such a sequence exists since the compact space $\overline{N}_\alpha$ is metrizable. Then $(S_\alpha)_{\alpha\in\w_1}$ is a $\ddot S^{\w_1}$-fan in $X$. But this contradicts our assumption. This contradiction implies that the $k$-network $\overline{\mathcal N}$ is compact-countable and hence $X$ is an $\bar \aleph_k$-space. Now Proposition~\ref{p:decomp1} completes the proof.
\end{proof}

Some results of this section are summed up in the following theorem.

\begin{theorem}\label{t:Dcofan} Let $X$ be a topological (Tychonoff) space containing no (strong) $D_\w$-cofan.
\begin{enumerate}
\item If $X$ is a $\mu$-complete $\bar\aleph_k$-space, then the $k$-modification of $X$ is a topological sum of cosmic $k_\w$-spaces.
\item  If $X$ is a $\mu$-complete $\aleph_k$-space containing no $\bar S^{\w_1}$-fan, then $X$ is a $\bar\aleph_k$-space and its $k$-modification $kX$  is a topological sum of cosmic $k_\w$-spaces.
\item If $X$ is an $\aleph_0$-space, then $X$ is hemicompact.
\item If $X$ is cosmic, then $X$ is $\sigma$-compact.
\end{enumerate}
\end{theorem}

\begin{proof} The four statement of this theorem are proved in Propositions~\ref{p:decomp1}, \ref{p:a+s->ba} and \ref{p:XD->s+h+kw}.
\end{proof}

\section{(Strong) $\Fin$-fans in products}\label{s:prod}

In the section we shall apply $\dot S$-fans and $D_\w$-cofans to construct $\Fin$-fans in products of topological spaces. The proofs of the following two theorems in the simplest form display two methods of constructing (strong) $\Fin$-fans, which will be exploited many times in this paper.
These methods are known and were used in the papers of Tanaka \cite{Tan76}, \cite{Tan79}, \cite{Tan89}, \cite{Tan97}, Gruenhage \cite{Gru80}, Chen \cite{Chen90}, Banakh \cite{Ban98}.

\begin{theorem}\label{t:product1} If a topological space $X$ contains a (strong) $D_\w$-cofan $(D_n)_{n\in\w}$ and a topological space $Y$ contains a (strong) $\bar S^\w$-fan $(S_n)_{n\in\w}$, then the product $X\times Y$ contains a (strong) $\Fin^\w$-fan.
\end{theorem}

\begin{proof} Let $(S_n)_{n\in\w}$ be a (strong) $\bar S^\w$-fan in $Y$. Since $(S_n)_{n\in\w}$ is compact-finite, we can assume that $\{\lim S_n\}_{n\in\w}$ is disjoint with $\bigcup_{n\in\w}S_n$. For every $n\in\w$ let $\{y_{n,m}\}_{m\in\w}$ be an enumeration of the convergent sequence $S_n$. Since the set $\{\lim S_n\}_{n\in\w}$ has compact closure in $Y$, there is a point $y_\infty\in Y$ such that every neighborhood $O_y\subset Y$ of $y$ contains the limit points $\lim S_n$ of infinitely many sequences $S_n$, $n\in\w$.

If the fan $(S_n)_{n\in\w}$ is strong, then each convergent sequence $S_n$ has an $\IR$-open neighborhood $V_n\subset Y$ such that the family $(V_n)_{n\in\w}$ is compact-finite in $Y$. If the fan $(S_n)_{n\in\w}$ is not strong, then we put $V_n=S_n$ for all $n\in\w$.

Let $x_\infty$ be the limit point of the cofan $(D_n)_{n\in\w}$.
Replacing $(D_n)_{n\in\w}$ by a suitable subfamily, we can assume that $x_\infty\notin\bigcup_{n\in\w}D_n$. For every $n\in\w$ write the compact-finite set $D_n$ as $D_n=\{x_{n,m}\}_{m\in\w}$. If $D_n$ is strongly compact-finite, then each point $x_{n,m}\in D_n$ has an $\IR$-open neighborhood $U_{n,m}\subset X$ such that the family $(U_{n,m})_{m\in\w}$ is compact-finite in $X$.
Replacing the sequence $(x_{n,m})_{m\in\w}$ by a suitable subsequence, we can assume that $x_\infty\notin\bigcup_{m\in\w}U_n$. If the set $D_n$ is not strongly compact-finite, then we put $U_{n,m}=\{x_{n,m}\}$ an observe that the family $(U_{n,m})_{m\in\w}$ is compact-finite in $X$.

We claim that the family $(z_{n,m})_{n,m\in\w}$ of the singletons $z_{n,m}=(x_{n,m},y_{n,m})$ is a (strong) $\Fin^\w$-fan in $X\times Y$.

First we show that this family is not locally finite at the point $(x_\infty,y_\infty)$. Given any open neighborhood $U\times V\subset X\times Y$ of $(x_\infty,y_\infty)$, use the convergence of the sequence $(D_n)_{n\in\w}$ to $x_\infty$ and find a number $k\in\w$ such that $\bigcup_{n\ge k}D_n\subset U$.
By the choice of $y_\infty$ the open set $V$ contains the limit point $\lim S_n$ of the sequence $S_n$ the some $n\ge k$. Since the sequence $(y_{n,m})_{m\in\w}$ converges to $\lim S_n\in V$, there exists a number $m\in\w$ such that $y_{n,m}\in V$. Then the point $z_{n,m}=(x_{n,m},y_{n,m})$ belongs to $U\times V$.

Next, we show that the family $(z_{n,m})_{n,m\in\w}$ is (strongly) compact-finite. This will follow as soon as we check that the family $(U_{n,m}\times V_n)_{n\in\w}$ is compact-finite in $X\times Y$. Given any compact set $K_X\times K_Y$, use the compact-finiteness of the family $(V_n)_{n\in\w}$ and find a number $n_0\in\w$ such that $K_Y\cap V_n=\emptyset$ for all $n\ge n_0$. The compact-finiteness of the families $(U_{n,m})_{m\in\w}$ for $n\le n_0$ yields a number $m_0\in\w$ such that $K_X\cap U_{n,m}=\emptyset$ for all $n\le n_0$ and $m\ge m_0$. Then the set $$\{(n,m)\in\w\times\w:(K_X\times K_Y)\cap(U_{n,m}\times V_n)\ne\emptyset\}\subset [0,n_0]\times [0,m_0]$$ is finite, witnessing that the family $(U_{n,m}\cap V_n)_{n,m\in\w}$ is compact-finite.
\end{proof}

\begin{theorem}\label{t:product2} If topological spaces $X,Y$ contain (strong) $\bar S^{\w_1}$-fans, then the product $X\times Y$ contains a (strong) $\Fin^{\w_1}$-fan.
\end{theorem}

\begin{proof} Let $(X_\alpha)_{\alpha\in\w_1}$ and $(Y_\alpha)_{\alpha\in\w_1}$  be  (strong) $\bar S^{\w_1}$-fans in $X$ and $Y$, respectively. Since the sets $\{\lim X_\alpha\}_{\alpha\in\w_1}$ and $X_\alpha$, $\alpha\in\w_1$, have compact closures in $X$  we can replace $(X_\alpha)_{\alpha\in\w_1}$ by a suitable uncountable subfamily and assume that the family $(X_\alpha)_{\alpha\in\w_1}$ is disjoint and its union does not intersect the closure of the set $\{\lim X_\alpha\}_{\alpha\in\w_1}$.
By analogy, we can replace $(Y_\alpha)_{\alpha\in\w_1}$ by a suitable uncountable subfamily and assume that the family $(Y_\alpha)_{\alpha\in\w_1}$ is disjoint and its union does not intersect the closure of the set $\{\lim Y_\alpha\}_{\alpha\in\w_1}$.

For every $\alpha\in\w_1$ chose sequences $(x_{\alpha,n})_{n\in\IN}$ and $(y_{\alpha,n})_{n\in\w}$ of pairwise distinct points in the convergent sequences $X_\alpha$ and $Y_\alpha$, respectively.

Choose any almost disjoint family $(A_\alpha)_{\alpha\in\w_1}$ of infinite subsets of $\w$.
Let $\Lambda=\{(\alpha,\beta)\in\w_1\times w_1:\alpha\ne \beta\}$ and for any pair $(\alpha,\beta)\in\Lambda$ consider the finite subset $$D_{\alpha,\beta}=\{(x_{\alpha,n},y_{\beta,n}):n\in A_\alpha\cap A_\beta\}$$of the product $X\times Y$.

We claim the family $(D_{\alpha,\beta})_{(\alpha,\beta)\in\Lambda}$ is a $\Fin^{\w_1}$-fan in $X\times Y$. First we check that this family is not locally finite in $X\times Y$. Fix any free ultrafilter $\U$ on $\w_1$ whose elements are uncountable sets of $\w_1$. Since the set $X_\infty=\{\lim X_\alpha\}_{\alpha\in\w_1}$ has compact closure $\bar X_\infty$ in $X$, the $\w_1$-sequence $(\lim X_\alpha)_{\alpha\in\w_1}$ is $\U$-convergent to some point $x\in \bar X_\infty$. The latter means that for every neighborhood $O_x\subset X$ of $x$ the set $\{\alpha\in\w_1:\lim X_\alpha\in O_x\}$ belongs to the ultrafilter $\U$. By analogy, the $\w_1$-sequence $(Y_\alpha)_{\alpha\in\lambda}$ is $\U$-convergent to some point $y\in Y$.

We claim that the family $(D_{\alpha,\beta})_{(\alpha,\beta)\in\Lambda}$ is not locally finite at the point $(x,y)$. Observe hat for any open neighborhood $U\times V\subset X\times Y$ of $(x,y)$ the set $\Omega=\{\alpha\in\w_1:(\lim X_\alpha,\lim Y_\alpha)\in U\times V\}$ belongs to the ultrafilter $\U$ and hence is uncountable.
For every $\alpha\in\Omega$ the open sets $U$ and $V$ are neighborhoods of the limit points $\lim X_\alpha$, $\lim Y_\alpha$ of the sequences $X_\alpha=\{x_{\alpha,n}\}_{n\in\w}$ and $Y_\alpha=\{y_{\alpha,n}\}_{n\in\w}$, respectively. So there is a number $\varphi(\alpha)\in\w$ such that $\{x_{\alpha,n}\}_{n\ge\varphi(\alpha)}\subset U$ and $\{y_{\alpha,n}\}_{n\ge\varphi(\alpha)}\subset V$.

By the Pigeonhole Principle, for some $m\in\w$ the set $\Omega_m=\{\alpha\in\Omega:\varphi(\alpha)=m\}$ is uncountable. By Lemma~\ref{l:ad}, the set $\Lambda'=\{(\alpha,\beta)\in\Lambda:\alpha,\beta\in\Omega_m,\;A_\alpha\cap A_\beta\not\subset[0,m]\}$ is uncountable. Fix any pair $(\alpha,\beta)\in\Lambda'$ and choose a number $n\in A_\alpha\cap A_\beta\setminus[0,m]$. Then $(x_{\alpha,n},y_{\beta,n})\in (U\times V)\cap D_{\alpha,\beta}$, which means that the family $(D_{\alpha,\beta})_{(\alpha,\beta)\in\Lambda}$ is not locally finite at $(x,y)$.

Next, we show that the family $(D_{\alpha,\beta})_{(\alpha,\beta)\in\Lambda}$ is compact-finite in $X\times Y$. Given a compact set $K\subset X\times Y$, find compact sets $K_X\subset X$ and $K_Y\subset Y$ such that $K\subset K_X\times K_Y$. Since the $\bar S^{\w_1}$-fans $(X_\alpha)_{\alpha\in\w_1}$ and $(Y_\alpha)_{\alpha\in\w_1}$ are compact-finite, the sets $F_X=\{\alpha\in\w_1:K_X\cap X_\alpha\ne\emptyset\}$ and $F_Y=\{\beta\in\w_1:K_Y\cap Y_\beta\ne\emptyset\}$ are finite. Then the set $\{(\alpha,\beta)\in\Lambda:D_{\alpha,\beta}\cap K\ne\emptyset\}\subset F_X\times F_Y$ is finite too. This shows that  the family $(D_{\alpha,\beta})_{(\alpha,\beta)\in\Lambda}$ is a $\Fin^{\w_1}$-fan.

If the fans $(X_\alpha)_{\alpha\in\w_1}$ and $(Y_\alpha)_{\alpha\in\w_1}$ are strong, then we can find compact-finite families of $(U_\alpha)_{\alpha\in\w_1}$ and $(V_\alpha)_{\alpha\in\w_1}$ of $\IR$-open sets in the spaces $X$ and $Y$, respectively, such that $X_\alpha\subset U_\alpha$ and $Y_\alpha\subset V_\alpha$ for all $\alpha\in\w_1$. Then for every $(\alpha,\beta)\in\Lambda$ the set $U_\alpha\times V_\beta$ is an $\IR$-open neighborhood of the set $D_{\alpha,\beta}$ and the family $(U_{\alpha}\times V_\beta:(\alpha,\beta)\in\Lambda\}$ is compact-finite in $X\times Y$, which means that the $\Fin^{\w_1}$-fan $(D_{\alpha,\beta})_{(\alpha,\beta)\in\Lambda}$ is strong.
\end{proof}

\begin{corollary}\label{c:prod1} If the product $X\times Y$ of two $k^*$-metrizable spaces $X,Y$ contains no $\Fin^\w$-fans, then one of the following alternatives hold:
\begin{enumerate}
\item[\textup{1)}] $X$ and $Y$ are $k$-metrizable;
\item[\textup{2)}] $X$ and $Y$ are $k$-homeomorphic to topological sums of cosmic $k_\w$-spaces;
\item[\textup{3)}] $X$ or $Y$ is $k$-homeomorphic to a metrizable locally compact space.
\end{enumerate}
\end{corollary}

\begin{proof} Observe that each metrizable $k_\w$-space is locally compact. This implies that a topological space $X$ is $k$-homeomorphic to a metrizable locally compact space if and only if $X$ is $k$-metrizable and $k$-homeomorphic to a topological sum of cosmic $k_\w$-spaces. Therefore, if none of the three alternatives holds, then one of the spaces (say $X$) is not $k$-metrizable and the other space (then $Y$) is not $k$-homeomorphic to a topological sum of cosmic $k_\w$-spaces. By Proposition~\ref{p:k-metr}, $X$ contains an $\ddot S^\w$-fan and by Proposition~\ref{p:decomp1}, the space $Y$ contains a $D_\w$-cofan. By Theorem~\ref{t:product1}, the product $X\times Y$ contains a $\Fin^\w$-fan, which is a desired contradiction.
\end{proof}

Corollary~\ref{c:prod1} combined with Corollary~\ref{c:fan-in-aleph-space} implies another corollary, generalizing a result of Y.Tanaka \cite{Tan76}.

\begin{corollary}\label{c:prod2} If the product $X\times Y$ of two Tychonoff $\aleph$-spaces $X,Y$ contains no strong $\Fin^\w$-fans, then one of the following alternatives hold:
\begin{enumerate}
\item[\textup{1)}] $X$ and $Y$ are $k$-metrizable;
\item[\textup{2)}] $X$ and $Y$ are $k$-homeomorphic to topological sums of cosmic $k_\w$-spaces;
\item[\textup{3)}] $X$ or $Y$ is $k$-homeomorphic to a metrizable locally compact space.
\end{enumerate}
\end{corollary}

Repeating the proof of Corollary~\ref{c:prod1} and applying Proposition~\ref{p:XD->s+h+kw} instead of Proposition~\ref{p:decomp1} we obtain the following $\aleph_0$-version of Corollary~\ref{c:prod2}.

\begin{corollary}\label{c:prod3} If the product $X\times Y$ of two $\aleph_0$-spaces $X,Y$ contains no strong $\Fin^\w$-fans, then one of the following alternatives hold:
\begin{enumerate}
\item[\textup{1)}] $X$ and $Y$ are $k$-metrizable;
\item[\textup{2)}] $X$ and $Y$ are cosmic and hemicompact;
\item[\textup{3)}] $X$ or $Y$ is metrizable separable and locally compact.
\end{enumerate}
\end{corollary}

Corollary~\ref{c:prod1} and Proposition~\ref{k-no-Cld-fan} imply the following characterization of the $k$-space property in products (proved for $\aleph$-spaces by Tanaka \cite{Tan76}).

\begin{corollary}\label{c:tanaka-gen} For $k^*$-metrizable spaces $X,Y$ the product $X\times Y$ is a $k$-space if and only if one of the following conditions holds:
\begin{enumerate}
\item[\textup{1)}] $X$ and $Y$ are metrizable;
\item[\textup{2)}] $X$ and $Y$ are topological sums of cosmic $k_\w$-spaces;
\item[\textup{3)}] $X$ and $Y$ are $k$-spaces and one of them is metrizable and locally compact.
\end{enumerate}
\end{corollary}

\section{Constructing (strong) $\Fin^\w$-fans in rectifiable spaces}\label{s:rectif}

In this section we prove a simple criterion for detecting topological groups (more generally, rectifiable spaces) that contain strong $\Fin^\w$-fans.

A topological space $X$ is called \index{topological space!rectifiable}{\em rectifiable} if there exists a homeomorphism $h:X\times X\to X\times X$ such that $h(\Delta_X)=X\times\{e\}$ for some point $e\in X$ and $h(\{x\}\times X)=\{x\}\times X$ for all $x\in X$.  Here by $\Delta_X=\{(x,y)\in X\times X:x=y\}$ we denote the diagonal of $X\times X$. Rectifiable spaces were introduced by Arhangelskii as non-associative generalizations of topological groups and were studied in \cite{Ar02}, \cite{Gul}, \cite{Lin13}, \cite{LLL}, \cite{LS11}, \cite{Ba15}, \cite{BL}.

 By \cite[3.2]{BR}, a topological space $X$ is rectifiable if and only if $X$ admits a continuous binary operation $\cdot:X\times X\to X$ with a two-sided unit $e\in X$ such that the map $h:X\times X\to X\times X$, $h:(x,y)\mapsto (x,xy)$, is a homeomorphism. Topological spaces endowed with such operation are called \index{topological lop}{\em topological lops} (see \cite[\S3]{BR}, \cite{Ba15}, \cite{BL}). 

 It is clear that the class of rectifiable spaces includes all topological groups.
 For topological groups the following equivalence was proved in \cite{NST} (see also \cite{Lin13} and \cite{LS11}).

 \begin{proposition}\label{p:dotSfan} A rectifiable space $X$ contains a (strong) $\dot S^\w$-fan if and only if $X$ contains a (strong) ${\bar S}^\w$-fan.
 \end{proposition}

 \begin{proof} Assume that $(S_n)_{n\in\w}$ is a (strong) ${\bar S}^\w$-fan in $X$.
 The space $X$, being rectifiable, admits a continuous binary operation $\cdot :X\times X\to X$ with a two-sided unit $e$ such that the map $h:X\times X\to X\times X$, $h(x,y)\mapsto (x,x\cdot y)$, is a homeomorphism.

Since $(S_n)_{n\in\w}$ is a ${\bar S}^\w$-fan, the set $\{\lim S_n\}_{n\in\w}$ has compact closure $L$ in $X$. For every $n\in\w$ consider the homeomorphism $h_n:X\to X$, $h_n:x\mapsto (\lim S_n)\cdot x$ and observe that $h_n^{-1}(\lim S_n)=e$. Then the preimage $S_n'=h_n^{-1}(S_n)$ is a sequence in $X$, convergent to the unit $e$ on $X$. We claim that $(S'_n)_{n\in\w}$ is a (strong) $\dot S^\w$-fan in $X$.

Given a compact set $K\subset X$ and $n\in\w$, observe that the set $S_n'=h_n^{-1}(S_n)$ meets $K$ if and only if the set $S_n$ meets the compact set $h_n(K)\subset L\cdot K$. The compact-finiteness of the family $(S_n)_{n\in\w}$ implies that the sets $\{n\in\w:K\cap S_n'\}\subset\{n\in\w:S_n\cap L\cdot K\ne\emptyset\}$ is finite. This means that the family $(S'_n)_{n\in\w}$ is compact-finite and hence $(S'_n)_{n\in\w}$ is a $\dot S^\w$-fan.

If the $\bar S^\w$-fan $(S_n)_{n\in\w}$ is strong, then every set $S_n$ is contained in an $\IR$-open set $U_n\subset X$ such that the family $(U_n)_{n\in\w}$ is compact-finite. Then for every $n\in\w$ the preimage $U_n'=h_n^{-1}(U_n)$ is $\IR$-open in $X$ and the family $(U_n')_{n\in\w}$ is compact-finite in $X$, witnessing that $(S_n)_{n\in\w}$ is a strong $\dot S^\w$-fan in $X$.
\end{proof}


\begin{theorem}\label{t:rec-fan} If a rectifiable space $X$ contains a strong $D_\w$-cofan and (strong) $\bar S^\w$-fan, then $X$ contains a (strong) $\Fin^\w$-fan.
\end{theorem}

\begin{proof} By \cite[3.2]{BR}, the space $X$, being rectifiable, admits a continuous binary operation $\cdot:X\times X\to X$ with a two-sided unit $e\in X$ such that the map $h:X\times X\to X\times X$, $h:(x,y)\mapsto (x,xy)$, is a homeomorphism. For any points $x,y\in X$ let $x^{-1}y$ denote the unique point of $X$ such that $h^{-1}(x,y)=(x,x^{-1}y)$. It follows that the binary operation $X\times X\to X$, $(x,y)\mapsto x^{-1}y$, is continuous and satisfies the identities $x(x^{-1}y)=y=x^{-1}(xy)$ for all $x,y\in X$.

Let $(D_n)_{n\in\w}$ be a strong $D_\w$-cofan in $X$ and $(S_n)_{n\in\w}$ be a (strong) $\dot S^\w$-fan in $X$ (which exists by Proposition~\ref{p:dotSfan}). Since the rectifiable space $X$ is topologically homogeneous, we can assume that the sequences $S_n$, $n\in\w$, and $(D_n)_{n\in\w}$ converge to the unit $e$ of $X$. For every $n\in\w$ write the strict compact-finite set $D_n$ as $D_n=\{x_{n,m}\}_{m\in\w}$ and for every $n,m\in\w$ find an $\IR$-open neighborhood $U_{n,m}\subset X$ of $x_{n,m}$ such that the family $(U_{n,m})_{m\in\w}$ is strongly compact-finite. Replacing each sequence $D_n=\{x_{n,m}\}_{m\in\w}$ by suitable subsequence, we can assume that $e\notin U_{n,m}$ for all $n,m\in\w$.

For every $n\in\w$ let $\{y_{n,m}\}_{m\in\w}$ be an injective enumeration of the convergent sequence $S_n\setminus\{e\}$.
If the fan $(S_n)_{n\in\w}$ is strong, then each convergent sequence $S_n$ has an $\IR$-open neighborhood $V_n\subset X$ such that the family $(V_n)_{n\in\w}$ is compact-finite in $X$. If the fan $(S_n)_{n\in\w}$ is not strong, then we put $V_n=X$ for all $n\in\w$.

For every $n,m\in\IN$ the left shift $z\mapsto x_{n,m}\cdot z$ is a homeomorphism of $X$, which implies that the set $x_{n,m}^{-1}U_{n,m}=\{z\in Z: x_{n,m}\cdot z\in U_{n,m}\}$ is an $\IR$-open neighborhood of the point $e=\lim_{k\to\infty}y_{n,m}$. So, we can choose a number $k_m>m$ such that $y_{n,k_m}\in x_{n,m}^{-1}U_{n,m}$ and $x_{n,m}\cdot y_{n,k_m}\in U_{n,m}$.

It follows that for every $n,m\in\w$ the set $W_{n,m}=U_{n,m}\cap (x_{n,m}\cdot V_n)$ is an $\IR$-open neighborhood of the point $z_{n,m}=x_{n,m}\cdot y_{n,k_m}$.
We claim that the family $(\{z_{n,m}\})_{n,m\in\w}$ is compact-finite in $X$.

 Fix any compact subset $K\subset X$.
For every $n\in\w$ the family $(U_{n,m})_{m\in\w}$ is compact-finite, which implies that the sets $\{m\in\w:z_{n,m}\in K\}\ne\emptyset\}\subset\{m\in\w:K\cap U_{n,m}\ne\emptyset\}$ are finite and hence the set $C=\{e\}\cup\{x_{n,m}:n,m\in\w,\;z_{n,m}\in K\}$ is compact.

It follows that the set $C^{-1}K=\{x^{-1}y:x\in C,\;y\in K\}$ is compact and by the choice of the family $(S_n)_{n\in\w}$ there is $N\in\w$ such that $(C^{-1}K)\cap S_{n}=\emptyset$ for every $n>N$. Then the set  $\{(n,m)\in\w\times\w:z_{n,m}\in K\}
\subset\bigcup_{n=1}^N\{m\in \w:K\cap U_{n,m}\ne\emptyset\}$ is finite, witnessing that the family $(\{z_{n,m}\})_{n,m\in\w}$ is compact-finite in $X$.

If the family $(V_n)_{n\in\w}$ is compact-finite, then by the same argument we can prove that the family $(W_{n,m})_{n,m\in\w}$ is compact-finite in $X$ and hence the family $(\{z_{n,m}\})_{n,m\in\w}$ is strongly compact-finite in $X$.

Finally, we check that the family   $(\{z_{n,m}\})_{n,m\in\w}$ is not locally finite at $e$. Take any neighborhood $O_e\subset X$ of $e$. By the continuity of the binary operation, there is a neighborhood $U_e\subset X$ of $e=ee$ such that $U_e\cdot U_e\subset O_e$. Since the sets $D_n=\{x_{n,m}\}_{m\in\IN}$ converge to $e$, there is a number $n\in\w$ such that $\{x_{n,m}\}_{m\in\IN}\subset U_e$. Since $\lim_{k\to\infty}y_{n,k}=e\in U_e$, we can choose a number $m_0$ such that $y_{n,k_m}\in U_e$ for all $m\ge m_0$. Then for every $m\ge m_0$ we get $x_{n,m}\cdot y_{n,k_m}\subset U_e\cdot U_e\subset O_e$, which means that the family $(\{z_{n,m}\})_{n,m\in\w}$ is not locally finite at $e$ and hence this family is a $\Fin^\w$-fan in $X$. If the $\dot S^\w$-fan $(S_n)_{n\in\w}$ is strong, then the family $(W_{n,m})_{n,m\in\w}$ is compact-finite, which implies that the $\Fin^\w$-fan $(\{z_{n,m}\})_{n,m\in\w}$ is strong.
\end{proof}

Theorem~\ref{t:rec-fan} and Propositions~\ref{p:k-metr}, \ref{p:decomp1} imply:

\begin{corollary} If a rectifiable Tychonoff space $X$ is $k^*$-metrizable and contains no $\Fin^\w$-fan, then $X$ is $k$-homeomorphic to a metrizable space or to a topological sum of cosmic $k_\w$-spaces.
\end{corollary}

Combining this corollary with Corollary~\ref{c:fan-in-aleph-space}, we get

\begin{corollary}\label{c:rect-aleph} If a rectifiable Tychonoff $\aleph$-space $X$ contains no strong $\Fin^\w$-fan, then $X$ is $k$-homeomorphic to a metrizable space or to a topological sum of cosmic $k_\w$-spaces.
\end{corollary}

\begin{corollary}\label{c:rect-aleph0} If a rectifiable $\aleph_0$-space $X$ contains no strong $\Fin^\w$-fan, then $X$ is either $k$-metrizable or hemicompact.
\end{corollary}

Since each topological group is rectifiable, Theorem~\ref{t:rec-fan} implies:

\begin{corollary}\label{c:g->fan} If a topological group $G$ contains a strong $D_\w$-cofan and (strong) $\bar S^\w$-fan, then $G$ contains a (strong) $\Fin^\w$-fan.
\end{corollary}

\chapter{Strong $\Fin$-fans in functions spaces}\label{ch:Funct}

In this section we shall construct strong $\Fin$-fans in function spaces $C_\K(X)$. Here for topological spaces $X,Y$ by $C(X,Y)$ we denote the linear space of all continuous functions $f:X\to Y$. Each class $\K$ of compact spaces determines the \index{topology!$\K$-open}{\em $\K$-open topology} $\tau_\K$ on $C(X,Y)$ generated by the sub-base consisting of the sets $$[K;U]=\{f\in C(X,Y):f(K)\subset U\}$$where $K\in\K$ is a compact subset of $X$ and $U$ is an open set in $Y$. The topological space $(C(X,Y),\tau_\K)$ is denoted by $C_\K(X,Y)$. If  $\K$ is the class of all compact (finite) subsets of $X$, then the function space $C_\K(X,Y)$ is denoted by $C_k(X,Y)$ (resp. $C_p(X,Y)$). If $Y$ is the real line, then we write $C_\K(X)$ instead of $C_\K(X,\IR)$.

The construction of the function space $C_\K(X)$ determines a contravariant functor from the category $\Top$ of topological spaces and their continuous maps to the category of locally convex linear topological spaces and their linear operators.
To each continuous map $f:X\to Y$ between topological spaces the functor $C_\K$ assigns the linear operator $C_\K f:C_\K(Y)\to C_\K(X)$, $C_\K f:\varphi\mapsto \varphi\circ f$. If the class $\K$ is closed under continuous images, then the linear operator $C_\K$ is continuous, which means that $C_\K$ is a functor from the category $\Top$ to the category $\LCS$ of locally convex spaces and their linear continuous operators.

In particular $C_k$ and $C_p$ are contravariant functors from $\Top$ to $\LCS$.

In this chapter we assume that $\K$ is a class of compact topological spaces containing all countable compact spaces with at most one non-isolated point.

\section{A duality between $D_\w$-cofans and $S^\w$-fans in function spaces}

The following two propositions display certain kind of functional duality between strong $D_\w$-cofans and strong $\dot S^\w$-fans.

\begin{proposition}\label{p:XD->FS} If a topological space $X$ contains a strong $D_\w$-cofan, then the function space $C_\K(X)$ contains a strong $\dot S^\w$-fan $(S_n)_{n\in\w}$. Moreover, if the space $X$ is zero-dimensional, then $\bigcup_{n=1}^\infty S_n\subset C_\K(X,2)\subset C_\K(X)$.
\end{proposition}

\begin{proof} Let $(D_n)_{n\in\w}$ be a strong $D_\w$-cofan converging to a point $x_\infty$ of the space $X$. 
For every $n\in\w$ the set $D_n$ is strongly compact-finite in $X$. Hence each point $x\in D_n$ has an $\IR$-open neighborhood $U_{n}(x)\subset X$ such that the family $\big(U_n(x)\big)_{x\in D_n}$ is compact-finite in $X$.

For every $n\in\w$ we can inductively choose a sequence of points $(x_{n,m})_{m\in\w}$ in $D_n$ such that $U_n(x_{n,m})\cap\{x_{n,k}\}_{k<m}=\emptyset$ for every $m\in\w$.
By definition of an $\IR$-open set, for every $n,m\in\w$ there exists a continuous function $f_{n,m}:X\to[0,1]$ such that $f_{n,m}(x_{n,m})=1$ and $f_{n,m}^{-1}(X\setminus U_{n}(x_{n,m}))\subset\{0\}$. If the space $X$ is zero-dimensional, then we can additionally assume that $f_{n,m}(X)\subset \{0,1\}$. For every $n\in\w$ the compact-finiteness of the family $(U_n(x))_{x\in D_n}$ implies that the sequence of functions $S_n=\{f_{n,m}\}_{m\in\w}$ converges to zero in the function space $C_\K(X)$. For every $n,m\in\w$ consider the open neighborhood $$W_{n,m}=\{g\in C_k(X):g(x_{n,m})>\tfrac34\mbox{ \ and \ }g(x_{n,k})<\tfrac14 \mbox{ for all $k<m$}\}$$ of the function $f_{n,m}$ in $C_\K(X)$. We claim that the family of open sets  $W_{n}=\bigcup_{m=1}^\infty W_{n,m}$, $n\in\w$, is compact-finite in $C_\K(X)$. Assuming that $(W_n)_{n\in\w}$ is not compact-finite, we can find a compact set $K\subset C_\K(X)$ such that the set $\Omega=\{n\in\w:K\cap W_n\ne\emptyset\}$ is infinite.

For every $n\in\Omega$ choose a function $g_n\in K\cap W_n$ and find a number $m_n\in\w$ such that $g_n\in W_{n,m_n}$. By the compactness of $K$, the sequence $(g_n)_{n\in\Omega}$ has an accumulation point $g_\infty$. The convergence of the sequence $(D_n)_{n\in\w}$ to $x_\infty$ implies that the set $C=\{x_\infty\}\cup\bigcup_{n\in\Omega}\{x_{n,m}\}_{m\le m_n}$ is a compact convergent sequence in $X$. So, $C\in\K$ and the set $W_\infty=\{g\in C_\K(X):\sup_{x\in C}|g(x)-g_\infty(x)|<\frac16\}$ is a neighborhood of the function $g_\infty$ in $C_\K(X)$. The continuity of $g_\infty$ at $x_\infty$ yields a neighborhood $O_x\subset X$ whose image $g_\infty(O_x)$ has diameter $<\frac16$ in the real line $\IR$. Since $g_\infty$ is an accumulation point of the sequence $(g_n)_{n\in\w}$, the neighborhood $W_\infty$ of $g_\infty$ contains some function $g_n$, $n>1$. For this function $g_n$ we get $$\diam(g_n(O_x\cap C))\le \diam(g_\infty(O_x\cap C))+2\sup_{x\in C}|g_n(x)-g_\infty(x)|\le \frac16+\frac26=\frac12.$$ In particular, $|g_n(x_{n,m_n})-g_n(x_{n,m_{n-1}})|\le \frac12$, which contradicts the inclusion $g_n\in W_{n,m_n}$. This contradiction shows that the family $(W_n)_{n\in\w}$ is compact-finite in $C_\K(X)$ and hence the family $(S_n)_{n\in\w}$ is a strong $\dot S^\w$-fan in $C_\K(X)$.
\end{proof}

\begin{proposition}\label{p:s->d} If a topological space $X$ contains a strong $\bar S^\w$-fan $(S_n)_{n\in\w}$, then the function space $C_\K(X)$ contains a strong $D_\w$-cofan $(D_n)_{n\in\w}$. If the space $X$ is zero-dimensional, then  $\bigcup_{n\in\w}D_n\subset C_\K(X,2)\subset C_\K(X)$.
\end{proposition}

\begin{proof} Let $(S_n)_{n\in\w}$ be a strong $\bar S^\w$-fan in $X$. 
Since the fan $(S_n)_{n\in\w}$ is strong, each set $S_n$ is contained in an $\IR$-open set $U_n\subset X$ such that the family $(U_n)_{n\in\w}$ is compact-finite in $X$.
Replacing the sequence $(S_n)_{n\in\w}$ by a suitable subsequence, we can assume that the family $(S_n)_{n\in\w}$ is disjoint, and the union $\bigcup_{n\in\w}U_n$ does not intersect the (compact) closure of the set  $\{\lim S_n\}_{n\in\w}$.

For every $n\in\w$ denote by $x_n$ the limit point of the convergent sequence $S_n$ and let $(x_{n,m})_{m\in\w}$ be an injective enumeration of the set $S_n\setminus\{x_n\}$.
For every $n,m\in\w$ the set $U_n$ is an $\IR$-open neighborhood of the points $x_{n,m}$, $m\in\w$, which allows us to construct a continuous function $f_{n,m}:X\to \IR$ such that $f(x_{n,k})=1$ for all $k\le m$ and $f_{n,m}(X\setminus U_{n})\subset\{0\}$. If the space $X$ is zero-dimensional, then we can additionally require that $f_{n,m}(X)\subset\{0,1\}$. Since the union $\bigcup_{n\in\w}U_n$ does not intersect the set $\{\lim S_n\}_{n\in\w}=\{x_n\}_{n\in\w}$, for every $n\in \w$ the point $x_n$ does not belong to the $\IR$-open set $U_n$ and hence $f_{n,m}(x_n)\in f_{n,m}(X\setminus U_n)\subset\{0\}$.
Then for every $n,m\in\w$ the set
$$W_{n,m}=\{f\in C_\K(X):\mbox{$f(x_{n})<\tfrac14$ and $f(x_{n,k})>\tfrac34$ for all $k\le m$}\}$$is an open neighborhood of the function $f_{n,m}$ in $C_\K(X)$. Since the space $C_\K(X)$ is Tychonoff, the open set $W_{n,m}$ is $\IR$-open.

We claim that for every $n\in\w$ the family $(W_{n,m})_{m\in\w}$ is compact-finite in $C_\K(X)$. To derive a contradiction, assume that for some compact set $K\subset C_\K(X)$ the family $\Omega=\{m\in\w:K\cap W_{n,m}\ne\emptyset\}$ is infinite. For every $m\in\Omega$ choose a function $g_m\in K\cap W_{n,m}$. By the compactness of $K$ the sequence $(g_m)_{m\in\w}$ has an accumulation point $g_\infty$ in $K$. By our assumption, the compact convergent sequence $\bar S_n$ belongs to the family $\K$. So, the set $[\bar S_n,\frac16]=\{g\in C_\K(X):\sup_{x\in \bar S_n}|g(x)-g_\infty(x)|<\frac16\}$ is a neighborhood of $g_\infty$ in the function space $C_\K(X)$.
The continuity of the function $g_\infty$ yields a neighborhood $O(x_n)$ of the point $x_n$ such that the set $g_\infty(O(x_n))$ has diameter $<\frac16$ in the real line. Choose a number $m_0$ such that $\{x_{n,m}\}_{m\ge m_0}\subset O(x_n)$.
 Since $g_\infty$ is an accumulation point of the sequence $(g_k)_{k\in\w}$, the neighborhood $[\bar S_n,\frac16]$ contains some function $g_k$ with $k\ge m_0$. Then we get the upper bound $$|g_k(x_n)-g_k(x_{n,m})|\le|g_\infty(x_n)-g_\infty(x_{n,m})|+2\sup_{x\in\bar S_n}|g_\infty(x)-g_k(x)|\le \frac16+\frac26=\frac12,$$which contradicts the inclusion $g_k\in W_{n,m}$. This contradiction shows that the family $(W_{n,m})_{m\in\w}$ is compact-finite and hence the set $D_n=\{f_{n,m}\}_{m\in\w}$ is strongly compact-finite in the function space $C_\K(X)$.

 The compact-finiteness of the family $(U_n)_{n\in\w}$ implies that the sequence $(D_n)_{n\in\w}$ converges to zero in $C_\K(X)$ and hence $(D_n)_{n\in\w}$ is a strong $D_\w$-cofan in $C_\K(X)$.
 \end{proof}

For functionally Hausdorff spaces Proposition~\ref{p:s->d} can be improved by replacing the strong $\bar S^\w$-fan by a strong $S^\w$-semifan. We recall that a topological space $X$ is functionally Hausdorff if for any distinct points $x,y\in X$ there is a continuous function $f:X\to[0,1]$ such that $f(x)=0$ and $f(y)=1$.

\begin{proposition}\label{p:XS->FD} If a functionally Hausdorff space $X$ contains a strong $S^\w$-semifan, then the function space $C_\K(X)$ contains a strong $D_\w$-cofan $(D_n)_{n\in\w}$. If the space $X$ is zero-dimensional, then  $\bigcup_{n\in\w}D_n\subset C_\K(X,2)\subset C_\K(X)$.
\end{proposition}

\begin{proof} Let $(S_n)_{n\in\w}$ be a strong $S^\w$-semifan in $X$.
By definition, every set $S_n$ has an $\IR$-open neighborhood $U_n\subset X$ such that the family $(U_n)_{n\in\w}$ is compact-finite in $X$. For every $n\in\w$ denote by $x_n$ the limit point of the convergent sequence $S_n$ and let $(x_{n,m})_{m\in\w}$ be an injective enumeration of the set $S_n\setminus\{x_n\}$.

For every $n,m\in\w$ the set $U_n$ is an $\IR$-open neighborhood of the points $x_{n,m}$, $m\in\w$, which allows us to construct a continuous function $\tilde f_{n,m}:X\to \IR$ such that $\tilde f(x_{n,k})=1$ for all $k\le m$ and $\tilde f_{n,m}(X\setminus U_{n})\subset\{0\}$.
Using the functional Hausdorff property of $X$, choose a continuous function $\lambda_{n,m}:X\to [0,1]$ such that $\lambda_{n,m}(x_n)=0$ and $\lambda(\{x_{n,k}\}_{k\le m})\subset\{1\}$. The the function $f_{n,m}=\lambda_{n,m}\cdot \tilde f_{n,m}$ has properties: $f_{n,m}(\{x_{n,k}\}_{k\le m})\subset\{1\}$ and $f_{n,m}(U_n\cup\{x_n\})\subset\{0\}$.

 If the space $X$ is zero-dimensional, then we can additionally require that $\tilde f_{n,m}(X)\cup \lambda_{n,m}(X)\subset\{0,1\}$ and hence $f_{n,m}(X)\subset\{0,1\}$. Proceeding as in the proof of Proposition~\ref{p:s->d}, we can prove that for every $n\in\w$ the set $D_n=\{f_{n,m}\}_{m\in\w}$ is strongly compact-finite in $C_\K(X)$ and the sequence $(D_n)_{n\in\w}$ converge to zero, so is a strong $D_\w$-cofan in $C_k(X)$.
 \end{proof}

There is also a duality between strict $\Cld^\w$-fans in a space and  strong $D_\w$-cofan in its function space.

\begin{proposition}\label{p:XCld->FD} Assume that a topological space $X$ contains a strict $\Cld^\w$-fan $(F_n)_{n\in\w}$ such that the union $\bigcup_{n\in\w}F_n$ is contained in the closure of the union $\bigcup_{n\in\w}K_n$ of some countable subfamily $\{K_n\}_{n\in\w}\subset \K$. Then the function space $C_\K(X)$ contains a strong $D_\w$-cofan $(D_n)_{n\in\w}$, which is contained in the space $C_\K(X,2)$ if the $\Cld^\w$-fan $(F_n)_{n\in\w}$ is a $\Clop^\w$-fan.
\end{proposition}

\begin{proof} Let $x\in X$ be a point at which the $\Cld^\w$-fan $(F_n)_{n\in\w}$ is not locally finite. For every $n\in\w$ choose a functional neighborhood $U_n\subset X$ of $F_n$ such that the family $(U_n)_{n\in\w}$ is compact-finite. Then for some $m$ the union $\bigcup_{n\ge m}U_n$ does not contain the point $x$. Replacing $(F_n)_{n\in\w}$ by a suitable subsequence, we can assume that $m=0$ and hence $x\notin\bigcup_{n\in\w}U_n$. Since $\bigcup\K=X$, we can attach the singleton $\{x\}$ to the family $\{K_n\}_{n\in\w}$ and assume that $x\in K_0$. For every $n\in\w$ choose a function $f_n:X\to[0,1]$ such that $f(F_n)\subset \{1\}$ and $f_n(X\setminus U_n)\subset\{0\}$. If some set $F_n$ is clopen in $X$ then we shall additionally assume that $f_n$ is the characteristic function of $F_n$ in $X$ and hence $f_n\in C_\K(X,2)$.

For any numbers $n\le m$ consider the continuous function $f_{n,m}=\max_{n\le k\le m}f_k\in C_\K(X)$ and its neighborhood  $$W_{n,m}=\bigcap_{i\le m}\{f\in C_\K(X):\max_{x\in K_i}|f(x)-f_{n,m}(x)|<\tfrac14\}$$  in $C_\K(X)$. For every $n\in\w$ the family $(W_{n,m})_{m\ge n}$ witnesses that the set $D_n=\{f_{n,m}\}_{m\ge n}$ is strongly compact-finite in $C_\K(X)$ (any accumulation point $f_\infty$ of the sequence $(W_{n,m})_{m\ge n}$ in $\IR^X$ has $f(x)=0$ and $f|\bigcup_{m\ge n}F_n=1$ and hence is discontinuous). The compact-finiteness of the family $(U_n)_{n\in\w}$ implies that the sequence $(D_n)_{n\in\w}$ converges to zero in $C_\K(X)$ and hence is a strong $D_\w$-cofan in $C_\K(X)$. If the fan $(F_n)_{n\in\w}$ is a $\Clop^\w$-fan, then $\bigcup_{n\in\w}D_n\subset C_\K(X,2)$ by the choice of the functions $f_n$, $n\in\w$.
\end{proof}

\begin{proposition}\label{p:XI->FD}  If a topological space $X$ contains a strongly compact-finite infinite set $D\subset X$, then the  function space $C_\K(X)$ contains a strong $D_\w$-cofan.
\end{proposition}

\begin{proof} Since $D$ is strongly compact-finite in $X$, each point $x\in D$ has an $\IR$-open neighborhood $U_x\subset X$ such that the family $(U(x))_{x\in D}$ is compact-finite in $X$. Inductively we can choose a sequence of points $(x_n)_{n\in\w}$ in $D$ such that $U(x_n)\cap \{x_k\}_{k<n}=\emptyset$ for every $n\in\w$.

For every $n\in\w$ choose a continuous function $f_n:X\to[0,1]$ such that $f_n(x_n)=1$ and $f_n(X\setminus U(x_n))\subset\{0\}$. It is easy to see that the set $D_n=\{m\cdot f_n\}_{m\in\w}$ is strongly compact-finite in $C_\K(X)$ and the family $(D_n)_{n\in\w}$ converges to zero in $C_\K(X)$. This means that $(D_n)_{n\in\w}$ is a strong $D_\w$-cofan in $X$.
\end{proof}

\section{Function spaces $C_k(X)$}

In this section we apply the results proved in the preceding sections to the problem of detecting function spaces containing no $\Fin^\w$-fans.

\begin{theorem}\label{t:F-prop} Assume that for a topological space $X$ the function space $C_\K(X)$ contains no strong $\Fin^\w$-fan. Then:
\begin{enumerate}
\item[\textup{1)}] The space $X$ contains no strong $D_\w$-cofan.
\item[\textup{2)}] If $X$ is cosmic, then $X$ is $\sigma$-compact;
\item[\textup{3)}] If $X$ is an $\aleph_0$-space, then $X$ is hemicompact;
\item[\textup{4)}] If $X$ is a sequential $\aleph_0$-space, then $X$ is a $k_\w$-space;
\item[\textup{5)}] If $X$ is a $\mu$-complete Tychonoff $\bar\aleph_k$-space, then $X$ is a $k$-sum of hemicompact spaces;
\item[\textup{6)}] If $X$ is a $\mu$-complete Tychonoff $\bar\aleph_k$-$k_\IR$-space, then $X$ is a topological sum of cosmic $k_\w$-spaces.
\end{enumerate}
\end{theorem}

\begin{proof} First we show that $X$ contains no $D_\w$-fan. To derive a contradiction, assume that the space $X$ contains a strong $D_\w$-cofan. Then $X$ contains a strongly compact-finite infinite subset. By Propositions~\ref{p:XD->FS} and \ref{p:XI->FD}, the function space $C_\K(X)$ contains a strong $\dot S^\w$-fan and a strong $D_\w$-cofan. Since $C_\K(X)$ is a topological group, we can apply Corollary~\ref{c:g->fan} and conclude that the space $C_\K(X)$ contains a strong $\Fin^\w$-fan. This is the desired contradiction completing the proof of the first statement.

All other statements follow from Theorem~\ref{t:Dcofan}.
\end{proof}

Theorem~\ref{t:F-prop} implies the following characterization extending the results of Arens \cite{Ar46}, Pol \cite{Pol74} and Gabriyelyan \cite[1.2]{Gab}.

\begin{corollary}\label{c:aleph0Ck} For an $\aleph_0$-space $X$ the following conditions are equivalent:
\begin{enumerate}
\item[\textup{1)}] the function space $C_k(X)$ is metrizable;
\item[\textup{2)}] $C_k(X)$ contains no strong $\Fin^\w$-fan;
\item[\textup{3)}] the space $X$ is hemicompact.
\end{enumerate}
\end{corollary}

\begin{proof} The implication $(1)\Ra(2)\Ra(3)$ follow from Proposition~\ref{k-no-Cld-fan} and Theorem~\ref{t:F-prop}(3) and $(3)\Ra(1)$ is a classical result of Arens \cite{Ar46} (see \cite[3.4.E]{En}).
\end{proof}

This characterization is completed by the following characterization, which generalizes Theorem 1.4 in \cite{GKP}. In this characterization by $d(X)$ we denote the density of a topological space (equal to the smallest cardinality of a dense set in $X$).

\begin{theorem}\label{t:Ck-char} For a Tychonoff $\mu$-complete $\bar\aleph_k$-$k_\IR$-space $X$ the following conditions are equivalent:
\begin{enumerate}
\item[\textup{1)}] $C_k(X)$ is a $k_\IR$-space;
\item[\textup{2)}] $C_k(X)$ is Ascoli;
\item[\textup{3)}] $C_k(X)$ contains no strong $\Fin^\w$-fan;
\item[\textup{4)}] $X$ is a topological sum of cosmic $k_\w$-spaces;
\item[\textup{5)}] the function space $C_k(X)$ is homeomorphic to the space $\IR^{d(X)}$.
\end{enumerate}
\end{theorem}

\begin{proof} The implication $(5)\Ra(1)$ follows an old result of Noble \cite{Nob70} who proved that a Tychonoff product of metrizable spaces is a $k_\IR$-space. The implications $(1)\Ra(2)$ was proved in \cite{Nob69} and $(2)\Ra(3)$ follows from Corollary~\ref{c:A->noFan}, $(3)\Ra(4)$ was proved in Theorem~\ref{t:F-prop}(6). To prove the final implication $(4)\Ra(5)$, assume that $X=\bigoplus_{\alpha\in A}X_\alpha$ is a topological sum of cosmic $k_\w$-spaces $X_\alpha$. Then $C_k(X)$ is topologically isomorphic to the Tychonoff product $\prod_{\alpha\in A}C_k(X_\alpha)$ of the separable Fr\'echet spaces $C_k(X_\alpha)$. Applying the classical Anderson-Kadec Theorem \cite[VI.5.2]{BP75}, we can prove that for every $\alpha\in A$ the Fr\'echet space $C_k(X_\alpha)$ is homeomorphic to the product $\IR^{d(X_\alpha)}$ of $d(X_\alpha)$ many real lines. Then $C_k(X)$ is homeomorphic to $\prod_{\alpha\in A}\IR^{d(X_\alpha)}=\IR^{d(X)}$.
\end{proof}

\section{Function spaces $C_k(X,2)$}\label{s:Funct}

In this section we study the topological structure of the function spaces $C_\K(X,2)$ over zero-dimensional spaces $X$.

\begin{theorem}\label{t:C2} Assume that for a Hausdorff zero-dimensional  $\mu$-complete space $X$ the function space $C_\K(X,2)$ contains no strong $\Fin^\w$-fan. Then:
\begin{enumerate}
\item[\textup{1)}] $X$ contains no strong $D_\w$-cofan or no strong $S^\w$-semifan;
\item[\textup{2)}] If $X$ is an $\aleph$-space, then its $k$-modification $kX$ is either a topological sum of cosmic $k_\w$-spaces or $kX$ is metrizable and has compact set of non-isolated points.
\item[\textup{3)}] If $X$ is a $\bar\aleph_k$-$k_\IR$-space, then $X$ is either a topological sum of cosmic $k_\w$-spaces or $X$ is metrizable and has compact set of non-isolated points.
\item[\textup{4)}] If $X$ is cosmic, then $X$ is $\sigma$-compact.
\end{enumerate}
\end{theorem}

\begin{proof} First observe that $X$, being a zero-dimensional $T_0$-space, is Tychonoff.
 \smallskip

1. To prove the first statement (by contradiction), assume that the space $X$ contains a strong $D_\w$-cofan and a strong $S^\w$-semifan. By Propositions~\ref{p:XD->FS} and \ref{p:XS->FD}, the function space $C_\K(X,2)$ contains a strong $\dot S^\w$-fan and a strong $D_\w$-cofan. Since $C_\K(X,2)$ is a topological group, we can apply Corollary~\ref{c:g->fan} and conclude that the space $C_\K(X,2)$ contains a strong $\Fin^\w$-fan, which is a desired contradiction.
\smallskip

2. Assume that $X$ is an $\aleph$-space. By the preceding statement, the space $X$ contains no strong $D_\w$-cofan or no strong $S^\w$-semifan. If $X$ contains no strong $D_\w$-cofan, then by Proposition~\ref{p:decomp1}, the $k$-modification $kX$ of $X$ is a topological sum of cosmic $k_\w$-spaces. If $X$ contains no strong $S^\w$-semifan, then by Theorem~\ref{t:Ssemifan}(7), the $k$-modification $kX$ of $X$ is metrizable and has compact set of non-isolated points.
\smallskip

3. Assume that $X$ is a $\bar \aleph_k$-$k_\IR$-space. By the first statement, the space $X$ contains no strong $D_\w$-cofan or no strong $S^\w$-semifan. If $X$ contains no strong $D_\w$-cofan, then by Proposition~\ref{p:decomp1}, $X$ is a topological sum of cosmic $k_\w$-spaces. If $X$ contains no strong $S^\w$-semifan, then by Theorem~\ref{t:kR-noSemifan}, the space $X$ is metrizable and has compact set of non-isolated points.

4. If the space $X$ is cosmic, then by the first statement, $X$ contains no strong $D_\w$-cofan or no strong $S^\w$-semifan and by Theorem~\ref{t:Dcofan}(4) and \ref{t:Ssemifan}(5), $X$ is $\sigma$-compact.
\end{proof}

Theorem~\ref{t:C2} implies the following characterization of zero-dimensional spaces $X$ whose function spaces $C_k(X,2)$ are Ascoli or $k_\IR$-spaces. For metrizable spaces this characterization was proved by Gabriyelyan \cite[1.4]{Gab2}.

\begin{theorem}\label{t:C2-kR} For a zero-dimensional $\mu$-complete $\bar\aleph_k$-$k_\IR$-space $X$ and its function space $C_k(X,2)$ the following conditions are equivalent:
\begin{enumerate}
\item[\textup{1)}] $C_k(X,2)$ is a $k_\IR$-space;
\item[\textup{2)}] $C_k(X,2)$ is Ascoli;
\item[\textup{3)}] $C_k(X,2)$ contains no strong $\Fin^\w$-fan;
\item[\textup{4)}] $X$ is either a topological sum of cosmic $k_\w$-spaces or $X$ is metrizable with compact set $X'$ of non-isolated points.
\end{enumerate}
\end{theorem}

\begin{proof} The implication $(1)\Ra(2)$ is proved by Noble \cite{Nob69}, $(2)\Ra(3)$ follows from Corollary~\ref{c:A->noFan}, and $(3)\Ra(4)$ follows from Theorem~\ref{t:C2}. It remains to prove the implication $(4)\Ra(1)$. If $X$ is metrizable and has compact set of non-isolated points, then the function space $C_k(X,2)$ is a $k_\IR$-space by \cite[1.4]{Gab}(i). So, we assume that $X$ is a topological sum $X=\oplus_{\alpha\in A}X_\alpha$ of cosmic $k_\w$-spaces. In this case the space $C(X,2)$ is homeomorphic to the Tychonoff product $\prod_{\alpha\in A}C_k(X_\alpha,2)$. By \cite{Ar46}, each space $C_k(X_\alpha,2)$ is metrizable and by a result of Noble \cite{Nob70}, the Tychonoff product $\prod_{\alpha\in A}C_k(X_\alpha,2)$ is a $k_\IR$-space.
\end{proof}

Next, we characterize $\bar\aleph_k$-$k_\IR$-spaces whose function spaces $C_k(X,2)$ are $k$-spaces.
For metrizable spaces the following characterization was proved by Gabriyelyan \cite{Gab}.

\begin{theorem}\label{t:C2-k} For a zero-dimensional $\mu$-complete $\bar\aleph_k$-$k_\IR$-space $X$ its function space $C_k(X,2)$ is a $k$-space if and only if $X$ is either a topological sum $K\oplus D$ of a cosmic $k_\w$-space $K$ and a discrete space $D$ or else $X$ is metrizable and has compact set $X'$ of non-isolated points.
\end{theorem}

\begin{proof} To prove the ``only if'' part, assume that $C_k(X,2)$ is a $k$-space. By Theorem~\ref{t:C2-kR}, $X$ is either metrizable and has compact set of non-isolated points or $X$ is a topological sum of cosmic $k_\w$-spaces. In the first case we are done. Next, assume that $X$ is a topological sum $\oplus_{\alpha\in A}X_\alpha$ of non-empty cosmic $k_\w$-spaces $X_\alpha$ and hence $C_k(X,2)$ is homeomorphic to the Tychonoff product $\prod_{\alpha\in A}C_k(X_\alpha,2)$. Let $B=\{\alpha\in A:X_\alpha$ is non-discrete$\}$. Observe that for every $\alpha\in B$ the function space $C(X_\alpha,2)$ is a non-compact metrizable topological group and hence $C(X_\alpha,2)$ contains a closed topological copy of the countable discrete space $\IN$. If the set $B$ is uncountable, then the Tychonoff product $\prod_{\alpha\in A}C_k(X_\alpha,2)$ contains a closed subspace homeomorphic to $\IN^{\w_1}$.
 By \cite[3.3.E]{En}, the space $\IN^{\w_1}$ is not a $k$-space and then the spaces
 $\prod_{\alpha\in B}C_k(X_\alpha,2)$ and $C_k(X,2)\cong \prod_{\alpha\in A}C_k(X_\alpha,2)$ are not $k$-spaces. This contradiction shows that the set $B$ is countable and hence $X$ is a topological sum $K\oplus D$ of the cosmic $k_\w$-space $K=\oplus_{\alpha\in B}X_\alpha$ and the discrete space $D=\oplus_{\alpha\in A\setminus B}X_\alpha$.
 \smallskip

If $X$ is metrizable and has compact set of non-isolated points, then by Theorem 1.4 of \cite{Gab}, the function space $C_k(X,2)$ is a $k$-space. If $X$ is a topological sum $K\oplus D$ of a cosmic $k_\w$-space $K$ and a discrete space $D$, then the function space $C_k(X,2)$ is homeomorphic to the product $C_k(K,2)\times C_k(D,2)$ of the metrizable space $C_k(K,2)$ and the compact space $C_k(D,2)$.
By \cite[3.3.27]{En}, this product is a $k$-space.
\end{proof}

Next, we characterize $\bar\aleph_k$-$k_\IR$-spaces whose function spaces $C_k(X,2)$ are sequential.
For metrizable spaces the following characterization was proved by Gabriyelyan \cite{Gab2}.

\begin{theorem}\label{t:C2-s} For a zero-dimensional $\mu$-complete $\bar\aleph_k$-$k_\IR$-space $X$ its function space $C_k(X,2)$ is sequential if and only if $X$ is either a cosmic $k_\w$-space or else $X$ is Polish and has compact set $X'$ of non-isolated points.
\end{theorem}

\begin{proof} If the function space $C_k(X,2)$ is sequential, then by Theorem~\ref{t:C2-k}, $X$ is either metrizable with compact set of non-isolated points or $X$ is a topological sum $K\oplus D$ of a cosmic $k_\w$-space $K$ and a discrete space $D$. In the first case we can apply Theorem 1.4(iii) of Gabriyelyan \cite{Gab2} and conclude that $X$ is Polish. In the second case we observe that $C_k(X,2)$ contains a closed subspace homeomorphic to the Cantor cube $2^D$. If the set $D$ is uncountable, then the cube $2^D$ contains a closed copy of the ordinal segment $[0,\w_1]$ and hence is not sequential. This means that the discrete space $D$ is at most countable and $X=K\oplus D$ is a cosmic $k_\w$-space.

If $X$ is a cosmic $k_\w$-space, then by a classical result of Arens \cite{Ar46}, the function space $C_k(X,2)$ is metrizable and hence sequential. If $X$ is a Polish space with compact set of non-isolated points, then by \cite[3.11]{GTZ}, the function space $C_k(X,2)$ is a sequential cosmic $k_\w$-space.
\end{proof}

Finally, we characterize $\bar\aleph_k$-$k_\IR$-spaces whose function spaces $C_k(X,2)$ are Fr\'echet-Urysohn. For metrizable spaces the following characterization was proved by Gabriyelyan \cite{Gab2}.

\begin{theorem}\label{t:C2-FU} For a zero-dimensional $\mu$-complete $\bar\aleph_k$-$k_\IR$-space $X$ its function space $C_k(X,2)$ is Fr\'echet-Urysohn if and only if $C_k(X,2)$ is metrizable if and only if $X$ is a cosmic $k_\w$-space.
\end{theorem}

\begin{proof} If $X$ is a cosmic $k_\w$-space, then by \cite{Ar46}, the function space $C_k(X,2)$ is metrizable and hence Fr\'echet-Urysohn. If $C_k(X,2)$ is Fr\'echet-Urysohn, then by Theorem~\ref{t:C2-s}, $X$ is either a cosmic $k_\w$-space or a Polish space with compact set of non-isolated points. In the first case we are done. In the second case we can apply Theorem 1.4(iv) of \cite{Gab} and conclude that the space $X$ is Polish and locally compact, and hence $X$ is a cosmic $k_\w$-space.
\end{proof}

In turns out that for function spaces $C_k(X,2)$ over zero-dimensional $\aleph_0$-spaces $X$ the properties distinguished in Theorems~\ref{t:C2-kR}--\ref{t:C2-s} are equivalent.

\begin{theorem}\label{t:C2-aleph0} For a zero-dimensional $\aleph_0$-space $X$ the following conditions are equivalent:
\begin{enumerate}
\item[\textup{1)}] $C_k(X,2)$ is either metrizable or a cosmic $k_\w$-space;
\item[\textup{2)}] the function space $C_k(X,2)$ is sequential;
\item[\textup{3)}] $C_k(X,2)$ is a $k$-space;
\item[\textup{4)}] $C_k(X,2)$ is Ascoli;
\item[\textup{5)}] $C_k(X,2)$ contains no strong $\Fin^\w$-fans;
\item[\textup{6)}] $X$ is either hemicompact or $X$ is Polish with compact set $X'$ of non-isolated points.
\end{enumerate}
\end{theorem}

\begin{proof} The implication $(6)\Ra(1)$ follows from \cite{Ar46} and \cite{GTZ}, and the implications  $(1)\Ra(2)\Ra(3)\Ra(4)\Ra(5)$ are trivial. It remains to prove that $(5)\Ra(6)$. Assume that the function space $C_k(X,2)$ contains no strong $\Fin^\w$-fans. If $X$ contains no strong $D_\w$-cofan, then by Proposition~\ref{p:XD->s+h+kw} the $\aleph_0$-space $X$ is hemicompact. So, we assume that the space $X$ contains a strong $D_\w$-cofan. Then by Proposition~\ref{p:XD->FS}, the function space $C_k(X,2)$ contains a strong $\dot S^\w$-fan. Since $C_k(X,2)$ is a topological group, we can apply Corollary~\ref{c:g->fan} and conclude that $C_k(X,2)$ contains no strong $D_\w$-cofan. By Propositions~\ref{p:XS->FD} and \ref{p:XCld->FD}, the space $X$ contains no strong $S^\w$-semifan and no $\Clop^\w$-fan and by Corollary~\ref{c:k*-withoutSw}, the space $X$ is metrizable and has compact set of non-isolated points. Being cosmic, the metrizable space $X$ is separable. Being the union of a compact and discrete space, the metrizable separable space $X$ is Polish.
 \end{proof}

\chapter{Strong $\Fin^\w$-fans in locally convex and Banach spaces with the weak topology}\label{ch:Banach}

In this chapter we construct strong $\Fin^\w$-fans in (subsets of) locally convex (Banach) spaces endowed with the weak topology. For a locally convex space $X$ by $X^*$ we denote the linear space of all linear continuous functionals on $X$. A locally convex space $X$ endowed with the weak topology is denoted by $(X,w)$. We recall that the {\em weak topology} on $X$ is the smallest topology on $X$ making all functionals $f\in X^*$ continuous.

\section{Strong $\Fin^\w$-fans in locally convex spaces with the weak topology}\label{s:weak-lc}

In \cite{GKP} Gabriyelyan, K\c akol and Plebanek proved that for a normed space $X$ the space $(X,w)$ is Ascoli if and only if $X$ is finite dimensional. The following theorem generalizes this result in two directions and answers a question posed in \cite{GKP}.

\begin{theorem} For a locally convex linear metric space $X$ the following conditions are equivalent:
\begin{enumerate}
\item[\textup{1)}] $(X,w)$ is metrizable.
\item[\textup{2)}] $(X,w)$ is Ascoli.
\item[\textup{3)}] $(X,w)$ contains no strong $\Fin^\w$-fan.
\item[\textup{4)}] The weak topology on $X$ coincides with the original topology of $X$.
\item[\textup{5)}] The dual space $X^*$ has at most countable Hamel basis.
\end{enumerate}
\end{theorem}

\begin{proof} The implications $(5)\Ra(1)\Ra(2)\Ra(3)$ are trivial or follow from Corollary~\ref{c:A->noFan}.

$(3)\Ra(4)$ Assume that the weak topology on $X$ does not coincide with the original topology of $X$. Then $X$ contains an open convex neighborhood $U$ of zero, which is not a neighborhood of zero in the weak topology.  Taking into account that $\bar U\subset U+U=2U$, we conclude that the set $\bar U$ and its homothetic copies $n\bar U$ are not neighborhoods of zero in the weak topology. Then for every $n\in\w$ the weak closure of the set $A_n=X\setminus n\bar U$ contains zero. By Theorem 1.2 \cite{GKZ}, the space $(X,w)$ has countable fan tightness. So, in each set $X\setminus n\bar U$ we can choose a finite subset $F_n\subset X\setminus n\bar U$ such that the union $\bigcup_{n\in\w}F_n$ contains zero in its closure, which implies that the family $(F_n)_{n\in\w}$ is not locally finite at zero. To see that this family  is a strong $\Fin^\w$-fan in $(X,w)$, it remains to check that $(F_n)_{n\in\w}$ is strongly compact-finite in $(X,w)$. Since closed convex subsets of $X$ are weakly closed, for every $n\in\w$ the set $X\setminus n\bar U$ is a weakly open neighborhood of the set $F_n$. Since each weakly compact subset of $(X,w)$ is bounded in $X$ (see \cite[3.31]{Os14}), the family $(X\setminus n\bar U)_{n\in\w}$ is compact-finite in $(X,w)$, which means that $(F_n)_{n\in\w}$ is a strong $\Fin^\w$-fan in $(X,w)$.
\smallskip

$(4)\Ra(5)$ Assume that the weak topology of $X$ coincides with the original topology of $X$. Using the metrizability of $X$, fix a countable neighborhood basis $(U_n)_{n\in\w}$ at zero in $X$. Since the weak topology  coincides with the original topology of $X$, for every  $n\in\w$ there is a finite subset $F_n\subset X^*$ such that $\bigcap_{f\in F_n}\{x\in X:|f_n(x)|<1\}\subset U_n$. We claim that the linear hull of the countable set $\bigcup_{n\in\w}F_n$ coincides with the dual space $X^*$ of $X$.
Given any functional $g\in X^*$, consider the weakly open neighborhood $U_g=\{x\in X:|g(x)|<1\}$ of zero, and find a basic neighborhood $U_n$ such that $U_n\subset U_g$. It follows that
$\bigcap_{f\in F_n}\{x\in X:|f(x)|<1\}\subset U_n\subset U_g$ and hence $\bigcap_{f\in F_n}f^{-1}(0)\subset g^{-1}(0)$, which implies that $g$ belongs to the linear hull of the set $F_n$.
\end{proof}

\section{Strong $\Fin^\w$-fans in bounded subsets of Banach spaces with the weak topology}\label{s:Banach}

In this section we construct strong $\Fin^\w$-fans in bounded subsets of Banach spaces endowed with the weak topology. For a subset $B$ of a Banach space $X$ by $(B,w)$ or just $B_w$ we denote the set $B$ endowed with the subspace topology inherited from the weak topology of $X$. If $B$ is the closed unit ball of $X$, then the space $B_w$ will be called the {\em weak unit ball} of $X$.

In Theorem 1.9 of \cite{GKP} Gabriyelyan, K\c akol and Plebanek proved that the weak unit ball $B_w$ of a Banach space $X$ is Ascoli if and only if $X$ contains no isomorphic copy of $\ell_1$. The proof of this theorem is rather involved and used some non-trivial tools of probability theory. In this section we present an alternative (and shorter) proof using geometric methods of the classical Banach space theory. 

First we recall some definitions. We say that a subset $B$ of linear space $X$ is \index{subset!absolutely convex}{\em absolutely convex} if $\lambda x+\mu y\in B$ for any points $x,y\in B$ and any real numbers $\lambda,\mu$ with $|\lambda|+|\mu|\le 1$.

A compact space $K$ is \index{topological space!Rosenthal compact}{\em Rosenthal compact} if $K$ embeds into the space $B_1(P)\subset\IR^P$ of functions of the Baire class on a Polish space $P$. It is known \cite[4.1]{Debs14} that each Rosenthal compact space $K$ is \index{topological space!Fr\'echet-Urysohn}{\em Fr\'echet-Urysohn} in the sense that for any subset $A\subset K$ and a point $x\in\bar A$ there is a sequence $S\subset A$, convergent to $x$.

A sequence $(e_n)_{n\in\w}$ of points of a Banach space $X$ is called \index{sequence!$\ell_1$-basic} {\em $\ell_1$-basic} if there are two positive real constants $c,C$ such that $$c\cdot\sum_{n=0}^\infty|x_n|\le\Big\|\sum_{i=0}^\infty x_ne_n\Big\|\le C\cdot\sum_{n=0}^\infty|x_n|$$for any sequence $(x_n)_{n\in\w}\in\ell_1$.

By $X^*$ we denote the dual Banach space to $X$ and by $X^{**}$ the second dual. Besides the weak topology $w$ the dual Banach space $X^*$ carries the weak$^*$ topology $w^*$ inherited from the Tychonoff product topology on $\IR^X$.

The following theorem is the main technical result of this section. In case of the unit ball $B$ this theorem was proved in \cite[1.9]{GKP} (by a different method).

\begin{theorem} For a bounded absolutely convex subset $B$ of a Banach space $X$ the following conditions are equivalent:
\begin{enumerate}
\item[\textup{1)}] the space $B_w=(B,w)$ is Fr\'echet-Urysohn;
\item[\textup{2)}] $B_w$ is a $k$-space;
\item[\textup{3)}] $B_w$ is Ascoli;
\item[\textup{4)}] $B_w$ contains no strong $\Fin^\w$-fan;
\item[\textup{5)}] $B$ contains no $\ell_1$-basic sequences;
\item[\textup{6)}] for any separable subspace $S\subset B$ its closure $\overline S_{w^*}$ in $(X^{**},w^*)$ is Rosenthal compact.
\end{enumerate}
\end{theorem}

\begin{proof} The implication $(1)\Ra(2)\Ra(3)\Ra(4)$ are trivial (or follow from Corollary~\ref{c:A->noFan}).
\vskip5pt

To prove that $(4)\Ra(5)$, assume that the condition (4) holds but (5) does not. Then the set $B$ contains an $\ell_1$-basic sequence $(e_n)_{n\in\w}$. This basic sequence spans a closed linear subspace which is isomorphic to $\ell_1$ and hence can be identified with $\ell_1$.
By \cite[p.142]{CFA}, the Banach space $X$ embeds into $\ell_\infty(\Gamma)$ for some set $\Gamma$. By (the proof of) Josefson-Nissenzweig Theorem \cite[p.223]{Dis}, the natural inclusion operator $\ell_1\to c_0$ defined on the subspace $\ell_1\subset X\subset\ell_\infty(\Gamma)$ extends to a bounded operator $T:\ell_\infty(\Gamma)\to c_0$. Consider the dual operator $T^*:c_0^*\to (\ell_\infty(\Gamma))^*$ and for every $n\in\w$ denote by $\bar e_n^*=T^*(e^*_n)$  the image of the $n$-coordinate functional $e_n^*\in c_0^*$. Then $\{(e_n,\bar e_n^*)\}_{n\in\w}$ is a biorthogonal sequence in $\ell_\infty(\Gamma)\times \ell_\infty(\Gamma)^*$. Since the sequence $(e^*_n)_{n\in\w}$ converges to zero in the weak$^*$ topology of $c_0^*$, its image $(\bar e^*_n)_{n\in\w}$ converges to zero in the weak$^*$ topology of the dual Banach space $\ell_\infty(\Gamma)^*$. By the Grothendieck property of the Banach space $\ell_\infty(\Gamma)$ (see Theorem~VII.15 \cite{Dis}), the sequence $(\bar e_n^*)_{n\in\w}$ converges to zero in the weak topology of $\ell_\infty(\Gamma)^*$.

For every $n,m\in\w$ consider the point $a_{n,m}=\frac12 e_n-\frac12 e_m\in B$ and its weakly open neighborhood $$U_{n,m}=\big\{x\in B:\max_{k\le m}|\bar e_k^*(x)-\bar e^*_k(a_{n,m})|<\tfrac1{2^{m+2}}\big\}\subset B_w,$$ which does not contain zero.

Next, for every $n\in\w$ consider the finite subset $F_n=\{a_{n,m}\}_{n<m\le 2n}$ of $B$ and its weakly open neighborhood $U_n=\bigcup_{n<m\le 2n}U_{n,m}$. We claim that $(F_n)_{n\in\w}$ is a strong $\Fin^\w$-fan in $B_w$.

First we show that the family $(F_n)_{n\in\w}$ is not locally finite at zero. We should prove that any weakly open neighborhood $U\subset B_w$ of zero meets infinitely many sets $F_n$, $n\in\w$. It suffices to show that for every $n_0\in\w$ the neighborhood $U$ meets some set $F_n$ with $n\ge n_0$.

By the definition of the weak topology on $X$, there exists a finite subset $\F\subset X^*$ of functionals such that $\{x\in B_w:\sup_{x^*\in \F}|x^*(x)|<1\}\subset U$. The family $\F$ can be considered as the operator $\F:X\to\ell_\infty(\F)$, $\F:x\mapsto (f(x))_{f\in\F}$, into the finite-dimensional Banach space $\ell_\infty(\F)$. Since the set $\F(B)$ has compact closure in $\ell_\infty(\F)$, there is a number $k\ge n_0$ such that each subset $A\subset \F(B)\subset\ell_\infty(\F)$ of cardinality $|A|\ge k$ contains two distinct points $a,b\in A$ with $\|a-b\|<1$. Such  choice of the number $k$ guarantees the existence of two numbers $n,m$ such that $k<n<m\le 2k$ and $\|\F(e_n)-\F(e_m)\|<1$. Then the point $a_{n,m}=\frac12e_n-\frac12e_m$ belongs to the set $F_n\cap U$, which completes the proof.

It remains to prove that the family $(F_n)_{n\in\w}$ is strongly compact-finite in $B_w$. This will follow as soon as we check that the family $(U_n)_{n\in\w}$ is compact-finite in $B_w$. Assuming the opposite, we could find a compact subset $K\subset B_w$ such that the set $\{n\in\w:K\cap U_n\ne\emptyset\}$ is infinite. Then we can choose  an infinite subset $\Omega\subset \{(n,m)\in\w\times\w:n<m\le 2n\}$ such that for every pair $(n,m)\in\Omega$ the intersection $K\cap U_{n,m}$ contains some point $x_{n,m}$.

By \cite[4.50]{CFA}, the (Eberlein) compact space $K\subset B_w\subset X_w$ is Fr\'echet-Urysohn, which allows us to replace the set $\Omega$ by a smaller infinite subset and assume that the sequence $(x_{n,m})_{(n,m)\in\Omega}$ converges to some point $x_\infty\in K$. Then the sequence $(x_{n,m}-x_\infty)_{(n,m)\in\Omega}$ weakly converges to zero in $X\subset\ell_\infty(\Gamma)$.

For every pair $(n,m)\in\Omega$ consider the functional $a_{n,m}^*=\bar e^*_n-\bar e^*_m$ and observe that $a_{n,m}^*(a_{n,m})=\frac12+\frac12=1$.
The inclusion $x_{n,m}\in U_{n,m}$ implies that $$\sup_{k\le m}|\bar e^*_k(x_{n,m}-a_{n,m})|<\frac14.$$ Then $$
\begin{aligned}
|a_{n,m}^*(x_{n,m}-a_{n,m})|&=|\bar e_n^*(x_{n,m}-a_{n,m})-\bar e^*_m(x_{n,m}-a_{n,m})|\le\\
 &\le|\bar e_n^*(x_{n,m}-a_{n,m})|+|\bar e^*_m(x_{n,m}-a_{n,m})|<\tfrac12
\end{aligned}
$$
and finally $a_{n,m}^*(x_{n,m})>a_{n,m}^*(a_{n,m})-\frac12=\frac12$.

Next, observe that for every pair $(p,q)\in\Omega$ the set $\Omega_{p,q}=\{(n,m)\in\Omega:q<n<m\}$ is infinite and for every $(n,m)\in\Omega_{p,q}$ we get
$$
\begin{aligned}
|a^*_{p,q}(x_{n,m})|&=|(\bar e^*_p-\bar e^*_q)(x_{n,m})|\le |(\bar e^*_p-\bar e^*_q)(a_{n,m})|+|(\bar e^*_p-\bar e_q^*)(a_{n,m}-x_{n,m})|\le\\
&\le|(\bar e^*_p-\bar e^*_q)(a_{n,m})|+|\bar e^*_p(a_{n,m}-x_{n,m})|+|\bar e^*_q(a_{n,m}-x_{n,m})|<\\
&<0+\frac1{2^{m+2}}+\frac1{2^{m+2}}=\frac1{2^{m+1}}.
\end{aligned}
$$
Now the weak convergence of the sequence $(x_{n,m})_{(n,m)\in\Omega}$ to $x_\infty$ implies that $a^*_{p,q}(x_\infty)=0$ for every $(p,q)\in\Omega$ and hence $a_{n,m}^*(x_{n,m}-x_\infty)=a_{n,m}^*(x_{n,m})>\frac12$.

The weak convergence of the sequence $(\bar e^*_n)_{n\in\w}$ to zero implies the weak convergence of the sequence $(a^*_{n,m})_{n<m}$ to zero. So, we obtain a weak null sequence $(a^*_{n,m})_{(n,m)\in\Omega}$ in $\ell_\infty(\Gamma)^*$ and a weak null sequence $(x_{n,m}-x_\infty)_{(n,m)\in\Omega}$ in $\ell_\infty(\Gamma)$  such that $$\lim_{(n,m)\to\infty}a^*_{n,m}(x_{n,m})\not=0.$$ But this contradicts the Dunford-Pettis property of $\ell_\infty(\Gamma)=C(\beta\Gamma)$ (see, \cite[11.36]{CFA}). This contradiction shows that the family $(U_{n})_{n\in\w}$ is compact-finite in $B_w$.
\vskip5pt

$(5)\Ra(6)$. Assume that $B$ contains no $\ell_1$-basic sequence and fix a separable subspace $S\subset B$. We should prove that the closure $\bar S_{w^*}$ of $S$ in $(X^{**},\mathrm{weak}^*)$
is Rosenthal compact. We shall follow closely the lines of the proof of Odell-Rosenthal Theorem in \cite[p.237]{Dis}. Replacing $B$ by $S$, we can assume that $B=S$ is separable and so is the Banach space $X$. Then the closed unit ball $K^*$ of the dual space $X^*$ endowed with the weak$^*$ topology is compact and metrizable. The second dual space $X^{**}$ can be identified with a subspace of $\ell_\infty(K^*)$. Under this identification, $X\subset C(K^*)\subset\ell_\infty(K^*)$. We claim that $\overline{S}_{w^*}\subset B_1(K^*)$ where $B_1(K^*)$ stands for the space of functions of the first Baire class on the compact metrizable space $K^*$. Assuming that some function $f\in\overline{S}_{w^*}$ does not belong to $B_1(K^*)$, we can apply Baire's Theorem~\cite[p.232]{Dis} and find a closed subset $D\subset K^*$ such that $f|D$ has no continuity points. By Lemma 8 in \cite[p.235]{Dis}, Lemma 9 in \cite[p.236]{Dis} and Proposition 3 in \cite[p.207]{Dis}, the set $B$ contains an $\ell_1$-basic sequence, which is a desired contradiction showing that $\bar S_{w^*}\subset B_1(K^*)$, which means that $\bar S_{w^*}$ is Rosenthal compact.
\vskip5pt

$(6)\Ra(1)$ Assume that for every separable subspace $S\subset B$ its closure $\overline{S}_{w^*}$ in $(X^{**},w^*)$ is Rosenthal compact. Since Rosenthal compacta are Fr\'echet-Urysohn \cite{Debs14}, the space $S$ is Fr\'echet-Urysohn. Now we are able to prove that the space $B_w$ is Fr\'echet-Urysohn.
Fix any subset $A\subset B_w$ and a point $a\in\bar A\setminus A\subset B_w$. By Kaplanski Theorem~\cite[4.49]{CFA}, the weak topology of any Banach space has countable tightness. Consequently, we can find a countable subset $S\subset A$ containing the point $a$ it its closure. The space $S\cup\{a\}\subset B_w$, being separable, is Fr\'echet-Urysohn, which allows us to choose a sequence $\{x_n\}_{n\in\w}\subset S\subset A$ convergent to $a$.
\end{proof}

\chapter{Topological functors and their properties}\label{ch:functor}

In this chapter we shall construct (strong) $\Fin$-fans in \index{functor-space}\index{functor!functor-space}{\em functor-spaces}, i.e., spaces of the form $FX$ where $X$ is a topological space and $F:\Top\to \Top$ is a functor in the category $\Top$ of topological spaces and their continuous maps. The results obtained in this chapter will be widely used in the next two chapters. These results can be considered as a self-contained presentation of the theory of functors with finite supports in categories of topological spaces.

Sometimes, functors naturally appearing in Topological Algebra are defined on some smaller subcategories of the category $\Top$, in particular, in the subcategories $\Top_i$, $i\in\{1,2,2\frac12,3,3\frac12\}$, determined by the corresponding Separation Axioms.
\smallskip

We recall that a topological space $X$ satisfies the separation axiom
\begin{itemize}
\item[$T_1$:] \index{topological space!$T_1$-space}(or is a {\em $T_1$-space}) if for any two distinct points $x,y\in X$ the point $x$ has a neighborhood $O_x\subset X$ such that $y\notin O_x$;
\item[$T_2$:] \index{topological space!Hausdorff}(or is {\em Hausdorff\/}) if any two distinct points $x,y\in X$ have disjoint neighborhoods $O_x,O_y\subset X$;
\item[$T_{2\frac12}$:] \index{topological space!functionally Hausdorff}(or is {\em functionally Hausdorff\/}) if for any distinct points $x,y\in X$ there is a continuous function $f:X\to[0,1]$ such that $f(x)=0$ and $f(y)=1$;
\item[$T_3$:] \index{topological space!regular}(or is {\em regular}) if $X$ is a $T_1$-space and for any open set $U\subset X$ and a point $x\in U$ there exists an open set $V\subset X$ such that $x\in V\subset\bar V\subset U$;
\item[$T_{3\frac12}$:] \index{topological space!Tychonoff}(or is {\em Tychonoff\/}) if $X$ is a $T_1$-space and for any open set $U\subset X$ and a point $x\in U$ there exists a continuous function $f:X\to[0,1]$ such that $x\in f^{-1}([0,1))\subset U$.
\end{itemize}
By $\Top_i$ for $i\in\{1,2,2\frac12,3,3\frac12\}$ we denote the category whose objects are $T_i$-spaces and morphisms are continuous maps between $T_i$-spaces. The category $\Top_i$ is a full subcategory of the category $\Top$. We recall that a subcategory $\T$ of $\Top$ is \index{subcategory!full}{\em full} if any continuous map between objects of $\T$ is a morphism of $\T$. A full subcategory of $\Top$ can be identified with the class of its objects.
The implications $T_{3\frac12}\Ra T_{2^\frac12}\Ra T_2\Ra T_1$ yields the following chain of the full subcategories of $\Top$:
$$\Top_{3\frac12}\subset \Top_{2^\frac12}\subset \Top_2\subset \Top_1\subset \Top.$$
In this chapter by $\Top_i$ we shall understand any full subcategory of the category $\Top$ which contains all finite discrete space and is \index{subcategory!hereditary}{\em hereditary} in the sense that  for any object $X$ of $\Top_i$ any subspace of $X$ also is an object of $\Top_i$.

We recall that a functor $F:\Top_i\to\Top$ is \index{functor!monomorphic}{\em monomorphic} it is preserves monomorphisms (which coincide with injective maps in the category $\Top_i$).

By $\w$ we denote the set of all finite ordinals and by $\IN=\w\setminus\{0\}$ the set of natural numbers. Each finite ordinal $n$ will be identified with the finite set $\{0,\dots,n-1\}$ endowed with the discrete topology. For a set $X$ by $X_d$ we denote the set $X$ endowed with the discrete topology.

For a set $X$ by $[X]^{<\w}$ we denote the family of all finite subsets of $X$. It is clear that
$[X]^{<\w}=\bigcup_{n\in\w}[X]^{\le n}$ where $[X]^{\le n}=\{A\in X^{<\w}:|A|\le n\}$ for $n\in\w$.



\section{Functors with finite supports} In this section we assume that $F:\Top_i\to\Top$ is a functor defined on a full hereditary subcategory $\Top_i$ containing all finite discrete spaces as objects. We shall identify the category $\Top_i$ with the class of its objects. So, finite ordinals $n\in\w$ endowed with the discrete topology belong to the class $\Top_i$.

For a finite ordinal $n\in\w$ and a space $X\in\Top_i$ any function $\xi:n\to X$ is a morphism of the category $\Top_i$. So, we can apply the functor $F$ to this morphism and obtain the continuous map $F\xi:Fn\to FX$ whose image $F\xi(Fn)$ is a subset of $FX$. Unifying all such images we obtain the subspace $$F_n(X)=\bigcup_{\xi\in X^n}F\xi(Fn)\subset FX.$$ Observe that the construction of the spaces $F_n(X)$ determines a subfunctor $F_n$ of the functor $F$. Indeed, for any continuous map $f:X\to Y$ between spaces $X,Y\in\Top_i$ and any map $\xi\in X^n$ the map $f\circ\xi$ belongs to $Y^n$ and hence $Ff(F\xi(Fn))=F(f\circ\xi)(Fn)\subset F_n(Y)$. So, $Ff(F_n(X))\subset F_n(Y)$ and we can put $F_nf=Ff|F_n(X):F_n(X)\to F_n(Y)$.

In fact, the space $F_n(X)$ admits an inner definition using finite subsets of $X$ endowed with the discrete topology (instead of maps defined on finite ordinals).
 We recall that a set $A$ endowed with the discrete topology is denoted by $A_d$. If $X$ is a $T_1$-space, then any finite subspace $A\subset X$ is discrete and hence $A_d=A$.

  For a topological space $X$ and a subset $A\subset X$ let $i_{A,X}:A\to X$ be the identity embedding. By  $i^d_{A,X}:A_d\to X$ (and $i^{dd}_{A,X}:A_d\to X_d$) we denote the same function $i_{A,X}$ but with discrete domain (and range).
It is clear that for two sets $A\subset B$ in a topological space $X$ we get $i^d_{A,X}=i^d_{B,X}\circ i^{dd}_{A,B}$.

If a topological space $X$ is an object of the category $\Top_i$ and $A$ is a finite subspace of  $X$, then the space $A$ and its discrete modification $A_d$ are objects of the category $\Top_i$. So, we can consider the continuous maps $Fi_{A,X}:FA\to FX$, $Fi^d:FA_d\to FX$ and their images $F(A;X)=Fi_{A,X}(FA)$ and  $F(A_d;X)=Fi^d_{A,X}(FA_d)$ in $FX$. Applying the functor $F$ to the equality $i^d_{A,X}=i_{A,X}\circ i^d_{A,A}$, we get the equality
$Fi^d_{A,X}=Fi_{A,X}\circ Fi^{dd}_{A,A}$, implying $F(A_d;X)\subset F(A;X)$.

Observe that for any (finite) subspaces $A\subset B$ of $X$ the obvious equality $i_{A,X}=i_{B,X}\circ i_{A,B}$ (and $i^d_{A,X}=i^d_{B,X}\circ i^{dd}_{A,B}$) implies the equality $Fi_{A,X}=Fi_{B,X}\circ Fi_{A,B}$ (and $Fi^d_{A,X}=Fi^d_{B,X}\circ Fi^{dd}_{A,B}$), which yields the inclusion $F(A;X)\subset F(B;X)$ (and $F(A_d;X)\subset F(B_d;X)$).

The following proposition yields alternative inner description of the subfunctor $F_n$.

\begin{proposition} For any space $X\in\Top_i$ and finite cardinal $n$ we get
$$F_n(X)=\bigcup_{A\in [X]^{\le n}}F(A_d;X).$$
\end{proposition}

\begin{proof} If $a\in F_n(X)$, then $a\in F\xi(Fn)$ for some map $\xi:n\to X$. For the set $A=Fn\in[X]^{\le n}$ choose a (unique) map $\tilde \xi:n\to A_d$ such that $\xi=i^d_{A,X}\circ \tilde \xi$ and observe that $a\in F\xi(Fn)=Fi^d_{A,X}(F\tilde\xi(Fn))\subset Fi^d_{A,X}(FA_d)=F(A_d;X)$.

On the other hand, for any finite subset $A\in [X]^{\le n}$ with $a\in F(A_d;X)$ we can choose a homeomorphism $\zeta:A_d\to k$ to the cardinal $k=|A|\subset n$, then take any retraction $r:n\to k$ and observe that the map $\xi=i^d_{A,X}\circ\zeta^{-1}\circ r:n\to X$ has the desired property: $a\in F\xi(Fn)$. This can be shown by applying the functor $F$ to the obvious equality $i^d_{A,X}=\xi\circ i_{k,n}\circ \zeta$:
 $$a\in F(A_d;X)=Fi^d_{A,X}(FA_d)=F\xi\circ Fi_{k,n}\circ F\zeta(FA_d)\subset F\xi\circ Fi_{k,n}(Fk)\subset F\xi(Fn).$$
\end{proof}

For any space $X\in\Top_i$ consider the subspace $$F_{{<}\w}(X)=\bigcup_{n\in\w}F_n(X)=\bigcup_{n\in\w}\bigcup_{A\in[X]^{\le n}}F(A_d;X).$$
It follows that for any element $a\in F_{{<}\w}(X)$ the family $$\Supp(a)=\{A\in[X]^{<\w}:a\in F(A_d;X)\}$$is not empty and hence its intersection
$$\supp(a)=\bigcap\Supp(a)$$is a well-defined finite subset of $X$, called the \index{functor!support of an element}{\em support} of $a$.

The family $\Supp(a)$ has the following useful intersection property.

\begin{lemma}\label{l:supp1} Let $F:\Top_i\to\Top$ be a monomorphic functor, $X\in \Top_i$ and $a\in FX$. For any sets $A,B\in\Supp(a)\subset[X]^{<\w}$, the family $\Supp(a)$ contains any finite non-empty set $C\subset X$ with $A\cap B\subset C$.
\end{lemma}

\begin{proof}  Consider the finite subset $U=A\cup B\cup C$ of $X$ and the identity inclusion $i^d_{U,X}\colon U_d\to X$. Since the functor $F$ is monomorphic, the map $Fi_{U,X}\colon FU_d\to FX$ is injective. Consider the identity embeddings $i^{dd}_{A,U}\colon A_d\to U_d$ and $i^{dd}_{B,U}\colon B_d\to U_d$ of the corresponding finite discrete spaces and observe that $i^d_{A,X}=i^{d}_{U,X}\circ i^{dd}_{A,U}$ and $i^d_{B,X}=i^d_{U,X}\circ i^{dd}_{B,U}$. Since $a\in F(A_d;X)=Fi^d_{A,X}(FA_d)$, we can find an element $a_A\in FA_d$ such that $Fi^d_{A,X}(a_A)=a$. By analogy, there exists an element $a_B\in FB_d$ such that $a=Fi^d_{B,X}(a_B)$.  Consider the elements $a_A'=Fi^{dd}_{A,U}(a_A)$ and $a_B'=Fi^{dd}_{B,U}(a_B)$ in the functor-space $FU_d$ and observe that
$$
\begin{aligned}
Fi^d_{U,X}(a'_A)&=Fi^d_{U,X}\circ Fi^{dd}_{A,U}(a_A)=Fi^d_{A,X}(a_A)=a=\\
&=Fi^d_{B,X}(a_B)=Fi^d_{U,X}\circ Fi^{dd}_{B,U}(a_B)=Fi^d_{U,X}(a_B').
\end{aligned}
$$The injectivity of the map $Fi^d_{U,X}$ implies $a_A'=a_B'$.
Now choose any map $r_A:U_d\to U_d$ such that $r_A|A_d=\id$ and $r_A(B\setminus A)\subset C$. Observe that $r_A(B)=r_A(A\cap B)\cup r_A(B\setminus A)\subset (A\cap B)\cup C=C$, which implies that the map $r_{B}:B\to C$, $r_{B}:x\mapsto r_A(x)$, is well-defined and satisfies the equality $i^d_{C,X}\circ r_B=i^d_{U,X}\circ r_A\circ i^{dd}_{B,U}$.

Observe that $i^d_{U,X}\circ r_A\circ i^{dd}_{A,U}=i^d_{A,X}$ and hence $$
\begin{aligned}
a=&Fi^d_{A,X}(a_A)=Fi^d_{U,X}\circ Fr_A\circ Fi^{dd}_{A,U}(a_A)=Fi^d_{U,X}\circ Fr_A(a_A')=Fi^d_{U,X}\circ Fr_A(a_B')=\\
&=Fi^d_{U,X}\circ Fr_A\circ Fi^{dd}_{B,U}(a_B)=Fi^{d}_{C;X}\circ Fr_B(a_B)\in Fi^d_{C,X}(FC_d)=F(C_d,X),
\end{aligned}
$$
which implies $C\in\Supp(a)$.
\end{proof}

Lemma~\ref{l:supp1} will be used in the proof of the following important fact (cf. \cite{BMZ}).

\begin{theorem}\label{t:supp} Let $F:\Top_i\to\Top$ be a monomorphic functor and $X\in\Top_i$. For any $a\in F_{{<}\w}(X)$ and non-empty finite subset $A\subset X$ with $\supp(a)\subset A$ the element $a$ belongs to the set $F(A_d;X)$. In particular, $a\in F(\supp(a)_d;X)$ if $\supp(a)\ne \emptyset$.
\end{theorem}

\begin{proof} Since $a\in F_{{<}\w}(X)$, the set $\Supp(a)\subset [X]^{<\w}$ is not empty and has well-defined finite intersection $\supp(a)=\bigcap\Supp(a)\subset A$. If the family $\Supp(a)$ contains two disjoint finite sets $B,C$, then $B\cap C=\emptyset\subset  A$ and $a\in F(A_d;X)$ by Lemma~\ref{l:supp1}. So, we assume that for any two sets $B,C\in\Supp(a)$ the intersection $B\cap C$ is not empty and hence belongs to the family $\Supp(a)$ by Lemma~\ref{l:supp1}. This means that the family $\Supp(a)$ is closed under finite intersections, which implies that the finite set $\supp(a)$ belongs to $\Supp(a)$ and hence $a\in F(\supp(a)_d;X)\subset F(A_d;X)$.
\end{proof}

Theorem~\ref{t:supp} has many corollaries.

\begin{corollary}\label{c:Fn} Let $F:\Top_i\to\Top$ be a monomorphic functor. For any space $X\in\Top_i$ the inequality
$$F_n(X)=\{a\in F_{{<}\w}(X):|\supp(a)|\le n\}$$holds for all $n\in\IN$.
\end{corollary}

\begin{proposition}\label{p:supp=image} Let $f:X\to Y$ be a continuous map between spaces $X,Y\in\Top_i$ such that the image $f(X)$ contains more than one point. For any $a\in F_{<\w}(X)$ we get $Ff(a)\in F_{{<}\w}(Y)$ and $\supp(Ff(a))\subset f(\supp(a))$. Moreover, if the restriction $f|\supp(a)$ is injective, then $\supp(Ff(a))=f(\supp(a))$.
\end{proposition}

 \begin{proof}  First we consider the case of empty support $\supp(a)$.
Since $|f(X)|>1$, we can choose two points $x_1,x_2$ such that $f(x_1)\ne f(x_2)$. Theorem~\ref{t:supp} implies that $a\in F(\{x_1\}_d;X)\cap F(\{x_2\}_d;X)$ and then $Ff(a)\subset F(\{f(x_1)\}_d;Y)\cap F(\{f(x_2)\}_d;Y)$ and $\supp(Ff(a))\subset \{f(x_1)\}\cap \{f(x_2)\}=\emptyset=f(\supp(a))$.

It remains to consider the case of non-empty support $A=\supp(a)$.
 In this case Theorem~\ref{t:supp} guarantees that $a\in F(A_d;X)=Fi^d_{A,X}(FA_d)$, so we can find an element $\tilde a\in FA_d$ with $a=Fi^d_{A,X}(\tilde a)$. Consider the set $B=f(A)\subset Y$ and the continuous map $g:A_d\to B_d$, $g:x\mapsto f(x)$. Applying the functor $F$ to the obvious equality $f\circ i^d_{A,X}=i^d_{B,Y}\circ g$, we get the equality
 $Ff(a)=Ff\circ Fi^d_{A,X}(\tilde a)=Fi^d_{B,Y}\circ Fg(\tilde a)\in Fi^d_{B,Y}(FB_d)=F(B_d;Y)$ and hence $\supp(Ff(a))\subset B=f(\supp(a))$.

 Now assuming that the map $f|\supp(a)$ is injective, we shall prove that $\supp(Ff(a))=f(\supp(a))$. To derive a contradiction, assume that the element $b=Ff(a)$ has support $\supp(b)\ne f(\supp(a))$. In this case the set $B'=\supp(b)$ is a proper subset of $f(\supp(a))=B$. So, $\supp(a)\ne\emptyset$ and by Theorem~\ref{t:supp}, $a\in F(\supp(a)_d;X)=F(A_d;X)$. Choose an element $\tilde a\in F(A_d)$ such that $a=Fi^d_{A,X}(\tilde a)$.

 If $B'\ne\emptyset$, then Theorem~\ref{t:supp} guarantees that $b\in F(B_d';X)$. So, we can find an element $b'\in F(B'_d)$ such that $Fi^d_{B',Y}(b')=b$.  Since the map $f|\supp(A)$ is injective, the map $g:A_d\to B_d$, $g:x\mapsto f(x)$, is a homeomorphism. Consider the proper subset $A'=g^{-1}(B')$ of $A$ and the homeomorphism $h:B'_d\to A'_d$, $h:y\mapsto g^{-1}(y)$, inducing the homeomorphism $Fh:FB'_d\to FA'_d$ that maps the element $b'$ onto the element $a'=Fh(b')\in FA'_d$.
  The obvious equality
 $i^d_{B',Y}=f\circ i^d_{A,X}\circ i^{dd}_{A',A}\circ h$ implies
 $b=Fi^d_{B',Y}(b')=Ff\circ Fi^d_{A,X}\circ Fi^{dd}_{A',A}\circ Fh(b')=Ff\circ Fi^d_{A,X}\circ Fi^{dd}_{A',A}(a')$. On the other hand, $b=Ff(a)=Ff\circ Fi^d_{A,X}(\tilde a)$.

 Since the functor $F$ is monomorphic, the injectivity of the map $f\circ i^d_{A,X}$ implies the injectivity of the map
 $Ff\circ Fi^d_{A,X}$.  Taking into account that $$Ff\circ Fi^d_{A,X}(\tilde a)=b=
 Ff\circ Fi^d_{A,X}\circ Fi^{dd}_{A',A}(a'),$$ we conclude that  $\tilde a=Fi^{dd}_{A',A}(a')$ and
 $a=Fi^d_{A,X}(\tilde a)=Fi^d_{A,X}\circ Fi^{dd}_{A',A}(a')=Fi^d_{A',X}(a')\in F(A_d';X)$.
 So, $A=\supp(a)\subset A'$, which is not possible as $A'$ is a proper subset of $A$.

 Now assume that $B'=\supp(b)=\emptyset$. By Theorem~\ref{t:supp}, $b$ belongs to $F(C_d;Y)$ for any non-empty set $C\subset Y$. Since $|f(X)|>1$, we can choose two points $x_1,x_2\in X$ such that $f(x_1)\ne f(x_2)$. Then $Ff(a)=b\in F(\{f(x_1)\}_d;Y)\cap F(\{f(x_2)\}_d;Y)$ implies $a\in F(\{x_1\}_d;X)\cap F(\{x_2\}_d;X)$, which yields $\supp(a)\subset\{x_1\}\cap \{x_2\}=\emptyset$ and hence $\supp(b)=f(\supp(a))$.
\end{proof}

\begin{definition} A functor $F:\Top_i\to\Top$ is defined to \index{functor!with finite supports}{\em have finite supports} if for every space $X\in\Top_i$ we get $FX=F_{{<}\w}(X)$. This happens if and only if for every space $X\in\Top_i$ and point $a\in FX$ there is a finite subset $A\subset X$ such that $a\in F(A_d;X)$.
\end{definition}

\begin{proposition} If a functor $F:\Top_i\to\Top$ has finite supports, then it preserves surjective maps.
\end{proposition}

\begin{proof} Given any surjective map $f:X\to Y$ between objects of the category $\Top_i$ and any element $b\in FY$, find a finite subset $B\subset Y$ such that $b=Fi^d_{B,Y}(b')$ for some $b'\in FB_d$. Since $f$ is surjective, there exists a map $g:B_d\to X$ such that $f\circ g=i^d_{B,Y}$ and hence $Ff\circ Fg=Fi^d_{B,Y}$. Then for the elements $a=Fg(b')\in FX$ we get $b=Fi^d_{B,Y}(b')=Ff\circ Fg(b')=Ff(a)$, which means that the map $Ff$ is surjective.
\end{proof}

\begin{proposition}\label{p:supp=cap} If a monomorphic functor $F:\Top_i\to\Top$  has finite supports, then for any $T_1$-space $X\in\Top_i$ and any $a\in FX$ we get
$\supp(a)=\bigcap\{A\subset X:a\in F(A;X)\}$.
\end{proposition}

\begin{proof} The inclusion $\supp(a)\subset \bigcap\{A\subset X:a\in F(A;X)\}$ will follow as soon as we check that any subset $A\subset X$ with $a\in F(A;X)$ contains $\supp(a)$. It follows that $a=Fi_{A,X}(a')$ for some $a'\in FA$. Since $F$ has finite supports, there exists a finite subset $B\subset A$ such that $a'=Fi^d_{B,A}(b)$ for some $B\in F(B_d)$. Taking into account that $i^d_{B,X}=i_{A,X}\circ i^d_{B,A}$, we conclude that $Fi^d_{B,X}(b)=Fi_{A,X}\circ Fi^d_{B,A}(b)=Fi_{A,X}(a')=a$ and hence $a\in F(B_d;X)$ and $\supp(a)\subset B\subset A$.

On the other hand, the inclusions $\bigcap\{A\subset X:a\in F(A;X)\}\subset \bigcap\{A\in[X]^{<\w}:a\in F(A;X)\}=\bigcap\{A\in [X]^{<\w}:a\in F(A_d;X)\}=\supp(a)$ follows from the fact that each finite subspace $A$ of the $T_1$-space $X$ is discrete and hence $F(A;X)=F(A_d;X)$.
\end{proof}




\section{$\II$-regular functors}

In this section we assume that $\Top_i$ is a full hereditary subcategory of the category $\Top$.

A functor $F:\Top_i\to\Top$ is defined to be \index{functor!$\II$-regular}{\em $\II$-regular} if $\II\in\Top_i$ and
for any closed subset $A$ of the unit interval $\II=[0,1]$ the set $F(A;\II)$ is $\IR$-closed in $F(\II)$.

We recall that a subset $F$ of a topological space $X$ is \index{subset!$\IR$-closed}{\em $\IR$-closed} if its complement $X\setminus F$ is \index{subset!$\IR$-open}{\em $\IR$-open} in $X$. The latter means that for any point $x\in X\setminus F$ there exists a continuous function $f:X\to[0,1]$ such that $f(x)=1$ and $f(F)\subset\{0\}$.

\begin{lemma}\label{l:supp-Ropen}  For any $\II$-regular monomorphic functor  $F:\Top_i\to\Top$ with finite supports and any functionally Hausdorff space $X\in\Top_i$ the support map $\supp:FX\to [X]^{<\w}$ is \index{map!lower semicontinuous}{\em lower semicontinuous} in the sense that  for any $\IR$-open set $U\subset X$ the set $W=\{a\in FX:\supp(a)\cap U\ne\emptyset\}$ is $\IR$-open in $FX$.
\end{lemma}

\begin{proof} Given any point $a\in W$ we should find a continuous function $\hbar:F X\to[0,1]$ such that $\hbar (a)=1$ and $\hbar(F X\setminus W)\subset\{0\}$.

Since $a\in W$, there exists a point $x\in\supp(a)\cap U$. Consider the canonical map $\delta:X\to\beta X$ of $X$ to its Stone-\v Cech compactification. Since $X$ is functionally Hausdorff, this map $\delta$ is injective. So, we can identify $X$ with the subset $\delta(X)$ of $\beta X$. Observe that the topology on $X$ inherited from $\beta X$ coincides with the family of all $\IR$-open subsets of $X$.

So, we can choose an open neighborhood $O_x\subset \beta X$ of the point $x\in X\subset\beta X$ such that $\bar O_x\cap \supp(a)=\{x\}$ and $O_x\cap X\subset U$. By the normality of $\beta X$, there is a continuous map $\xi_1:\beta X\to [\frac12,1]$ such that $\xi_1(x)=1$ and $\xi_1(\beta X\setminus O_x)\subset\{\frac12\}$. Let $\partial O_x$ be the boundary of $O_x$ in $\beta X$. Using Tietze-Urysohn Theorem we can construct a continuous map $\xi_2:\beta X\setminus O_x\to[0,\frac12]$ such that $\xi_2(\partial O_x)\subset \{\frac12\}$ and $\xi_2|\supp(a)\setminus \{x\}$ is injective. Then the map $\xi:X\to [0,1]$ defined by $\xi|X\cap\bar O_x=\xi_1|X\cap\bar O_x$ and $\xi|X\setminus O_x=\xi_2|X\setminus O_x$ is continuous and the restriction $\xi|\supp(a)$ is injective. Consider the map $F\xi:FX\to F\II$ and observe that the element $b=F\xi(a)$ has support $\supp(b)=\xi(\supp(a))\ni 1$ and hence does not belong to the $\IR$-closed subset $F([0,\frac12];\II)$ of $F(\II)$.
So, we can find a continuous function $h:F(\II)\to[0,1]$ such that $h(b)=1$ and $h\big(F([0,\frac12];\II)\big)\subset \{0\}$ and consider the continuous function $\hbar=h\circ F\xi:FX\to [0,1]$. Observe that $\hbar(a)=h(b)=1$.

We claim that the $\IR$-open neighborhood  $V=\hbar^{-1}((0,1])$ of $a$ is contained in $W$. Assuming that $V\not\subset W$, we could find an element $v\in V\setminus W$. Then $\supp(v)\cap U=\emptyset$ and hence $\supp(v)\subset X\setminus O_x$ and $\xi(\supp(v))\subset \xi(X\setminus O_x)\subset [0,\frac12]$. Proposition~\ref{p:supp=image} implies that $\supp(F\xi(v))\subset \xi(\supp(v))\subset [0,\frac12]$ and then $F\xi(v)\in F([0,\frac12];\II)$ and $\hbar(v)=h\circ F\xi(v)\in h(F([0,\frac12];\II))\subset\{0\}$, which is a desired contradiction.
Therefore, $V\subset W$ and $\hbar (F(\II)\setminus W)\subset \hbar(F X\setminus V)\subset\{0\}$.
\end{proof}

\begin{corollary}\label{c:Fn-Rclosed} Let  $F:\Top_i\to\Top$ be an $\II$-regular monomorphic functor  with finite supports. For any functionally Hausdorff space $X\in\Top_i$ and every $n\in\IN$ the set $F_n(X)$ is $\IR$-closed in $FX$.
\end{corollary}

\begin{proof} Given any element $a\in FX\setminus F_n(X)$, we can apply Corollary~\ref{c:Fn} and conclude that $\supp(a)>n$. Choose pairwise disjoint $\IR$-open sets $U_0,\dots,U_n\subset X$ such that $\supp(a)\cap U_i\ne\emptyset$ for all $i\le n$. By Lemma~\ref{l:supp-Ropen}, the set $W=\bigcap_{i=0}^n\{b\in FX:\supp(b)\cap U_i\ne\emptyset\}$ is $\IR$-open in $FX$ and is disjoint with the set  $F_n(X)=\{b\in FX:|\supp(b)|\le n\}$, witnessing that this set is $\IR$-closed in $FX$.
\end{proof}

\section{Bounded and strongly bounded functors}

In this section we assume that $\Top_i$ is a full hereditary subcategory of the category $\Top$.

A functor $F:\Top_i\to\Top$ is defined to be \index{functor!bounded}\index{functor!strongly bounded}({\em strongly}) {\em bounded} if for each space $X\in\Top_i$ and a compact subset $K\subset FX$ there exists a bounded subset $B\subset X$ (and a number $n\in\IN$) such that $K\subset F(B;X)$ (and $K\subset F_n(X)$). We recall that a subset $B$ of a topological space $X$ is \index{subset!bounded}{\em bounded} if $X$ contains no infinite locally finite family of open sets meeting the set $B$.

In the following characterization of bounded functors by $\IR_+$ we denote the closed half-line $[0,\infty)$.

\begin{proposition}\label{p:F-bounded} Let $F:\Top_i\to\Top$  be a monomorphic functor with finite supports. Then the following conditions are equivalent:
\begin{enumerate}
\item[\textup{1)}] the functor $F$ is bounded;
\item[\textup{2)}] for every space $X\in\Top_i$ and compact set $K\subset FX$ the set $\supp(K)=\bigcup_{a\in K}\supp(a)$ is bounded in $X$.
\end{enumerate}
If $\IR_+\in\Top_i\subset\Top_{3\frac12}$, then the equivalent conditions \textup{(1)--(2)} are equivalent to:
\begin{enumerate}
\item[3)] for every compact subset $K\subset F(\IR_+)$ the set $\supp(K)$ is bounded in $\IR_+$.
\end{enumerate}
\end{proposition}

\begin{proof} $(1)\Ra(2)$ Assume that the functor $F$ is bounded. Take any
space $X\in\Top_i$ and  compact subset $K\subset FX$. By the boundedness of $F$, $K\subset F(B;X)$ for some bounded set $B\subset X$. Then for any $a\in K$ the inclusion $a\in F(B;X)=Fi_{B,X}(FB)$ implies the existence of $b\in FB$ such that $a=Fi_{B,X}(b)$. Since the functor $F$ has finite supports, there exists a finite subset $A\subset B$ such that $b\in Fi^d_{A,B}(FA_d)$. Then $a\in Fi_{B,X}(b)\in Fi_{B,X}\circ Fi^d_{A,B}(FA_d)=Fi_{A,X}(A_d)=F(A_d;X)$ and hence $\supp(a)\subset A\subset B$ by the definition of $\supp(a)$. So, the set $\supp(K)=\bigcup_{a\in K}\supp(a)\subset B$ is bounded in $X$.
\smallskip

$(2)\Ra(1)$ Assume that for any space $X\in \Top_i$ and any compact set $K\subset FX$ the set $\supp(K)$ is bounded in $X$. If $X$ is bounded in $X$, then
$K\subset F(X)=F(X;X)$ and we are done. So, we assume that $X$ is not bounded and hence not empty. Then we can choose a non-empty bounded subset $B\subset X$ containing $\supp(K)$ and applying Theorem~\ref{t:supp} conclude that $K\subset F(B;X)$.
\smallskip

 Now assume that $\IR_+\in\Top_i\subset\Top_{3\frac12}$.
 \smallskip

 The implication $(2)\Ra(3)$ is trivial. To prove the implication $(3)\Ra(2)$,
assume  that for some Tychonoff space $X$ and compact subset $K\subset X$ the set $\supp(K)$ is unbounded in $X$. Since each unbounded set contains a countable unbounded set, we can choose a countable subset $K_0\subset K$ such that the countable set $\supp(K_0)=\bigcup_{a\in K_0}\supp(a)$ is unbounded in $X$.  By Lemma~\ref{l:Runbound}, there exists a continuous function $f:X\to\IR_+$ such that the restriction $f|\supp(K)$ is injective and the set $f(\supp(K))$ is unbounded in $\IR_+$. Consider the continuous map $Ff:FX\to F\IR_+$ and observe that the set $Ff(K)$ is compact in $F\IR_+$. By $(3)$, the set $\supp(Ff(K))$ is bounded in $\IR_+$. Since $f|\supp(K_0)$ is injective, we can apply Proposition~\ref{p:supp=image} and conclude that the set $f(\supp(K_0))=\supp(Ff(K_0))$ is bounded in $\IR_+$, which contradicts the choice of $f$.
\end{proof}

\begin{proposition}\label{p:F-strong-bound} Assume that $\Top_{3\frac12}\subset\Top_i\subset\Top_{2\frac12}$. A monomorphic functor   $F:\Top_i\to\Top$ is strongly bounded if and only if $F$ is bounded, has finite supports, and for every compact Hausdorff  space $X$, any compact subset $K\subset F(X)$ is contained in the subspace $F_n(X)$ for some $n\in\IN$.
\end{proposition}

\begin{proof}The ``only if'' part is trivial. To prove the ``if'' part, take any space  $X\in\Top_i$ and compact subset $K\subset FX$. If $X$ is empty, then $K\subset F(X)=F_0(X)$ and we are done. So, we assume that $X$ is not empty.   Let $\delta:X\to\beta X$ be the canonical map of $X$ into its Stone-\v Cech compactification. Since the functor $F$ is monomorphic, the continuous map $F\delta:FX\to F(\beta X)$ is injective. By our assumption, the compact subset $F\delta(K)\subset F(\beta X)$ is contained in $F_n(\beta X)$ for some $n\in\IN$. Then for every $a\in K$ the element $F\delta(a)\in F_n(\beta X)$ has support of cardinality $|\supp(F\delta(a))|\le n$.
Since the map $\delta$ is injective, we can apply Proposition~\ref{p:supp=image} and conclude that $\delta(\supp(a))=\supp(F\delta(a))$ and $|\supp(a)|=|\supp(F\delta(a))|\le n$. Choose any non-empty subset $B\subset X$ that contains the set $\supp(a)$ and has cardinality $|B|\le n$. By Theorem~\ref{t:supp}, $a\in F(B_d;X)$ and hence $a\in F_n(X)$. Therefore, $K\subset F_n(X)$.
\end{proof}

We shall often use the following property of bounded functors.

\begin{lemma}\label{l:cont-supp} Let $F:\Top_i\to\Top$ be a bounded $\II$-regular monomorphic functor with finite supports. If $(U_\alpha)_{\alpha\in A}$ is a compact-finite family of sets in a $\mu$-complete space $X\in\Top_i$, then the family $(V_\alpha)_{\alpha\in A}$ of the sets
$$V_\alpha=\{a\in FX:\supp(a)\cap U_\alpha\ne\emptyset\}, \;\alpha\in A,$$
is compact-finite in $FX$.
\end{lemma}

\begin{proof} Without loss of generality, the space $X$ is not empty. To see that the family $(V_\alpha)_{\alpha\in A}$ is compact-finite in $FX$, fix any compact set $K\subset FX$. By the boundedness of the functor $F$, there exists a non-empty bounded closed set $B\subset X$ such that $K\subset F(B;X)$. By the $\mu$-completeness of $X$ the bounded closed set $B$ is compact. Since the family $(U_\alpha)_{\alpha\in A}$ is compact-finite in $X$, the subfamily $A'=\{\alpha\in A:B\cap U_\alpha\ne\emptyset\}$ is finite. For any $\alpha\in A\setminus A'$ and point $a\in K\subset F(B;X)\subset FX$ we get $\supp(a)\subset B\subset X\setminus U_\alpha$ (by Theorem~\ref{t:supp})  and hence $a\notin V_\alpha$. So, $K\cap V_\alpha=\emptyset$ for all $\alpha\in A\setminus A'$, which means that the family $(V_\alpha)_{\alpha\in A}$ is compact-finite in $FX$.
\end{proof}

\section{$\HM$-commuting functors}

In the previous subsection we have seen that the $\II$-regularity is a nice property of functors with important implications. In this section we shall prove that each $\HM$-commuting functor $F:\Top_i\to\Tych$ is $\II$-regular.

First we recall some information on the Hartman-Mycielski construction $\HM(X)$, which was introduced in \cite{HM} and in general form in \cite{BM}. This construction has many applications in Topological Algebra, see \cite{BM}, \cite{HM}, \cite[\S3.8]{AT}, \cite{BGG}, \cite{BH}.

For a topological
space $X$ let $\HM(X)$ be the set of all functions $f:[0;1) \to
X$ for which there exists a sequence $0=a_{0}<a_{1}<\dots<a_{n}=1$
such that $f$ is constant on each interval $[a_{i-1},a_{i})$,
$1\leq i\leq n$. A
neighborhood sub-base of the Hartman-Mycielski topology of
$\HM(X)$ at an $f\in \HM(X)$ consists of sets $N(a,b,V,\varepsilon)$,
where
\begin{itemize}
\item[1)] $0\leq a<b\leq 1$, $f$ is constant on $[a;b)$, $V$ is
    neighborhood of $f(a)$ in $X$ and $\varepsilon>0$;
\item[2)] $g\in N(a,b,V,\varepsilon)$ means that $|\{t\in [a;b):
    g(t)\notin V\} |<\varepsilon$, where $|\cdot|$ denotes the Lebesgue
    measure on $[0,1)$.
\end{itemize}

If $X$ is a Hausdorff (Tychonoff) space, then so is the space
$\HM(X)$, see \cite{BM}. Moreover, if the topology of $X$ is generated by a bounded metric $d$, then the Hartman-Mycielski topology on $\HM(X)$ is generated by the metric $$d_{\HM}(f,g)=\int_0^1d(f(t),g(t))dt.$$

The construction $\HM$ is functorial in the sense
that for any continuous map $p:X\to Y$ between topological spaces the map
$\HM(p):\HM(X)\to \HM(Y)$, $\HM(p):f\mapsto p\circ f$, is continuous,
see \cite{BM}. So, $\HM:\Top\to\Top$ is a functor in the category $\Top$ such that $\HM(\Top_i)\subset\Top_i$ for $i\in\{1,2,2\frac12,3\frac12\}$.

The space $X$ can be identified with a subspace of
$\HM(X)$ via the embedding $hm_X:X\to \HM(X)$ assigning to each point $x$ the constant
function $hm_X(x):t\mapsto x$. This embedding $hm_X:X\to \HM(X)$
is closed if $X$ is Hausdorff (see \cite{BM}). The maps $hm_X$ determine a natural transformation $hm:\Id\to \HM$ of the identity functor to the Hartman-Mycielski functor $\HM$.

We shall say that a subspace $X\subset Y$ of a topological space $Y$ is an \index{subspace!$\HM$-valued retract}{\em $\HM$-valued retract} of $Y$ if there exists a continuous map $r:Y\to\HM(X)$ such that $r\circ i_{X,Y}=hm_X$. If for each compact subset $K\subset Y$ there is a compact subset $B\subset X$ such that $r(K)\subset HM(B;X)$, then we say that $X$ is an \index{subspace!$\HM_k$-valued retract}{\em $\HM_k$-valued retract} of $Y$.

\begin{proposition}\label{p:BB} A closed subspace $X$ of a topological space $Y$ is an $\HM_k$-valued retract of $X$ if one of the following conditions is satisfied:
\begin{enumerate}
\item[\textup{1)}] the space $Y$ is stratifiable;
\item[\textup{2)}] $X$ is compact metrizable and  $Y$ is functionally Hausdorff;
\item[\textup{3)}] the space $X$ is Polish and the space $Y$ is normal.
\end{enumerate}
\end{proposition}

\begin{proof}  The first statement was proved in \cite{BB}.
To prove the other two statements, assume that $Y$ is functionally Hausdorff (normal) and $X$ is compact metrizable (Polish). Let $\xi:X\to \xi(X)\subset \IR^\w$ be a closed topological embedding. The homeomorphism $h=\xi^{-1}:\xi(X)\to X$ induces a homeomorphism $\HM h:\HM(\xi(X))\to \HM(X)$. By Lemma~\ref{l:fH-ext} (resp. the Tietze-Urysohn Theorem \cite[2.1.8]{En}), the map $\xi$ can be extended to a continuous map $\xi:X\to\IR^\w$. By the stratifiability of the metrizable space $\IR^\w$, there is an $\HM_k$-valued retraction $r:\IR^\w\to \HM(\xi(X))$. Then the map $\HM h\circ r\circ \bar\xi:Y\to\HM(X)$ is a required $\HM_k$-valued retraction of $Y$ on $X$.
\end{proof}


\begin{proposition}\label{p:rHM} Each $\HM_k$-valued retract $X$ of a topological space $Y$ is $C_k$-embedded into $Y$.
\end{proposition}

\begin{proof} Let $r:Y\to\HM(X)$ be an $\HM_k$-valued retraction.
 For a point $y\in Y$ the function $r(y):[0,1)\to X$ will be denoted by $\bar y$. Define an extension operator $E:C_k(X)\to C_k(Y)$ assigning to each function $f\in C_k(X)$ the function $Ef:Y\to \IR$ given by the formula
$$Ef(y)=\sum_{x\in \bar y([0,1))}\lambda(\bar y^{-1}(x))f(x)$$where $\lambda$ be the Lebesgue measure on the interval $[0,1)$. It is clear that the operator $E$ is linear. So, it suffices to check the continuity of $E$ at zero. Fix a compact subset $K\subset Y$ and $\e>0$. Since $r$ is an $\HM_k$-valued retraction, there is a compact subset $B\subset X$ such that $r(K)\subset \HM(B;X)$. Then $[B,\e]=\{f\in C_k(X):\max_{x\in B}|f(x)|<\e\}$ is an open neighborhood of zero such that for any $f\in [B,\e]$ we get $\max_{y\in K}|Ef(y)|\le\max_{x\in B}|f(x)|<\e$. By Proposition~\ref{p:lin->0}, the operator $E$ is 0-continuous.
\end{proof}

Proposition~\ref{p:BB}(1) and Proposition~\ref{p:rHM} imply the following result of Borges \cite{Bor66}.

\begin{corollary}\label{c:s->C_k-emb} Each closed subset $X$ of a stratifiable space $Y$ is $C_k$-embedded in $X$.
\end{corollary}

Let $\Top_i$ be a full hereditary subcategory of $\Top$, containing all compact metrizable spaces. We shall say that a functor $F:\Top_i\to\Top$ is \index{functor!$\HM$-commuting}{\em $\HM$-commuting} if for every  space $X\in\Top_i$ the space $\HM X$ is an object of the category $\Top_i$ and there exists a continuous map $c_X:F\circ \HM X\to \HM\circ F X$ making the following diagram  commutative:
$$
\xymatrix{
F(\HM X)\ar^{c_X}[rr]&&\HM(FX)\\
&FX\ar_{hm_{FX}}[ru]\ar^{F(hm_X)}[ul]}
$$

\begin{proposition}\label{p:Femb+HMr} Let $F:\Top_i\to\Top$ be an $\HM$-commuting functor, $Y\in\Top_i$ be a topological space and $X$ be a subspace of $Y$. If $X$ is an $\HM$-valued retract of $Y$, then the map $Fi_{X,Y}:FX\to FY$ is injective. Moreover, if the  space $FX$ is Hausdorff, then the map $Fi_{X,Y}$ is a closed topological embedding and the subspace $F(X;Y)$ is an $\HM$-valued retract of $FY$.
\end{proposition}

\begin{proof} We shall follow the idea of the proof of Theorem 2 in \cite{BH}. Since the functor $F$ is $\HM$-commuting, for the space $X$ there exists a continuous map
$c_X:F(\HM X)\to\HM(FX)$ such that $c_X\circ Fhm_X=hm_{FX}$.
 Since $X$ is an $\HM$-valued retract of $Y$, there exists a continuous map $r:Y\to \HM X$ such that $r\circ i_{X,Y}=hm_X$. Applying to the latter equality the functor $F$ we get the commutative diagram:
$$\xymatrix{
FY\ar_{Fr}[rd]&&FX\ar_{Fi_{X,Y}}[ll]\ar^{Fhm_X}[ld]\ar^{hm_{FX}}[rd]\\
&F(\HM X)\ar_{c_X}[rr]&&\HM(FX)
}
$$ Taking into account that the map $hm_{FX}=c_X\circ Fr\circ  Fi_{X,Y}$ is injective, we conclude that $Fi_{X,Y}:FX\to FY$ is injective. Moreover, if the space $FX$ is Hausdorff, then $hm_{FX}=c_X\circ Fr\circ Fi_{X,Y}$ a closed topological embedding and so is the map $Fi_{X,Y}$.

To see that the space $F(X;Y)=Fi_{X,Y}(FX)\subset FY$ is an $\HM$-valued retract of $FY$, consider the homeomorphism $f:FX\to F(X;Y)$ defined by $f(x)=Fi_{A,X}(a)$ for $a\in FX$, and observe that $ Fi_{A,X}\circ f^{-1}=i_{F(X;Y),FY}$. It follows that $\HM f:\HM(FX)\to\HM (F(X;Y))$ a homeomorphism too and the following diagram commutes.
$$\xymatrix{
FY\ar_{Fr}[rd]&&FX\ar_{Fi_{X,Y}}[ll]\ar^{Fhm_X}[ld]\ar_{hm_{FX}}[rd]
&F(X;Y)\ar^{hm_{F(X;Y)}}[rd]\ar@/_20pt/_{i_{F(X;Y),FY}}[lll]\ar^{f^{-1}}[l]\\
&F(\HM X)\ar_{c_X}[rr]&&\HM(FX)\ar_{\HM f}[r]&\HM(F(X;Y))
}
$$
Then the continuous map $R=\HM f\circ c_X\circ Fr:FY\to \HM(F(X;Y))$ witnesses that $F(X;Y)$ is an $\HM$-valued retract of $FY$.
\end{proof}

Combining Proposition~\ref{p:Femb+HMr} with Proposition~\ref{p:BB}, we get the following two corollaries.

\begin{corollary}\label{c:Fs->emb} If $F:\Top_i\to\Top$ is an $\HM$-commuting functor, and $f:X\to Y$ be a closed topological embedding of two spaces $X,Y\in\Top_i$.
Assume that one of the following conditions is satisfied:
\begin{enumerate}
\item[\textup{1)}] the space $Y$ is stratifiable;
\item[\textup{2)}]  $X$ is compact metrizable and  $Y$ is functionally Hausdorff;
\item[\textup{3)}] the space $X$ is Polish and the space $Y$ is normal.
\end{enumerate}
Then the map $Ff:FX\to FY$ is injective. Moreover, if the functor-space $FX$ is Hausdorff, then $Ff$ is a closed topological embedding.
\end{corollary}

\begin{proof} Let $X'=f(X)$ and $f':X\to X'$ be the homeomorphism defined by the formula $f'(x)=f(x)$ for $x\in X$. It follows that the map $Ff':FX\to FX'$ is a homeomorphism,  $f=i_{X',Y}\circ f'$ and hence $Ff=Fi_{X',Y}\circ Ff'$. By
 Proposition~\ref{p:Femb+HMr} with Proposition~\ref{p:BB}, the map $Fi_{X',Y}$ is injective (and is a closed topological embedding if the functor-space $FX$ is Hausdorff). Then so is the composition $Ff=Fi_{X',Y}\circ Ff'$ of $Fi_{X',Y}$ with the homeomorphism $Ff'$.
 \end{proof}

\begin{corollary}\label{c:FHM->mono} If a functor $F:\Top_i\to\Top$ has finite supports and is $\HM$-commu\-ting, then for any injective map $f:X\to Y$ between functionally Hausdorff spaces $X,Y\in \Top_i$ the map $Ff:FX\to FY$ is injective.
\end{corollary}

\begin{proof} Given any two distinct points $a_1,a_2\in FX$, find a finite set $A\subset X$ such that $a_1,a_2\in Fi_{A,X}(FA)$ and hence $a_i=Fi_{A,X}(a'_i)$, $i\in\{1,2\}$, for some (necessarily distinct) points $a_1',a_2'\in FA$.
Taking into account that the finite functionally Hausdorff space $A$ is compact and metrizable, apply Corollary~\ref{c:Fs->emb}(2) and conclude that the map $F(f\circ i_{A,X}):FA\to FY$ is injective and hence
$$Ff(a_1)=Ff\circ Fi_{A,X}(a'_1)=F(f\circ i_{A,X})(a_1')\ne F(f\circ i_{A,X})(a_1')=
Ff(a_2).$$
\end{proof}

\begin{corollary}\label{c:F-Ireg} An $\HM$-commuting functor $F:\Top_i\to\Top$ is $\II$-regular if the space $F(\II)$ is Tychonoff.
\end{corollary}

\begin{proof} Assume that the space $F(\II)$ is Tychonoff. By Corollary~\ref{c:Fs->emb}, for every closed subspace $X\subset \II$ the map $Fi_{X,\II}:FX\to F(\II)$ is injective, which implies that the space $FX$ is Hausdorff. Applying  Corollary~\ref{c:Fs->emb} once more, we obtain that the map $Fi_{X,\II}:FX\to F(\II)$ is a closed topological embedding. So, the set $F(X;\II)=Fi_{X,\II}(FX)$ is closed in $F(\II)$ and hence $\IR$-closed in $F(\II)$ (as $F(\II)$ is Tychonoff).
\end{proof}

\section{Functors preserving closed embeddings of metrizable compacta}

We shall say that a functor $F:\Top_i\to\Top$ \index{functor!preserves closed embeddings}{\em preserves closed embeddings of metrizable compacta} if for any continuous injective map $f:X\to Y$
between compact metrizable spaces the map $Ff:FX\to FY$ is a closed topological embedding. Observe that $F$ preserves closed embedding of metrizable compacta if and only if for every compact metrizable space $Y$ and a closed subspace $X\subset Y$ the set $F(X;Y)$ is closed in $FX$ and the map $Fi_{X,Y}:F(X)\to F(X;Y)\subset FX$ is a homeomorphism between the spaces $F(X)$ and $F(X;Y)$.

Corollary~\ref{c:Fs->emb} implies the following fact.

\begin{corollary}\label{c:HM-presemb} Any $\HM$-commuting functor $F:\Top_i\to\Top$ with Hausdorff  functor-space $F(\II^\w)$ preserves closed embeddings of metrizable compacta.
\end{corollary}

\begin{proposition}\label{p:FembTych} If a functor $F:\Top_i\to\Top$ preserves closed embeddings of metrizable compacta, then for every compact metrizable subset $K$ of a functionally Hausdorff space $X\in\Top_i$ the map $Fi_{K,X}:FK\to FX$ is a closed topological embedding.
\end{proposition}

\begin{proof} Since $X$ is functionally Hausdorff, there exists a continuous map $f:X\to\II^\w$ into the Hilbert cube whose restriction $\bar f=f|K$ to the compact metrizable subspace $K\subset X$ is injective. Since the functor $F$ preserves closed embedding of metrizable compacta, the map $F\bar f:FK\to F(\II^\w)$ is a closed topological embedding.

  Applying the functor $F$ to the equality $f\circ i_{K,X}=\bar f$, we get the equality $F\bar f\circ Fi_{K,X}=F\bar f$. Since $F\bar f$ is a closed topological embedding, the map $Fi_{K,X}:F(K)\to F(X)$ is a closed topological embedding too.
\end{proof}

\begin{proposition}\label{p:FX-fH} Assume that a functor $F:\Top_i\to\Top$ has finite supports and preserves closed embedding of metrizable compacta. If the space $f(\II)$ is (functionally) Hausdorff, then for every functionally Hausdorff space $X$ the functor-space $FX$ is (functionally) Hausdorff, too.
\end{proposition}

\begin{proof} Given two distinct elements $a_1,a_2\in FX$ it suffices to construct a continuous map $g:FX\to F(\II)$ such that $g(a_1)\ne g(a_2)$. Since $F$ has finite supports, there exists a finite subset $A\subset X$ such that $a_1,a_2\in Fi_{A,X}(FA)$. Take any closed embedding $\xi:A\to\II$ and using Lemma~\ref{l:fH-ext}, extend it to a continuous map $\bar \xi:X\to\II$. Then the map $F\bar\xi:FX\to F(\II)$ has the required property: $F(\bar \xi)(a_1)\ne F(\bar \xi)(a_2)$. Indeed, for the points $a_1,a_2$, we can find points $a_1',a_2'\in FX$ such that $Fi_{A,X}(a_i')=a_i$ for $i\in\{1,2\}$. Since $F$ preserves closed embeddings for metrizable compacta, the map $F\xi:FA\to F(\II)$ is injective and hence $F\xi(a_1')\ne F\xi(a_2')$. Applying the functor $F$ to the obvious equality $\bar\xi\circ i_{A,X}=\xi$, we conclude that $F\bar\xi\circ Fi_{A,X}=F\xi$ and hence $F\bar\xi(a_1)=F\bar\xi(Fi_{A,X}(a_1'))=F\xi(a_1')\ne F\xi(a_2')=F\bar \xi(a_2)$.
\end{proof}

\section{Functors preserving preimages}

This section is devoted to functors that preserve preimages. In this section we assume that $\Top_i$ is a full hereditary subcategory of the category $\Top$, containing all finite discrete spaces.

\begin{definition} We say that a monomorphic functor $F:\Top_i\to\Top$ \index{functor!preserves preimages}{\em preserves preimages} if for any continuous surjective map $f:X\to Y$ between spaces $X,Y\in\Top_i$ and any non-empty subset $A\subset Y$ we get $Ff^{-1}(F(A;Y))=F(f^{-1}(A);X)$.
\end{definition}

\begin{proposition}\label{p:non-preim} For a monomorphic functor $F:\Top_i\to\Top$ with finite supports the following conditions are equivalent:
\begin{enumerate}
\item[\textup{1)}] The functor $F$ preserves preimages.
\item[\textup{2)}] For any  surjective map $f:X\to Y$ between finite discrete spaces and any non-empty subset $A\subset Y$ we get $Ff^{-1}(F(A;Y))=F(f^{-1}(A);X)$.
\item[\textup{3)}] For any surjective map $f:X\to Y$ between finite discrete spaces with $|Y|=|X|-1>1$ and any non-empty set $Z\subset Y$ we get $F(f^{-1}(Z);X)\subset(Ff)^{-1}(F(Z;Y))$.
\end{enumerate}
\end{proposition}

\begin{proof} The implication $(1)\Ra(3)$ is trivial.

To prove $(3)\Ra(2)$, take any surjective map $f:X\to Y$ between finite discrete spaces and any non-empty subset $Z\subset Y$. Given an element $a\in FX$ we should prove that $a\in F(f^{-1}(Z);X)$ if and only if $Ff(a)\in F(Z;Y)$.
This is clear if $Z=Y$. So, we assume that $Z\ne Y$ and hence $Y$ contains more than one point. If $a\in F(f^{-1}(Z);X)$, then by the functoriality of $F$, we get $Ff(a)\in F(Z;X)$. Next, assuming that $Ff(a)\in F(Z;Y)$, we shall show that $a\in F(f^{-1}(Z);X)$. Since the space $X$ is finite, the map $f$ can be written as a composition $f=g_1\circ\dots\circ g_m$ of maps $g_i:X_{i+1}\to X_i$ such that $X_{m+1}=X$, $X_1=Y$ and $|X_{i+1}|=|X_i|+1$ for every $i\le m$. Let $a_m=a$ and $a_{i-1}=Fg_i(a_i)\in FX_i$ for $i\in\{m,\dots,2\}$.
Also put $Z_1=Z$ and $Z_{i+1}=g_i^{-1}(Z_i)$ for all $i\le m$. By our assumption, for every $i\in\{1,\dots,m\}$ we get $(Fg_i)^{-1}(F(Z_i;X_i))=F(Z_{i+1},X_{i+1})$.
For every $i\in\{1,\dots,m\}$ consider the map $f_i=g_i\circ\cdots\circ g_1:X_{i+1}\to X_1=Y$ and observe that $f_i^{-1}(Z)=Z_{i+1}$.
By induction, we shall prove that for every $i\in\{1,\dots,m\}$ we get $(Ff_i)^{-1}(F(Z;Y))=F(Z_{i+1};X_{i+1})$. For $i=1$ this equality follows from the equality $f_1=g_1$. Assume that for some $1<i\le m$ the equality $(Ff_{i-1})^{-1}(F(Z;Y))=F(Z_{i},X_{i})$ has been proved. Since $f_{i}=g_{i}\circ f_{i-1}$ and $Ff_i=Fg_i\circ Ff_{i-1}$, we conclude that $Ff_i^{-1}(F(Z;Y))=Fg_i^{-1}(Ff_{i-1}^{-1}(F(Z;Y))=Fg_i^{-1}(F(Z_i,X_i))=F(Z_{i+1},X_{i+1})$.
This completes the inductive step. For $i=m$, we get $f_m=f$ and $Ff^{-1}(F(Z;Y))=F(Z_m;X)=F(f^{-1}(Z);X)$.
\smallskip

To prove $(2)\Ra(1)$, take any continuous surjective map $f:X\to Y$ between spaces $X,Y\in\Top_i$ and any non-empty subset $Z\subset Y$. Given an element $a\in FX$ we need to prove that $a\in F(f^{-1}(Z);X)$ if and only if $Ff(a)\in F(Z;Y)$. If $a\in F(f^{-1}(Z);X)$, then by the functoriality of $F$, we get $Ff(a)\in F(Z;Y)$. So, assume that $Ff(a)\in F(Z;Y)$. Since $F$ has finite supports, there exists a finite subset $Z'\subset Z$ such that $Ff(a)\in F(Z'_d;Y)$. Since $F$ has finite supports, we can choose a finite set $X'\subset X$ such that $a\in F(X'_d;X)$. Replacing $X'$ by a larger finite set we can assume that $Z'\subset f(X')$. Put $Y'=f(X')$ and consider the map $f':X'_d\to Y'_d$, $f':x\mapsto f(x)$. Since $a\in F(X'_d;X)$, there is an element $a'\in F(X'_d)$ such that $a=Fi^d_{X',X}(a')$. Consider the element $b'=Ff'(a')\in FY'_d$.
Applying the functor $F$ to the equality $f\circ i^d_{X',X}=i^d_{Y',Y}\circ f'$, we get the equality  $Ff(a)=Ff\circ Fi^d_{X'X}(a')=Fi^d_{Y',Y}\circ Ff'(a')=Fi^d_{Y',Y}(b')$. So, $Fi^d_{Y',Y}(b')=Ff(a)\in F(Z'_d;Y)$ and we can choose an element $c\in FZ'_d$ such that $Ff(a)=Fi^d_{Z';Y}(c)$. Observe that $i^d_{Z',Y}=i^d_{Y',Y}\circ i^{dd}_{Z',Y'}$ and hence
$$Fi^d_{Y',Y}(b')=Ff(a)=Fi^d_{Z',Y}(c)=Fi^d_{Y',Y}(Fi^{dd}_{Z',Y'}(c)).$$
The injectivity of the map $Fi^b_{Y',Y}$ guarantees that $$Ff'(a')=b'=Fi^{dd}_{Z',Y'}(c)\in F(Z'_d;Y'_d).$$ Now we can apply the assumption (2) and conclude that $a'\in F(f'^{-1}(Z'_d);X_d')$. Then $a=Fi^d_{X',X}(a')\in F(X'\cap f^{-1}(Z);X)\subset F(f^{-1}(Z);X)$ and we are done.
\end{proof}

\begin{corollary}\label{c:preim-fin} Let $F:\Top_i\to\Top$ be a monomorphic  functor with finite supports. If $F$ does not preserve preimages, then there exists a surjective map $f:A\to B$ between non-empty finite discrete spaces such that $|B|=|A|-1$ and $Ff(a)\in F_{|B|-1}(B)$ for some $a\in F(A)\setminus F_{|A|-1}(A)$.
\end{corollary}

\begin{proof} By Proposition~\ref{p:non-preim}, there is a surjective map $f:X\to Y$ between finite discrete spaces of cardinality $|Y|=|X|-1$ and a non-empty subset $Z\subset Y$ such that $(Ff)^{-1}(F(Z;Y))\not\subset F(f^{-1}(Z);X)$. In this case $Z\ne Y$ and $|Y|\ge 2$. So, we can find an element $a\in FX$ such that $Ff(a)\in F(Z;Y)$ such that $a\notin F(f^{-1}(Z),X)$. Put $A=\supp(a)$ and $B=f(\supp(a))$. Taking into account that $a\notin F(f^{-1}(Z),X)$ we can apply Theorem~\ref{t:supp} and conclude that the set $A=\supp(a)$ is not empty and is not contained in the set $f^{-1}(Z)$. Then $B\not\subset Z$ and hence $B\cap Z\ne B$.

 If the map $f_a=f|A:A\to B$ is injective, then  $\supp(Ff(a))=f(\supp(a))$ by Proposition~\ref{p:supp=image}. Then $\supp(Ff(a))\subset Z$ implies $\supp(a)\subset f^{-1}(Z)$ and $a\in F(\supp(a);X)\subset F(f^{-1}(Z);X)$, which contradicts the choice of $a$. Therefore, the map $f|A$ is not injective and hence $|B|=|A|-1$.
Since the functor $F$ is monomorphic, the set $F(A;X)$ can be identified with $FA$. Then $a\in FA$ is a required element such that $a\notin F_{|A|-1}(A)$ (by Theorem~\ref{t:supp}) but $Ff(a)\in F_{|B|-1}(B)$. To see that the latter inclusion holds, observe that $Ff(a)\in F(Z;Y)\cap F(B;Y)$ implies that $\supp(Ff(a))\subset Z\cap B$ and hence $|\supp(Ff(a))|\le|Z\cap B|<|B|$.
\end{proof}

For a positive integer $n=\{0,\dots,n-1\}\in\IN$ let $r^{n+1}_n:n+1\to n$ be the retraction such that $r^{n+1}_n|n=\id$ and $r_n^{n+1}(n)=n-1$. Since any finite discrete space $X$ is homeomorphic to the cardinal $|X|$ endowed with the discrete topology, Corollary~\ref{c:preim-fin} implies its own canonical version:

\begin{corollary}\label{c:F-preim} Let $F:\Top_i\to\Top$ be a monomorphic  functor with finite supports. If $F$ does not preserve preimages, then for some $n\in\IN$ there exists an element $a\in F(n+1)\setminus F_n(n+1)$ such that $\supp(Fr^{n+1}_n(a))\subset n-1$.
\end{corollary}

\begin{definition}\label{d:stong-preim} We shall say that a monomorphic functor $F:\T\to \Top$ \index{functor!strongly fails to preserve preimages}{\em strongly fails to preserve preimages} if for some $n\ge 2$ there is an element $a\in F(n+1)\setminus F_n(n+1)$ such that $\supp(Fr^{n+1}_n(a))\subset n-2$.
\end{definition}

Functors which strongly fail to preserve preimages will be used in Section~\ref{s:monad}.

Finally we shall discuss the relation of functors that preserve preimages to functors with continuous finite supports.

\begin{definition} Let $F:\Top_i\to \Top$ be a monomorphic functor with finite supports, defined on a full subcategory $\Top_i$ of $\Top$, containing all finite discrete spaces.
We say that the functor $F$
\begin{itemize}
\item is \index{functor!continuous}{\em continuous} if for any object $X$ of $\Top_i$, number $n\in\w$ and element
$a\in F(n)$ the map $a_X:X^n\to FX$, $a_X:f\mapsto Ff(a)$, is continuous;
\item has \index{functor!with continuous supports}{\em continuous supports} if for every object $X$ of $\Top_i$ the support map $\supp:FX\to[X]^{<\w}$ is continuous with respect to the Vietoris topology on $[X]^{<\w}$ (this happens if and only if for any open sets $U\subset X$ the sets $\{a\in FX:\supp(a)\subset U\}$ and $\{a\in FX:\supp(a)\cap U\ne\emptyset\}$ are open in $FX$).
\end{itemize}
\end{definition}

\begin{proposition} Let $F:\Top_i\to\Top$ be a monomorphic functor with finite supports defined on a full subcategory of $\Top$, containing all finite discrete spaces. If $F$ is continuous and has continuous supports, then $F$ preserves preimages.
\end{proposition}

\begin{proof} Assuming that $F$ does not preserve preimages, we can apply Corollary~\ref{c:F-preim} and find a number $l\in\IN$ and an element $a\in F(n+1)\setminus F_n(n+1)$ such that $\supp(Fr^{n+1}_n(a))\subset  n-1$. Consider the map $f_\infty:n+1\to [0,1]$ defined by $f_\infty(i)=\frac1{n}r^{n+1}_n(i)$ for $i\in n+1$. Next, for every $m\in\IN$ consider the map $f_m:n+1\to[0,1]$ such that $f_m|n=f_\infty|n$ and $f_m(n)=\frac{n-1}n+\frac1m$. It is clear that the sequence $(f_m)_{m=1}^\infty$ converges to the map $f_\infty$ in the space $X^{n+1}$. By the continuity of $F$, the sequence $(a_m)_{m\in\w}$ of points $a_m=Ff_m(a)$ converges to the point $a_\infty=Ff_\infty(a)=Ff_\infty\circ Fr^{n+1}_n(a)$. It follows from $\supp(Fr^{n+1}_n(a))\subset n-1$ that $\supp(a_\infty)\subset \{\frac{i}n\}_{i\in n-1}\subset [0,\frac{n-1}n)$. Assuming that the functor $F$ has continuous supports, we would conclude that the set $W=\{b\in F\II:\supp(b)\subset [0,\frac{n-1}n)\}$ is an open neighborhood of $a_\infty$ in $FX$.
So, $W$ contains some points $a_m$, which is not possible as  $\supp(a_m)=\{\frac{i}{n}\}_{i\in n}\cup\{\frac{n-1}n+\frac1m\}\not\subset [0,\frac{n-1}n)$.
\end{proof}

    \section{Strong $\Cld$-fans in functor-spaces}\label{s:monads}

In this section we give conditions on a functor $F:\Top_i\to\Top$ guaranteeing that each Ascoli space $X\in\Top_i$ with Ascoli functor-space $FX$ is discrete.


In this subsection we assume that $\Top_i\subset\Top_{2\frac12}$ is a full hereditary subcategory of the category of functionally Hausdorff spaces and $\Top_i$ contains all metrizable compacta. We also assume that $F:\Top_i\to\Top$ is a functor, preserving closed embeddings of metrizable compacta.

\begin{lemma}\label{l:cfFI->cfFX} Let $\lambda$ be an infinite cardinal, $S$ be a convergent sequence in a space $X\in\Top_i$ and $f:S\to\II$ be an embedding such that for every strongly (strictly) compact-finite family $(A_\alpha)_{\alpha\in\lambda}$ of subsets of the space $F(S)$ the family $(Ff(A_\alpha))_{\alpha\in\lambda}$ is strongly (strictly) compact-finite in $F(\II)$. Then every strongly (strictly) compact-finite family of subsets $(B_n)_{n\in\w}$ in the space $F(S;X)\subset F(X)$ remains strongly (strictly) compact-finite in $FX$.
\end{lemma}

\begin{proof} Taking into account that the space $X$ is functionally Hausdorff, we can apply Lemma~\ref{l:fH-ext} and extend the function $f:S\to\II$ to a continuous function $\bar f:X\to \II$. Applying the functor $F$ to the equality $\bar f\circ i_{S,X}=f$, we obtain the equality $F\bar f\circ Fi_{S,X}=Ff$. By Proposition~\ref{p:FembTych}, the continuous maps $Ff$ and $Fi_{S,X}$ are closed topological embeddings.

Now take any strongly (strictly) compact-finite family $(B_\alpha)_{\alpha\in\lambda}$ of subsets of the space $F(S;X)$. Since $Fi_{S,X}:FS\to F(S;X)$ is a topological embedding, the family $(A_\alpha)_{\alpha\in\lambda}$ of the sets $A_\alpha=(Fi_{S,X})^{-1}(B_\alpha)$ is strongly (strictly) compact-finite in the space $FS$. By our assumption, the family $(Ff(A_\alpha))_{\alpha\in\lambda}=(F\bar f(B_\alpha))_{\alpha\in\lambda}$ is strongly (strictly) compact-finite in $F(\II)$.  Consequently, each set $F\bar f(B_\alpha)$, $\alpha\in\lambda$, has an $\IR$-open (functional) neighborhood $U_\alpha\subset F(\II)$ such that the family $(U_\alpha)_{\alpha\in\lambda}$ is compact-finite in $F(\II)$. It is easy to see that for every $\alpha\in\lambda$ the set $W_\alpha=(F\bar f)^{-1}(U_\alpha)$ is an $\IR$-open (functional)  neighborhood of the set $B_\alpha$ in $FX$ and the family $(W_\alpha)_{\alpha\in\lambda}$ is compact finite in $FX$, witnessing that the family $(B_\alpha)_{\alpha\in\lambda}$ is strongly (strictly) compact-finite in $FX$.
\end{proof}

Combining Lemma~\ref{l:cfFI->cfFX} with Corollary~\ref{c:s-sfan1}, we obtain:

\begin{corollary}\label{c:remain1} Let $S$ be a convergent sequence in a space $X\in \Top_i$. If the functor-space $F(\II)$ is stratifiable, then every strongly compact-finite family of (closed) subsets in the space $F(S;X)\subset F(X)$ is  (strictly) strongly compact-finite in $FX$.
\end{corollary}

\begin{corollary}\label{c:remain2} Let $S$ be a convergent sequence in a space $X\in\Top_i$. If the functor-space $F(\II)$ is a (normal) $\aleph$-space, then every  compact-finite countable family of (closed) subsets in the space $F(S;X)\subset F(X)$ is (strictly)  strongly compact-finite in $FX$.
\end{corollary}

\begin{corollary} Let $S$ be a convergent sequence in a space $X\in\Top_i$. If the functor-space $F(\II)$ is stratifiable and $FS$ is an $\aleph$-space, then every  compact-finite countable family of (closed) subsets in the space $F(S;X)\subset F(X)$ is (strictly)  strongly compact-finite in $FX$.
\end{corollary}

We recall that $\w+1=\w\cup\{\w\}$ stands for the standard convergent sequence.

\begin{theorem}\label{t:F-noseq} Assume that the space $F(\w+1)$ contains a (strong) $\Fin^\w$-fan and the space $F(\II)$ is stratifiable or an $\aleph$-space. If a Tychonoff space $X\in\Top_i$ contains a convergent sequence, then the functor-space $FX$ contains a (strong) $\Fin^\w$-fans.
\end{theorem}

\begin{proof} Assume that $S$ is a compact convergent sequence in $X$. By our assumption, the space $F(S)$ contains a (strong) $\Fin^\w$-fan $(D_n)_{n\in\w}$. By Proposition~\ref{p:FembTych}, the map $Fi_{S,X}:F(S)\to F(X)$ is a closed topological embedding. So, we can identify the space $F(S)$ with its image $F(S;X)$.
By Corollaries~\ref{c:remain1} and \ref{c:remain2} the (strongly) compact-finite family $(D_n)_{n\in\w}$ in $F(S;X)$ remains (strongly) compact-finite in $FX$.
So, it is a (strong) $\Fin^\w$-fan in $FX$, which contradicts our assumption.
\end{proof}

Theorems~\ref{t:F-noseq} and \ref{t:disc-char} imply:

\begin{corollary}\label{c:F-disc} Assume that the space $F(\w+1)$ contains a strong $\Fin^\w$-fan and the space $F(\II)$ is stratifiable or an $\aleph$-space. A Tychonoff space $X\in\Top_i$ is discrete if the following conditions are satisfied:
\begin{enumerate}
\item[\textup{1)}] the space $FX$ contains no strong $\Fin^\w$-fans,
\item[\textup{2)}] each infinite compact set in $X$ contains a convergent sequence;
\item[\textup{3)}] the space $X$ contains no $\Clop$-fan;
\item[\textup{4)}] $X$ is zero-dimensional or contain no strict $\Cld^\w$-fan.
\end{enumerate}
\end{corollary}

Since Ascoli spaces contain not strict $\Cld$-fans and no strong $\Fin$-fans, Corollary~\ref{c:F-disc} implies the following corollary.

\begin{corollary} Assume that the space $F(\w+1)$ contains a strong $\Fin^\w$-fan and the space $F(\II)$ is stratifiable or an $\aleph$-space. A Tychonoff space $X\in\Top_i$ is discrete if each infinite compact set in $X$ contains a convergent sequence and the spaces $X$ and $FX$ are Ascoli.
\end{corollary}

For functors mapping compact maps to quotient maps we can prove more.
We recall that a map $f:X\to Y$ between topological spaces is called
\begin{itemize}
\item \index{map!closed}{\em closed} (resp. {\em open}) if for each closed (resp. open) set $A\subset X$ the set $f(A)$ is closed (resp. open) in $Y$;
\item \index{map!perfect}{\em perfect} if $f$ is closed and for each point $y\in Y$ the preimage $f^{-1}(y)$ is compact;
\item \index{map!compact}{\em compact} if $f$ is perfect, surjective and there is a compact subset $K\subset Y$ such that for every point $y\in Y\setminus K$ the preimage $f^{-1}(y)$ is a singleton.
\end{itemize}

A functor $F:\Top_{3\frac12}\to\Top$ will be called
 \begin{itemize}
 \item  \index{functor!compact-to-quotient}{\em compact-to-quotient} if for any compact map $f:X\to Y$ between Tychonoff spaces the map $Ff:FX\to FY$ is quotient;
 \item \index{functor!quotient-to-open}{\em quotient-to-open}     if for any quotient map $f:X\to Y$ between Tychonoff spaces the map $Ff:FX\to FY$ is open.
  \end{itemize}
  It is clear that each quotient-to-open functor is compact-to-quotient.

\begin{theorem}\label{t:discrete2} Assume a functor $F:\Top_{3\frac12}\to \Top$ preserves closed embeddings of metrizable compacta and is compact-to-quotient. Assume that  the space $F(\w+1)$ contains a strong $\Fin^\w$-fan and the space $F(\II)$ is stratifiable or an $\aleph$-space. A Tychonoff space $X$ is discrete if $X$ contains no strong $\Cld^\w$-fans, no $\Clop$-fans and the space $FX$ contains no strong $\Cld^\w$-fans.
\end{theorem}

\begin{proof} Assume that a Tychonoff space $X$ contains no strong $\Cld^\w$-fan and no $\Clop$-fan, and  the functor-space $FX$ contains no strong $\Cld^\w$-fan. We claim that the space $X$ contains no compact infinite set. To derive a contradiction, assume that $K$ is an infinite compact set in $X$. Fix any continuous map $f:K\to M$ onto an infinite compact metrizable space $M$. The map $f$ determines a closed equivalence relation $$E=\{(x,y)\in X\times X:\mbox{$x=y$ or $x,y\in K$ and $f(x)=f(y)$}\}.$$It can be shown that the quotient map $q:X\to X/E$ is compact and the quotient space $X/E$ is Tychonoff. Since the functor $F$ is compact-to-quotient, the map $Fq:FX\to F(X/E)$ is quotient. Since the space $FX$ contains no strong $\Cld^\w$-fan, we can apply Proposition~\ref{p:quot-fan} and conclude that the space $X/E$ contains no strong $\Cld^\w$-fan. On the other hand, the space $X/E$ contains an infinite compact metrizable space and hence contains a convergent sequence $S$.
By our assumption, the functor-space $FS$ contains a strong $\Cld^\w$-fan and then by Corollary~\ref{c:remain1} or \ref{c:remain2}, the space $F(X/E)$ contains a strong $\Cld^\w$-fan too and this is a desired contradiction, which shows that the space $X$ contains no infinite compact set. Now Corollary~\ref{c:F-disc} implies that the space $X$ is discrete.
\end{proof}

\begin{corollary} Assume a functor $F:\Top_{3\frac12}\to \Top$ preserves closed embeddings of metrizable compacta and is compact-to-quotient. Assume that  the space $F(\w+1)$ contains a strong $\Fin^\w$-fan and the space $F(\II)$ is stratifiable or an $\aleph$-space. Then each Ascoli space $X$ with Ascoli functor-space $F(X)$ is discrete.
\end{corollary}

\chapter{Constructing $\Fin$-fans in topological algebras}\label{ch:algebra}

In this section we obtain many quite general results on the existence of (strong) $\Fin$-fans in functor-spaces possessing a kind of algebraic structure. We shall consider six such  algebraic structures, called topological algebras of types $xy$, $x(y^*z)$, $(x^*y){}(z^*s)$,  $x(s^*y_i)^{<\w}$, $x(y^*(s^*z)y)$,  and $(x^*(s^*z)x){}(y^*(s^*z)y)$.

These six structures are introduced with the help of $F$-valued $n$-ary operations, defined as follows. Let $\Top_i$ be a full hereditary subcategory of the category $\Top$, $F:\Top_i\to\Top$ be a functor and $n\in\w$. By an \index{operation!$n$-ary $F$-valued}{\em $n$-ary $F$-valued operation} on a space $X\in\Top_i$ we understand a function $p:X^n\to FX$, (not necessarily continuous). Elements of the power $X^n$ are functions $x:n\to X$, which can be written as vectors $(x(0),\dots,x(n-1))$.

In the following six subsections we shall consider six $n$-ary $F$-valued operations possessing some special properties. {\em In these subsections we assume that $F:\Top_i\to\Top$ is a bounded $\II$-regular monomorphic functor with finite supports}. In this case Lemma~\ref{l:supp-Ropen} guarantees that for every functionally Hausdorff space $X$ the support map $\supp:FX\to[X]^{<\w}$ is lower semicontinuous in the sense that for every $\IR$-open subset $U\subset X$ the set $W=\{a\in FX:\supp(a)\cap U\ne\emptyset\}$ is $\IR$-open in $FX$.

For a topological space $X$ by $X_d$ we denote the space $X$ endowed with the discrete topology. By $\Delta$  we shall denote the diagonal $\{(x,y)\in X\times X:x=y\}$ of the square $X\times X$.

\section{Topological algebras of type $x{\cdot}y$}

\begin{definition}\label{d:type-xy}
Given a space $X\in\Top_i$, we say that the functor-space $FX$ is a \index{topological algebra!of type $x{\cdot}y$}{\em topological algebra of type $x{\cdot}y$} if there are a finite subset $C\subset X$ and a binary $F$-valued operation $p:X^{2}\to FX$ satisfying the following properties:
\begin{enumerate}
\item[1)] $\{x,y\}\subset \supp\big(p(x,y)\big)$ for any distinct points $x,y\in X\setminus C$;
\item[2)] for any points $x,y\in X\setminus C$ and any neighborhood $W\subset FX$ of $p(x,y)$ there are neighborhoods $O_x,O_y\subset X$ and $x,y$ such that $p(O_x\times O_y)\subset W$.
\end{enumerate}
\end{definition}

\begin{lemma}\label{l:type-xy} Assume that a $\mu$-complete (functionally Hausdorff) space $X\in\Top_i$ contains a (strong) $D_\w$-cofan $(D_n)_{n\in\w}$ and a (strong) $\bar S^\w$-fan $(S_n)_{n\in\w}$. If the functor-space $FX$ is a topological algebra of type  $x{\cdot}y$, then the functor-space $FX$ contains a (strong) $\Fin^\w$-fan.
\end{lemma}

\begin{proof} Fix a finite set $C\subset X$ and an operation $p:X^2\to FX$ witnessing that $FX$ is a topological algebra of type $xy$.

Let $(S_n)_{n\in\w}$ be a (strong) $\bar S^\w$-fan in $X$. Replacing the sequence  $(S_n)_{n\in\w}$ by a suitable subsequence, we can assume that the compact-finite family $(S_n)_{n\in\w}$ is disjoint and the set of limit points $\{\lim S_n\}_{n\in\w}$ has compact closure in $X$, which does not intersect the union $\bigcup_{n\in\w}S_n$. Moreover, simultaneously replacing $(S_n)_{n\in\w}$ by a smaller infinite subfamily and each compact-finite set $D_n$, $n\in\w$ by a suitable infinite subset, we can assume that the unions $\bigcup_{n\in\w}D_n$ and $\bigcup_{n\in\w}S_n$ are disjoint and are contained in $X\setminus C$.

For every $n\in\w$ fix an injective enumeration $(x_{n,m})_{m\in\w}$ of the set $D_n$ and an injective enumeration $(y_{n,m})_{m\in\w}$ of the convergent sequence $S_n$.

By Definition~\ref{d:type-xy}, for any distinct elements $x,y\in X\setminus C$ the element
$w=p(x,y)$ has support $\supp(w)\supset\{x,y\}$.
In particular,  for every $n,m\in\w$ the point $w_{n,m}=p(x_{n,m},y_{n,m})$ of $FX$ has support $\supp(w_{n,m})\supset \{x_{n,m},y_{n,m}\}$.

We claim that the family of singletons $(\{w_{n,m}\})_{n,m\in\w}$ is a  $\Fin^\w$-fan in $FX$. Let $x$ be the limit point of the $D_\w$-cofan $(D_n)_{n\in\w}$. Since the set $\{\lim S_n\}_{n\in\w}$ has compact closure in $X$, the sequence $(\lim S_n)_{n\in\w}$ has an accumulation point $y\in X$.

We claim that the family of singletons $(\{w_{n,m}\})_{n,m\in\w}$ is not locally finite at the point $p(x,y)$. Given any neighborhood $W\subset FX$ of $p(x,y)$, use Definition~\ref{d:type-xy}(2) to find two open sets $O_x,O_y\subset X$ such that $x\in O_x$, $y\in O_y$ and $p(O_x\times O_y)\subset W$.

Since the sequence $(D_n)_{n\in\w}$ converges to $x$, there exists a number $n_0\in\w$ such that $\{x_{n,m}:n\ge n_0,\;m\in\w\}\subset O_x$. Since $y$ is an accumulation point of the sequence $(\lim S_n)_{n\in\w}$, there exists $n\ge n_0$ such that $\lim S_n\in O_y$. Then $O_y$ is an open neighborhood of the point $\lim S_n$ and we can find a number $m_0\in\w$ such that $\{y_{n,m}\}_{m\ge m_0}\subset O_y$. Then
$w_{n,m}=p(x_{n,m},y_{n,m})\in p(O_x\times O_y)\subset W$ for every $m\ge m_0$,  which means that the family $(\{w_{n,m}\})_{n,m\in\w}$ is not locally finite at the point $p(x,y)$.

Next, we prove that the family of singletons $(\{w_{n,m}\})_{n,m\in\w}$ is compact-finite in $FX$.  Fix any compact subset $K\subset FX$. The boundedness of the functor $F$ and the $\mu$-completeness of $X$ guarantee that $K\subset F(B;X)$ for some closed bounded (and hence compact) non-empty subset $B\subset X$.  Since the family $(S_n)_{n\in\w}$ is compact-finite, there is $n_0\in\w$ such that $B\cap S_n=\emptyset$ for all $n>n_0$. For every $n\le n_0$ the set $D_n$ is compact-finite in $X$ and hence the intersection $B\cap D_n$ is finite. So, we can find a number $m_0\in\w$ such that $B\cap\bigcup_{n\le n_0}D_n\subset\{x_{n,m}:n\le n_0,\;m\le m_0\}$.
 We claim that $\{(n,m)\in\w\times\w:w_{n,m}\in K\}\subset [0,n_0]\times [0,m_0]$. Indeed, if for some $n,m\in\w$ the point $w_{n,m}$ belongs to $K$, then $w_{n,m}\in F(B;X)$ and hence
$\{x_{n,m},y_{n,m}\}\subset\supp(w_{n,m})\subset B$. The choice of the numbers $n_0$ and $m_0$ guarantee that $n\le n_0$ and then $m\le m_0$, which implies that the sets $$\{(n,m)\in\w\times \w:w_{n,m}\in K\}\subset \{(n,m)\in\w\times \w:w_{n,m}\in F(B;X)\}\subset [0,n_0]\times [0,m_0],$$are finite. Therefore, the family $(\{w_{n,m}\})_{n,m\in\w}$ if compact-finite and hence is a $\Fin^\w$-fan.

Now assuming that the space $X$ is functionally Hausdorff, and the cofan $(D_n)_{n\in\w}$ and the fan $(S_n)_{n\in\w}$
are strong, we shall prove that the fan $(\{w_{n,m}\})_{n,m\in\w}$ is strong. Since the cofan $(D_n)_{n\in\w}$ is strong, every set $D_n$ is strongly compact-finite and hence each point $x_{n,m}\in D_n$ has an $\IR$-open neighborhood $U_{n,m}\subset X$ such that the family $(U_{n,m})_{m\in\w}$ is compact-finite. Since the fan $(S_n)_{n\in\w}$ is strong, every set $S_n$ has an $\IR$-open neighborhood $V_n\subset X$ such that the family $(V_n)_{n\in\w}$ is compact-finite. Since $\supp(w_{n,m})\supset\{x_{n,m},y_{n,m}\}$ the set $$W_{n,m}=\{a\in F(X):\supp(a)\cap U_{n,m}\ne \emptyset\ne\supp(a)\cap V_n\}$$ is an $\IR$-open neighborhood of $w_{n,m}$ in $FX$ (see Lemma~\ref{l:supp-Ropen}). We claim that the family $(W_{n,m})_{n,m\in\w}$ is compact-finite. For this observe that $W_{n,m}=\widetilde U_{n,m}\cap \widetilde V_n$ where $\widetilde U_{n,m}=\{a\in FX:\supp(a)\cap U_{n,m}\ne\emptyset\}$ and $\widetilde V_n=\{a\in FX:\supp(a)\cap V_n\ne\emptyset\}$.

Take any compact set $K\subset FX$. By Lemma~\ref{l:cont-supp}, the family $(\widetilde V_n)_{n\in\w}$ is compact-finite in $FX$. So, there is a number $n_0\in\w$ such that $K\cap\widetilde V_n=\emptyset$ for all $n> n_0$. By Lemma~\ref{l:cont-supp}, for every $n\le n_0$ the family $(\widetilde U_{n,m})_{m\in\w}$ is compact-finite in $FX$, which allows us to find a number $m_0\in\w$ such that $K\cap \widetilde U_{n,m}=\emptyset$ for all $n\le n_0$ and $m>m_0$. Now we see that the set
$$\{(n,m)\in\w^2:K\cap W_{n,m}\ne\emptyset\}=\{(n,m)\in\w^2:K\cap\widetilde U_{n,m}\ne\emptyset\ne K\cap\widetilde V_n\}\subset[0,n_0]\times[0,m_0]$$is finite, which implies that the family $(W_{n,m})_{n,m\in\w}$ is compact-finite and hence the family $(\{w_{n,m}\})_{n,m\in\w}$ is strongly compact-finite, which means that the $\Fin^\w$-fan $(\{w_{n,m}\})_{n,m\in\w}$ is strong.
\end{proof}

Topological algebras of type $x{\cdot}y$ also have another useful property:

\begin{lemma}\label{l:x.y} Assume that a (functionally Hausdorff) $\mu$-complete space $X\in\Top_i$ contains a (strong) $\bar S^{\w_1}$-fan. If the functor-space $FX$ is a topological algebra of type $x{\cdot}y$, then the functor-space $FX$ contains a (strong) $\Fin^{\w_1}$-fan.
\end{lemma}

\begin{proof} Let a finite subset $C\subset X$ and a binary operation $p:X^2\to FX$ witness that the functor-space $FX$ is a topological algebra of type $x{\cdot}y$.

Let $(S_\alpha)_{\alpha\in\w_1}$ be a (strong) $\bar S^{\w_1}$-fan in $X$. Since the sets $\{\lim S_\alpha\}_{\alpha\in\w_1}$ and $S_\alpha$, $\alpha\in\w_1$, have compact closures in $X$  we can replace $(S_\alpha)_{\alpha\in\w_1}$ by a suitable uncountable subfamily and assume that this family is disjoint and its union does not intersect the closure of the set $C\cup\{\lim S_\alpha\}_{\alpha\in\w_1}$.

For every $\alpha\in\w_1$ chose a sequence $(x_{\alpha,n})_{n\in\IN}$ of pairwise distinct points in $S_\alpha$.
Choose any almost disjoint family $(A_\alpha)_{\alpha\in\w_1}$ of infinite subsets of $\w$.
Let $\Lambda=\{(\alpha,\beta)\in\w_1\times w_1:\alpha\ne \beta\}$ and for any pair $(\alpha,\beta)\in\Lambda$ consider the finite subset $$D_{\alpha,\beta}=\{p(x_{\alpha,n},x_{\beta,n}):n\in A_\alpha\cap A_\beta\}$$of $FX$.

We claim the family $(D_{\alpha,\beta})_{(\alpha,\beta)\in\Lambda}$ is not locally finite in $FX$. Since the set $L=\{\lim S_\alpha\}_{\alpha\in\w_1}$ has compact closure $\bar L$ in $X$, the $\w_1$-sequence $(\lim S_\alpha)_{\alpha\in\w_1}$ has a condensation point $x\in \bar L$. The latter means that every neighborhood $O_x\subset X$ of $x$ contains uncountably many points $\lim S_\alpha$.

We claim that the family $(D_{\alpha,\beta})_{(\alpha,\beta)\in\Lambda}$ is not locally finite at the point $p(x,x)$. Given any neighborhood $W\subset FX$ of $p(x,x)$ use Definition~\ref{d:type-xy}(2) and find a neighborhood $O_x\subset X$ of $x$ such that $p(O_x\times O_x)\subset W$. Since $x$ is a condensation point of the $\w_1$-sequence $(\lim S_\alpha)_{\alpha\in\w_1}$, the set $\Omega=\{\alpha\in\w_1:\lim S_\alpha\in O_x\}$ is uncountable.

For every $\alpha\in\Omega$ the open set $O_x$ is a neighborhood of the limit point $\lim S_\alpha$ of the sequence $S_\alpha=\{x_{\alpha,n}\}_{n\in\w}$. So there is a number $\varphi(\alpha)\in\w$ such that $\{x_{\alpha,n}\}_{n\ge\varphi(\alpha)}\subset O_x$.
By the Pigeonhole Principle, for some $m\in\w$ the set $\Omega_m=\{\alpha\in\Omega:\varphi(\alpha)=m\}$ is uncountable. By Lemma~\ref{l:ad}, the set $\Lambda'=\{(\alpha,\beta)\in\Lambda:\alpha,\beta\in\Omega_m,\;A_\alpha\cap A_\beta\not\subset[0,m]\}$ is uncountable. Fix any pair $(\alpha,\beta)\in\Lambda'$ and choose a number $n\in A_\alpha\cap A_\beta\setminus[0,m]$. Then $x_{\alpha,n},x_{\beta,n}\in O_x$ and hence $$D_{\alpha,\beta}\ni  p(x_{\alpha,n},x_{\beta,n})\subset p(O_x\times O_x)\subset W$$which implies that $D_{\alpha,\beta}\cap W\ne\emptyset$ and the set $\{(\alpha,\beta)\in\Lambda:D_{\alpha,\beta}\cap W\ne\emptyset\}\supset\Lambda'$ is uncountable. This means that the family $(D_{\alpha,\beta})_{(\alpha,\beta)\in\Lambda}$ is not locally countable and hence not locally finite at $p(x,x)$.

Next, we show that the family $\{D_{\alpha,\beta}\}_{(\alpha,\beta)\in\Lambda}$ is compact-finite in $FX$. Since $$\supp\big(p(x_{\alpha,n},x_{\beta,n})\big)
\supset\{x_{\alpha,n},x_{\beta,n}\},$$ we can use  the compact-finiteness of the family $(S_\alpha)_{\alpha\in\w_1}$ and Lemma~\ref{l:cont-supp} to conclude that the family $(D_{\alpha,\beta})_{(\alpha,\beta)\in\Lambda}$ is compact-finite in $FX$.

If the $\bar S$-fan $(S_\alpha)_{\alpha\in\w_1}$ is strong and the space $X$ is functionally Hausdorff, then each set $S_\alpha$ has an $\IR$-open neighborhood $V_\alpha\subset X$ such that the family $(V_\alpha)_{\alpha\in\w_1}$ is compact-finite. Definition~\ref{d:type-xy}(1) and Lemma~\ref{l:supp-Ropen} imply that for every pair $(\alpha,\beta)\in\Lambda$ the set  $W_{\alpha,\beta}=\{a\in FX:\supp(a)\cap U_\alpha\ne\emptyset\ne\supp(a)\cap U_\beta\}$ is an $\IR$-open neighborhood of the finite set $D_{\alpha,\beta}$ in $FX$. Applying Lemma~\ref{l:cont-supp}, we can show that the family $(W_{\alpha,\beta})_{(\alpha,\beta)\in\Lambda}$ is compact-finite in $FX$, which implies that the family $(D_{\alpha,\beta})_{(\alpha,\beta)\in\Lambda}$ is a strong $\Fin^\w$-fan in $FX$.
\end{proof}

Lemma~\ref{l:type-xy} and Theorems~\ref{t:Sfan} and \ref{t:Dcofan} imply:

\begin{theorem}\label{t:xy} Let $X\in\Top_i$ be a $\mu$-complete space such that the space $FX$ is a topological algebra of type $x{\cdot}y$. Assume that one of the following conditions is satisfied:
\begin{enumerate}
\item[\textup{1)}] $FX$ contains no strong $\Fin^\w$-fan and $X$ is a Tychonoff $\aleph$-space;
\item[\textup{2)}] $FX$ contains no $\Fin^\w$-fan and $X$ is a $k^*$-metrizable space.
\end{enumerate}
Then $X$ is either $k$-metrizable or a $k$-sum of hemicompact spaces. Moreover, if $X$ is a $k_\IR$-space, then $X$ is either metrizable or a topological sum of $k_\w$-spaces.
\end{theorem}

\section{Topological algebras of type $x{} (y^*z)$}\label{s:xyz}

\begin{definition}\label{d:xyz}
Given a space $X\in\Top_i$, we say that the functor-space $FX$ is a \index{topological algebra!of type $(x(y^*z)$}{\em topological algebra of type $x{}(y^*z)$} if there exist a finite subset $C\subset X$ and a ternary $F$-valued operation $p:X^3\to FX$ satisfying the following properties:
\begin{enumerate}
\item[1)] $p(x,y,y)=p(x,x,x)$ for all $x,y\in X$;
\item[2)] $\{x,z\}\subset \supp(p(x,y,z))$ for any pairwise distinct points $x,y,z\in X\setminus C$;
\item[3)] for any point $x\in X$ and neighborhood $U\subset FX$ of $p(x,x,x)$ there are a neighborhood $O_x\subset X$ of $x$ and a neighborhood $U_\Delta$ of the diagonal $\Delta_X$ in the space $X_d\times X$ such that $p(O_x\times U_\Delta)\subset U$.
\end{enumerate}
\end{definition}

\begin{lemma}\label{l:xyz} Assume that a $\mu$-complete (functionally Hausdorff) space $X\in\Top_i$ contains a (strong) $D_\w$-cofan $(D_n)_{n\in\w}$ and a (strong) $S^\w$-semifan $(S_n)_{n\in\w}$ with infinite set of limit points $\{\lim S_n\}_{n\in\w}$. If the functor-space $FX$ is a topological algebra of type $x{}(y^*z)$, then the space $FX$ contains a (strong) $\Fin^\w$-fan.
\end{lemma}

\begin{proof} Let a finite subset $C\subset X$ and an operation $p:X^3\to FX$ witness that $FX$ is a topological algebra of type $x{}(y^*z)$. Let $x$ be the limit point of the $D_\w$-cofan $(D_n)_{n\in\w}$.
Since the semifan $(S_n)_{n\in\w}$ is compact-finite, we can replace this family by a suitable subfamily and assume that it is disjoint, $x\notin\bigcup_{n\in\w}S_n$ and the infinite set $\{\lim C_n\}_{n\in\w}$ is disjoint with $C$. Moreover, replacing simultaneously the sets $D_n$ of the $D_\w$-cofan by suitable infinite subsets of $D_n$ and the family $(S_n)_{n\in\w}$ by a suitable subfamily, we can assume that the family $(D_n)_{n\in\w}$ is disjoint and the sets $\bigcup_{n\in\w}D_n$,$\{x\}\cup\bigcup_{n\in\w}\bar S_n$, and $C\setminus \{x\}$ are pairwise disjoint.

For every $n\in\w$ fix an injective enumeration $\{x_{n,m}\}_{m\in\w}$ of the set $D_n$.
For every $n\in\w$ let $y_n$ be the limit point of the convergent sequence $S_n$ and $\{y_{n,m}\}_{m\in\w}$ be an injective enumeration of the set $S_n\setminus\{y_n\}$.

By Definition~\ref{d:xyz}(2), for every $n,m\in\w$ the point $w_{n,m}=p(x_{n,m},y_{n},y_{n,m})$ of $FX$ has support\break $\supp\big(p(x_{n,m},y_n,y_{n,m})\big)\supset\{x_{n,m},y_{n,m}\}$.

We claim that the family of singletons $(\{w_{n,m}\})_{n,m\in\w}$ is a  $\Fin^\w$-fan in $FX$. First we check that this family is not locally finite at the point $p(x,x,x)$. Given any neighborhood $W\subset FX$ of $p(x,x,x)$, apply Definition~\ref{d:xyz}(3) and find a neighborhood $O_x\subset X$ of $x$ and an open neighborhood $U_\Delta\subset X_d\times X$ of the diagonal in the space $X_d\times X$ such that $p(O_x\times U_\Delta)\subset W$.
Since the sequence $(D_n)_{n\in\w}$ converges to $x$, there exists a number $n\in\w$ such that $\{x_{n,m}\}_{m\in\w}\subset O_x$. Since the sequence $(y_{n,m})_{m\in\w}$ converges to $y_n$, we can find a  number $m_0\in\w$ such that $\{(y_n,y_{n,m})\}_{m\ge m_0}\subset U_\Delta$. Then
$w_{n,m}=p(x_{n,m},y_n,y_{n,m})\in p(O_x\times U_\Delta)\subset W$ for every $m\ge m_0$,  which means that the family $(\{w_{n,m}\})_{n,m\in\w}$ is not locally finite at the point $p(x,x,x)$.

Next, we prove that the family of singletons $(\{w_{n,m}\})_{n,m\in\w}$ is compact-finite in $FX$.  Fix any compact subset $K\subset FX$. The boundedness of the functor $F$ and the $\mu$-completeness of $X$ guarantee that $K\subset F(B;X)$ for some compact non-empty subset $B\subset X$. Since the family $(S_n)_{n\in\w}$ is compact-finite, there is $n_0\in\w$ such that $B\cap S_n=\emptyset$ for all $n>n_0$. For every $n\le n_0$ the set $D_n$ is compact-finite in $X$ and hence the intersection $B\cap D_n$ is finite. So, we can find a number $m_0\in\w$ such that $B\cap\bigcup_{n\le n_0}D_n\subset\{x_{n,m}:n\le n_0,\;m\le m_0\}$.
 We claim that $\{(n,m)\in\w\times\w:w_{n,m}\in K\}\subset [0,n_0]\times [0,m_0]$. Indeed, if for some $n,m\in\w$ the point $w_{n,m}$ belongs to $K$, then $w_{n,m}\in F(B;X)$ and hence
$\{x_{n,m},y_{n,m}\}\subset\supp(w_{n,m})\subset B$. The choice of the numbers $n_0$ and $m_0$ guarantee that $n\le n_0$ and then $m\le m_0$, which implies that the sets $$\{(n,m)\in\w\times \w:w_{n,m}\in K\}\subset \{(n,m)\in\w\times \w:w_{n,m}\in F(B;X)\}\subset [0,n_0]\times [0,m_0],$$are finite. Therefore, the family $(\{w_{n,m}\})_{n,m\in\w}$ if compact-finite and hence is a $\Fin^\w$-fan.

Now assuming that the cofan $(D_n)_{n\in\w}$ and the semifan $(S_n)_{n\in\w}$
are strong and the space $X$ is functionally Hausdorff, we shall prove that the fan $(\{w_{n,m}\})_{n,m\in\w}$ is strong. Since the cofan $(D_n)_{n\in\w}$ is strong, every set $D_n$ is strongly compact-finite and hence each point $x_{n,m}\in D_n$ has an $\IR$-open neighborhood $U_{n,m}\subset X$ such that the family $(U_{n,m})_{m\in\w}$ is compact-finite. Since the semifan $(S_n)_{n\in\w}$ is strong, every set $S_n$ has an $\IR$-open neighborhood $V_n\subset X$ such that the family $(V_n)_{n\in\w}$ is compact-finite. By Definition~\ref{d:xyz}(2), for every $n,m\in\w$ the point $w_{n,m}= p(x_{n,m},y_n,y_{n,m})$ has support $\supp(w_{n,m})\supset\{x_{n,m},y_{n,m}\}$ and then by Lemma~\ref{l:supp-Ropen}, the set $$W_{n,m}=\{a\in F(X):\supp(a)\cap U_{n,m}\ne \emptyset\ne\supp(a)\cap V_n\}$$ is an $\IR$-open neighborhood of $w_{n,m}$ in $FX$. We claim that the family $(W_{n,m})_{n,m\in\w}$ is compact-finite. For this observe that $W_{n,m}=\widetilde U_{n,m}\cap \widetilde V_n$ where $\widetilde U_{n,m}=\{a\in FX:\supp(a)\cap U_{n,m}\ne\emptyset\}$ and $\widetilde V_n=\{a\in FX:\supp(a)\cap V_n\ne\emptyset\}$.

Take any compact set $K\subset FX$. By Lemma~\ref{l:cont-supp}, the family $(\widetilde V_n)_{n\in\w}$ is compact-finite in $FX$. So, there is a number $n_0\in\w$ such that $K\cap\widetilde V_n=\emptyset$ for all $n> n_0$. By Lemma~\ref{l:cont-supp}, for every $n\le n_0$ the family $(\widetilde U_{n,m})_{m\in\w}$ is compact-finite in $FX$, which allows us to find a number $m_0\in\w$ such that $K\cap \widetilde U_{n,m}=\emptyset$ for all $n\le n_0$ and $m>m_0$. Now we see that the set
$$\{(n,m)\in\w^2:K\cap W_{n,m}\ne\emptyset\}=\{(n,m)\in\w^2:K\cap\widetilde U_{n,m}\ne\emptyset\ne K\cap\widetilde V_n\}\subset[0,n_0]\times[0,m_0]$$is finite, which implies that the family $(W_{n,m})_{n,m\in\w}$ is compact-finite and hence the family $(\{w_{n,m}\})_{n,m\in\w}$ is strongly compact-finite, which means that the $\Fin^\w$-fan $(\{z_{n,m}\}_{n,m\in\w}$ is strong.
\end{proof}

Lemmas~\ref{l:type-xy} and \ref{l:xyz} imply:

\begin{corollary}\label{c:xy+xyz} Assume that a $\mu$-complete (functionally Hausdorff) space $X\in\Top_i$ contains a (strong) $D_\w$-cofan $(D_n)_{n\in\w}$ and a (strong) $S^\w$-semifan $(S_n)_{n\in\w}$. If the functor-space $FX$ is a topological algebra of types $x{\cdot}y$ and $x{}(y^*z)$, then the space $FX$ contains a (strong) $\Fin^\w$-fan.
\end{corollary}

\begin{proof} If the set $\{\lim S_n\}_{n\in\w}$ is finite, then the (strong) semifan $(S_n)_{n\in\w}$ is a (strong) $\bar S^\w$-fan and we can apply Lemma~\ref{l:type-xy} to conclude that $FX$ contains a (strong) $\Fin^\w$-fan. If the set $\{\lim S_n\}_{n\in\w}$ is infinite, then the functor-space $FX$ contains a (strong) $\Fin^\w$-fan by Lemma~\ref{l:xyz}.
\end{proof}

Corollary~\ref{c:xy+xyz} and Theorems~\ref{t:Dcofan} and \ref{t:Ssemifan} imply the following theorem:

\begin{theorem}\label{t:xy+xyz} Let $X\in\Top_i$ be a $\mu$-complete space such that the functor-space $FX$ is a topological algebra of types $xy$ and $x(y^*z)$.
\begin{enumerate}
\item If $X$ is Tychonoff and the functor-space $FX$ contains no strong $\Fin^\w$-fan, then:
\begin{enumerate}
\item If $X$ is cosmic, then $X$ is $\sigma$-compact;
\item If $X$ is a $\bar\aleph_k$-space, then  either $X$ contains a compact set with $k$-discrete complement or $X$ is a $k$-sum of hemicompact spaces.
\item If $X$ is a $\bar\aleph_k$-$k_\IR$-space, then either $X$ a topological sum of cosmic $k_\w$-spaces or $X$ is metrizable and has compact set of non-isolated points.
\item If $X$ is an $\aleph$-space, then the $k$-modification $kX$ of $X$ either is a topological sum of cosmic $k_\w$-spaces or $kX$ is a metrizable space with compact set of non-isolated points.
\end{enumerate}
\item If the space $FX$ contains no $\Fin^{\w}$-fan and $X$ is $k^*$-metrizable, then its $k$-modification $kX$ is either a metrizable space with compact set of non-isolated points or $kX$ is a topological sum of  cosmic $k_\w$-spaces.
\end{enumerate}
\end{theorem}
\newpage

\section{Topological algebras of type $(x^*y){}(z^*s)$}\label{s:ta:xy.zs}

\begin{definition}\label{d:xy.zs}
Given a space $X\in\Top_i$, we say that the functor-space $FX$ is a \index{topological algebra!of type $(x^*y)(z^*s)$}{\em topological algebra of type $(x^*y){}(z^*s)$} if there exist a finite subset $C\subset X$ and an $F$-valued operation $p:X^4\to FX$ having the properties:
\begin{enumerate}
\item[1)] $p(x,x,x,x)=p(x,x,z,z)=p(z,z,z,z)$ for all $x,z\in X$;
\item[2)]  $\{y,s\}\subset \supp\big(p(x,y,z,s)\big)$ for any pairwise distinct points $x,y,z,s\in X\setminus C$;
\item[4)] for any point $x\in X$ and neighborhood $W\subset FX$ of $p(x,x,x,x)$ there is a neighborhood $U_\Delta\subset X_d\times X$ of the diagonal $\Delta$ of $X^2$ such that $p(U_{\Delta}\times U_\Delta)\subset W$.
\end{enumerate}
A topological algebra of type $(x^*y){}(z^*s)$ is called a \index{topological superalgebra of type  $(x^*y)(z^*s)$}{\em topological superalgebra of type} $(x^*y){}(z^*s)$ if  it has an additional property:
\begin{enumerate}
\item[4)] for a compact subset $B\subset X$ and a compact subset $K\subset FX$ there exists a finite subset $A\subset X$ such that for any pairwise distinct points $x,y,z,s\in B\setminus C$ with $p(x,y,z,s)\in K$ we get $\{x,z\}\subset A$.
\end{enumerate}
\end{definition}

\begin{lemma}\label{l:xy.zs} Assume that a (functionally Hausdorff) $\mu$-complete space $X$ contains a (strong) $S^{\w_1}$-semifan $(S_\alpha)_{\alpha\in\w_1}$ such that the set $\{\lim S_\alpha\}_{\alpha\in\w_1}$ is uncountable. If the functor-space $FX$ is a topological algebra of type  $(x^*y){}(z^*s)$, then $FX$ contains a (strong) $\Fin^{\w_1}$-fan.
\end{lemma}

\begin{proof}
Fix a finite subset $C\subset X$ and an operation $p:X^4\to FX$ witnessing that the functor-space $FX$ is a topological algebra of type $(x^*y){}(z^*s)$.

By our assumption, the set $\{\lim S_\alpha\}_{\alpha\in\w_1}$ is uncountable and the family $(S_\alpha)_{\alpha\in\w_1}$ is compact-finite and consists of sets with compact closures. Using these facts, by transfinite induction we can choose an uncountable subset $\Omega\subset\w_1$ such that every ordinal $\alpha\in\Omega$ the convergent sequence $\bar S_\alpha$ does not intersect the countable set $C\cup \bigcup\{\bar S_\beta:\beta\in\Omega,\;\beta<\alpha\}$. For such set $\Omega$ the family $(\bar S_\alpha)_{\alpha\in\Omega}$ is disjoint.
Let $\Lambda=\{(\alpha,\beta)\in\Omega:\alpha\ne\beta\}$.

For every $\alpha\in\w_1$ denote by $x_\alpha$ the limit point of the convergent sequence $S_\alpha$ and let $\{x_{\alpha,n}\}_{n\in\w}$ be an injective enumeration of the set $S_\alpha\setminus\{x_\alpha\}$.

Choose any almost disjoint family $(A_\alpha)_{\alpha\in\Omega}$ of infinite subsets of $\w$, and for any pair $(\alpha,\beta)\in\Lambda$   consider the finite subset $D_{\alpha,\beta}=\{p(x_\alpha, x_{\alpha,n},x_\beta,x_{\beta,n}):n\in A_\alpha\cap A_\beta\}.$

Definition~\ref{d:xy.zs}(2) implies the inclusion $$D_{\alpha,\beta}\subset\{a\in FX:\supp(a)\cap S_{\alpha}\ne \emptyset\ne \supp(a)\cap S_\beta\}$$ for all $(\alpha,\beta)\in\Lambda$. Since the family $(S_\alpha)_{\alpha\in\Omega}$ is compact-finite, Lemma~\ref{l:cont-supp} implies that the family $(D_{\alpha,\beta})_{(\alpha,\beta)\in\Lambda}$ is compact-finite in $FX$.

Let us show that  this family is not locally finite at the point $p(x,x,x,x)$ for any $x\in X$. Given a neighborhood $W\subset FX$ of $p(x,x,x,x)$ we can apply Definition~\ref{d:xy.zs}(3) and find a neighborhood $U_\Delta\subset X_d\times X$ of the diagonal such that $p(U_\Delta\times U_\Delta)\subset W$. Since for every $\alpha\in\Omega$ the sequence $(x_{\alpha,n})_{n\in\w}$ converges to $x_\alpha$, we can find a number $\varphi(\alpha)\in\w$ such that $\{(x_\alpha,x_{\alpha,n})\}_{n\ge\varphi(\alpha)}\subset U_\Delta$.
By the Pigeonhole Principle, for some $m\in\w$ the set $\Omega_m=\{\alpha\in\Omega:\varphi(\alpha)=m\}$ is uncountable.
By Lemma~\ref{l:ad}, the set $$\Lambda_m=\{(\alpha,\beta)\in\Lambda:\alpha,\beta\in\Omega_m,\;A_{\alpha}\cap A_{\beta}\not\subset [0,m]\}$$is uncountable. For any pair $(\alpha,\beta)\in\Lambda_m$ we can find a number $n\in A_\alpha\cap A_\beta\setminus[0,m]$ and conclude that
$$D_{\alpha,\beta}\ni p(x_\alpha,x_{\alpha,n},x_\beta,x_{\beta,n})\in p(U_\Delta\times U_\Delta)\subset W$$and hence the family $\{(\alpha,\beta)\in\Lambda:W\cap D_{\alpha,\beta}\ne\emptyset\}\supset \Lambda_m$ is uncountable, which implies that the family $(D_{\alpha,\beta})_{(\alpha,\beta)\in\Lambda}$ is not locally countable (and not locally finite) at $p(x,x,x,x)$.

If the $S^{\w_1}$-semifan $(S_\alpha)_{\alpha\in\w_1}$ is strong and $X$ is functionally Hausdorff, then each set $S_\alpha$ has an $\IR$-open neighborhood $V_\alpha\subset X$ such that the family $(V_\alpha)_{\alpha\in\w_1}$ is compact-finite. By Definition~\ref{d:xy.zs}(2) and  Lemma~\ref{l:supp-Ropen}, for every pair $(\alpha,\beta)\in\Lambda$ the set
$$W_{\alpha,\beta}=\{a\in FX:\supp(a)\cap U_\alpha\ne\emptyset\ne\supp(a)\cap U_\beta\}$$is an $\IR$-open neighborhood of the set $D_{\alpha,\beta}$. Using Lemma~\ref{l:cont-supp}, we can show that the family $(W_{\alpha,\beta})_{(\alpha,\beta)\in\Lambda}$ is compact-finite in $FX$ witnessing that the $\Fin$-fan $(D_{\alpha,\beta})_{(\alpha,\beta)\in\Lambda}$ is strong.
 \end{proof}

\begin{corollary}\label{c:xy+xyzs} Assume that a (functionally Hausdorff) $\mu$-complete space $X$ contains a (strong) $S^{\w_1}$-semifan. If the functor-space $FX$ is a topological algebra of types $x{\cdot}y$ and   $(x^*y){}(z^*s)$, then $FX$ contains a (strong) $\Fin^{\w_1}$-fan.
\end{corollary}

\begin{proof} By our assumption, the space $X$ contains a (strong) $S^{\w_1}$-semifan $(S_\alpha)_{\alpha\in\w_1}$. If the set $L=\{\lim S_\alpha\}_{\alpha\in\w_1}$ is uncountable, then by Lemma~\ref{l:xy.zs}, the space $FX$ contains a (strong) $\Fin^{\w_1}$-fan.

If the set $L$ is countable, then for some point $x\in L$ the set $\Omega=\{\alpha\in\w_1:x=\lim S_\alpha\}$ is uncountable. Then $(S_\alpha)_{\alpha\in\Omega}$ is a (strong) $\dot S^{\w_1}$-fan in $X$ and we can apply Lemma~\ref{l:x.y} to find a (strong) $\Fin^{\w_1}$-fan in $FX$.
\end{proof}

We recall that a  subset $B\subset X$ of a topological space $X$ is \index{subset!$\w_1$-bounded}{\em $\w_1$-bounded} if $X$ contains no uncountable locally finite family $\U$ of open subsets of $X$ that meet $B$. A topological space $X$ is \index{topological space!$\w_1$-bounded}{\em $\w_1$-bounded} if $X$ is ${\w_1}$-bounded in itself.

\begin{corollary}\label{c:xy.zs3} Let $X\in\Top_i$ be a $\mu$-complete space such that the space $FX$ is a topological algebra of types $x{\cdot}y$ and $(x^*y)(z^*s)$.
\begin{enumerate}
\item[\textup{1)}] If $X$ is a functionally Hausdorff ($\aleph_k$-)space and  $FX$ contains no strong $\Fin^{\w_1}$-fan, then $X$ contains a closed $\w_1$-bounded subset $B\subset X$ whose complement $X\setminus B$ contains no infinite metrizable compact set (and is $k$-discrete).
\item[\textup{2)}] If the space $FX$ contains no $\Fin^{\w_1}$-fan and $X$ is $k^*$-metrizable, then $X$ contains a $k$-closed $\aleph_0$-subspace $A\subset X$ with $k$-discrete complement $X\setminus A$.
\end{enumerate}
\end{corollary}

\begin{proof}  1. If $X$ is functionally Hausdorff and $FX$ contains no strong $\Fin^{\w_1}$-fan, then by Corollary~\ref{c:xy+xyzs}, the space $X$ contains no strong $S^{\w_1}$-semifan. By Theorem~\ref{t:Ssemifan}(1), $X$ contains a closed $\w_1$-bounded set $B\subset X$ whose complement $X\setminus B$ contains no infinite compact metrizable subspaces. If $X$ is an $\aleph_k$-space, then all compact subsets of $X$ are metrizable, which implies that $X\setminus B$ is $k$-discrete.
\smallskip

2. If $FX$ contains no $\Fin^{\w_1}$-fan, then by Corollary~\ref{c:xy+xyzs}, the space $X$ contains no $S^{\w_1}$-semifan and by Proposition~\ref{p:k*->a}, the $k^*$-metrizable space $X$ contains a $k$-closed $\aleph_0$-subspace with $k$-discrete complement.
\end{proof}

\begin{theorem}\label{t:xy+xyz+xy.zs} Let $X\in\Top_i$ be a $\mu$-complete  space such that the space $FX$ is a topological algebra of types $x{\cdot}y$, $x(y^*z)$ and $(x^*y)(z^*s)$.
\begin{enumerate}
\item If $X$ is a functionally Hausdorff ($\aleph_k$-)space and $FX$ contains no strong $\Fin^{\w_1}$-fans, then $X$ contains a closed $\w_1$-bounded subset $B\subset X$ whose complement $X\setminus B$ contains no infinite metrizable compact set (and is $k$-discrete).
\item If the space $FX$ contains no $\Fin^{\w_1}$-fan and $X$ is $k^*$-metrizable, then $X$ contains a $k$-closed $\aleph_0$-subspace $A\subset X$ with $k$-discrete complement $X\setminus A$.
\item If the $X$ is Tychonoff and the space $FX$ contains no strong $\Fin^\w$-fan, then:
\begin{enumerate}
\item If $X$ is cosmic, then $X$ is $\sigma$-compact.
\item If $X$ is an $\bar\aleph_k$-space, then  either $X$ contains a compact set with $k$-discrete complement or $X$ is $k$-homeomorphic to a topological sum of cosmic $k_\w$-spaces.
\item If $X$ is an $\aleph$-space, then the $k$-modification $kX$ of $X$ either is a topological sum of cosmic $k_\w$-spaces or $kX$ is a metrizable space with compact set of non-isolated points.
\item If $X$ is a $\bar\aleph_k$-$k_\IR$-space, then  either
 $X$ is a topological sum $K\oplus D$ of a cosmic $k_\w$-space $K$ and a discrete space $D$ or else $X$ is metrizable and has compact set of non-isolated points.
\end{enumerate}
\item If the space $FX$ contains no $\Fin^{\w}$-fan, then:
\begin{enumerate}
\item if $X$ is an $\aleph_k$-space, then $X$ either $X$ contains a compact set with $k$-discrete complement or the $k$-modification of $X$ is a topological sum $D\oplus K$ of a discrete space $D$ and a cosmic $k_\w$-space $K$.
\item if $X$ is Tychonoff and $k^*$-metrizable, then its $k$-modification $kX$ is either a metrizable space with compact set of non-isolated points or $kX$ is a topological sum of a discrete space and a cosmic $k_\w$-space.
\end{enumerate}
\end{enumerate}
\end{theorem}

\begin{proof} 1,2. The first two statements are proved in Corollary~\ref{c:xy.zs3}.
\smallskip

3. Assume that $X$ is Tychonoff and $FX$ contains no strong $\Fin^{\w}$-fan. Then $FX$ contains no strong $\Fin^{\w_1}$-fan, too. In this case Corollary~\ref{c:xy+xyzs} ensures that $X$ contains no strong $S^{\w_1}$-semifan and Corollary~\ref{c:xy+xyz} guarantees that $X$ contains no strong $D_\w$-cofan or no strong $S^\w$-semifan. If $X$ is cosmic, then $X$ is $\sigma$-compact by Theorem~\ref{t:Dcofan}(4) or \ref{t:Ssemifan}(5).

So, assume that $X$ is a $\bar\aleph_k$-space. If $X$ contains no strong $D_\w$-cofan, then Theorem~\ref{t:Dcofan}(1) guarantees that the $k$-modification $kX$ of $X$ is a topological sum $\oplus_{\alpha\in A}X_\alpha$ of cosmic $k_\w$-spaces and hence $kX$ is a Tychonoff space. If $X$ is a $k_\IR$-space, then $X=kX$ is a topological sum of cosmic $k_\w$-spaces. Since $X$ contains no strong $S^{\w_1}$-semifan, at most countably many cosmic $k_\w$-spaces $X_\alpha$, $\alpha\in A$, can be non-discrete, which implies that the topological sum $\bigoplus_{\alpha\in A}X_\alpha$ decomposes into a topological sum of a cosmic $k_\w$-space and a discrete space.

If the $\mu$-complete $\bar\aleph_k$-space contains no strong $S^\w$-semifan, then  by Theorem~\ref{t:Ssemifan}(2), $X$ contains a compact subset with $k$-discrete complement.
Moreover, if $X$ is an $k_\IR$-space, then by Theorem~\ref{t:Ssemifan}(3), $X$ is metrizable and has compact set of non-isolated points. If $X$ is an $\aleph$-space, then by Theorem~\ref{t:Ssemifan}(7), the $k$-modification $kX$ of $X$ is metrizable and has compact set of non-isolated points.
\smallskip

4. Assume that $X$ is a (Tychonoff $k^*$-metrizable) space and $FX$ contains no $\Fin^\w$-fan.  By Corollary~\ref{c:xy+xyzs}, the space $X$ contains no $S^{\w_1}$-semifan and by Corollary~\ref{c:xy+xyz}, $X$ contains no $D_\w$-cofan or no $S^\w$-semifan. By Theorems~ \ref{t:Ssemifan}(2,6) and \ref{t:Dcofan}(3), the $k$-modification of $X$ either (is metrizable and) has compact set of non-isolated points or is a topological sum $\bigoplus_{\alpha\in A}X_\alpha$ of cosmic $k_\w$-spaces. Since $X$ (and $kX$) contains no $S^{\w_1}$-semifan, at most countably many cosmic $k_\w$-spaces $X_\alpha$ are non-discrete, which implies that the
topological sum $\bigoplus_{\alpha\in A}X_\alpha$ decomposes into a topological sum of a cosmic $k_\w$-space and a discrete space.
\end{proof}

Now we prove a modification of Lemma~\ref{l:xy.zs} for topological superalgebras of type $(x^*y){}(z^*s)$.

\begin{lemma}\label{l:xy..zs} Assume that a  $\mu$-complete space $X\in\Top_i$ contains an uncountable disjoint family $(\bar S_\alpha)_{\alpha\in\w_1}$ of compact convergent sequences. If the functor-space $FX$ is a topological superalgebra of type $(x^*y){}(z^*s)$, then $FX$ contains a  $\Fin^{\w_1}$-fan.
\end{lemma}

\begin{proof}  Fix a finite subset $C\subset X$ and an operation $p:X^4\to FX$ witnessing that $FX$ is a topological superalgebra of type $(x^*y){}(z^*s)$.

For every $\alpha\in\w_1$ denote by $x_\alpha$ the limit point of the convergent sequence $\bar S_\alpha$ and let $\{x_{\alpha,n}\}_{n\in\w}$ be an injective enumeration of the set $\bar S_\alpha\setminus\{x_\alpha\}$.

Let $\Lambda=\{(\alpha,\beta)\in\w_1\times\w_1:\alpha\ne\beta\}$.
Choose any almost disjoint family $(A_\alpha)_{\alpha\in\w_1}$ of infinite subsets of $\w$, and for any pair $(\alpha,\beta)\in\Lambda$   consider the finite subset $D_{\alpha,\beta}=\{ p(x_\alpha, x_{\alpha,n},x_\beta,x_{\beta,n}):n\in A_\alpha\cap A_\beta\}.$
Repeating the argument from Lemma~\ref{l:xy.zs}, we can prove that  the family $(D_{\alpha,\beta})_{(\alpha,\beta)\in\Lambda}$ is not locally finite in $FX$.
\smallskip

To prove that this family is compact-finite in $FX$, take any compact subset $K\subset F(X)$. The $\mu$-completeness of $X$ and the boundedness of the functor $F$ yields a compact subset $B\subset X$ such that $K\subset F(B;X)$. Definition~\ref{d:xy.zs}(4) yields a finite subset $E\subset X$ such that for any pairwise distinct points $x,y,z,s\in B$ the inclusion $p(x,y,z,s)\in K$ implies $x,z\in E$.
Then the set
$$
\begin{aligned}
\{(\alpha,\beta)\in\Lambda&:D_{\alpha,\beta}\cap K\ne\emptyset\}=\\
&=\{(\alpha,\beta)\in\Lambda:\exists n\in A_\alpha\cap A_\beta\;\; p(x_\alpha,x_{\alpha,n},x_\beta,x_{\beta,n})\in K\}\subset\\
&\subset\{(\alpha,\beta)\in\Lambda:x_\alpha,x_\beta\subset E\}
\end{aligned}
$$is finite, witnessing that the family $(D_{\alpha,\beta})_{(\alpha,\beta)\in\Lambda}$ is compact-finite and hence is a $\Fin^{\w_1}$-fan in $FX$.
\end{proof}

Lemma~\ref{l:xy..zs} implies:

\begin{corollary}\label{c:cosmic-countable} If for a $\mu$-complete space $X$ the  functor-space $FX$ is a topological superalgebra of type $(x^*y){}(z^*s)$ and the space $FX$ contains no $\Fin^{\w_1}$-fans, then each cosmic subspace of $X$ is at most countable.
\end{corollary}

\begin{proof} To derive a contradiction, assume that $X$ contains a uncountable cosmic subspace $Z$. Taking into account that each uncountable cosmic space contains a non-trivial convergent sequence, by transfinite induction we can construct a disjoint uncountable family $(\bar S_\alpha)_{\alpha\in\w_1}$ of compact convergent sequences in $Z\subset X$. By Lemma~\ref{l:xy..zs}, the space $FX$ contains a $\Fin^{\w_1}$-fan, which is forbidden by our assumption.
\end{proof}

\section{Topological algebras of type $x{}(s^*y_i)^{<\w}$}

\begin{definition}\label{d:xyy_w}
Given a space $X\in\Top_i$, we  say that the functor-space $FX$ is a \index{topological algebra!of type $x(s^*y_i)^{<\w}$}{\em topological algebra of type $x{}(s^*y_i)^{<\w}$} if for every $s\in Y$ there exist sequences $(C_n)_{n\in\w}$ and $(p_n)_{n\in\w}$ satisfying the following axioms:
\begin{enumerate}
\item[1)] for every $n\in\w$ \; $C_n$ is a finite subset of $X$ and $p_n:X^{1+n}\to FX$ is an $(1+n)$-ary $F$-valued operation such that $p_n(s,s,\dots,s)=p_0(s)$;
\item[2)] $\{x,y_1,\dots,y_n\}\subset \supp(p_n(x,y_1,\dots,y_n))$ for any $n\in\IN$ and pairwise distinct points $x,y_1,\dots,y_n\in X\setminus C_n$;
\item[3)] for any neighborhood $U\subset FX$ of $p_0(s)$ there is a sequence $(U_n)_{n\in\w}$ of neighborhoods of $s$ in $X$ such that $p_n(U_0\times \cdots \times U_n)\subset U$ for all $n\in\IN$.
\end{enumerate}
\end{definition}

\begin{lemma}\label{l:xyy_w} Assume that a (functionally Hausdorff) $\mu$-complete space $X$ contains a (strong) $D_\w$-cofan $(D_n)_{n\in\w}$. If the functor $F$ is strongly bounded and the functor-space $FX$ is a topological algebra of type $x{}(s^*y_i)^{<\w}$, then the space $FX$ contains a (strong) $\Fin^\w$-fan.
\end{lemma}

\begin{proof} Let $s$ be the limit point of the cofan $(D_n)_{n\in\w}$. Fix a sequence $(C_n)_{n\in\w}$ of finite sets in $X$ and a sequence $(p_n)_{n\in\w}$ of $(1+n)$-ary operations witnessing that the functor-space $FX$ is a topological algebra of type $x{}(s^*y_i)^{<\w}$.

Since $X$ contains an $D_\w$-cofan, it contains a convergent sequence $Y\subset X\setminus \{s\}$ convergent to $s$. By induction choose pairwise distinct points $y_{n,m}$, $n,m\in\w$, in $Y$ such that for every $n\in\w$ the sequence $\{y_{n,m}\}_{m\in\w}$ is contained in $Y\setminus C_n$.
Observe that for every $n\in\IN$ the sequence $(y_{n,m})_{m\in\IN}$ converges to $y$.
For every $n\in\w$ chose pairwise distinct points $x_{n,m}$, $m\in\w$, in the set $D_n\setminus (C_n\cup Y)$.

For every $n,m\in\IN$ consider the element $w_{n,m}= p_n(x_{n,m},y_{1,m},\dots,y_{n,m})\in FX$. By Definition~\ref{d:xyy_w}(2), $\{x_{n,m},y_{1,m},\dots,y_{n,m}\}\subset \supp(w_{n,m})$ and hence $w_{n,m}\notin F_{n}(X)$.

We claim that the family $(\{w_{n,m}\})_{n,m\in\w}$ is a (strong) $\Fin^\w$-fan in $FX$. First we check that this family is not locally finite at the point $p_0(s)$. Given any neighborhood $U\subset FX$ of the point $p_0(s)$, use Definition~\ref{d:xyy_w}(3), to find a sequence $(U_n)_{n\in\w}$ of neighborhoods of $s$ in $X$ such that $p_n(U_0\times\cdots \times U_n)\subset U$ for all $n\in\IN$.

Since the sets $D_n\supset\{x_{n,m}\}_{m\in\w}$ tend to $s$, for the neighborhood $U_0$ of $s$ there is a number $n\in\w$ such that $\{x_{n,m}\}_{m\in\w}\subset U_0$. For the number $n$  find a number $m_0\in\w$ such that $y_{k,m}\in U_k$ for all $k\le n$ and $m\ge m_0$. Then
$$p_n(x_{n,m},y_{1,m},\dots,y_{n,m})\in p_n(U_0\times\cdots\times U_n)\subset U$$for every $m\ge m_0$, which ensures that the family $(\{w_{n,m}\})_{n,m\in\w}$ is not locally finite at $p_0(s)$.

Next, we show that this family is compact-finite in $FX$. Fix any compact set $K\subset FX$ and using the strong boundedness of the functor $F$ and the $\mu$-completeness of the space $X$, find a compact subset $B\subset X$ and a number $k\in\w$ such that $K\subset F_k(B;X)\subset FX$. Taking into account that  $w_{n,m}\notin F_n(X)$ for all $n,m\in\w$, we conclude that $w_{n,m}\notin K$ for all $n>k$ and $m\in\w$. Since the sets $D_n$, $n\le k$, are compact-finite, we can find a number $m_0\in\w$ such that $B\cap\bigcup_{n\le k}D_n\subset \{x_{n,m}:n\le k,\;m\le m_0\}$. Taking into account that $x_{n,m}\in\supp(w_{n,m})$ for all $n,m\in\w$, we conclude that the set
$$\{(n,m)\in\w\times\w:w_{n,m}\in K\}\subset\{(n,m)\in\w\times \w:w_{n,m}\in F_k(B;X)\}\subset[0,k]\times[0,m_0]$$is finite, which implies that the family $(\{w_{n,m}\})_{n,m\in\w}$ is compact-finite in $FX$.

Now assuming that the $D_\w$-cofan $(D_n)_{n\in\w}$ is strong and $X$ is functionally Hausdorff, we shall prove that the $\Fin^\w$-fan $(\{w_{n,m}\})_{n,m\in\w}$ is strong in $FX$. It follows that for every $n\in\w$ the set $D_n$ is strongly compact-finite in $X$. Consequently, each point $x_{n,m}\in D_n$ has an $\IR$-open neighborhood $U_{n,m}\subset X$ such that the family $(U_{n,m})_{m\in\w}$ is compact-finite in $X$.  Lemma~\ref{l:supp-Ropen} and Corollary~\ref{c:Fn-Rclosed} imply that the set$$W_{n,m}=\{a\in FX:\supp(a)\cap U_{n,m}\ne\emptyset\}\setminus F_n(X)$$is an $\IR$-open neighborhood of the point $w_{n,m}$ in $FX$. We claim that the family $(W_{n,m})_{n,m\in\w}$ is compact-finite in $FX$. Fix any compact subset $K\subset FX$. The strong boundedness of the functor $F$ and the $\mu$-completeness of $X$ guarantee that $K\subset F_k(B;X)$ for some compact subset $B\subset X$ and some $k\in\IN$.  Since the families $(U_{n,m})_{m\in\w}$, $n\le k$, are compact-finite, there is $m_0\in\w$ such that $B\cap U_{n,m}=\emptyset$ for all $n\le k$ and $m\ge m_0$. Then $\{(n,m)\in\w\times \w:K\cap W_{n,m}\ne\emptyset\}\subset [0,k]\times [0,m_0]$, which means that the family $(W_{n,m})_{n,m\in\w}$ is compact-finite and the $\Fin^\w$-fan $(\{w_{n,m}\})_{n,m\in\w}$ is strong.
\end{proof}

Lemma~\ref{l:xyy_w} and Theorem~\ref{t:Dcofan} implies:

\begin{corollary}\label{c:xyy_w} Let $X\in\Top_i$ be a $\mu$-complete (Tychonoff) functionally Hausdorff space such that the functor-space $FX$ is a topological algebra of type $x{} (s^*y_i)^{<\w}$. Assume that the functor $F$ is strongly bounded and the space $FX$ contains no (strong) $\Fin^\w$-fan.
\begin{enumerate}
\item[\textup{1)}] If $X$ is a $\bar\aleph_k$-space, then the $k$-modification of $X$ is a topological sum of cosmic $k_\w$-spaces.
\item[\textup{2)}] If $X$ is an $\aleph_0$-space, then $X$ is hemicompact.
\item[\textup{3)}] If $X$ is cosmic, then $X$ is $\sigma$-compact.
\end{enumerate}
\end{corollary}

\smallskip

\begin{corollary}\label{c:xy.zs+xsy_w} Let $X\in\Top_i$ be a $\mu$-complete (Tychonoff) functionally Hausdorff space such that  the functor-space $FX$ is a topological algebra of types $x{\cdot}y$,  $x{}(s^*y_i)^{<\w}$ and $(x^*y){} (z^*s)$. Assume that the functor $F$ is strongly bounded and the space $FX$ contains no (strong) $\Fin^\w$-fan.
\begin{enumerate}
\item[\textup{1)}] If $X$ is a $\bar \aleph_k$-$k_\IR$-space, then $X$ is a topological sum $K\oplus D$ of a cosmic $k_\w$-space $K$ and a discrete space $D$.
\item[\textup{2)}] If $X$ is an $\aleph_k$-space and $FX$ contains no $\Fin^{\w_1}$-fan, then the $k$-modification $kX$ of $X$ is a topological sum $K\oplus D$ of a cosmic $k_\w$-space $K$ and a discrete space $D$.
\end{enumerate}
\end{corollary}

\begin{proof} 1. Assume that $X$ is a $\bar \aleph_k$-$k_\IR$-space. By Corollary~\ref{c:xyy_w}, the space $X$ is a  topological sum $\bigoplus_{\alpha\in A}X_\alpha$ of  cosmic $k_\w$-subspaces of $X$. If the set $B=\{\alpha\in A:X_\alpha$ is not discrete$\}$ is uncountable, then the space $X$ contains a strong $S^{\w_1}$-semifan $(S_\alpha)_{\alpha\in\w_1}$ with uncountable set $\{\lim S_\alpha\}_{\alpha\in\w_1}$ and by Lemma~\ref{l:xy.zs} the space $FX$ contains a strong $\Fin^{\w_1}$-fan, which is forbidden by our assumptions. So, the set $B$ is countable and then $X$ is the topological sum of the cosmic $k_\w$-space $K=\bigcup_{\alpha\in B}X_\alpha$ and the discrete space $D=\bigcup_{\alpha\in A\setminus B}X_\alpha$.
\smallskip

2. Assume that  $X$ is an $\aleph_k$-space and $FX$ contains no $\Fin^{\w_1}$-fan. By Corollary~\ref{c:xy+xyzs}, the space $X$ contains no $S^{\w_1}$-semifan and by Lemma~\ref{l:xyy_w}, the space $X$ contains no strong $D_\w$-cofan. By Theorem~\ref{t:Dcofan}(2), the $k$-modification $kX$ of $X$ is a topological sum $\bigoplus_{\alpha\in A}X_\alpha$ of  cosmic $k_\w$-spaces. If the set $B=\{\alpha\in A:X_\alpha$ is not discrete$\}$ is uncountable, then the space $X$ contains an $S^{\w_1}$-semifan $(S_\alpha)_{\alpha\in\w_1}$ with uncountable set $\{\lim S_\alpha\}_{\alpha\in\w_1}$, which is forbidden by  Lemma~\ref{l:xy.zs}. So, the set $B$ is countable and then $kX$ is the topological sum of the cosmic $k_\w$-space $K=\bigcup_{\alpha\in B}X_\alpha$ and the discrete space $D=\bigcup_{\alpha\in A\setminus B}X_\alpha$.  If $FX$ is a topological superalgebra of type $(x^*y)(z^*s)$, then the cosmic space $K$ is countable by Corollary~\ref{c:cosmic-countable}.
\end{proof}
\smallskip

\section{Topological algebras of type $x{}(y(s^*z)y^*)$}\label{s:xyzs}
\begin{definition}\label{d:xyzs}
Given a space $X\in\Top_i$, we say that the functor-space $FX$ is a \index{topological algebra!of type $x(y(s^*z)y^*)$}{\em topological  algebra of type $x{}(y(s^*z)y^*)$}  if for every point $s\in X$ there exist a finite subset $C_s\subset X$ and a ternary operation $p:X^3\to FX$ having the following properties:
\begin{enumerate}
\item[1)] $p(x,y,s)=p(x,s,s)$ for all $x,y\in X$;
\item[2)] $\{x,y\}\subset \supp(p(x,y,z))$ for any pairwise distinct points $x,y,z\in X\setminus C_s$;
\item[3)] for any neighborhood $U\subset FX$ of the point $p(s,s,s)$ there exist a neighborhood $O_s\subset X$ of $s$ and a neighborhood $U_{X\times\{s\}}\subset X_d\times X$ of the set $X_d\times\{s\}$  such that $p(O_s\times U_{X\times\{s\}})\subset U$.
\end{enumerate}
\end{definition}

\begin{lemma}\label{l:xyzs} Assume that a (functionally Hausdorff) $\mu$-complete space $X$ contains a (strong) $D_\w$-cofan $(D_n)_{n\in\w}$. If the functor-space $FX$ is a topological algebra of type $x{}(y(s^*z)y^*)$, then the space $FX$ contains a (strong) $\Fin^\w$-fan.
\end{lemma}

\begin{proof} Let $s$ be the limit point of the cofan $(D_n)_{n\in\w}$. Fix a finite set $C_s\subset X$ and an operation $p:X^3\to FX$ witnessing that $FX$ is a topological algebra of type
 $x{}(y(s^*z)y^*)$.

Since $X$ contains a (strong) $D_\w$-cofan convergent to $s$, it contains a convergent sequence $Z\subset X\setminus C_s$ convergent to $s$ and a (strongly) compact-finite infinite set $Y\subset X$. Deleting from $Z$ its limit point $s$, we can assume that $s\notin Z$.
Replacing $Y$ by $Y\setminus (Z\cup C_s)$ we can also assume that $Y\cap (Z\cup C_s)=\emptyset$.

Fix an injective enumeration $\{z_m\}_{m\in\w}$ of the convergent sequence $Z=Z\setminus\{s\}$.

Replacing every set $D_n$ of the cofan by $D_n\setminus (Z\cup C_s)$ we can assume that the set $\bigcup_{n\in\w}D_n$ is disjoint with $Z\cup C_s$. Replacing  the compact-finite sets  $Y$ and $D_n$, $n\in\w$, by suitable compact-finite subsets we can assume that $Y$ is disjoint with $\{x\}\cup\bigcup_{n\in\w}D_n$. For every $n\in\w$ fix an injective enumeration $\{x_{n,m}\}_{m\in\w}$ of the set $D_{n}$.
Choose any sequence $(y_n)_{n\in\w}$ of pairwise distinct points in $Y$.

For every $n,m\in\IN$ consider the element $w_{n,m}=p_n(x_{n,m},y_n,z_{m})\in FX$. By Definition~\ref{d:xyzs}(2), $\{x_{n,m},y_n\}\subset \supp(w_{n,m})$.
We claim that the family of singletons $(\{w_{n,m}\})_{n,m\in\w}$ is a (strong) $\Fin^\w$-fan in $FX$.

First we check that this family is not locally finite at the point $p(s,s,s)$. Let $U$ be any neighborhood of the point $p(s,s,s)$ in $FX$.
By Definition~\ref{d:xyzs}(3), there are a neighborhood $O_s\subset X$ of $s$ and a neighborhood $U_{X\times \{s\}}\subset X_d\times X$ of the set $X\times \{s\}$ such that $p(O_s\times U_{X\times \{s\}})\subset U$.
Since the sequence $(D_n)_{n\in\w}$ converges to $s$, there is $n\in\IN$ such that $\{x_{n,m}\}_{m\in\w}=D_n\subset O_s$. Since the sequence $(z_m)_{m\in\w}$ converges to $s$, for the number $n$ we can find a number $m_0\in\w$ such that $\{(y_n,z_m)\}_{m\ge m_0}\in U_{X\times \{s\}}$. Then
$$w_{n,m}=p(x_{n,m},y_n,z_m)\in p(O_s\times U_{X\times \{s\}})\subset U,$$ for all $m\ge m_0$, which means that the family $(\{w_{n,m}\})_{n,m\in\w}$ is not locally finite at $p(s,s,s)$.

Using the inclusion $\{x_{n,m},y_n\}\subset\supp(w_{n,m})$ and taking into account that the sets $Y$ and $D_n$, $n\in\w$, are compact-finite in $X$, we can show that the family $\big(\{z_{n,m}\}\big)_{n,m\in\w}$ is compact-finite in $FX$.

Now assume that the $D_\w$-cofan $(D_n)_{n\in\w}$ is strong and the space $X$ is functionally Hausdorff. In this case the choice of the set $Y$ guarantees that it is strongly compact-finite. So, each point  $y_n\in D$ has an $\IR$-open neighborhood $V_n\subset X$ such that the family $(V_n)_{n\in\w}$ is compact-finite in $X$. By our assumption, for every $n\in\w$ the set $D_n=\{x_{n,m}\}_{m\in\w}$ is strongly compact-finite in $X$. Consequently, each point  $x_{n,m}$ has an $\IR$-open neighborhood $U_{n,m}\subset X$ such that for every $n\in\w$ the family $(U_{n,m})_{n,m\in\w}$ is compact-finite. Since $\{x_{n,m},y_n\}\subset\supp(w_{n,m})$, for every $n,m\in\w$ the set
$$W_{n,m}=\{a\in FX:\supp(a)\cap U_{n,m}\ne\emptyset\ne \supp(a)\cap V_n\}$$ contains the point $w_{n,m}$ and by Lemma~\ref{l:supp-Ropen} is $\IR$-open in $FX$. Repeating the argument from the proof of Lemma~\ref{l:xyz}, we can show that the family $(W_{n,m})_{n,m\in\w}$ is compact-finite in $FX$, which implies that $(\{w_{n,m}\})_{n,m\in\w}$ is a strong $\Fin^\w$-fan in $FX$.
\end{proof}

Lemma~\ref{l:xyzs} with Theorem~\ref{t:Dcofan} imply:

\begin{corollary}\label{c:xyszy} Let $X\in\Top_i$ be a $\mu$-complete (Tychonoff) functionally Hausdorff  space such that the functor-space $FX$ is a topological algebra of type $x(y(s^*z)y^*)$. Assume that the space $FX$ contains no (strong) $\Fin^\w$-fan.
\begin{enumerate}
\item[\textup{1)}] If $X$ is a $\bar\aleph_k$-space, then the $k$-modification of $X$ is a topological sum of cosmic $k_\w$-spaces.
\item[\textup{2)}] If $X$ is an $\aleph_0$-space, then $X$ is hemicompact.
\item[\textup{3)}] If $X$ is cosmic, then $X$ is $\sigma$-compact.
\end{enumerate}
\end{corollary}

The following corollary can be derived from Lemmas~\ref{l:xyy_w}, \ref{l:xy.zs}, \ref{l:xyzs} and  Theorem~\ref{t:Dcofan} by analogy with Corollary~\ref{c:xy.zs+xsy_w}.

\begin{corollary}\label{c:xy.zs+xyzs} Let $X\in\Top_i$ be a $\mu$-complete (Tychonoff) functionally Hausdorff space such that  the functor-space $FX$ is a topological algebra of types $x{\cdot}y$, $(x^*y){} (z^*s)$ and $x(y(s^*z)y^*)$. Assume that the space $FX$ contains no (strong) $\Fin^\w$-fan.
\begin{enumerate}
\item[\textup{1)}] If $X$ is a $\bar \aleph_k$-$k_\IR$-space, then $X$ is a topological sum $X\oplus D$ of a cosmic $k_\w$-space and a discrete space.
\item[\textup{2)}] If $X$ is an $\aleph_k$-space and $FX$ contains no $\Fin^{\w_1}$-fan, then the $k$-modification $kX$ of $X$ is a topological sum $K\oplus D$ of a cosmic $k_\w$-space $K$ and a discrete space $D$.
\end{enumerate}
\end{corollary}

\section{Topological algebras of type $(x(s^*z)x^*){\cdot}(y(s^*z)y^*)$}

\begin{definition}\label{d:long}
Given a space $X\in\Top_i$, we say that the functor-space $FX$ is a \index{topological algebra!of type $(x(s^*z)x^*)(y(s^*z)y^*)$}{\em topological  algebra of type $(x(s^*z)x^*){}(y(s^*z)y^*)$} if there exist a finite subset $C\subset X$ and an operation $p:X^3\to FX$ having the following properties:
\begin{enumerate}
\item[1)] $p(x,y,s)=p(s,s,s)$ for all $x,y\in X$;
\item[2)] $\{x,y\}\subset \supp(p(x,y,z))$ for any pairwise distinct points $x,y,z\in X\setminus C$;
\item[3)] for any neighborhood $W\subset FX$ of $p(s,s,s)$ there is a neighborhood $U_{X\times\{s\}}\subset X_d\times X$ of the set $X_d\times\{s\}$ such that  $p(x,y,z)\subset W$ or any triple $(x,y,z)\in X^3$ with $(x,z),(y,z)\in U_{X\times\{s\}}$.
\end{enumerate}
\end{definition}

\begin{lemma}\label{l:long} Assume that a (functionally Hausdorff) $\mu$-complete space $X$ contains a convergent sequence $S$ and a (strongly) compact-finite uncountable set $D$. If the functor-space $FX$ is a topological algebra of type   $(x(s^*z)x^*){\cdot}(y(s^*z)y^*)$, then $FX$ contains a (strong) $\Fin^{\w_1}$-fan.
\end{lemma}

\begin{proof} Let $s$ be the limit point of the convergent sequence $S$. Fix a finite subset $C\subset X$ and an operation $p:X^3\to FX$ witnessing that $FX$ is a topological algebra of type
$(x(s^*z)x^*){}(y(s^*z)y^*)$.

 Choose any sequence $(s_n)_{n\in\w}$ of pairwise distinct points in the set $S\setminus(C\cup \{s\})$.
Next, take any pairwise distinct points $x_\alpha$, $\alpha\in\w_1$, in the uncountable set $D\setminus(C\cup \bar S)$.

Fix an almost disjoint family $(A_\alpha)_{\alpha\in\w_1}$ of infinite subsets of $\w$ and let $\Lambda=\{(\alpha,\beta)\in\w_1\times\w_1):\alpha\ne\alpha\}$.
For any pair $(\alpha,\beta)\in\Lambda$  consider the finite subset $D_{\alpha,\beta}=\{p(x_\alpha,x_\beta,s_n):n\in A_\alpha\cap A_\beta\}.$ By Definition~\ref{d:long}(2), $\{x_\alpha,x_\beta\}\subset\supp(p(x_\alpha,x_\beta,s_n))$, which implies that the family $(D_{\alpha,\beta})_{(\alpha,\beta)\in\Lambda}$ is compact-finite in $FX$.

Now we show that this family is not locally finite at the point $p(s,s,s)$. Given a neighborhood $W\subset FX$ of $p(s,s,s)$, apply Definition~\ref{d:long}(3) to find a neighborhood $U_{X\times\{s\}}\subset X_d\times X$ of the set $X_d\times\{s\}$ such that $p(x,y,z)\in W$ for every $(x,y,z)\in X^3$ with $(x,z),(y,z)\in U_{X\times\{s\}}$.

For every $\alpha\in\w_1$ choose a number $\varphi(\alpha)\in\w$ such that $\{(x_\alpha,s_n)\}_{n\ge\varphi(\alpha)}\subset U_{X\times\{z\}}$.
By the Pigeonhole Principle, for some $m\in\w$ the set $\Omega=\{\alpha\in\w_1:\varphi(\alpha)=m\}$ is uncountable. By Lemma~\ref{l:ad}, the set $\Lambda_m=\{(\alpha,\beta)\in\Lambda:A_\alpha\cap A_\beta\not\subset[0,m]\}$ is uncountable. For any pair $(\alpha,\beta)\in\Lambda_m$ we can find a  number  $n\in A_\alpha\cap A_\beta\setminus[0,m]$ and conclude that $(x_\alpha,s_n),(x_\beta,s_n)\in U_{X\times\{z\}}$ and hence
$D_{\alpha,\beta}\ni p(x_\alpha,x_\beta,s_n)\in W$, which means that the family $(D_{\alpha,\beta})_{(\alpha,\beta)\in\Lambda}$ is not locally countable and not locally finite in $FX$.

If the set $D$ is strongly compact-finite and $X$ is functionally Hausdorff, then each point $x_\alpha\in D$ has an $\IR$-open neighborhood $U_\alpha\subset X$ such that the family $(U_\alpha)_{\alpha\in\w_1}$ is compact-finite in $X$. Definition~\ref{d:long}(2) and Lemma~\ref{l:supp-Ropen} imply that for every pair $(\alpha,\beta)\in\Lambda$ the set
$$W_{\alpha,\beta}=\{a\in FX:\supp(a)\cap U_\alpha\ne\emptyset\ne\supp(a)\cap U_\beta\}$$ is an $\IR$-open neighborhood of the finite set $D_{\alpha,\beta}$ in $FX$. On the other hand, Lemma~\ref{l:cont-supp} implies that the family $\{W_{\alpha,\beta}:(\alpha,\beta)\in\Lambda\}$ is compact-finite in $FX$, which means that the $(D_{\alpha,\beta})_{(\alpha,\beta)\in\Lambda}$ is a strong $\Fin^{\w_1}$-fan in $FX$.
\end{proof}

\begin{theorem}\label{t:long1} Let $X\in\Top_i$ be a $\mu$-complete  space such that the functor-space $FX$ is a topological algebra of types $x{\cdot}y$, $x{} (y(s^*z)y^*)$ and
$(x(s^*z)x^*){\cdot}(y(s^*z)y^*)$ (and a topological superalgebra of type $(x^*y){}(z^*s)$).
\begin{enumerate}
\item[\textup{1)}] If $X$ is functionally Hausdorff, $FX$ contains no strong $\Fin^{\w_1}$-fan and each infinite compact subset of $X$ contains a convergent sequence, then $X$ either $k$-discrete or $\w_1$-bounded.
\item[\textup{2)}] If $X$ is $k^*$-metrizable and $FX$ contains no $\Fin^{\w_1}$-fan, then either $X$ is $k$-discrete or $X$ is a (countable) $\aleph_0$-space.
\item[\textup{3)}] If $X$ is a Tychonoff $\bar\aleph_k$-space and $FX$ contains no strong $\Fin^{\w}$-fan, then $X$ is either $k$-discrete or $X$ is $\w_1$-bounded and is a $k$-sum of (countable) hemicompact spaces; moreover, if $X$ is  a $k_\IR$-space, then $X$ is either discrete or a (countable) cosmic $k_\w$-space.
\item[\textup{4)}] If $X$ is a functionally Hausdorff $\aleph_k$-space and $FX$ contains no strong $\Fin^{\w}$-fan and no $\Fin^{\w_1}$-fan, then $X$ is either $k$-discrete or (countable and) hemicompact.
\item[\textup{5)}] If $X$ is an $\aleph_k$-space and $FX$ contains no $\Fin^\w$-fan, then $X$ is either $k$-discrete or (countable and) hemicompact.
\end{enumerate}
\end{theorem}

\begin{proof} 1. Assume that $X$ is functionally Hausdorff, $FX$ contains no strong $\Fin^{\w_1}$-fan and each infinite compact subset of $X$ contains an infinite compact metrizable subspace. By Lemma~\ref{l:long}, the space $X$ either contains no infinite compact metrizable set or contains no unbounded strongly compact-finite set. In the first case the space $X$ is $k$-discrete. In the second case $X$ is $\w_1$-bounded.
\smallskip

2. Assume that $X$ is $k^*$-metrizable and $FX$ contains no $\Fin^{\w_1}$-fan.
By Lemma~\ref{l:long}, the space $X$ either is $k$-discrete or contains no uncountable compact-finite subset. In the latter case, the $k$-modification contains no uncountable closed discrete subsets, which means that $kX$ has countable extent. By Proposition 3.3 \cite{BBK}, the $k$-modification of the $k^*$-metrizable space $X$ is $k^*$-metrizable and Theorem 5.3 \cite{BBK}, the $k^*$-metrizable $k$-space $kX$ of countable extent is cosmic and so is its continuous image $X$. By Theorem 7.2 \cite{BBK} the cosmic $k^*$-metrizable space $X$ is an $\aleph_0$-space. (If $FX$ is a topological superalgebra of type $(x^*y)(s^*z)$, then by Corollary~\ref{c:cosmic-countable}, the $\aleph_0$-space $X$ is countable).
\smallskip

3. Assume that $X$ is a functionally Hausdorff $\bar\aleph_k$-space and $FX$ contains no strong $\Fin^{\w}$-fan. Then $FX$ contains no strong $\Fin^{\w_1}$-fan and by the first statement, the space $X$ is either $k$-discrete or $\w_1$-bounded.
Assume that $X$ is not $k$-discrete. In this case $X$ is $\w_1$-bounded.
  By Lemma~\ref{l:xyzs}, the space $X$ contains no strong $D_\w$-cofan.   Then by Theorem~\ref{t:Dcofan}(1), the $k$-modification $kX$ of $X$ is a topological sum $\oplus_{\alpha\in A}X_\alpha$ of non-empty cosmic $k_\w$-spaces.
 (If $FX$ is a topological superalgebra of type $(x^*y)(s^*z)$, then by Corollary~\ref{c:cosmic-countable}, each cosmic space $X_\alpha$ is countable).

  If $X$ is a $k_\IR$-space, then $X=kX$ and the $\w_1$-boundedness of $X$ guarantees that the index set $A$ is countable and hence $X=kX=\oplus_{\alpha\in A}X_\alpha$ is a cosmic $k_\w$-space.
(If $FX$ is a topological superalgebra of type $(x^*y)(s^*z)$, then by Corollary~\ref{c:cosmic-countable}, the cosmic space $X$ is countable).
\smallskip

  4. Assume that  $X$ is a Tychonoff $\aleph_k$-space and $FX$ contains no strong $\Fin^\w$-fans and no $\Fin^{\w_1}$-fans. Assuming that $X$ is not $k$-discrete, we shall prove that $X$ is hemicompact (and countable). By Lemma~\ref{l:long}, Corollary~\ref{c:xy+xyzs} and Lemma~\ref{l:xyzs}, $X$ contains no uncountable compact-finite set, no $S^{\w_1}$-semifan (and hence no $\bar S^{\w_1}$-fan) and no strong $D_\w$-cofan. Applying Theorem~\ref{t:Dcofan}(2), we conclude that $X$ is a $k$-sum $\bigcup_{\alpha\in A}X_\alpha$ of non-empty hemicompact spaces. Since $X$ contains no uncountable   compact-finite subset, the index set $A$ is at most countable and hence $X$ is hemicompact. If $FX$ is a topological superalgebra of type $(x^*y)(s^*z)$, then by Corollary~\ref{c:cosmic-countable}, the hemicompact space $X$ is countable.
\smallskip

5. Finally assume that  $X$ is an $\aleph_k$-space and $FX$ contains no $\Fin^\w$-fans and hence no $\Fin^{\w_1}$-fans. Assuming that $X$ is not $k$-discrete, we shall prove that $X$ is hemicompact (and countable). By Lemma~\ref{l:long}, Corollary~\ref{c:xy+xyzs} and Lemma~\ref{l:xyzs}, $X$ contains no uncountable compact-finite set, no $S^{\w_1}$-semifan (and hence no $\bar S^{\w_1}$-fan) and no $D_\w$-cofan. Applying Theorem~\ref{t:Dcofan}(2), we conclude that $X$ is a $k$-sum $\bigcup_{\alpha\in A}X_\alpha$ of non-empty hemicompact spaces. Since $X$ contains no uncountable   compact-finite subset, the index set $A$ is at most countable and hence $X$ is hemicompact. If $FX$ is a topological superalgebra of type $(x^*y)(s^*z)$, then by Corollary~\ref{c:cosmic-countable}, the hemicompact space $X$ is countable.
  \end{proof}

\section{Algebraic structures on monadic functors}\label{s:monad}

Let $\Top_i$ be a full hereditary subcategory of $\Top$ such that $\Top_i$ contains all finite discrete spaces. Let $F:\Top_i\to\Top$ be a monomorphic functor with finite supports. In this section we apply the results of the preceding sections to construct $\Fin$-fans in functor-spaces $FX$ using some canonical $n$-ary $F$-valued operations on $FX$.

Observe that each element $p\in F(n)$ determines an $n$-ary $F$-valued operation $p_X:X^n\to FX$ assigning to each function $x:n\to X$ the element $p_X(x):=Fx(p)$. Observe that $F_n(X)=\bigcup_{p\in F(n)} p_X(X^n)$.

We recall that a functor $F:\Top_i\to\Top$ is \index{functor!continuous}{\em continuous} if for every space $X\in\Top_i$ and every $p\in F(n)$ the $n$-ary operation $p_X:X^n\to FX$ is continuous.

The following lemma shows that the value of $p_X( x)$ depends only on the restriction $x|\supp(p)$.

\begin{lemma}\label{l:equal-supp} Let $n\in\IN$ and $p\in F(n)$. For any space $X\in\Top_i$ and vectors $x,y\in X^n$ with $x|\supp(p)=y|\supp(p)$ we get $p_X(x)=p_X(y)$.
\end{lemma}

\begin{proof} Take any non-empty set $P\subset n$ that contains the support $\supp(p)$ of $P$ and consider the identity embedding $i_{P,n}:P\to n$.
By Theorem~\ref{t:supp}, $p\in F(P;n)$. So, we can find an element $a\in FP$ such that $p=Fi_{P,n}(a)$.

For any vectors $x,y\in X^n$ the equality $x|P=y|P$ implies $x\circ i_{P,n}=y\circ i_{P,n}$. Applying the functor $F$ to the last equality we get $p_X(x)=Fx(p)=Fx\circ Fi_{P,n}(a)=Fy\circ Fi_{P,n}(b)=Fy(p)=p_X(y)$.

If $\supp(p)$ is not empty, then we can put $P=\supp(p)$ and conclude that for any $x,y\in X^n$ the equality $x|\supp(p)=y|\supp(y)$ implies $p_X(x)=p_X(y)$.

It remains to consider the case of empty support $\supp(p)$. By Theorem~\ref{t:supp}, $p\in F(P;X)$ for any non-empty set $P\subset X$.
For every number $k\in\{0,1\}=2$ consider the map $i_k:1\to \{k\}\subset 2$. In the functor-space $F1$ consider the set $F^\circ\emptyset=\{a\in F1:Fi_0(a)=Fi_1(a)\}$. Let $e:\emptyset\to X$ be the unique map. Define the map
$F^\circ e:F^\circ\emptyset\to FX$ letting $F^\circ e=Fz|F^\circ\emptyset$ where $z:1\to X$ is any map. By Theorem 1 of \cite{BMZ}, the map $F^\circ e$ is well-defined and does no depend on the choice of the map $z:1\to X$. Moreover, the set $F^\circ(\emptyset;X)=F^\circ e(F^\circ\emptyset)\subset FX$ contains all elements of $FX$ with empty support. In particular, $p\in F^\circ(\emptyset;X)$. So, we can find an element $a\in F^\circ \emptyset$ such that $F^\circ e(a)=p$. The definition of the image $F^\circ e(a)$ implies that for any maps $x,y:1\to X$ we get $p_X(x)=Fx(a)=F^\circ e(a)=Fy(a)=p_X(y)$. Then for any constant maps $\bar x,\bar y:n\to X$  we also get the equality $p_X(\bar x)=p_X(\bar y)$. Now take any maps $x,y\in X^n$ and let $\bar x,\bar y$ be the constant maps such that $\bar x(0)=x(0)$ and $\bar y(0)=y(0)$. Since $\supp(p)\subset\{0\}$, the latter equalities imply $p_X(x)=p_X(\bar x)=p_X(\bar y)=p_X(y)$.
\end{proof}

It turns out that functor-spaces $FX$ quite often are topological algebras of type $x{\cdot}y$.

\begin{proposition}\label{p:extype-xy} If $F:\Top_i\to\Top$ is a continuous functor with $F(n)\ne F_1(n)$ for some $n\in\w$, then for any infinite space $X\in\Top_i$ the functor-space $FX$ is a topological algebra of type $x{\cdot}y$.
\end{proposition}

\begin{proof} Taking into account that $\emptyset\ne F(n)\setminus F_1(n)=\bigcup_{k=0}^{n-2} F_{k+2}(n)\setminus F_{k+1}(n)$, we can find an ordinal
$k\le n-2$ and an element $q\in F_{k+2}(n)\setminus F_{n+1}(n)$. It follows that $|\supp(q)|=k+2$. Applying to $n$ a suitable permutation, we can assume that $\supp(q)=\{0,\dots,k+1\}=k+2$ and hence $q\in F(k+2;n)$. Find a unique element $p\in F(k+2)$ such that $q=Fi_{k+2,n}(p)$. It follows that $p\in F(k+2)\setminus F_{k+1}(k+2)$ and hence $\supp(p)=k+2$.

Fix any pairwise distinct points $c_1,\dots,c_k\in X$ and consider the set $C=\{c_1,\dots,c_k\}$. The element $p\in F(k+2)$ induces the binary $F$-valued operation $ p_c:X^2\to FX$, $p_c:(x,y)\mapsto p_X(c_1,\dots,c_k,x,y)$. Observe that for any distinct points $x,y\in X\setminus C$ we get $\supp(p_c(x,y))=C\cup\{x,y\}\supset\{x,y\}$.

The continuity of the functor $F$ implies the continuity of the operation $ p_c:X^2\to FX$. Now we see that the set $C$ and the operation $p_c$ witnesses that the functor-space $FX$ is a topological algebra of type $x{\cdot}y$.
\end{proof}

The constructions of more complicated operations on functor-spaces involves
the monadic structures on functors.

We say that a functor $F:\Top_i\to\Top_i$ is \index{functor!monadic}{\em monadic}
if there are natural transformations $\delta:\Id\to F$ and $\mu:F^2\to F$ making the following diagrams commutative:
$$\xymatrix{
F\ar_{F\delta}[d]\ar^{\delta F}[r]\ar^{\id}[rd]&F^2\ar^{\mu}[d]&&F^3\ar^{\mu_T}[r]\ar_{F\mu}[d]&F^2\ar^{\mu}[d]\\
F^2\ar_{\mu}[r]&F&&F^2\ar_\mu[r]&T
}
$$In this case the triple $(F,\delta,\mu)$ is called a \index{monad}{\em monad} and the functor $F$ is called its {\em functorial part}.
The natural transformations $\delta$ and $\mu$ are called the \index{monad!unit of}{\em unit} and the \index{monad!multiplication of}{\em multiplication} of the monad $(F,\delta,\mu)$.
More information on monads in categories of topological spaces can be found in the book \cite{TZ}.

If the functorial part $F:\Top_i\to\Top_i$ of a monad $(F,\delta,\mu)$ is a  monomorphic functor with finite supports, then for any space  $X\in\Top_i$ any element $\A\in F^2X$ has finite support $\supp(\A)\subset FX$ and any element $\alpha\in \supp(\A)$ has finite support $\supp(\alpha)\subset X$. Then the set $$\supp^2(\A)=\bigcup_{\alpha\in\supp(\A)}\supp(\alpha)$$ is finite and Theorem~\ref{t:supp} guarantees that $\A\subset F^2(A;X)=F(F(A;X);FX)$ for any non-empty subset $A\subset X$ containing $\supp^2(\A)$. We shall say that an element $\A\in F^2X$ has {\em disjoint support} if the indexed family $\big(\supp(\alpha)\big)_{\alpha\in \supp(\A)}$ is disjoint and consists of non-empty sets.

\begin{definition} Let $(F,\delta,\mu)$ be a monad in the category $\Top_i$ whose functorial part $F:\Top_i\to\Top_i$ is a monomorphic functor with finite supports. The monad $(F,\delta,\mu)$ is defined to \index{functor!preserves disjoint supports}{\em preserve disjoint supports} if  for any space $X\in\Top_i$ the following conditions hold:
\begin{itemize}
\item $\supp(\delta_X(x))=\{x\}$ for any $x\in X$;
\item $\supp(\mu(\A))=\supp^2(\A)$ for any element $\A\in F^2X$ with disjoint support $\big(\supp(\alpha)\big)_{\alpha\in\supp(\A)}$.
\end{itemize}
\end{definition}

\begin{example} The functor $F:\Top\to\Top$, $F:X\mapsto X\times X$, is the functorial part of a monad $(F,\delta,\mu)$ which does not preserve disjoint supports. The identity $\delta$ of this monad is defined by $\delta_X:x\mapsto (x,x)$ and the multiplication $\mu$ by $\mu_X:((x,y),(u,v))\mapsto (u,v)$. Then for any pairwise distinct points $x,y,u,v$ in a topological space $X$ the element $\A=((x,y),(u,v))\in (X\times X)\times (X\times X)$ has disjoint support $\{(x,y),(u,v)\}$ but $\supp(\mu(\A))=\supp((x,v))=\{x,v\}\ne \{x,y,u,v\}=\supp^2(\A)$.
\end{example}

Assume that a functor $F:\Top_i\to\Top_i$ can be completed to a monad $(F,\delta,\mu)$. Then for every space $X\in\Top_i$, every element $p\in F(n)$, $n\in\w$, induces the $n$-ary $F$-valued operation $p_{FX}:(FX)^n\to F(FX)$ which can be transformed into an $n$-ary operation $\dot p_{FX}=\mu_X\circ p_X:(FX)^n\to FX$ on $FX$. If the functor $F$ is continuous, then the operations $p_{FX}$ and $\dot p_{FX}$ are continuous. These two operations will be used in the proof of the following proposition.

\begin{proposition}\label{p:F->structure} Assume that a continuous monomorphic functor $F:\Top_i\to\Top_i$ with finite supports can be completed to a monad $(F,\delta,\mu)$ which preserves disjoint supports. Let $X\in\Top_i$ be an infinite space.
\begin{enumerate}
\item[\textup{1)}] If $F(X)\ne F_1(X)$, then the space $FX$ is a topological algebra of type $x{\cdot}y$.
\item[\textup{2)}] If $F$ does not preserve preimages, then $FX$ is a topological algebra of types $x{\cdot}y$, $x{}(y^*z)$ and  $(x^*y){}(z^*s)$.
\item[\textup{3)}] If $F$ strongly fails to preserve preimages, then $FX$ is a topological algebra of types $x{} (y^*(z^*s)y)$ and $(x^*(z^*s)y){\cdot} (y^*(z^*s)y)$.
\end{enumerate}
\end{proposition}

\begin{proof} 1. Assuming that $F(X)\ne F_1(X)$, choose an element $a\in FX\setminus F_1(X)$ and consider its support $\supp(a)$, which has cardinality $n=|\supp(a)|>1$ according to Theorem~\ref{t:supp}. Let $f:n\to\supp(a)$ be any bijective map and observe that $a\in F(f(n)_d,X)$ and hence $a=Ff(b)$ for some element $b\in F(n)\setminus F_1(n)$. Now we can apply Proposition~\ref{p:extype-xy} and conclude that the functor-space $FX$ is a topological algebra of type $x{\cdot}y$.
\smallskip

2. Assume that $F$ does not preserve preimages. By Corollary~\ref{c:F-preim}, there exist a number $l\in\IN$ and an element $a\in F(l+1)\setminus F_l(l+1)$ such that $\supp(Fr^{l+1}_l(a))\subset l-1$. Here $r^{l+1}_l:l+1\to l$, $r^{l+1}_l:i\mapsto\max\{j\in l:j\le i\}$, stands for the retraction of $l+1$ on $l$.

  For any space $Z\in\Top_i$ the element $a$ induces a continuous $(l+1)$-ary operation $a_Z:Z^{l+1}\to FZ$ such that
\begin{itemize}
\item[(i)] $\supp(a_Z(z_1,\dots,z_{l+1}))=\{z_1,\dots,z_{l+1}\}$ for any pairwise distinct points $z_1,\dots,z_{l+1}\in Z$;
\item[(ii)] $a_Z(z_1,\dots,z_{l+1})=a_Z(z_1',\dots,z_{l+1}')$ for any vectors $(z_i)_{i=1}^{l+1},(z_i')_{i=1}^{l+1}\in Z^{l+1}$ such that $z_l=z_{l+1}$, $z_l'=z'_{l+1}$ and $(z_i)_{i=1}^{l-1}=(z_i')_{i=1}^{l-1}$.
\end{itemize}
On the space $FX$ consider the continuous $(l+1)$-ary operation
 $\dot a_{FX}=\mu_X\circ  a_{FX}:(FX)^{l+1}\to FX$. If $l=1$, then the finite set $C=\emptyset$ and the continuous operation $p:X^3\to FX$ defined by $p(x,y,z)=\dot a_{FX}(\delta_X(x),a_X(y,z))$ witness that $FX$ is a topological algebra of type $x(y^*z)$. 

If $l>1$, then take any disjoint family $(C_i)_{i=1}^{l}$ of finite subsets of $X$ of cardinality $|C_i|=l+1$. For every $i<l$ fix an enumeration $\{c_{i,j}\}_{j=1}^{l+1}$ of the set $C_i$.
For every $i<l$ consider the element $w_i=a_X(c_{i,1},\dots,c_{i,l+1})\in FX$ and observe that $\supp(w_i)=C_i$.
This implies that the elements $w_1,\dots,w_l$ are pairwise distinct.
Let $\vec c_l=(c_{l,1},\dots,c_{l,l-1})\in X^{l-1}$.

Define a ternary operation $p:X^3\to FX$ by the formula $$p(x,y,z)=\dot a_{FX}\big(w_1,\dots,w_{l-1},\delta_X(x),a_X(\vec c_l,y,z)\big).$$
We claim that the finite set  $C=\bigcup_{i=1}^l C_i$ and the operation $p$ witness that the functor-space $FX$ is a topological algebra of type $x(y^*z)$.
We need to verify the conditions (1)--(3) of Definition~\ref{d:xyz}.

First we show that $p(x,y,y)=p(x,x,x)$ for all $x,y\in X$. Denote the vectors $(\vec c_l,x,x)$ and $(\vec c_l,y,y)$ by $\vec x$ and $\vec y$, respectively. We think of these vectors as functions $\vec x,\vec y:l+1\to X$. It is clear that $\vec x= \vec x\circ r^{l+1}_l$ and hence  $F\vec x=F\vec x\circ Fr^{l+1}_l$ and $a_X(\vec x)=F\vec x(a)=F\vec x\circ Fr^{l+1}_l(a)=F\vec x(b)=b(\vec x)$ where $b=Fr^{l+1}_l(a)$.
By analogy we can show that $a_X(\vec y)=b(\vec y)$.
Since $\supp(b)\subset l-1$ and $\vec x|l-1=\vec y|l-1$, Lemma~\ref{l:equal-supp} guarantees that $a_X(\vec x)=b(\vec x)=b(\vec y)=a_X(\vec y)$ and hence
$$
\begin{aligned}
p(x,x,x)&=\dot a_{FX}(w_1,\dots,w_{l-1},\delta_X(x),a_X(\vec x))=\\
&=\dot a_{FX}(w_1,\dots,w_{l-1},\delta_X(x),a_X(\vec y))=p(x,y,y).
\end{aligned}
$$
So, the condition (1) of Definition~\ref{d:xyz} holds.

Since the monad $(F,\delta,\mu)$ preserves disjoint supports, for any pairwise distinct points $x,y,z\in X\setminus C$ the element $p(x,y,z)=\dot a_X(w_1,\dots,w_{l-1},\delta_X(x),a_X(\vec c_l,y,z))$ has support $\supp(p(x,y,z))\supset\{x,y,z\}$. So the condition (2) of Definition~\ref{d:xyz} is satisfied too. The condition (3) of this definition follows from the continuity of the operation $p$, condition (1) of Definition~\ref{d:xyz} and the property (ii) of the operation $a_Z$.
\smallskip

To see that $FX$ is a topological algebra of type $(x^*y)\cdot(z^*s)$, consider the operation $q:X^4\to FX$ defined by the formula
$$q(x,y,z,s)=\dot a_{FX}(w_1,\dots,w_{l-1},a_X(\vec c_l,x,y),a_X(\vec c_l,z,s)).$$
By analogy with the case of topological algebra $x{}(y^*z)$, it can be shown that the finite set $C=\bigcup_{i=1}^lC_i$ and the operation $q$ witness that $FX$ is a topological algebra of type $(x^*y)\cdot (z^*s)$.
\smallskip

Taking into account that $a\in F(l+1)\setminus F_l(l+1)$, we conclude that $F(l+1)\ne F_1(l+1)$, so we can apply the first statement which states that $FX$ is a topological algebra of type $x{\cdot}y$.
\smallskip

3. If the functor $F$ strongly fails to preserve preimages, then Definition~\ref{d:stong-preim} yields a number $n\ge 2$ and an element $a\in F(n+1)\setminus F_n(n+1)$ such that $Fr^{n+1}_n(a)\in F_{n-2}(n)$. Then for every Tychonoff space $Z$ the $F$-valued operation $a_Z:Z^{n+1}\to FZ$ has the following two properties:
\begin{itemize}
\item[(i)] $\supp(\dot a_Z(z_1,\dots,z_{l+1}))=\{z_1,\dots,z_{l+1}\}$ for any pairwise distinct points $z_1,\dots,z_{l+1}\in Z$;
\item[(ii)] $\supp(\dot a_Z(z_1,\dots,z_{l+1}))\subset \{z_1,\dots,z_{l-2}\}$ for any points $z_1,\dots,z_{l+1}\in Z$ with $z_l=z_{l+1}$.
\end{itemize}
On the space $FX$ consider the continuous $(l+1)$-ary operation
 $$\dot a_{FX}=\mu_X\circ a_{FX}:(FX)^{l+1}\to FX.$$

 Take any disjoint family $(C_i)_{i=1}^{l}$ of finite subsets of $X$ of cardinality $|C_i|=l+1$. For every $i<l$ fix an enumeration $\{c_{i,j}\}_{j=1}^{l+1}$ of the set $C_i$.
For every $i<l$ consider the element $w_i= a_X(c_{i,1},\dots,c_{i,l+1})\in FX$ and observe that $\supp(w_i)=C_i$.
This implies that the elements $w_1,\dots,w_l$ are pairwise distinct.
Let $\vec c_l=(c_{l,1},\dots,c_{l,l-2})\in X^{l-2}$.

Next, define a continuous operation $p:X^4\to FX$ by the formula $$p(x,y,z,s)=\dot a_{FX}(w_1,\dots,w_{l-1},\delta_X(x),a_X(\vec c_l,y,z,s)).$$
By analogy with the case of topological algebra $x{}(y^*z)$, it can be shown that the finite set  $C=\bigcup_{i=1}^l C_i$ and the operation $p$ witness that $FX$ is a topological algebra of type $x\cdot(y^*(z^*s)y)$.

To see that $FX$ is a topological algebra of type $(x^*(z^*s)x)\cdot(y^*(z^*s)y)$, consider the operation $q:X^4\to FX$ defined by the formula
$$q(x,y,z,s)=\dot a_{FX}(w_1,\dots,w_{l-1},a_X(\vec c_l,x,z,s), a_X(\vec c_l,y,z,s)).$$
It can be shown that the finite set $C=\bigcup_{i=1}^lC_i$ and the operation $q$ witness that $FX$ is a topological algebra of type $(x^*(z^*s)y){} (y^*(z^*s)y)$.
\end{proof}

Proposition~\ref{p:F->structure}(1) and Lemmas~\ref{l:type-xy} and \ref{l:x.y} imply:

\begin{corollary}\label{c:funct-a1} Assume that a monomorphic functor $F:\Top_i\to\Top_i$ has finite supports, is bounded, $\II$-regular, and $F(n)\ne F_1(n)$ for some $n\in\w$. Let $X\in\Top_i$ be a $\mu$-complete (functionally Hausdorff) space.
\begin{enumerate}
\item If $FX$ contains no (strong) $\Fin^\w$-fan, then  $X$  contains no (strong) $\bar S^\w$-fan or no (strong) $D_\w$-cofan.
\item If $FX$ contains no (strong) $\Fin^{\w_1}$-fan, then $X$  contains no (strong) $\bar S^{\w_1}$-fan.
\end{enumerate}
\end{corollary}

Combining this corollary with Theorems~\ref{t:Dcofan}(1) and \ref{t:Sfan}, we obtain:

\begin{theorem}\label{t:F-xy} Assume that a monomorphic functor $F:\Top_i\to\Top_i$ has finite supports, is bounded, $\II$-regular, and $F(n)\ne F_1(n)$ for some $n\in\w$. Let $X\in\Top_i$ be a $\mu$-complete space and one of the following conditions is satisfied:
\begin{enumerate}
\item $X$ is a Tychonoff $\aleph$-space and $FX$ contains no strong $\Fin^\w$-fans.
\item $X$ is $k^*$-metrizable and $FX$  contains no $\Fin^{\w}$-fans.
\end{enumerate}
Then the $k$-modification of $X$ either is metrizable or is a topological sum of cosmic $k_\w$-spaces.
\end{theorem}

Now we consider functors which do not preserve preimages.
We recall that a monomorphic functor $F:\Top_i\to\Top_i$ with finite supports
\begin{itemize}
\item \index{functor!does not preserve preimages}{\em does not preserve preimages} if for some $n\ge 1$ and $a\in F(n+1)\setminus F_n(n+1)$ we get $\supp(Fr^{n+1}_n(a))\subset n-1$;
\item \index{functor!strongly fails to preserve preimages}{\em strongly fails to preserve preimages} if for some $n\ge 2$ and $a\in F(n+1)\setminus F_n(n+1)$ we get $\supp(Fr^{n+1}_n(a))\subset n-2$.
\end{itemize}
Here by $r^{n+1}_n\colon n+1\to n$ we denote the (monotone) retraction of $n+1$ onto $n$.

Combining Proposition~\ref{p:F->structure}(1,2) with Corollaries~\ref{c:xy+xyz} and \ref{c:xy+xyzs} we obtain:

\begin{corollary}\label{c:Funct2} Assume that a monomorphic functor $F:\Top_i\to\Top_i$ has finite supports, is bounded, $\II$-regular, can be completed to a monad $(F,\delta,\mu)$ that preserves disjoint supports, but $F$ does not preserve preimages. Let $X\in\Top_i$ be a $\mu$-complete (functionally Hausdorff) space.
\begin{enumerate}
\item[\textup{1)}] If $FX$ contains no (strong) $\Fin^\w$-fans, then $X$ contains no (strong) $D_\w$-cofan or no (strong) $S^\w$-semifan.
\item[\textup{2)}] If $FX$ contains no (strong) $\Fin^{\w_1}$-fan, then $X$ contains no (strong) $S^{\w_1}$-semifan.
\end{enumerate}
\end{corollary}

Proposition~\ref{c:Funct2} and Theorem~\ref{t:xy+xyz+xy.zs} imply:


\begin{theorem}\label{t:Functor2} Assume that a monomorphic functor $F:\Top_i\to\Top_i$ has finite supports, is bounded, $\II$-regular, can be completed to a monad $(F,\delta,\mu)$ that preserves disjoint supports, but $F$ does not preserve preimages. Let $X\in\Top_i$ be a $\mu$-complete space.
\begin{enumerate}
\item If $X$ is a functionally Hausdorff ($\aleph_k$-)space and $FX$ contains no strong $\Fin^{\w_1}$-fans, then $X$ contains a closed $\w_1$-bounded subset $B\subset X$ whose complement $X\setminus B$ contains no infinite metrizable compact set (and is $k$-discrete).
\item If the space $FX$ contains no $\Fin^{\w_1}$-fan and $X$ is $k^*$-metrizable, then $X$ contains a $k$-closed $\aleph_0$-subspace $A\subset X$ with $k$-discrete complement $X\setminus A$.
\item If the $X$ is Tychonoff and the space $FX$ contains no strong $\Fin^\w$-fan, then:
\begin{enumerate}
\item If $X$ is cosmic, then $X$ is $\sigma$-compact.
\item If $X$ is an $\bar\aleph_k$-space, then  either $X$ contains a compact set with $k$-discrete complement or $X$ is $k$-homeomorphic to a topological sum of cosmic $k_\w$-spaces.
\item If $X$ is an $\aleph$-space, then the $k$-modification $kX$ of $X$ either is a topological sum of cosmic $k_\w$-spaces or $kX$ is a metrizable space with compact set of non-isolated points.
\item If $X$ is a $\bar\aleph_k$-$k_\IR$-space, then  either
 $X$ is a topological sum $K\oplus D$ of a cosmic $k_\w$-space $K$ and a discrete space $D$ or else $X$ is metrizable and has compact set of non-isolated points.
\end{enumerate}
\item If the space $FX$ contains no $\Fin^{\w}$-fan, then:
\begin{enumerate}
\item if $X$ is an $\aleph_k$-space, then $X$ either $X$ contains a compact set with $k$-discrete complement or the $k$-modification of $X$ is a topological sum $D\oplus K$ of a discrete space $D$ and a cosmic $k_\w$-space $K$.
\item if $X$ is Tychonoff and $k^*$-metrizable, then its $k$-modification $kX$ is either a metrizable space with compact set of non-isolated points or $kX$ is a topological sum of a discrete space and a cosmic $k_\w$-space.
\end{enumerate}
\end{enumerate}\end{theorem}

Finally we consider functors that strongly fail to preserve preimages.
Proposition~\ref{p:F->structure}(3) and Lemmas~\ref{l:xyzs}, \ref{l:long} imply:

\begin{corollary}\label{c:Functor3} Assume that a monomorphic functor $F:\Top_i\to\Top_i$ has finite supports, is bounded, $\II$-regular, can be completed to a monad $(F,\delta,\mu)$ that preserves disjoint supports, but $F$ strongly fails to preserve preimages. Let $X\in\Top_i$ be a $\mu$-complete (functionally Hausdorff) space.
\begin{enumerate}
\item[\textup{1)}] If $FX$ contains no (strong) $\Fin^\w$-fans, then $X$ contains no (strong) $D_\w$-cofan.
\item[\textup{2)}] If $FX$ contains no (strong) $\Fin^{\w_1}$-fan, then $X$ contains no non-trivial convergent sequence or no (strongly) compact-finite uncountable subset.
\end{enumerate}
\end{corollary}

On the other hand, Proposition~\ref{p:F->structure}(3) and Theorem~\ref{t:long1} yield:

\begin{theorem}\label{t:Functor3} Assume that a monomorphic functor $F:\Top_i\to\Top_i$ has finite supports, is bounded, $\II$-regular, can be completed to a monad $(F,\delta,\mu)$ that preserves disjoint supports, but $F$ strongly fails to preserve preimages. Let $X\in\Top_i$ be a $\mu$-complete  space.
\begin{enumerate}
\item[\textup{1)}] If $X$ is functionally Hausdorff, $FX$ contains no strong $\Fin^{\w_1}$-fan and each infinite compact subset of $X$ contains a convergent sequence, then $X$ either $k$-discrete or $\w_1$-bounded.
\item[\textup{2)}] If $X$ is $k^*$-metrizable and $FX$ contains no $\Fin^{\w_1}$-fan, then either $X$ is $k$-discrete or $X$ is a (countable) $\aleph_0$-space.
\item[\textup{3)}] If $X$ is a Tychonoff $\bar\aleph_k$-space and $FX$ contains no strong $\Fin^{\w}$-fan, then $X$ is either $k$-discrete or $X$ is $\w_1$-bounded and is a $k$-sum of (countable) hemicompact spaces; moreover, if $X$ is  a $k_\IR$-space, then $X$ is either discrete or a (countable) cosmic $k_\w$-space.
\item[\textup{4)}] If $X$ is a functionally Hausdorff $\aleph_k$-space and $FX$ contains no strong $\Fin^{\w}$-fan and no $\Fin^{\w_1}$-fan, then $X$ is either $k$-discrete or (countable and) hemicompact.
\item[\textup{5)}] If $X$ is an $\aleph_k$-space and $FX$ contains no $\Fin^\w$-fan, then
$X$ is either $k$-discrete or (countable and) hemicompact.
\end{enumerate}
\end{theorem}

\section{Superbounded functors}

In this section we shall give a condition on a monadic functor $F:\Top_i\to\Top_i$ guaranteeing that for every infinite space $X\in\Top_i$ the functor space $FX$ is a topological superalgebra of type $(x^*y)(z^*s)$ (see Definition~\ref{d:xy.zs}). In this section we assume that $\Top_i$ is a full hereditary subcategory of $\Top_{2\frac12}$, containing all compact metrizable spaces, and $F:\Top_i\to\Top_i$ is a monomorphic functor with finite supports that preserves closed embeddings of metrizable compact and can be completed to a monad $(F,\delta,\mu)$. Let $E:\Top_i\to\Top_i$ be a subfunctor of $F$ such that  $\delta_X(X)\subset EX\subset FX$ for every object $X$ of $\Top_i$ and $Ef=Ff|EX$ for every morphism $f:X\to Y$ of $\Top_i$.

Given a space $X\in\Top_i$ and a subset $A\subset FX$, denote by $i_{A,FX}:A\to FX$ the identity inclusion and consider the subspace $Ei_{A,FX}(EA)\subset E(FX)\subset F(FX)$. Its image $\langle A\rangle_E=\mu_X(Ei_{A,FX}(EA))\subset FX$ under the multiplication map $\mu_X:F^2X\to FX$ of the monad is called the \index{subset!$E$-hull of}{\em $E$-hull} of $A$ in $FX$.

We claim that the $E$-hull $\langle A\rangle_E$ of $A$ contains $A$. Indeed, the inclusion $\delta_A(A)\subset EA$, the equality $Ei_{A,FX}=Fi_{A,FX}|EX$, and the naturality of $\delta$ imply that $$Ei_{A,FX}(EA)\supset Ei_{A,FX}(\delta_A(A))=Fi_{A,FX}(\delta_A(A))=\delta_{FX}\circ i_{A,FX}(A)$$ and hence
$$\langle A\rangle_E\supset \mu_X\circ \delta_{FX}\circ i_{A,FX}(A)=i_{A,FX}(A)=A.$$

In the proof of Theorem~\ref{t:superbound} below we shall use the following two properties of the $E$-hulls.

\begin{lemma}\label{l:E-hull1} For any morphism $f:X\to Y$ of the category $\Top_i$ and any subsets $A\subset FX$ and $B\subset FY$ the inclusion $Ff(A)\subset B$ implies $Ff(\langle A\rangle_E)\subset \langle B\rangle_E$.
\end{lemma}

\begin{proof} This inclusion can be easily derived from the commutativity of the following diagram:
$$\xymatrix{
EA\ar^{Ei_{A,FX}}[r]\ar_{Ef|A}[d]& EFX\ar[r]\ar^{EFf}[d]&F^2X\ar^{\mu_X}[r]\ar^{F^2f}[d]&FX\ar^{Ff}[d]\\
EB\ar_{Ei_{B,FY}}[r]& EFY\ar[r]&F^2Y\ar_{\mu_Y}[r]&FY.
}
$$
\end{proof}

\begin{lemma} Let $X\in\Top_i$ be a space and $A\subset B$ be finite subsets in $X$. Then $Fi_{B,X}(\langle F(A;B)\cup\delta_B(A)\rangle_E)=\langle F(A;X)\cup\delta_X(A)\rangle_E$.
\end{lemma}

 \begin{proof} Consider the sets $$U=F(A;X)\cup\delta_X(A)\subset FX\mbox{ \  and \ }V=F(A;B)\cup\delta_B(A)\subset FB$$ and observe that
 $$
 \begin{aligned}
 Fi_{B,X}(V)&=Fi_{B,X}(Fi_{A,B}(FA)\cup\delta_B(A))=\\
 &=Fi_{B,X}\circ Fi_{A,B}(FA)\cup Fi_{B,X}\circ\delta_B(A)=\\
 &=Fi_{A,X}(FA)\cup \delta_X\circ i_{B,X}(A)=U.
 \end{aligned}
 $$
 Since $F$ preserves closed embeddings of metrizable compacta, the map $Fi_{B,X}:FB\to FX$ is a closed topological embedding (see Proposition~\ref{p:FembTych}). Then the restriction $h=Fi_{B,X}|V:V\to U$ is a homeomorphism and so is the map $Eh:EV\to EU$.
 Now the equality $$Fi_{B,X}(\langle F(A;B)\cup\delta_B(A)\rangle_E=Fi_{B,X}\circ \mu_B(Ei_{V,FB}(EV))=\mu_X(Ei_{U,FX}(EU))$$ follows from the bijectivity of $Eh$ and the commutativity of the following diagram:
 $$\xymatrix{
EV\ar^{Ei_{V,FB}}[r]\ar_{Eh}[d]& EFB\ar[r]\ar^{EFi_{B,X}}[d]&F^2B\ar^{\mu_B}[r]\ar^{F^2i_{B,X}}[d]&FB\ar^{Fi_{B,X}}[d]\\
EU\ar_{Ei_{U,FX}}[r]& EFX\ar[r]&F^2X\ar_{\mu_Y}[r]&FX.
}
$$
\end{proof}

 Given an element $a\in FX$ with non-empty support $\supp(a)$, consider the family $$\Supp_E(a)=\big\{A\subset \supp(a): a\in \langle F(A;X)\cup\delta_X(\supp(a))\rangle_E\big\},$$
which contains the set  $A=\supp(a)$ and hence is not empty. Then its intersection $\supp_E(a)$ is a well-defined subset of $\supp(a)$, called the {\em $E$-support} of $a$. If $\supp(a)=\emptyset$, then the set $\supp_E(a)$ is not defined, so we put $\supp_E(x):=\emptyset$. In any case we get $\supp_E(a)\subset\supp(a)$.

\begin{definition} The functor $F$ is defined to be \index{functor!$E$-superbounded}{\em $E$-superbounded} if $F$ is strongly bounded and for any space $X\in\Top_i$ and compact subset $K\subset FX$ there exists a finite subset $B\subset X$ such that for every element $a\in K$ with non-empty support the family $\Supp_E(a)$ contains a subset of $B$.
\end{definition}

For $E$-superbounded functors $F$, it is easy to construct operations on functor-spaces $FX$ turning them into topological superalgebras of type $(x^*y)(z^*s)$.

\begin{proposition}\label{p:F->super} Assume that the functor $F$ is $E$-superbounded. For a space $X\in\Top_i$ the functor-space $FX$ is a topological superalgebra of type $(x^*y){}(z^*s)$ if there exist a finite subset $C\subset X$ and an $F$-valued operation $p:X^4\to FX$ having the properties:
\begin{enumerate}
\item[\textup{1)}] $p(x,x,x,x)=p(x,x,z,z)=p(z,z,z,z)$ for all $x,z\in X$;
\item[\textup{2)}]  $\{x,y,z,s\}\subset \supp_E\big(p(x,y,z,s)\big)$ for any pairwise distinct points $x,y,z,s\in X\setminus C$;
\item[\textup{3)}] for any point $x\in X$ and neighborhood $W\subset FX$ of $p(x,x,x,x)$ there is a neighborhood $U_\Delta\subset X_d\times X$ of the diagonal $\Delta$ of $X^2$ such that $p(U_{\Delta}\times U_\Delta)\subset W$.
\end{enumerate}
\end{proposition}

\begin{proof} Given compact sets  $B\subset X$ and $K\subset FX$, we need to find a finite subset $A\subset X$ such that for any pairwise distinct points $x,y,z,s\in B\setminus C$ with $p(x,y,z,s)\in K$ we get $\{x,z\}\subset A$.

Since the functor $F$ is $E$-superbounded, for the compact set $K$ there exists a finite set $A\subset X$ such that for every element $a\in K$ with non-empty support we get $A\cap\supp(a)\in\Supp_E(a)$. Now take any pairwise distinct elements $x,y,z,s\in B\setminus C$ and consider the element $a=p(x,y,z,s)$. By the condition (2), $\{x,z\}\subset\supp_E(a)\subset \supp(a)$, which implies that the support of $p(a)$ is not empty. Then the choice of $A$ guarantees that $A\cap \supp(a)\in\Supp_E(a)$. Then the condition (2) yields the required inclusion $\{x,z\}\subset \supp_E(a)=\bigcap\Supp_E(a)\subset A\cap\supp(a)\subset A$.
\end{proof}

The following theorem will help us to detect $E$-superbounded functors.

\begin{theorem}\label{t:superbound} The functor $F$ is $E$-superbounded if and only if $F$ is strongly bounded and for every compact metrizable space $Y$ with finite number of non-isolated points and every compact set $K\subset FY$ there exists a finite set $B\subset Y$ such that for every element $a\in K$ with non-empty support the family $\Supp_E(a)$ contains the set $B\cap\supp(a)$.
\end{theorem}

\begin{proof} The ``only if'' part is trivial. To prove the ``if'' part, assume that the corresponding assumptions hold.
Fix a space $X\in\Top_i\subset\Top_{2\frac12}$ and compact subset $K\subset FX$. Since $F$ is strongly bounded, the set $\supp(K)$ is contained in $[X]^{\le m}$ for some $m\in\IN$. To derive a contradiction, assume that for every finite subset $Z\subset X$ there exists an element $a_Z\in K$ with non-empty support such that $\supp(a_Z)\cap Z\notin \Supp_E(a_Z)$. Put $Z_0=\emptyset$ and for every $n\in\IN$ define the finite set $Z_n$ by the recursive formula:
$$Z_n=Z_{n-1}\cup \supp(a_{Z_{n-1}}).$$
Consider the countable set $Z_\w=\bigcup_{n\in\w}Z_n$.
Since the space $X$ is functionally Hausdorff, there exists a continuous function $f:X\to\II$ such that the restriction $f|Z_\w$ is injective. It follows that $\big(f(\supp(a_{Z_n})\big)_{n\in\w}$ is a sequence in the
space $[\II]^{\le m}$, which is compact and metrizable with respect to the Vietoris topology. Then we can choose an increasing number  sequence $(n(k))_{k\in\w}$ such that the sequence $\big(f(\supp(a_{Z_{n(k)}})\big)_{k\in\w}$ converges in $[\II]^{\le m}$ to some set $Y_\infty\in[\II]^{\le m}$. It follows that $Y=Y_\infty\cup\bigcup_{k\in\w}\supp(a_{Z_{n(k)}})$ is a compact metrizable space with at most $m$ non-isolated points.

Since $F$ preserves closed embeddings of metrizable compacta, the map $Fi_{Y,\II}:FY\to F(\II)$ is a closed embedding. So, we can identify the space $FY$ with the closed subspace $F(Y;\II)=Fi_{Y,\II}(FY)$ of $F(\II)$. The continuity of the map $Ff:FX\to F(\II)$ guarantees that the set $Ff(K)$ is compact and so is its intersection $K_Y=Ff(K)\cap F(Y;\II)$. By our assumption, there exists a finite subset $Y'\subset Y$ such that for every element $b\in K_Y$ with non-empty support the family $\Supp_E(b)$
contains the set $Y'\cap \supp(b)$ as an element. Now consider the finite set $X'=Z_\w\cap f^{-1}(Y')$ and choose $k\in\w$ so large that $X'\subset Z_{n(k)}$. The choice of the element $a=a_{Z_{n(k)}}\in K$ guarantees that $a$ has non-empty support $A=\supp(a)\subset Z_{n(k)+1}$ such that the set $A'=A\cap X'\subset A\cap Z_{n(k)}\notin \Supp_E(a)$. Consider also the element $b=Ff(a)\in F(\II)$ and its support $B=\supp(b)$. By Proposition~\ref{p:supp=image}, $B=\supp(b)=f(\supp(a))=f(A)\subset Y$. By Theorem~\ref{t:supp}, $b\in F(\supp(b);\II)\subset F(Y;\II)$ and hence $b\in K_Y$. The choice of the set $Y'$ guarantees that the set $B'=\supp(b)\cap Y'$ belongs to the family $\Supp_E(b)$. Observe that $f(A')=f(A\cap X')=f(A)\cap f(X')=B\cap Y'=B'\in \Supp_E(b)$.
Looking at the following commutative diagram, we can see that $a\in \langle F(A';X)\cup\delta_X(A)\rangle_E$, which means that $A'\in\Supp_E(a)$ and this is a desired contradiction completing the proof.

$$\xymatrix{
\langle F(A';X)\cup\delta_X(A)\rangle_E\ar[ddd]\ar[rd]&\{a\}\ar[r]\ar[d]&\{c\}\ar[r]\ar[d]\ar[l]&\{b\}\ar[d]\ar[r]\ar[l]&\langle F(B';Y)\cup\delta_Y(B)\rangle_E\ar[ddd]\ar[ld]\\
&FX\ar^{Ff}[r]&F\II&FY\ar_{Fi_{Y,\II}}[l]\\
&FA\ar[rr]\ar^{Fi_{A,X}}[u]&&FB\ar[ll]\ar_{Fi_{B,Y}}[u]\\
\langle F(A';A)\cup\delta_A(A)\rangle_E\ar[rrrr]\ar[uuu]\ar[ru]&&&&\langle F(B';B)\cup\delta_B(B)\rangle_E\ar[llll]\ar[uuu]\ar[lu]
}
$$
\end{proof}

\chapter{Constructing fans in free universal algebras}\label{ch:free}

In this Chapter we shall apply the general theory developed in Chapters~\ref{ch:functor} and \ref{ch:algebra} to study the $k$-space and Ascoli properties in free objects that naturally appear in Topological Algebra and Functional Analysis. We start with studying free topologized $E$-algebras of a given signature $E$ in varieties of topologized $E$-algebras. In Chapter~\ref{ch:magma} we apply the results obtained in this chapter to study free objects in some concrete varieties of topologized algebras. The theory presented in this chapter can be considered as a self-containment treatment of the theory of universal topologized algebras developed (in a bit different form) by Choban with coauthors (see \cite{Cho93} and references therein).

\section{Universal $E$-algebras}

In this section we recall known information on universal algebras (see, \cite{Gr08} for more details).

A \index{signature}\index{universal algebra!signature of}{\em signature} is a set $E$ endowed with a function $\alpha:E\to\w$ called the {\em arity} function. For every $n\in\w$ the preimage $E_n=\alpha^{-1}(n)$ is the set of symbols for operations of arity $n$ in the signature $E$. Elements of the set $E_0$ are called symbols of constants. Any subset $A\subset E$ endowed with the restricted arity function $\alpha|A$ is called a \index{subsignature}\index{signature!subsignature of}{\em subsignature} of the signature $E$.

 A \index{universal algebra!of signature $E$}{\em universal algebra of signature $E$} (briefly, an {\em $E$-algebra}) is a pair $(X,\star)$ consisting of a set $X$ and a  function $\star\colon\oplus_{n\in\w}E_n\times X^n\to X$, $\star\colon (e,x_1,\dots,x_n)\mapsto e(x_1,\dots,x_n)$ called the {\em multiplication}. This function turns each element $e\in E_n\subset E$ into an $n$-ary operation $e\colon X^n\to X$, $e\colon (x_1,\dots,x_n)\mapsto \star(e,x_1,\dots,x_n)$, on $X$. So, often we shall omit the symbol of multiplication and shall write $e(x_1,\dots,x_n)$ instead of $\star(e,x_1,\dots,x_n)$. Also by $E_n(X^n)$ we shall denote the set $\star(E_n\times X^n)$ and by $E[X]$ the set $\bigcup_{n\in\w}E_n(X^n)$.

For $e\in E^0$ the power $X^0$ is a singleton (consisting of the unique map $\emptyset\to X$) and the 0-ary operation $e:X^0\to X$ maps the unique element of $X^0$ to some point of $X$. The set $E_0(X^0)$ is called the set of constants of the $E$-algebra $X$. If $E_0\ne\emptyset$, then the set $E_0(X^0)$ is not empty and hence $X\ne\emptyset$.

 A subset $A\subset X$ of an $E$-algebra $X$ is called an \index{$E$-algebra}\index{$E$-subalgebra}{\em $E$-subalgebra} of $X$  if $E[A]\subset A$. In this case $A$ endowed with the restricted operation $\star|\oplus_{n\in\w}E_n\times A^n$ is an $E$-algebra.  Observe that any $E$-subalgebra $A$ necessarily contains the set $E_0(X^0)=E_0(A^0)$ of constants of $X$.

For every subset $A$ of a universal $E$-algebra $X$ the intersection $\langle A\rangle_E$ of all $E$-sub\-algebras of $X$ that contain $A$ is called the $E$-subalgebra generated by $A$. This subalgebra can be equivalently described as $\langle A\rangle_E=\bigcup_{m\in\w}E^m[A]$ where $E^0[A]=A$ and $$E^{m+1}[A]=E^m[A]\cup E[E^m[A]]\mbox{ \ \  for \ }m\in\w.$$

A function $f:X\to Y$ between two $E$-algebras is called a \index{$E$-homomorphism}{\em $E$-homomorphism} if for every $n\in\w$, $e\in E_n$ and points $x_1,\dots,x_n\in X^n$ we get $$f(e(x_1,\dots,x_n))=e(f(x_1),\dots,f(x_n)).$$ A bijective $E$-homomorphism between two $E$-algebras is called an \index{$E$-isomorphism}{\em $E$-isomorphism}.

For universal $E$-algebras $(X_\alpha)_{\alpha\in A}$ their product $\prod_{\alpha\in A}X_\alpha$ has a natural structure of a universal $E$-algebra endowed with the coordinatewise multiplication: for an element $e\in E_n\subset E$ and functions $x_1,\dots,x_n\in\prod_{\alpha\in A}X_\alpha$ the function $e(x_1,\dots,x_n)\in\prod_{\alpha\in A}X_\alpha$ assigns to each $\alpha\in A$ the element $e(x_1(\alpha),\dots,x_n(\alpha))$.

\section{Topologized and topological $E$-algebras}

In this section we discuss topologized versions of the algebraic notions discussed in the preceding section.

A \index{signature!topological}{\em topological signature} is a signature $E$ endowed a topology making the arity function $\alpha:E\to\w$ continuous. In this case the sets $E_n=\alpha^{-1}(n)$, $n\in\w$, are closed-and-open in $E$.

 A \index{$E$-algebra!topologized}{\em topologized universal $E$-algebra} (briefly, a {\em topologized $E$-algebra}) is an universal $E$-algebra $X$ endowed with a topology $\tau$. A topologized $E$-algebra $X$ is called a \index{$E$-algebra!topological}{\em topological $E$-algebra} if its multiplication $\star:\oplus_{n\in\w}E_n\times X^n\to X$ is continuous.

 For a topologized $E$-algebra $X$ any $E$-subalgebra $A$ of $X$ endowed with the subspace topology is a topologized $E$-algebra.

Topologized $E$-algebras are objects of the category $\TEA$ whose morphisms are continuous $E$-homomorphisms of topologized $E$-algebras. A function $f:X\to Y$  between two topologized $E$-algebras is called a \index{$E$-isomorphism!topological}{\em topological $E$-isomorphism} if $f$ is bijective and both maps $f$ and $f^{-1}$ are continuous $E$-homomorphisms.

Let $\K$ be a class of topologized $E$-algebras and $X$ be a topological space. By a \index{$E$-algebra!free}{\em free topologized $E$-algebra of $X$ in the class $\K$} we understand a pair $(F_\K(X),\delta_X)$ consisting of a topologized $E$-algebra $F_\K(X)\in\K$ and a continuous map $\delta_X:X\to F_\K(X)$ such that for every continuous map $f:X\to Y$ to a topologized $E$-algebra $Y\in\K$ there exists a unique continuous $E$-homomorphism $\bar f:F_\K(X)\to Y$ such that $f=\bar f\circ\delta_X$.
This definition implies that any two free topologized $E$-algebras $(F_\K(X),\delta_X)$ and $(F'_\K(X),\delta_X')$ over $X$ are topologically $E$-isomorphic in the sense that there is a unique topological $E$-isomorphism $h:F_\K(X)\to F'_\K(X)$ such that $h\circ \delta_X=\delta_X'$.

A free topologized $E$-algebra $(F_\K(X),\delta_X)$ exists if the class $\K$ is a variety.

\begin{definition} A class $\K$ of topologized $E$-algebras is called a \index{variety}{\em variety} if  $\K$ contains a non-empty topologized $E$-algebra and $\K$ is closed under Tychonoff products, taking $E$-subalgebras and taking images under topological $E$-isomorphisms.
\end{definition}

The following proposition is a modification of a classical argument of Birkhoff \cite[Ch.4]{Gr08}.

\begin{proposition}\label{p:FA-exists} For any variety $\K$ of topologized $E$-algebras and any topological space $X$ a free topologized $E$-algebra $(F_\K(X),\delta_X)$ exists and is unique up to a topological $E$-isomorphism.
\end{proposition}

\begin{proof} Given a topological space $X$, consider the {\em set} $\mathcal A_\kappa\subset \K$ of all topologized $E$-algebras whose underlying set is a subset of the cardinal $\kappa=\max\{|X|,|E|,\w\}$. Next, consider the set $\mathcal P$ of all pairs $(Y,f)$ where $Y\in\mathcal A_\kappa$ and $f:X\to Y$ is a continuous function. Let $\mathcal P=\{(Y_\alpha,f_\alpha)\}_{\alpha\in |\mathcal P|}$ be an enumeration of the set $\mathcal P$ by ordinals $\alpha<|\mathcal P|\le 2^{2^\kappa}$. Consider the Tychonoff product $Y=\prod_{\alpha<|\mathcal P|}Y_\alpha$ of the topologized $E$-algebras $Y_\alpha$, $\alpha<|\mathcal P|$, and the diagonal product $\delta_X:X\to  Y$, $\delta_X:x\mapsto (f_\alpha(x))_{\alpha<|\mathcal P|}$, of the continuous functions $f_\alpha$, $\alpha\in A$. Then for the $E$-subalgebra $F_\K(X)=\langle\delta_X(X)\rangle_E$ of $Y$, generated by the set $\delta_X(X)$, the pair $(F_\K(X),\delta_X)$ is a free topologized $E$-algebra over $X$ in the variety $\K$.

The uniqueness of a free topologized $E$-algebra follows from the definition.
\end{proof}

The construction of a free topologized $E$-algebra determines a functor $F_\K:\Top\to\TEA$ from the category of topological spaces and their continuous maps to the category $\TEA$ of topologized $E$-algebras and their continuous $E$-homomorphisms. Composing this functor with the functor $U:\TEA\to \Top$ forgetting the algebraic structure, we get a functor $F_\K:\Top\to\Top$ (denoted by the same symbol $F_\K$), which can be completed to a monad $(F_\K,\delta,\mu)$. In this monad, for every topological space $X$ the component $\delta_X$  of the natural transformation $\delta:\Id\to F_\K$ is just the map $\delta_X:X\to F_\K(X)$ (from the definition of a free topologized $E$-algebra) and the component $\mu_X$ of the multiplication $\mu:F_\K^2\to F_\K$ is the unique continuous $E$-homomorphism $\mu_X:F_\K(F_\K(X))\to F_\K(X)$ such that $\mu_X\circ\delta_{F_\K(X)}=\id_{F_\K(X)}$ (this  $E$-homomorphism $\mu_X$ exists and is unique by the definition of a free topologized $E$-algebra).

From now on we assume that $\K$ is a variety of topologized $E$-algebras. In this case for every topological space $X$ a free topologized $E$-algebra $(F_\K(X),\delta_X)$ exists and is unique up to a topological $E$-isomorphism.

\begin{proposition}\label{p:delta-inject} For a (Tychonoff) functionally Hausdorff space $X$ the map $\delta_X:X\to F_\K(X)$ is  injective (a topological embedding) if the closed interval  $\II$ admits a topological embedding $e:\II\to A$ into some topologized $E$-algebra $A\in\K$.
\end{proposition}

 \begin{proof}  Indeed, replacing $A$ by the $E$-subalgebra generated by the arc $I=e(\II)$ in $A$, we can assume that $A\in\mathcal A_\kappa$ where the set $\A_\kappa$ is taken from the proof of Proposition~\ref{p:FA-exists}. Next, consider the subset $\Omega=\{\alpha<|\mathcal P|:Y_\alpha=A,\;f_\alpha(X)\subset I\}$ and the projection $\pr:\prod_{\alpha<|\mathcal P|}Y_\alpha\to \prod_{\alpha\in\Omega}Y_\alpha$.  Taking into account that the space $X$ is functionally Hausdorff (Tychonoff), we conclude that
the map $\pr\circ\delta_X:X\to I^\Omega\subset M^\Omega= \prod_{\alpha\in\Omega}Y_\alpha$ is injective (a topological embedding). Then the map $\delta_X:X\to F_\K(X)\subset\prod_{\alpha<|\mathcal P|}Y_\alpha$ is injective (a topological embedding), too.
\end{proof}

Now we establish some properties of the functor $F_\K:\Top\to\Top$. We recall that a functor $F:\Top\to\Top$ has {\em finite supports} if for every topological space $X$ and every element $a\in FX$ there exists a map $f:n\to X$ defined on a finite ordinal $n=\{0,\dots,n-1\}$ such that $a\in Ff(Fn)$.

\begin{proposition}\label{p:TEA-fin-supp} The functor $F_\K:\Top\to\Top$ has finite supports.
\end{proposition}

\begin{proof} Fix a topological space $X$ and an element $a\in F_\K(X)$. By the construction of a free topologized $E$-algebra $F_\K(X)$, the element $a$ belongs to the $E$-subalgebra of $F_\K(X)$, generated by the set $\delta_X(X)$. Then there exists a finite set $Z\subset X$ such that $a$ belongs to the $E$-subalgebra $\langle\delta_X(Z)\rangle$ of $F_\K(X)$ generated by the set $Z$. Now take any function $f:n\to X$ with $Z\subset f(n)$ and observe that $a\in F_\K f(F_\K(n))$ (which follows from the preservation of the algebraic operations by the $E$-homomorphism $F_\K f$).
\end{proof}

Propositions~\ref{p:TEA-fin-supp} and \ref{c:FKX-surj} imply:

\begin{corollary}\label{c:FKX-surj} For any surjective continuous map $f:X\to Y$ between topological spaces the continuous $E$-homomorphism $F_\K f:F_\K(X)\to F_\K(Y)$ is surjective.
\end{corollary}

For a topological space $X$ and an element $a\in F_\K(X)$ the finite set
$$\supp(a)=\bigcap\{A\in[X]^{<\w}:a\in\langle\delta_X(A)\rangle_E\}$$is called the {\em support} of $a$.

\begin{proposition}\label{p:supp} If for some full subcategory $\Top_i\subset\Top$ containing all finite discrete spaces the functor $F_\K|\Top_i$ is monomorphic, then for every space $X\in\Top_i$ every element $a\in F_\K(X)$ is contained in the set $\langle \delta_X(A)\rangle_E$ for any non-empty subset $A\subset X$ containing $\supp(a)$.
\end{proposition}

\begin{proof}
Following notations from Chapter~\ref{ch:functor}, for a topological space $X$ and a finite subset $A\subset X$ by $i_{A,X}:A\to X$ denote the identity inclusion and by $i^d_{A,X}:A_a\to X$ the same map $i_{A,X}$ defined on the discrete modification of $A$, i.e., the set $A$ endowed with the discrete topology. The continuous map $i^d_{A,X}:A_d\to X$ induces a continuous $E$-homomorphism $F_\K i^d_{A,X}:F_\K(A_d)\to F_\K(X)$ whose image $F_\K i^d_{A,X}(F_\K(A_d))$ is denoted by $F_\K(A_d;X)$.
Taking into account that $Fi^d_{A,X}:F_\K(A_d)\to F_\K(X)$ is an $E$-homomorphism and the $E$-algebra $F_\K(A_d)$ coincides with the $E$-hull $\langle \delta_{A_d}(A_d)\rangle_E$ of the set $\delta_{A_d}(A_d)$, we conclude that $F(A_d;X)=\langle Fi^d_{A,X}(\delta_{A_d}(A_d))\rangle_E=\langle \delta_{X}(i^d_{A,X}(A_d))\rangle_E=\langle\delta_X(A)\rangle_E$. This implies that $\supp(a)=\bigcap\{A\in[X]^{<\w}:a\in F_\K(A_d;X)\}$. Applying Theorem~\ref{t:supp}, we conclude that $a\in F(A_d;X)=\langle\delta_X(A)\rangle_E$ for any non-empty subset $A\subset X$ containing $\supp(a)$.
\end{proof}

The following example shows that in general the inclusion $a\in\langle \delta_X(A)\rangle_E$ in Proposition~\ref{p:supp} cannot be improved to $a\in \langle \delta_X(\supp(a))\rangle_E$.

\begin{example} \textup{Consider the signature $E$ such that the space $E_1=\{e_1\}$ is a singleton and all other spaces $E_n$, $n\ne 1$, are empty. Let $\K$ be the variety of topological $E$-algebras $X$ satisfying the identity $e_1(x)=e_1(y)$ for all $x,y\in X$. This means that the unary operation $e_1:X\to X$ is constant. Then for the empty topological space its free topologized $E$-algebra $F_\K\emptyset$ is empty. On the other hand, for a non-empty topological space $X$ the free topologized $E$-algebra $F_\K(X)=\delta_X(X)\oplus \{1_X\}$ where $1_X$ is some element that does not belong to $\delta_X(X)$ and $1_X=e_1(x)$ for all $x\in F_\K(X)$. If $X$ contains more than one point, then $\supp(1_X)=\emptyset$ but $1_X\notin \langle \emptyset\rangle_E$.}
\end{example}

\begin{proposition}\label{p:supp-delta} Assume that the variety  $\K$ contains a topologized $E$-algebra $Y$ of cardinality $|Y|>1$ and for some full subcategory $\Top_i\subset\Top$ containing all finite discrete spaces the functor $F_\K|\Top_i$ is monomorphic. Then for every topological space $X$ and point $x\in X$ the element $a=\delta_X(x)\subset F_\K(X)$ has support $\supp(a)=x$.
\end{proposition}

\begin{proof} It is clear that $\supp(a)\subset\{x\}$. Assuming that $\supp(a)\ne\{x\}$, we conclude that $\supp(a)=\emptyset$ and hence $a\in \langle \delta_X(A)\rangle_E$ for some finite set $A\subset X\setminus\{x\}$. Consider the finite set $B=A\cup\{x\}$. Since the functor $F_\K|\Top_i$ is monomorphic, the $E$-homomorphism $F_\K i^d_{B,X}:F_\K(B_d)\to F_\K(X)$ is injective. Identifying the $E$-algebra $F_\K(A_d)$ with the $E$-subalgebra $F_\K i^d_{B,X}(F_\K(B_d))$, we conclude that $a\in\langle \delta_{B_d}(A)\rangle_E$. Choose a cardinal $\kappa$ such that $|Y^\kappa|>|F_\K(\{x\})|$. Fix any point $y\in Y^\w$ and using the inequality $|Y^\kappa|>|F_\K(\{x\})|\ge\langle \{y\}\rangle_E$, find a point $z\in Y^\kappa\setminus\langle\{y\}\rangle_E$. Now consider the continuous map $f:B_d\to \{y,z\}\subset Y^\kappa$ such that $f(A)\subset y$ and $f(x)=z$. By the definition of the free $E$-algebra
$F_\K(B_d)$, there is a unique $E$-homomorphism $\bar f:F_\K(B_d)\to Y^\kappa$ such that $\bar f\circ\delta_{B_d}=f$. Then $$z=f(x)=\bar f\circ\delta_{B_d}(x)=\bar f(a)\in \langle f(A)\rangle_E\subset\langle \{y\}\rangle_E,$$ which contradicts the choice of $z$.
\end{proof}

By the \index{topological space!discrete modification of}{\em discrete modification} $X_d$ of a topologized $E$-algebra $X$ we understand this $E$-algebra endowed with the discrete topology.
We shall say that a class $\K$ of topologized $E$-algebras is \index{variety!$d$-stable}{\em $d$-stable} if for any topologized $E$-algebra $X\in\K$ the class $\K$ contains the discrete modification $X_d$ of $X$.

\begin{proposition}\label{p:FKX-disc} If a variety $\K$ of topologized $E$-algebras is $d$-stable, then for every discrete topological space $X$ its free topologized $E$-algebra $F_\K(X)$ is discrete and every element $a\in F_\K(X)$  belongs to the $E$-hull $\langle \delta_X(A)\rangle_E$ of any non-empty subset $A\subset X$ containing $\supp(a)$.
\end{proposition}

\begin{proof} Let $F^d_\K:\Top\to\Top$ be the functor assigning to each topological space $X$ the discrete modification $F_\K(X)_d$ of its free topologized $E$-algebra.
To any continuous map $f:X\to Y$ between topological spaces the functor $F_\K^d$ assigns the map $F_\K f$ considered as the (necessarily continuous) map between the discrete modifications of $F_\K(X)$ and $F_\K(Y)$.

 Since the variety $\K$ is $d$-stable, for every topological space $X$ the discrete $E$-algebra $F^d_\K(X)$ belongs to $\K$. Observe that for a discrete topological space $X$ the canonical map $\delta_X:X\to F^d_\K(X)$ remains continuous. By the definition of a free topologized $E$-algebra, there exists a unique continuous $E$-homomorphism $h:F_\K(X)\to F^d_\K(X)$ such that $h\circ \delta_X=\delta_X$. The uniqueness of $h$ implies that $h$ is equal to the  identity map $F_\K(X)\to F_\K(X)_d$ whose continuity means that the space $F_\K(X)$ is discrete.

It is clear that the category $\Top_d$ of discrete topological spaces and their (necessarily continuous) maps is a full hereditary subcategory of $\Top$, containing all finite discrete spaces. Observe that the restriction $F_\K|\Top_d$ is monomorphic. This follows from the fact that each monomorphism  $f:X\to Y$ in the category $\Top^d$ is left-invertible in the sense that there exists a map $r:Y\to X$ such that $r\circ f=\id_X$. Applying to this equality the functor $F_\K$, we obtain $F_\K r\circ F_\K f=\id_{F_\K X}$, which implies that the $E$-homomorphism $F_\K f:F_\K X\to F_\K Y$ is injective. So, we can apply Proposition~\ref{p:supp} and conclude that for every discrete topological space $X$, every element $a\in F_\K(X)$ belongs to $\langle \delta_X(A)\rangle_E$ for every non-empty subset $A\subset X$ containing $\supp(a)$.
\end{proof}

\section{Free topologized $E$-algebras in $\HM$-stable varieties}

Observe that for every topologized $E$-algebra $X$ the Hartman-Mycielski space $\HM(X)$ is a topologized $E$-algebra (endowed with the algebraic structure  inherited from the power $X^{[0,1)}$ of $X$). The following fact is proved in \cite{BH}.

\begin{lemma}\label{l:BanHryn} If $X$ is a topological $E$-algebra, then so is the topologized $E$-algebra $\HM(X)$.
\end{lemma}

We say that the variety $\K$ is \index{variety!$\HM$-stable}{\em $\HM$-stable} if for any topologized $E$-algebra $X$ the topologized $E$-algebra $\HM(X)$ belongs to $\K$. In this case the functor $F_\K$ has many nice properties.

\begin{theorem}\label{t:FKX1} If a variety $\K$ of topologized $E$-algebras is $\HM$-stable, then the functor $F_\K:\Top\to\Top$ has the following properties:
\begin{enumerate}
\item[\textup{1)}] $F_\K$ has finite supports.
\item[\textup{2)}] $F_\K$ is $\HM$-commuting.
\item[\textup{3)}] $F_\K$ preserves the injectivity (and bijectivity) of continuous functions between functionally Hausdorff spaces.
\item[\textup{4)}] $F_\K$ preserves closed embeddings of metrizable compacta (more generally, closed embeddings of stratifiable spaces).
\item[\textup{5)}] For any functionally Hausdorff space $X$ every element $a\in F_\K(X)$ belongs to the $E$-hull $\langle\delta_X(A)\rangle_E$ of any non-empty subset $A\subset X$ containing $\supp(a)$.
\item[\textup{6)}] If the space $F(\II)$ is (functionally) Hausdorff, then for any functionally Hausdorff space $X$ the functor-space $FX$ is (functionally) Hausdorff.
\item[\textup{7)}] If the space $F_\K(\II)$ is Tychonoff, then the functor $F_\K$ is $\II$-regular.
\item[\textup{8)}] If $\K$ contains a topologized $E$-algebra $A$ of cardinality $|A|>1$, then for every functionally Hausdorff (Tychonoff) space $X$ the map $\delta_X:X\to F_\K(X)$ is injective (topological embedding).
\item[\textup{9)}] The functor $F_\K|\Top_{3\frac12}$ is bounded if and only if  every compact subset of $F_\K(\IR_+)$ is contained in the set $F_\K(B;\IR_+)\subset F_\K(\IR_+)$ for some bounded subset $B$ of the half-line $\IR_+=[0,\infty)$.
\end{enumerate}
\end{theorem}

\begin{proof} 1. The first statement was proved in Proposition~\ref{p:TEA-fin-supp}.
\smallskip

2. To see that $F_\K$ is $\HM$-commuting, choose any topological space $X$ and  consider its free topologized $E$-algebra $F_\K(X)\in\K$. Since the variety $\K$ is $\HM$-closed, the topologized $E$-algebra $\HM(F_\K(X))$ belongs to the variety $\K$. Now the definition of the free topologized $E$-algebra $F_\K(X)$ guarantees that for the continuous map $\HM \delta_X:\HM(X)\to \HM(F_\K(X))$ there exists a unique continuous $R$-homomorphism $c_X:F_\K(\HM(X))\to \HM F_\K(X)$ such that $c_X\circ \delta_{\HM X}=\HM\delta_X$. To complete the proof we need to check that $c_X\circ F hm_X=hm_{FX}$. For this consider the following diagram.
$$
\xymatrix{
X\ar^{\delta_X}[rr]\ar_{hm_X}[dd]&&FX\ar^{F hm_X}[dd]\ar^{hm_{FX}}[dl]\\
&\HM(FX)\\
HM(X)\ar_{\delta_{\HM X}}[rr]\ar_{\HM\delta_X}[ru]&&F(\HM X)\ar^{c_X}[ul]
}
$$
The naturality of the transformations $hm:\Id\to \HM$ and $\delta:\Id\to F$ yield two equalities (or two commutative squares in the diagram):  $$hm_{FX}\circ\delta_X=\HM\delta_X\circ hm_X\mbox{ \ and \ }Fhm_X\circ\delta_X=\delta_{\HM X}\circ hm_X.$$ Composing these two equalities with the equality $\HM\delta_X=c_X\circ \delta_{\HM X}$ multiplied by $hm_X$ from the right, we obtain the equality
$$hm_FX\circ\delta_X=\HM\delta_X\circ hm_X=c_X\circ\delta_{\HM X}\circ hm_X=c_X\circ F hm_X\circ\delta_X.$$
Taking into account that $hm_{FX}$ and $c_X\circ F hm_X$ are two continuous $E$-homomorphisms from the free topologized $E$-algebra  $FX$ to the topologized $E$-algebra $\HM(FX)$ satisfying the equality $hm_X\circ \delta_X=(c_X\circ \delta_{\HM X})\circ\delta_X$, we obtain the equality $hm_X=c_X\circ\delta_{\HM X}$ by (the uniqueness in) the definition of a free topologized $E$-algebra.
\smallskip

3. If $f:X\to Y$ is an injective continuous map between functionally Hausdorff spaces, then by the statement (1,2) and Corollary~\ref{c:FHM->mono} the map $F_\K f:F_\K(X)\to F_\K(Y)$ is injective. If $f$ is surjective, then $F_\K f$ is surjective by Corollary~\ref{c:FKX-surj}. This means that the functor $F_\K$ preserves bijective maps between functionally Hausdorff spaces.

4. By the statement (2) and Corollary~\ref{c:Fs->emb}, the functor $F_\K$ preserves closed embeddings of metrizable compacta (more generally, closed embeddings of stratifiable spaces).

5. The fifth statement follows from the third statement and Proposition~\ref{p:FKX-disc}.

 6,7.  The statements (6) and (7) follow from the first and the second statements,  and Proposition \ref{p:FX-fH} and Corollary~\ref{c:F-Ireg}, respectively.

8. The eighth statement follows from Proposition~\ref{p:delta-inject} and the path-connectedness of the Hartman-Mycielski spaces $HM(X)$.

9. The characterization of bounded functors $F_\K$ follows from Proposition~\ref{p:F-bounded} and the statements (1) and (3).
 \end{proof}

 \section{Free topologized $E$-algebras in $d{+}\HM$-stable varieties}

In this section we assume that $\K$ is a variety of topologized $E$-algebras and $\K$ is $d$-stable and $\HM$-stable. Such varieties will be called \index{variety!$d{+}\HM$-stable}{\em $d{+}\HM$-stable}.

First observe the following implication of Theorem~\ref{t:FKX1}(3) and Proposition~\ref{p:FKX-disc}.

\begin{corollary}\label{c:algebr-free} If a variety $\K$ of topologized $E$-algebras is $d{+}\HM$-stable, then for every functionally Hausdorff space $X$ its free topologized $E$-algebra $F_\K(X)$ is {\em algebraically free} in the sense that for the identity map $f:X_d\to X$ from the discrete modification of $X$ the continuous $E$-homomorphism $F_\K f:F_\K(X_d)\to F_\K(X)$ is bijective.
\end{corollary}

Given a variety $\K$ of topologized $E$-algebras, denote by $\bar\K$ the subvariety of $\K$ consisting of topological $E$-algebras, which are Tychonoff spaces. If the variety $\K$ is $d$-stable (and $\HM$-stable), then so is the variety $\bar\K$.

For every topological space $X$ we can consider its free topologized $E$-algebra $(F_\K(X),\delta_X)$ and its free topological $E$-algebra $(F_{\bar\K}(X),\bar\delta_X)$ in the varieties $\K$ and $\bar\K$, respectively. Since $F_{\bar \K}(X)\in\bar \K\subset\K$, there exists a unique continuous $E$-homomorphism $\bar\i_X:F_\K(X)\to F_{\bar \K}(X)$ such that $\bar\delta_X=\bar\i_X\circ\delta_X$. In fact, the $E$-homomorphisms $\bar\i_X$ are components of the natural transformation $\bar\i:F_\K\to F_{\bar\K}$.

 \begin{theorem}\label{t:bar-i-bijective} If a variety $\K$ of topologized $E$-algebras is $d{+}\HM$-stable, then for any functionally Hausdorff space $X$ the $E$-homomorphism $\bar\i_X:F_\K(X)\to F_{\bar \K}(X)$ is bijective and the space $F_\K(X)$ is functionally Hausdorff. Consequently, the functor $F_\K$ is $\II$-regular.
\end{theorem}

\begin{proof} Let $i^d_{X,X}:X_d\to X$ denote the identity map. By Corollary~\ref{c:algebr-free}, the maps $F_\K i^d_{X,X}:F_\K(X_d)\to F_\K(X)$ and $F_{\bar\K} i^d_{X,X}:F_{\bar\K}(X_d)\to F_{\bar\K}(X)$ are bijective. The naturality of the transformation $\bar \i:F_\K\to F_{\bar\K}$ implies the commutativity of the diagram:
$$
\xymatrix{
F_\K(X_d)\ar^{F_\K i^d_{X,X}}[rr]\ar_{\bar\i_{X_d}}[d]&&F_\K(X)\ar^{\bar\i_{X}}[d]\\
F_{\bar\K}(X_d)\ar_{F_{\bar \K}i^d_{X,X}}[rr]&&F_{\bar\K}(X).
}
$$
By Proposition~\ref{p:FKX-disc}, the topologized $E$-algebra $F_\K(X_d)$ is discrete and hence it belongs to the variety $\bar \K$. Now the definition of the free $E$-algebra $F_{\bar \K}(X_d)$ yields a unique continuous $E$-homomorphism $h:F_{\bar \K}(X_d)\to F_\K(X_d)$ such that $h\circ \bar\delta_X=\delta_X$. The uniqueness of $h$ guarantees that $h$ is inverse to the $E$-homomorphism $\bar\i_{X_d}$. Then the $E$-homomorphism $\bar \i_{X_d}$ is bijective and so is the map $\bar\i_X$.

Taking into account that the space $F_{\bar\K}(X)\in\bar\K$ is Tychonoff and the map $\bar\i_X:F_\K(X)\to F_{\bar \K}(X)$ is continuous and injective, we conclude that the space $F_{\K}(X)$ is functionally Hausdorff.
\smallskip

Since the unit interval $\II$ is Tychonoff, the continuous $E$-homomorphism $\bar \i_{\II}:F_{\K}(\II)\to F_{\bar\K}(\II)$ is  bijective. By Theorem~\ref{t:FKX1}(4), the functor $F_{\bar K}$ preserves closed embeddings of metrizable compacta. Consequently, for every closed subset $X\subset\II$ the set $F_{\bar \K}(X;\II)$ is closed in the Tychonoff space $F_{\bar \K}(\II)$ and hence is $\IR$-closed in $F_{\bar\K}(\II)$. Then the preimage $\bar\i_X^{-1}(F_{\bar \K}(X;\II))=F_\K(X;\II)$ of the set $F_{\bar\K}(X;\II)$ is $\IR$-closed in the functor-space $F_\K(\II)$, which means that the functor $F_\K$ is $\II$-regular.
\end{proof}

\begin{corollary}\label{c:delta-closed}  If a variety $\K$ of topologized $E$-algebras is $d{+}\HM$-stable and contains a $E$-algebra of cardinality $>1$, then for any functionally Hausdorff (Tychonoff) space $X$ the map $\delta_X:X\to F_\K(X)$ is injective (topological embedding) and the set $\delta_X(X)$ is closed in $F_\K(X)$.
\end{corollary}

\begin{proof} By Theorem~\ref{t:FKX1}(8), the map  $\delta_X:X\to F_\K(X)$ is injective (topological embedding). It remains to prove that the image $\delta_X(X)$ is closed in $F_\K(X)$. Consider the canonical (injective) map $f:X\to \beta X$ of $X$ into its Stone-\v Cech compactification $\beta X$. By Theorem~\ref{t:FKX1}(3), the map $F_\K f: F_\K(X)\to F_\K(\beta X)$ is injective. By Theorem~\ref{t:bar-i-bijective}, the space $F_\K(\beta X)$ is Hausdorff and hence $\delta_{\beta X}(\beta X)$ is closed in $F_\K(\beta X)$. The continuity of the map $F_\K f$ implies that the set $C=\{a\in F_\K(X):F_\K f(a)\in\delta_{\beta X}(\beta X)\}$ is closed in $F_\K(X)$. We claim that $C=\delta_X(X)$. The naturality of the transformation $\delta:\Id\to F_\K$ guarantees that $\delta_X(X)\subset C$. To prove the reverse inclusion, fix any element $c\in C$ and consider its image $\bar c=F_\K f(c)\in\delta_{\beta X}(\beta X)$. Find a point $\bar x\in\beta X$ such that $\bar c=\delta_{\beta X}(\bar x)$. By Proposition~\ref{p:supp-delta}, $\supp(\bar c)=\{\bar x\}$ and by Proposition~\ref{p:supp=image}, $\{\bar x\}=\supp(\bar c)=f(\supp(c))\subset f(X)$.
So, $\bar x=f(x)$ for some $x\in X$. The naturality of the transformation $\delta:\Id\to F_\K$ guarantees that $F_\K f\circ \delta_X=\delta_{\beta X}\circ f$. Consequently, $F_\K f(\delta_X(x))=\delta_{\beta X}(f(x))=\delta_{\beta X}(\bar x)=\bar c=F_\K f(c)$. Now the injectivity of the map $F_\K f$ guarantees that $c=\delta_X(x)\in\delta_X(X)$.
\end{proof}

 \section{Free topological $E$-algebras in $k_\w$-stable varieties}

We recall that a topological space $X$ is a \index{topological space!$k_\w$-space}{\em $k_\w$-space} if $X=\bigcup_{n\in\w}X_n$ for an increasing sequence $(X_n)_{n\in\w}$ of compact subsets such that a subset $U\subset X$ is open in $X$ if and only if for every $n\in\w$ the intersection $U\cap X_n$ is relatively open in the compact space $X_n$. In this case the sequence $(X_n)_{n\in\w}$ is called a \index{$k_\w$-sequence}{\em $k_\w$-sequence} for the $k_\w$-space $X$.

\begin{definition} A variety $\K$ of topologized $E$-algebras is \index{variety!$k_\w$-stable}{\em $k_\w$-stable} if $\K$ contains any topological $E$-algebra $X$ which is a $k_\w$-space and admits a continuous bijective $E$-homomorphism $h:X\to Y$ onto a Hausdorff topological $E$-algebra $Y\in\K$.
 \end{definition}

We recall that for a subset $B$ of a topologized $E$-algebra $X$ its $E$-hull $\langle B\rangle_E$ in $X$ equals the union $\bigcup_{m\in\w}E^m[B]$ where $E^0[ B]=B$ and $$E^{m+1}[B]=E^m[B]\cup E[E^m[B]]\mbox{ \ for \ }m\in\w,$$
and $E[B]=\bigcup_{n\in\w}E_n[B^n]$. Here $E=\oplus_{n\in\w}E_n$ and $E_n=\alpha^{-1}(n)$ are the sets of symbols of $n$-ary operations in the signature $E$.

If $X$ is a topological $E$-algebra and the spaces $B$ and $E$ are compact, then by induction it can be shown that every set $E^m[B]$, $m\in\w$, is compact. Let us observe that any $E$-algebra $X$ remains an $E'$-algebra for any subsignature $E'\subset E$. So, for any subset $B\subset X$ we can consider its $E'$-hull $\langle B\rangle_{E'}$ which a subset of the $E$-hull $\langle B\rangle_E\subset X$.

The following known theorem describes the structure of compact subsets in free topological $E$-algebras of $k_\w$-signature $E$ over $k_\w$-spaces (cf. \cite[4.3]{Cho93}, \cite[7.5.2]{AT}).

 \begin{theorem}\label{t:kw} Assume that the signature $E$ is a (cosmic) $k_\w$-space and $\K$ is a $k_\w$-stable variety of topological $E$-algebras. If the free topological $E$-algebra $F_\K(X)$ of a (cosmic) $k_\w$-space $X$ is Hausdorff, then $F_\K(X)$ is a (cosmic) $k_\w$-space. Moreover, for every bounded subset $K\subset F_\K(X)$ there are compact sets $A\subset E$ and $B\subset X$ such that $K\subset A^m[\delta_X(B)]$ for some $m\in\w$.
 \end{theorem}

  \begin{proof} Let $(B_n)_{n\in\w}$ and $(A_n)_{n\in\w}$ be $k_\w$-sequences for the $k_\w$-spaces $X$ and $E$, respectively.
  The continuity of the multiplication in the topological $E$-algebra $F_\K(X)$ implies that $\big(A_n^n[\delta_X(B_n)]\big)_{n\in\w}$ is an increasing sequence of compact sets in $F_\K(X)$ such that $F_\K(X)=\bigcup_{n\in\w}A_n^n[\delta_X(B_n)]$.
Denote by $F_\K^\tau(X)$ the $E$-algebra $F_\K(X)$  endowed with the $k_\w$-topology $\tau$ determined by the $k_\w$-sequence $\big(A_n^n[\delta_X(B_n)]\big)_{n\in\w}$.
The topology $\tau$ consists of sets $U\subset F_\K(X)$ such that for every $n\in\w$ the intersection $U\cap A_n^n[\delta_X(B_n)]$ is open in the compact Hausdorff space $A_n^n[\delta_X(B_n)]$.

We claim that $F^\tau_\K(X)$ is a topological $E$-algebra. This will follow as soon as we check that for every $n\in \w$ the multiplication map
 $\star_n:E_n\times F_\K^\tau(X)^n\to F_\K^\tau(X)$ is continuous. Since $E_n\times (F^\tau_\K(X))^n$ is a $k_\w$-space (see Lemma~\ref{l:kw-product}), it suffices to prove that for every compact subset $K\subset E_n\times F_\K^\tau(X)^n$ the restriction $\star_n|K$ is continuous.
Taking into account that $(A_m)_{m\in\w}$ and $(A_m^m[\delta_X(B_m)])_{m\in\w}$ are $k_\w$-sequences in the $k_\w$-spaces $E$ and $F_\K^\tau(X)$, we can find a number $m\in\w$ such that $K\subset (E_n\cap A_m)\times (A_m^m[\delta_X(B_m)])^n$. Observing that $\star_n(K)\subset A_m[A_m^m[\delta(B_m)]]\subset A_m^{m+1}[\delta_X(B_m)]\subset A_{m+1}^{m+1}[\delta_X( B_{m+1})]\subset F_\K(X)$ and using the continuity of the multiplication in the topological $E$-algebra $F_\K(X)$, we conclude that the map $\star_n|K:K\to A_{m+1}^{m+1}[\delta_X(B_{m+1})]\subset F_\K^\tau(X)$ is continuous. So, $F^\tau_\K(X)$ is a Hausdorff topological $E$-algebra and the identity map $\id:F^\tau_\K(X)\to F_\K(X)$ is a continuous $E$-homomorphism of topological $E$-algebras.  For every $m\in\w$ the inclusion $\delta_X(B_m)\subset A_m^m[\delta_X(B_m)]\subset F_\K^\tau(X)$ ensures that the restriction $\delta|B_m\to F^\tau_\K(X)$ is continuous. This implies the continuity of the map $\delta_X:X\to F^\tau_\K(X)$ on the $k_\w$-space $X$.

 The $k_\w$-stability of the variety $\K$ guarantees that $F^\tau_\K(X)\in\K$ and then by the definition of the free topological $E$-algebra, there exists a unique continuous $E$-homomorphism $i: F_\K(X)\to F^\tau_\K(X)$ such that $i\circ \delta_X=\delta_X$. The unique $E$-homomorphism with this property is the identity map of $F_\K(X)$. So, the identity map $i$ is a homeomorphism and $F_\K(X)$ is a $k_\w$-space. By Lemma~\ref{l:bound-kw}, each bounded subset $K\subset F_\K(X)$ is contained in some set $A_n^n[\delta_X(B_n)]$, $n\in\w$.

If the spaces $X$ and $E$ are cosmic, then by induction we can prove that the compact sets $A_n^n[\delta_X(B_n)]$, $n\in\w$, are cosmic and so is the space $F_\K(X)$.
\end{proof}

\begin{lemma}\label{l:Ak} Let $X$ be a topological space, $A\subset E$ be a compact subset of the topological signature $E$ and $B\subset X$. Then there exists $m\in\w$ such that for every $k\in\w$ and $a\in A^k[\delta_X(B)]$ there is a subset $S\subset B$ of cardinality $|B|\le m^k$ such that $a\in A^k[\delta_X(S)]$.
\end{lemma}

\begin{proof} Since $E=\oplus_{n\in\w}E_n$, the compact set $A\subset E$ is contained in the finite sup $\oplus_{n=0}^mE_n$ for some $m\in\w$. By induction on $k\in\w$ we shall prove that for every $k\in\w$ and $a\in A^k[\delta_X(B)]$ there is a subset $S\subset B$ of cardinality $|S|\le m^k$ such that $a\in A^k[\delta_X(S)]$.

For $k=0$ the inclusion $a\in A^0[\delta_X(B)]=\delta_X(B)$ implies that $a=\delta_X(x)$ for some $x\in B$ and hence for the singleton $S=\{x\}$ we have $a\in A^0[\delta_X(S)]=\delta_X(S)$ and $|S|=1=m^0$. Assume that for some $k\in\w$ we have proved that each element $a\in A^k[\delta_X(B)]$ is contained in $A^k[\delta_X(S)]$ for some subset $S\subset B$ of cardinality $|S|\le m^k$.

Take any element $a\in A^{k+1}[\delta_X(B)]=A^k[\delta_X(B)]\cup A[A^k[\delta_X(B)]]$.
If $a\in A^k[\delta_X(B)]$, then by the inductive assumption, $a\in A^k[\delta_X(S)]$ for some set $S\subset B$ of cardinality $|S|\le m^k\le m^{k+1}$ and we are done. If $a\in A[A^k[\delta_X(B)]$, then $a=e_p(a_1,\dots,a_p)$ for some $p\le m$,
$e_p\in A\cap E_p$, and points $a_1,\dots,a_p\in A^k[\delta_X(B)]$. By the inductive assumption, there are sets $S_1,\dots,S_p\subset B$ of cardinality $\le m^k$ such that $a_i\in A^k[\delta_X(S_i)]$ for every $i\le p$. Then the set $S=\bigcup_{i=1}^pS_i$ has cardinality $|S|\le\sum_{i=1}^p|S_i|\le p\cdot m^k\le m^{k+1}$ and $a=e_p(a_1,\dots,a_p)\in A[A^k[\delta_X(S)]]=A^{k+1}[\delta_X(S)]$, which completes the inductive step.
\end{proof}

\section{Free topologized $E$-algebras in $d{+}\HM{+}k_\w$-stable varieties}

 We shall say that a variety $\K$ of topologized $E$-algebras is \index{variety!$d{+}\HM{+}k_\w$-stable}{\em $d{+}\HM{+} k_\w$-stable} if it is $d$-stable, $\HM$-stable and $k_\w$-stable.

\begin{theorem}\label{t:d+HM+kw} If the signature $E$ is a $k_\w$-space and $\K$ is a $d{+}\HM{+}k_\w$-stable variety of topologized $E$-algebras, then for every Tychonoff space $X$ every bounded subset $K\subset F_\K(X)$ the set $\supp(K)=\bigcup_{a\in K}\supp(a)$ is bounded in $X$ and $\{\supp(a):a\in K\}\subset  [X]^{\le n}$ for some $n\in\IN$. Consequently, the functor $F_\K|\Top_{3\frac12}$ is strongly bounded.
\end{theorem}

\begin{proof} If the space $X$ is empty, then $\supp(K)=\emptyset$ and $\{\supp(a):a\in K\}=\{\emptyset\}\subset [\emptyset]^{\le 0}$.
So, we can assume that the Tychonoff space $X$ is not empty.
Let $\bar \K$ be the subvariety of $\K$ consisting of Tychonoff topological $E$-algebras. The definitions and Lemma~\ref{l:BanHryn} imply that the variety $\bar\K$ is $d{+}\HM{+}k_\w$-stable.

By Theorem~\ref{t:bar-i-bijective}, the map $\bar\i_X:F_\K(X)\to F_{\bar \K}(X)$ is bijective.  By Theorem~\ref{t:FKX1}(1,3), the functor $F_{\bar\K}|\Top_{3\frac12}$ is monomorphic and has finite supports. So, the set $\supp(K)=\bigcup_{a\in K}\supp(a)$ is well-defined. We claim that this set is bounded in $X$. To derive a contradiction, assume that $\supp(K)$ is not bounded. Since each unbounded subset contains a countable unbounded set, there is a countable set $K_0\subset K$ such that the countable set $\supp(K_0)$ is unbounded in $X$. By Lemma~\ref{l:Runbound}, there is a continuous map $f:X\to\IR_+$ such that $f|\supp(K_0)$ is injective and $f(\supp(K_0))$ is unbounded in $\IR_+$. The map $f$ induces a continuous $E$-homomorphism $F_{\bar \K}f:F_{\bar \K}(X)\to F_{\bar \K}(\IR_+)$. The continuity of the map $F_{\bar\K}f\circ \bar\i_X:F_\K(X)\to F_{\bar\K}(\IR_+)$ implies the boundedness of the set $K'=F_{\bar\K}f\circ\bar i_X(K)$ in the space $F_{\bar \K}(\IR_+)$. Since the half-line $\IR_+$ is a $k_\w$-space, by Theorem~\ref{t:kw}, $K'\subset
\langle \delta_X([0,n])\rangle_E\subset F_{\bar \K}(\IR_+)$ for some $n\in\IN$. Since the set $f(\supp(K_0))$ is unbounded in $\IR_+$, there exists an element $a\in K_0$ such that $f(\supp(a))\not\subset[0,n]$.
Let $\bar a=\bar\i_X(a)\in F_{\bar \K}(X)$. Since $\bar \i_X$ is an (algebraic) $E$-isomorphism of the $E$-algebras $F_\K(X)$ and $F_{\bar\K}(X)$, $\supp(\bar a)=\supp(a)$. The injectivity of the map $f|\supp(a)=f|\supp(\bar a)$ and Proposition~\ref{p:supp=image} guarantee that the element $a'=F_{\bar\K}f(\bar a)$ has support  $\supp(a')=f(\supp(\bar a))=f(\supp(a))\not\subset [0,n]$ and hence $a'\notin \langle \delta_X([0,n])\rangle_E$, which contradicts the inclusion $a'\in K'\subset\langle\delta_X([0,n])\rangle_E\subset F_{\bar\K}(\IR_+)$.
This contradiction completes the proof of the boundedness of the set $\supp(K)$ in $X$.

If $\supp(K)$ is not empty, then put $B=\supp(K)$. If $\supp(K)$ is empty, let $B$ be any singleton in $X$. In any case $B$ is a non-empty bounded subset of $X$ containing $\supp(K)$.

Next we prove that the set $\{\supp(a):a\in K\}$ is contained in $[X]^{\le n}$ for some $n\in\w$. To derive a contradiction, assume that for every $n\in\w$ there is an element $a_n\in K$ with $|\supp(a_n)|>n$. Consider the countable set $S=\bigcup_{n\in\w}\supp(a_n)$. Using Lemma~\ref{l:fH-inj}, find a continuous map $f:X\to\II$ such that $f|S$ is injective.

 By Theorem~\ref{t:bar-i-bijective}, for the Tychonoff space $X$ the map $\bar\i_\II:F_\K(\II)\to F_{\bar \K}(\II)$ is bijective and by Theorem~\ref{t:kw}, $F_{\bar\K}(\II)$ is a $k_\w$-space such that the bounded subset $K'=\bar\i_\II\circ Ff(K)$ is contained in $A^k[\delta_\II(\II)]$ for some compact subset $A\subset E$ and some $k\in\w$.

 Now consider the element $a_n\in K$ and its images $\bar a_n=F_\K f(a_n)\in F_\K(\II)$ and $a_n'=\bar \i_\II(\bar a_n)$. The injectivity of the map $f|\supp(a_n)$ and Proposition~\ref{p:supp=image} imply that $\supp(\bar a_n)=f(\supp(a_n))$. Since $\bar\i_\II$ is an $E$-isomorphism of the $E$-algebras $F_\K(\II)$ and $F_{\bar\K}(\II)$, $\supp(a'_n)=\supp(\bar a_n)=f(\supp(a_n))$ and hence $|\supp(a'_n)|=|\supp(a_n)|$.
By Lemma~\ref{l:Ak}, for the set $A^k[\delta_\II(\II)]$ there exists a number $m\in\w$ such that for every $a\in A^k[\delta_\II(\II)]$ there is a subset $S_a\subset \II$ of cardinality $|S_a|\le m^k$ such that $a\in A^k[\delta_\II(S_a)]$, which implies $|\supp(a)|\le |S_a|\le m^k$. In particular,
for every $n\in\w$ the inclusion $a_n'\in K'\subset A^k[\delta_\II(\II)]$ implies
$|\supp(a_n)|=|\supp(a_n')|\le m^k$, which contradicts the choice of the elements $a_n$ for $n>m^k$. This contradiction completes the proof of the inclusion $\{\supp(a):a\in K\}\subset [X]^{\le n}$ for some $n\in\w$.
\end{proof}

We recall that a topological space $X$ is  \index{topological space!$\mu_s$-complete}{\em $\mu_s$-complete} if each bounded  closed subset $B$ of $X$ is sequentially compact. This means that each sequence in $B$ contains a convergent subsequence. A $\mu$-complete space $X$ is $\mu_s$-complete if each compact subset of $X$ is sequentially compact.

\begin{theorem}  If the signature $E$ is a $k_\w$-space and $\K$ is a $d{+}\HM{+}k_\w$-stable variety of topologized $E$-algebras, then for every $\mu_s$-complete Tychonoff space $X$ every bounded subset $K\subset F_\K(X)$ is contained in $A^n[\delta_X(B)]$ for some compact set $A\subset E$, bounded set $B\subset X$, and some $n\in\IN$.
\end{theorem}

\begin{proof} By Theorem~\ref{t:d+HM+kw}, for the bounded subset $K\subset F_\K(X)$ the set $\supp(K)=\bigcup_{a\in K}\supp(a)$ is bounded and $\{\supp(a):a\in K\}\subset [X]^{\le m}$ for some $m\in\w$. So, we can find a non-empty closed bounded subset $B\subset X$ containing $\supp(K)$.

Let $(A_n)_{n\in\w}$ be a $k_\w$-sequence for the $k_\w$-space $E$. We claim that $K\subset A_n^n[\delta_X(B)]$ for some $n\in\w$. To derive a contradiction, assume that for every $n\in\w$ there is an element $a_n\in K\setminus A_n^n[\delta_X(B)]$. Observe that $\{\supp(a_n)\}_{n\in\w}\subset [B]^{\le m}$. Since the space $X$ is $\mu_s$-complete, the sequence $(a_n)_{n\in\w}$ contains a subsequence $(a_{n_k})_{k\in\w}$ such that the sequence of supports $(\supp(a_{n_k})\big)_{k\in\w}$ converges to some set $S_\infty\in [B]^{\le m}$ in the Vietoris topology on $[X]^{\le m}$. Replacing the sequence $(a_n)_{n\in\w}$ by the subsequence $(a_{n_k})_{k\in\w}$ we can assume that the sequence $\big(\supp(a_n)\big)_{n\in\w}$ converges to $S_\infty$ and hence the countable set $S_\infty\cup\bigcup_{n\in\w}\supp(a_n)$ is compact.
Choose any non-empty compact countable set $C\subset X$ containing the set $S_\infty\cup\bigcup_{n\in\w}\supp(a_n)$. Then for every $n\in \w$ we get  $\supp(a_n)\subset C$. Applying Theorem~\ref{t:supp}, we can find an element $c_n\in F_\K(C)$ such that $F_{\K}i_{C,X}(c_n)=a_n$.

Let $i_{X,\beta X}:X\to\beta X$ be the embedding of $X$ into its Stone \v Cech compactification. Observe that the composition $i_{X,\beta X}\circ i_{C,X}$ coincides with the embedding $i_{C,\beta X}:C\to\beta X$.

  Let $\bar\K$ be the subvariety of $\K$ consisting of Tychonoff topological $E$-algebras. By Theorem~\ref{t:kw}, the space $F_{\bar \K}(\beta X)$ is a $k_\w$-space. Being Lindel\"of, the $k_\w$-space $F_{\bar \K}(\beta X)$ is $\mu$-complete. By Theorem~\ref{t:bar-i-bijective}, the components $\bar\i_{C}:F_\K(C)\to F_{\bar\K}(C)$ and $\bar\i_{\beta X}:F_\K(\beta X)\to F_{\bar\K}(\beta X)$ of the natural transformation $\bar\i:F_\K\to F_{\bar\K}$ are bijective continuous $E$-homomorphisms.

The continuity of the map $j=\bar\i_{\beta X}\circ F_{\bar \K} i_{X,\beta X}:F_{\K}(X)\to F_{\bar\K}(\beta X)$ implies that the set $K_\beta=j(K)$ is bounded in the $k_\w$-space $F_{\bar \K}(\beta X)$ and hence has compact closure $\bar K_\beta$ in $F_{\bar \K}(\beta X)$.

The naturality of $\bar\i$ ensures that
for every $n\in\w$ we get
$$
\begin{aligned}
F_{\bar\K}i_{C,\beta X}\circ \bar\i_C(c_n)&=\bar\i_{\beta X}\circ F_\K i_{C,\beta X}(c_n)=\bar \i_{\beta X}\circ F_\K i_{X,\beta X}\circ F_\K i_{C,X}(c_n)=\\
&=\bar \i_{\beta X}\circ
F_\K i_{X,\beta X}(a_n)\in \bar\i_{\beta X}\circ F_\K i_{X,\beta X}(K)=j(K)\subset\bar K_\beta.
\end{aligned}
$$

By Theorem~\ref{t:FKX1}, the map $F_{\bar\K} i_{C,\beta X}:F_{\bar\K}(C)\to F_{\bar \K}(\beta X)$ is a closed topological embedding. Then preimage $K_C=(F_{\bar\K} i_{C,\beta X})^{-1}(\bar K_\beta)$ is compact in $F_{\bar \K}(C)$ and contains the set $\{\bar i_C(c_n)\}_{n\in\w}$. By Theorem~\ref{t:kw}, $K_C\subset A_m^m[\bar\delta_C(C)]$ for some $m\in\w$. Since the $E$-homomorphism $\bar \i_C:F_\K(C)\to F_{\bar \K}(C)$ is bijective, $\{c_n\}_{n\in\w}\subset A_m^m[\delta_C(C)]$ and then $$\{a_n\}_{n\in\w}=\{F_\K i_{C,X}(c_n)\}_{n\in\w}\subset F_\K i_{C,X}(A_m^m[\delta_C(C)])\subset A_m^m[\delta_X(C)]\subset A_m^m[\delta_X(B)],$$
which contradicts the choice of the points $a_n$ for $n\ge m$.

This contradiction completes the proof of the inclusion $K\subset A_n^n[\delta_X(B)]$ for some $n\in\w$.
\end{proof}

 \section{Paratopological $(E,C)$-algebras}

In this section we introduce and study para\-topological $(E,C)$-algebras. Paratopological $(E,C)$-algebras generalize the well-known notion of a paratopological group (i.e, a group endowed with a topology making the multiplication continuous but the inversion not necessarily).

Let $E$ be a topological signature and $C\subset E$ be a subsignature of $E$.

A topologized $E$-algebra $X$ is called a \index{$E$-algebra!paratopological}\index{paratopological $(E,C)$-algebra}{\em paratopological $(E,C)$-algebra} if $X$ is a topological $C$-algebra. This means that the multiplication map $\star$ of $X$ restricted to $\oplus_{n\in\w}C_n\times X^n$ is continuous.

Any paratopological $(E,C)$-algebra $X$ has two hull operators. Given a subset $A\subset X$ we can consider its $E$-hull $\langle A\rangle_E$ and its $C$-hull $\langle A\rangle_C$ (which is contained in $\langle A\rangle_E$).

Let $\K$ be a variety of topologized $E$-algebras. We say that $\K$ is \index{variety!$k_\w^C$-superstable}{\em $k_\w^C$-superstable} if $\K$ contains any paratopological $(E,C)$-algebra $X$ which is a $k_\w$-space and  admits a continuous bijective $E$-homomorphism $f:X\to Y$ to some Hausdorff paratopological $(E,C)$-algebra $Y$ in the class $\K$. Since topological $E$-algebras are paratopological $(E,C)$-algebras, any $k_\w$-superstable variety of topologized $E$-algebras is $k_\w$-stable.

 In this section we shall be interested in describing compact subset in the free paratopological $(E,C)$-algebras $F_\K(X)$ for a given $k_\w^C$-superstable variety $\K$ of topologized $E$-algebras.

The following lemma can be considered as the paratopological counterpart of Theorem~\ref{t:kw}.

 \begin{lemma}\label{l:para-kw} Assume that the signature $E$ is countable and its subsignature $C\subset E$ is a $k_\w$-space. Assume that $\K$ is a $k_\w^C$-superstable variety of paratopological $(C,E)$-algebras. If the free topologized   $E$-algebra $F_\K(X)$ of a countable $k_\w$-space $X$ is Hausdorff, then $F_\K(X)$ is a $k_\w$-space. Moreover, every bounded subset $K\subset F_\K(X)$ is contained in $A^n[D\cup\delta_X(B)]$ for some compact sets $A\subset C$ and $B\subset X$, some finite set $D\subset F_\K(X)$ and some $n\in\w$.
 \end{lemma}

\begin{proof} The countability of the signature $E$ and the space $X$ implies the countability of the free topologized $E$-algebra $F_\K(X)$. Write the countable set $F_\K(X)$ as the union $\bigcup_{m\in\w}D_m$ of an increasing sequence $(D_m)_{m\in\w}$ of finite subsets. Let $(A_m)_{m\in\w}$ and $(B_m)_{m\in\w}$ be $k_\w$-sequences for the $k_\w$-spaces $C$ and $X$, respectively.

Since the variety $\K$ consists of paratopological $(E,C)$-algebras, $F_\K(X)$ is a topological $C$-algebra. The continuity of the multiplication in the topological $C$-algebra $F_\K(X)$ implies that $\big(A_n^n[D_n\cup \delta_X(B_n)]\big)_{n\in\w}$ is an increasing sequence of compact sets in $F_\K(X)$ such that $F_\K(X)=\bigcup_{n\in\w}A_n^n[D_n\cup \delta_X(B_n)]$.
Denote by $F_\K^\tau(X)$ the $E$-algebra $F_\K(X)$  endowed with the $k_\w$-topology $\tau$ determined by the $k_\w$-sequence $\big(A_n^n[D_n\cup \delta_X(B_n)]\big)_{n\in\w}$. We claim that $F^\tau_\K(X)$ is a paratopological $(E,C)$-algebra. This will follow as soon as we check that for every $n\in \w$ the multiplication map
 $\star_n:C_n\times F_\K^\tau(X)^n\to F_\K^\tau(X)$ is continuous. Since $C_n\times (F^\tau_\K(X))^n$ is a $k_\w$-space, it suffice to prove that for every compact subset $K\subset C_n\times F_\K^\tau(X)^n$ the restriction $\star_n|K$ is continuous.
Taking into account that $(A_m)_{m\in\w}$ and $(\langle\delta_X(B_m)\rangle_{A_m}^m)_{m\in\w}$ are $k_\w$-sequences in the $k_\w$-spaces $E$ and $F_\K^\tau(X)$, we can find a number $m\in\w$ such that $K\subset (C_n\cap A_m)\times A_m^m[D_m\cup \delta_X(B_m)]$. Observing that $\star_n(K)\subset A_m[A_m^m[D_m\cup\delta_X(B_m)]]\subset A_m^{m+1}[D_{m}\cup \delta_X(B_m)]\subset A_{m+1}^{m+1}[D_{m+1}\cup \delta_X( B_{m+1})]\subset F_\K(X)$ and using the continuity of the multiplication in the topological $C$-algebra $F_\K(X)$, we conclude that the map $\star_n|K:K\to A_{m}^{m+1}[D_{m}\cup\delta_X(B_m)]\subset F_\K^\tau(X)$ is continuous. So, $F^\tau_\K(X)$ is a Hausdorff paratopological topological $(E,C)$-algebra and the identity map $\id:F^\tau_\K(X)\to F_\K(X)$ is a continuous $E$-homomorphism of paratopological $(E,C)$-algebras. The $k_\w^C$-superstability of the variety $\K$ guarantees that $F^\tau_\K(X)\in\K$.

Next, we check that the map $\delta_X:X\to F_\K^\tau(X)$ is continuous. Since $(B_n)_{n\in\w}$ is a $k_\w$-sequence for the $k_\w$-space $X$, it suffices to check that for every $n\in\w$ the restriction $\delta_X|B_n$ is continuous. This follows from the inclusion $\delta_X(B_n)=A_n^0[\delta_X(B_n)]\subset A_n^n[\delta_X(B_n)]$ and the continuity of the map $\delta_X:X\to A_n^n[\delta_X(B_n)]\subset F_\K(X)$.

By the definition of the free topological $E$-algebra, there exists a unique continuous $E$-homomorphism $i: F_\K(X)\to F^\tau_\K(X)$ such that $i\circ \delta_X=\delta_X$. The unique $E$-homomorphism with this property is the identity map of $F_\K(X)$. So, the identity map $i$ is a homeomorphism and $F_\K(X)$ is a $k_\w$-space. By Lemma~\ref{l:bound-kw}, each bounded subset $K\subset F_\K(X)$ is contained in some set $A_n^n[D_n\cup \delta_X(B_n)]$, $n\in\w$.
\end{proof}

 We shall say that a variety $\K$ of topologized $E$-algebras is \index{variety!$d{+}\HM{+}k_\w^C$-superstable}{\em $d{+}\HM{+} k_\w^C$-superstable} if it is $d$-stable, $\HM$-stable and $k_\w^C$-superstable.

A subset $K$ of a topological space $X$ is called \index{topological space!countably compact}{\em countably compact} in $X$ if each countable subset $K\subset X$ has an accumulation point in $X$. It is easy to see that each countably compact set in $X$ is bounded in $X$.
 We recall that a topological space $X$ is called \index{topological space!$\mu_s$-complete}{\em $\mu_s$-complete} if each closed bounded subset of $X$ is sequentially compact.

\begin{theorem}\label{t:super} Assume that the signature $E$ is a countable $k_\w$-space and $C$ is a $k_\w$-subspace of $E$. Let $\K$ be a $d{+}\HM{+}k^C_\w$-superstable variety of topologized $E$-algebras. Then for every $\mu_s$-complete Tychonoff space $X$, every countably compact set $K\subset F_\K(X)$ is contained in $A^n[D\cup\delta_X(B)]$ for some $n\in\w$, compact subset $A\subset C$, finite subset $D\subset F_\K(X)$, and bounded set $B\subset X$.
\end{theorem}

\begin{proof}
If the space $X$ is empty, then $X$ is a countable $k_\w$-space and we can apply Lemma~\ref{l:para-kw}. So, we assume that the space $X$ is not empty and hence $X$ contains some point $x^+$.

Let $(A_n)_{n\in\w}$ be a $k_\w$-sequence for the $k_\w$-space $C$.
Since the signature $E$ is countable, it can be written as the union $E=\bigcup_{n\in\w}\Gamma_n$ of an increasing sequence $(\Gamma_n)_{n\in\w}$ of finite sets.

Observe that the countably compact subset $K$ of $F_\K(X)$ is bounded in $F_\K(X)$. By Theorem~\ref{t:d+HM+kw}, the set $\supp(K)=\bigcup_{a\in K}\supp(a)$ is bounded in $X$ and  $\{\supp(a):a\in K\}\subset [X]^{\le m}$ for some $m\in\IN$. Let $B$ be the closure of the non-empty bounded set $\{x^+\}\cup\supp(K)$.

We claim that $K\subset A_n^n[D\cup \delta_X(B)]$ for some $n\in\w$ and some finite set $D\subset F_\K(X)$. To derive a contradiction, assume that for every $n\in\w$ and every finite subset $D\subset F_\K(X)$ there is an element $a_{n,D}\in K\setminus A_n^n[D\cup \delta_X(B)]$.
Let $a_0=a_{0,\emptyset}$ and for every $n>0$ put $a_n=a_{n,D_n}$ where $D_n=\bigcup_{k<n}\Gamma^n_n[\delta_X(S^+_k)]$ where $S^+_k=\supp(a_k)\cup\{x^+\}$.

Observe that $\{\supp(a_n)\}_{n\in\w}\subset [B]^{\le {m}}$. Since the space $X$ is $\mu_s$-complete, the set $\w$ contains an infinite subset $\Omega$ such that  the sequence of supports $\big(\supp(a_{n})\big)_{n\in\Omega}$ converges to some set $L\in [B]^{\le m}$ in the Vietoris topology on $[X]^{\le m}$. Then the countable set $M=\{x^+\}\cup L\cup\bigcup_{n\in\Omega}\supp(a_{n})$ is compact and not empty.

By Theorem~\ref{t:FKX1}(4) and Proposition~\ref{p:FembTych}, the map $F_\K i_{M,X}:F_\K(M)\to F_\K(X)$ is a closed topological embedding, which allows us to identify $F_\K(M)$ with the closed subspace $F_\K i_{M,X}(F_\K(M))$ of $F_\K(X)$. It follows that the set $K_M=K\cap F_\K(M)\supset\{a_n\}_{n\in\Omega}$ is countably compact and hence bounded in $F_\K(M)$.

Let $\bar \K$ be the subvariety of $\K$ consisting of Tychonoff topological $E$-algebras and $\tilde\K$ be the subvariety of $\K$ consisting of Tychonoff paratopological $(E,C)$-algebras. Applying Lemma~\ref{l:BanHryn}, we can show that the variety $\bar\K$ is $d{+}\HM{+}k_\w$-stable and the variety $\tilde\K$ is $d{+}\HM{+}k^C_\w$-superstable. Let $(F_{\bar\K}(M),\bar\delta_M)$ and  $(F_{\tilde\K}(M),\tilde\delta_M)$ be the free topologized $E$-algebras over the compact countable space $M$ in the varieties $\bar\K$ and $\tilde \K$, respectively.
By Lemma~\ref{l:para-kw}, $F_{\tilde \K}(M)$ is a $k_\w$-space.
Since $F_{\tilde \K}(M)\in\tilde\K\subset\K$, there exists a unique continuous $E$-homomorphism $\tilde \i_M:F_\K(M)\to F_{\tilde \K}(M)$ such that $\tilde \i_M\circ\delta_M=\tilde\delta_M$. Since $F_{\bar \K}(M)\in\bar\K\subset\tilde\K$, there exists a unique continuous $E$-homomorphism $\tilde{\bar\i}_M:F_{\tilde \K}(M)\to F_{\bar\K}(M)$ such that $\tilde{\bar\i}_M\circ \tilde\delta_M=\bar\delta_M$.
Observe that the composition $\tilde{\bar\i}_M\circ\tilde\i_M:F_\K(M)\to F_{\bar\K}(M)$ coincides with the unique continuous $E$-homomorphism $\bar\i_M:F_\K(M)\to F_{\bar\K}(M)$ such that $\bar\i_M\circ \delta_M=\bar\delta_M$. By Theorem~\ref{t:bar-i-bijective}, the homomorphism $\bar\i_M$ is bijective. This implies that the $E$-homomorphism $\tilde\i_M:F_\K(M)\to F_{\tilde \K}(M)$ is bijective too.

It follows that the set $\tilde K_M=\tilde\i_M(K_M)$ is bounded in the $k_\w$-space $F_{\tilde \K}(M)$. By Lemma~\ref{l:para-kw}, $\tilde K_M\subset A_q^q[\mathcal D\cup \tilde \delta_M(M)]$ for some $q\in\w$ and some finite set $\mathcal D\subset F_{\tilde\K}(M)=\bigcup_{n\in\w}\Gamma^n_n(\delta_M(M))$.
Using Lemma~\ref{l:Ak}, we can find a number $p\ge q$ and a finite subset $D\subset M$ such that $\mathcal D\subset\Gamma^p_p[\tilde\delta_M(D)]$.
 Then $\tilde K_M\subset A^p_p[\Gamma^p_p[\tilde\delta_M(D)]\cup\tilde\delta_M(M)]$.
 Replacing the set $D$ by $D\cup\{x^+\}$, we can additionally assume that $x^+\in D$.

Taking into account that $\tilde i_{M}:F_\K(M)\to F_{\tilde\K}(M)$ is an $E$-isomorphism, we conclude that $\{a_n\}_{n\in\Omega}\subset K_M\subset A_p^p[\Gamma^p_p[\delta_M(D)]\cup\delta_M(M)]$.

Since the countable space $M$ is zero-dimensional, for every $n\in\Omega$ we can choose a retraction $r_n:M\to S_n^+=\supp(a_n)\cup\{x^+\}$ such that $r_n(D)\subset D\cap S^+_{n}$. Such a retraction $r_n$ exists since the set $D\cap S^+_n$ contains the point $x^+$ and hence is not empty.

By Theorem~\ref{t:FKX1}(4), the map $F_\K i_{S_n^+,M}:F_\K(S^+_n)\to F_\K(M)$ is a closed topological embedding, which allows us to identify the space $F_\K(S^+_n)$ with a closed  subspace of $F_\K(M)$. After such identification, we see that $a_n\in F_\K(S^+_n)$ and the map $F_\K r_n:F_\K(M)\to F_\K(S^+_n)$ is a retraction such that $$
\begin{aligned}
a_n&=F_\K r_n(a_n)=F_\K r_n(A^p_p[\Gamma^p_p[\delta_M(D)]\cup\delta_M(M)])\subset\\
&\subset A^p_p[\Gamma^p_p[\delta_M(r_n(D))]\cup\delta_M(S^+_n)]\subset A^p_p[\Gamma^p_p[\delta_M(D\cap S^+_n)]\cup\delta_M(S^+_n)].
\end{aligned}
$$

Since the set $D$ is finite, by the Pigeonhole principle, there exists a subset $D'\subset D$ such that the set $\Omega'=\{n\in\Omega:D'=D\cap S^+_n\}$ is infinite and hence $\Omega'$ contains two numbers $k<n$, greater that $p$.
It follows that $D\cap S^+_n=D'=D\cap S_k^+\subset S^+_k$ and hence $$a_n\in A^p_p[\Gamma_p^p[\delta_M(D\cap S^+_n)]\cup\delta_M(S^+_n)]\subset A^n_n[\Gamma^n_n[\delta_X(S_k^+)]\cup \delta_X(B)]\subset A^n_n[D_n\cup\delta_X(B)],$$
which contradicts the choice of $a_n=a_{n,D_n}$.
\end{proof}

\section{Fans in free paratopological $\dot E$-algebras}

Let $E$ be a topological signature with a distinguished symbol $\cdot\in E_2$ of binary operation.

A topologized $E$-algebra $X$ is called a \index{$\dot E$-algebra!paratopological}\index{paratopological $\dot E$-algebra}{\em paratopological $\dot E$-algebra} if the binary operation $\cdot:X^2\to X$ (interpreting the symbol $\cdot$) is continuous.

A paratopological $\dot E$-algebra is $X$ called \index{paratopological $\dot E$-algebra!$x{\cdot}y$-mixing}{\em $x{\cdot}y$-mixing} if there are two  points $x,y\in X$ such that $x{\cdot} y\notin\langle x\rangle_E\cup\langle y\rangle_E$.

\begin{lemma}\label{l:FK-xy} If an $\HM$-stable variety of paratopological $\dot E$-algebras contains an $x{\cdot}y$-mixing $\dot E$-algebra, then for every functionally Hausdorff space $X$ the free paratopological $\dot E$-algebra $F_\K(X)$ is a topological algebra of type $x{\cdot}y$.
\end{lemma}

\begin{proof} We claim that the set $C=\emptyset$ and the $F_\K$-valued operation
$$p:X\times X\to F_\K(X),\;\;p:(x,y)\mapsto \delta_X(x)\cdot\delta_X(y),$$
witness that $F_\K(X)$ is a topological algebra of type $x{\cdot}y$.
We need to check the conditions (1) and (2) of Definition~\ref{d:type-xy}.

The condition (2) follows from the continuity of the operation $p$ (which follows from the continuity of the canonical map $\delta_X:X\to F_\K(X)$ and the continuity of the binary operation $\cdot: F_\K(X)\times F_\K(X)\to F_\K(X)$~).

So, it remains to check the condition (1). Given two distinct points $x,y\in X$, we should check that $\{x,y\}\subset \supp\big(p(x,y)\big)$. By our assumption, the variety $\K$ contains an $x{\cdot}y$-mixing $\dot E$-algebra $Y$. By definition, $Y$ contains two points $a,b\in Y$ such that $a\cdot b\notin\langle a\rangle_E\cup\langle b\rangle_E$. Consider the doubleton $D=\{x,y\}$ and the map $f:D\to Y$ such that $f(x)=a$ and $f(b)=b$. By the definition of a free topologized $\dot E$-algebra, there exists a unique continuous $E$-homomorphism $\bar f:F_\K(D)\to Y$ such that $f=\bar f\circ\delta_D$.
We claim that $\delta_D(x)\cdot \delta_D(y)\notin\langle \delta_D(x)\rangle_E\cup\langle \delta_D(y)\rangle_E$. In the opposite case, applying the $E$-homomorphism $\bar f$ to the inclusion $\delta_D(x)\cdot \delta_D(y)\in\langle \delta_D(x)\rangle_E\cup\langle \delta_D(y)\rangle_E$, we would get
$$
\begin{aligned}
a\cdot b&=f(x){\cdot}f(y)=\bar f(\delta_D(x))\cdot\bar f(\delta_D(y))=\bar f(\delta_D(x)\cdot\delta_D(y))\in \bar f(\langle \delta_D(x)\rangle_E\cup\langle \delta_D(y)\rangle_E)=\\
&=\langle \bar f\circ \delta_D(x)\rangle_E\cup\langle\bar f\circ\delta_D(y)\rangle_E=\langle f(x)\rangle_E\cup\langle f(y)\rangle_E=\langle a\rangle_E\cup\langle b\rangle_E,
\end{aligned}
$$ which contradicts the choice of the points $a,b$. This contradiction shows that $\delta_D(x)\cdot \delta_D(y)\notin\langle \delta_D(x)\rangle_E\cup\langle \delta_D(y)\rangle_E$.

By Theorem~\ref{t:FKX1}(3), the $E$-homomorphism $F_\K i_{D,X}:F_\K(D)\to F_\K(X)$ is injective. Consequently,
$$
\begin{aligned}
p(x,y)&=\delta_X(x)\cdot\delta_X(y)=F_\K i_{D,X}(\delta_D(x)\cdot \delta_D(y))\notin F_\K i_{D,X}(\langle \delta_D(x)\rangle_E\cup\langle\delta_D(y)\rangle_E)=\\
&=
\langle \delta_X(x)\rangle_E\cup\langle \delta_X(y)\rangle_E=F_\K(\{x\};X)\cup F_\K(\{y\};X),
\end{aligned}
$$ which implies that $\supp(p(x,y))$ is not contained in $\{x\}$ or $\{y\}$.
Combining this fact with the obvious inclusion $\supp(p(x,y))\subset\{x,y\}$, we conclude that $\supp(p(x,y))=\{x,y\}$. This yields the condition (1) of Definition~\ref{d:type-xy}.
\end{proof}

\begin{theorem}\label{t:dotE} Assume that a $d{+}\HM{+}k_\w$-stable variety of paratopological $\dot E$-algebras contains an $x{\cdot}y$-mixing $\dot E$-algebra.
Let $X$ be a $\mu$-complete Tychonoff space satisfying one of the following conditions:
\begin{enumerate}
\item[\textup{1)}]  $F_\K(X)$ contains no strong $\Fin^\w$-fan and $X$ is an $\aleph$-space.
\item[\textup{2)}] $F_\K(X)$  contains no $\Fin^{\w}$-fan and $X$ is $k^*$-metrizable.
\end{enumerate}
Then the $k$-modification $kX$ of $X$ is either metrizable or is a topological sum of cosmic $k_\w$-spaces.
\end{theorem}

\begin{proof} By Theorems~\ref{t:FKX1}(1,3), \ref{t:bar-i-bijective}, \ref{t:d+HM+kw}, the functor $F_\K|\Top_{3\frac12}$ is monomorphic, $\II$-regular, bounded and has finite supports.
By Lemma~\ref{l:FK-xy}, the functor-space $F_\K(X)$ is a topological algebra of type $x{\cdot}y$. Now it is legal to apply Theorem~\ref{t:xy}  and finish the proof.
\end{proof}

\begin{definition}\label{d:xsy-mixing} A paratopological $\dot E$-algebra $X$ is called
\begin{itemize}
\item \index{paratopological $\dot E$-algebra!idempotent}{\em idempotent} if $x{\cdot}x=x$ for all $x\in X$;
\item \index{paratopological $\dot E$-algebra!$x(sy_i)^{<\w}$-mixing}{\em $x(sy_i)^{<\w}$-mixing} if for every $n\in\IN$ there exist pairwise distinct points $s,x,y_1,\dots,y_n\in X$ such that for the sequence of points $z_n=s\cdot y_n$ and $z_{k}=(s\cdot y_k)\cdot z_{k+1}$, $1\le k<n$, the point $x\cdot z_1$ is not contained in the $E$-hull $\langle A\cup\{s\}\rangle_E$ for any proper subset $A$ of $\{x,y_1,\dots,y_n\}$.
\end{itemize}
\end{definition}

\begin{lemma}\label{l:xxy_w->xy} Each $x(sy_i)^{<\w}$-mixing $\dot E$-algebra $X$ is $x{\cdot}y$-mixing.
\end{lemma}

\begin{proof} The $x(sy_i)^{<\w}$-mixing $\dot E$-algebra $X$ contains pairwise distinct points $x,s,y$ such that $x(sy)\notin \langle \{x,s\}\rangle_E\cup\langle\{y,s\}\rangle_E$. Taking into account that  $\langle\{x\}\rangle_E\cup\langle\{sy\}\rangle_E\subset \langle\{x,s\}\rangle_E\cup\langle\{y,s\}\rangle_E$, we conclude that $x(sy)\notin\langle\{x\}\rangle_E\cup\langle\{sy\}\rangle_E$, which means that the $\dot E$-algebra $X$ is $x{\cdot}y$-mixing.
\end{proof}

\begin{lemma}\label{l:FK-xxy_w} If an $\HM{+}k_\w$-stable variety $\K$ of idempotent paratopological $\dot E$-algebras contains an $x(sy_i)^{<\w}$-mixing $\dot E$-algebra, then for every functionally Hausdorff space $X$ the free paratopological $\dot E$-algebra $F_\K(X)$ is a topological algebra of types $x{\cdot}y$ and $x(s^*y_i)^{<\w}$.
\end{lemma}

\begin{proof} By Lemma~\ref{l:xxy_w->xy}, the variety $\K$ contains an $x{\cdot}y$-mixing $\dot E$-algebra and by Lemma~\ref{l:FK-xy}, $F_\K(X)$ is a topological algebra of type $x{\cdot}y$.

Next, we show that $F_\K(X)$ is a topological algebra of type $x(s^*y_i)^{<\w}$.
Given any point $s\in X$, for every $n\in\IN$ and positive $k\le n$ consider the $(1+n)$-ary operation $p_{k,n}:X^{1+n}\to F_\K(X)$ defined by the recursive formulas $p_{n,n}(x,y_1,\dots,y_n)=\delta_X(s)\cdot \delta_X(y_n)$ and $$p_{k,n}(x,y_1,\dots,y_n)=(\delta_X(s)\cdot\delta_X(y_k))\cdot p_{k+1,n}(x,y_1,\dots,y_n),$$
 for $0<k<n$. Finally, put $p_0(x)=\delta_X(x)$ and $$p_n(x,y_1,\dots,y_n)=\delta_X(x_0)\cdot p_{1,n}(x,y_1,\dots,y_n)$$ for $n\in\IN$.

We claim that the sets $C_n=\{s\}$ and the operations $p_n$, $n\in\w$, witness that the functor-space $F_\K(X)$ is a topological algebra of type $x(s^*y_i)^{<\w}$. We need to verify the conditions (1)--(3) of Definition~\ref{d:xyy_w}. Fix any $n\in\IN$ and observe that $p_{n,n}(s,\dots,s)=\delta_X(s)\cdot\delta_X(s)=\delta_X(s)$.
Assume that for some positive $k<n$ we have proved that $p_{k{+}1,n}(s,\dots,s)=\delta_X(s)$. Then
$$
\begin{aligned}
p_{k,n}(s,\dots,s)&=(\delta_X(s){\cdot}\delta_X(s))\cdot p_{k{+}1,n}(s,\dots,s)=\\
&=(\delta_X(s){\cdot}\delta_X(s)){\cdot} \delta_X(s)=\delta_X(s){\cdot}\delta_X(s)=\delta_X(s).
\end{aligned}
$$
In particular, $p_{1,n}(s,\dots,s)=\delta_X(s)$ and $$p_n(s,\dots,s)=\delta_X(s)\cdot  p_{1,n}(s,\dots,s)=\delta_X(s)\cdot\delta_X(s)=\delta_X(s)=p_0(s).$$

To check the second condition of Definition~\ref{d:xyy_w}, fix $n\in \IN$ and pairwise distinct points $x,y_1,\dots,y_n\in X\setminus\{s\}$. We need to prove that $\{x,y_1,\dots,y_n\}\subset\supp(p_n(x,y_1,\dots,y_n))$. The definition of $p_n(x,y_1,\dots,y_n)$ implies that $\supp(p_n(x,y_1,\dots,y_n))$ is contained in the set $S=\{s,x,y_1,\dots,y_n\}$. Assuming that $\{x,y_1,\dots,y_n\}\not\subset\supp(p_n(x,y_1,\dots,y_n))$ for some $n\in\IN$, we conclude that $p_n(x,y_1,\dots,y_n)\in\langle \delta_X(A\cup\{s\})\rangle_E$ for some proper subset $A\subset\{x,y_1,\dots,y_n\}$. By Theorem~\ref{t:FKX1}(4) and Proposition~\ref{p:FembTych}, the map $F_\K i_{S,X}:F_\K(S)\to F_\K(X)$ is a closed topological embedding, which allows us to identify $F_\K(S)$ with a closed subspace of $F_\K(X)$.

By our assumption, the variety $\K$ contains a $x(sy_i)^{<\w}$-mixing $\dot E$-algebra $Y$. By Definition~\ref{d:xsy-mixing}(2), the $\dot E$-algebra $Y$ contains pairwise distinct points $\bar x,\bar s,\bar y_1,\dots,\bar y_n\in Y$ such that for the  points $\bar z_n=\bar s{\cdot}\bar y_n$ and $\bar z_{k}=(\bar s{\cdot}\bar y_k)\cdot \bar z_{k+1}$, $1\le k<n$, the point $\bar x{\cdot}\bar z_1$ is not contained in the $E$-hull $\langle \bar A\cup\{s\}\rangle_E$ for any proper subset $\bar A$ of $\{\bar x,\bar y_1,\dots,\bar y_n\}$.

Consider the function $f:S\to Y$ such that $f(x)=\bar x$, $f(s)=\bar s$, and $f(y_i)=\bar y_i$ for $1\le i\le n$. By the definition of the free $E$-algebra $F_\K(S)$, there exists a continuous $E$-homomorphism $\bar f:F_\K(S)\to Y$ such that $\bar f\circ\delta_S=f$.
The definition of the element $p_n(x,y_1,\dots,y_n)$ implies that
$$
p_n(x,y_1,\dots,y_n)\in F_\K i_{S,X}(F_\K(S))= F_\K(S)
$$ and $\bar x{\cdot}\bar z_1=\bar f(p_n(x_0,y_1,\dots,y_n))\in\bar f(\langle A\cup\{s\})\rangle_E=\langle f(A)\cup\{\bar s\}\rangle_E$ for the proper subset $f(A)$ of $\{\bar x_0,\bar y_1,\dots,\bar y_n\}$. But this contradicts the choice of the $\dot E$-algebra $Y$. This contradiction completes the proof of the condition (2) of Definition~\ref{d:xyy_w}.

To check the condition (3), fix any neighborhood $U\subset F_\K(X)$ of the point $p_0(s)=\delta_X(s)$. By the continuity of the multiplication in $F_\K(X)$, there is a neighborhood $V_0\subset F_\K(X)$ of the point $\delta_X(s)$ such that $V_0{\cdot}V_0\subset U$. By induction, we can construct a sequence $(V_n)_{n\in\w}$ of neighborhoods of $s$ in $F_\K(X)$ such that $V_n{\cdot}V_n\subset V_{n-1}$ and $V_n\subset V_{n-1}$ for all $n\in\IN$. Using the continuity of the map $\delta_X:X\to F_\K(X)$ and the continuity of the multiplication in the paratopological $\dot E$-algebra $F_\K(X)$, for every $n\in\w$ choose an open neighborhood $U_n\subset X$ of the point $s$ such that $\delta_X(U_n)\subset V_{n}$ and $\delta_X(s)\cdot\delta_X(U_n)\subset V_{n+1}$. We claim that $p_n(U_0\times \cdots\times U_n)\subset U$ for every $n\in\IN$. Fix any sequence $(u_0,\dots,u_n)\in U_0\times\dots \times U_n$. First we prove that $p_{k,n}(u_0,\dots,u_n)\in V_k$ for every $k\le n$. Indeed, for $k=n$, we get $p_{n,n}(u_0,\dots,u_n)=\delta_X(s){\cdot}\delta_X(u_n)\subset V_{n+1}\subset V_n$. Assuming that for some positive $k<n$ the inclusion $p_{k+1,n}(u_0,\dots,u_n)\in V_{k+1}$ has been proved, observe that $$
\begin{aligned}
p_{k,n}(u_0,\dots,u_n)&=(\delta_X(s)\cdot\delta_X(u_k))\cdot p_{k+1,n}(u_0,\dots,u_n)\in\\
&\in (\delta_X(s)\cdot\delta_X(U_k))\cdot V_{k+1}\subset V_{k+1}\cdot V_{k+1}\subset V_k.
\end{aligned}
$$
Then $p_n(u_0,\dots,u_n)=\delta_X(u_0)\cdot p_{1,n}(u_0,\dots,u_n)\in V_0{\cdot}V_1\subset U$, which completes the verification of the conditions (1)--(3) of Definition~\ref{d:xyy_w}.
\end{proof}

Combining Lemma~\ref{l:FK-xxy_w} with Corollaries~\ref{c:xyy_w} and Theorems~\ref{t:FKX1}(1), \ref{t:bar-i-bijective}, \ref{t:d+HM+kw}, we obtain:

\begin{corollary}\label{c:i-xyy_w} Let $\K$ be a $d{+}\HM{+}k_\w$-stable variety of idempotent topological $\dot E$-algebras such that $\K$ contains an  $x(sy_i)^{<\w}$-mixing topological algebra. Let $X$ be a $\mu$-complete Tychonoff space such that $F_\K X$ contains no strong $\Fin^\w$-fan.
\begin{enumerate}
\item[\textup{1)}] If $X$ is a $\bar\aleph_k$-space, then the $k$-modification of $X$ is a topological sum of cosmic $k_\w$-spaces.
\item[\textup{2)}] If $X$ is a $\bar\aleph_k$-space and a $k_\IR$-space, then the space $X$ is a topological sum of cosmic $k_\w$-spaces.
\item[\textup{3)}] If $X$ is an $\aleph_0$-space, then $X$ is hemicompact.
\item[\textup{4)}] If $X$ is cosmic, then $X$ is $\sigma$-compact.
\end{enumerate}
\end{corollary}

\section{Fans in free paratopological $\dot E^*$-algebras}

Let $E$ be a topological signature with a distinguished symbol $\cdot\in E_2$ of binary operation and a distinguished symbol ${}^*\in E_1$ of a unary operation.

A topologized $E$-algebra $X$ is called a \index{paratopological $\dot E^*$-algebra}{\em paratopological $\dot E^*$-algebra} if the binary operation $\cdot\colon X^2\to X$, $\cdot\colon (x,y)\mapsto xy$,  (interpreting the symbol $\cdot$) is continuous. According to this definition the unary operation ${}^*\colon X\to X$,
${}^*\colon x\mapsto x^*$, is defined but is not necessarily continuous.

\begin{definition}\label{d:mixing} A paratopological $\dot E^*$-algebra is called
\begin{itemize}
\item \index{paratopological $\dot E^*$-algebra!$x(y^*z)$-mixing}{\em $x(y^*z)$-mixing} if there exist pairwise distinct points $x,y,z\in M$ such that\newline $x(y^*z)\notin\langle \{y,z\}\rangle_E\cup\langle \{x,z\}\rangle_E$;
\item \index{paratopological $\dot E^*$-algebra!$(x^*y)(z^*s)$-mixing}{\em $(x^*y)(z^*s)$-mixing} if there exist  pairwise distinct points $x,y,z,s$ such that\newline $(x^*y)(z^*s)\notin \langle \{x,z,s\}\rangle_E\cup\langle\{x,y,z\}\rangle_E$;
\item \index{paratopological $\dot E^*$-algebra!$x(y(s^*z)y^*)$-mixing} {\em $x(y(s^*z)y^*)$-mixing} if there exist  pairwise distinct points $x,y,z,s\in\K$ such that\newline $x((y(s^*z))y^*)\notin \langle \{x,s,z\}\rangle_E\cup \langle \{y,s,z\}\rangle_E$;
\item \index{paratopological $\dot E^*$-algebra!$(x(s^*z)x^*)(y(s^*z)y^*)$-mixing}{\em $(x(s^*z)x^*){\cdot}(y(s^*z)y^*)$-mixing}  if there exist  pairwise distinct points $x,y,z,s\in X$ such that\newline $((x(s^*z))x^*){\cdot}((y(s^*z))y^*)\notin \langle \{x,s,z\}\rangle_E\cup \langle \{y,s,z\}\rangle_E$;
\item \index{paratopological $\dot E^*$-algebra!$x(s^*y_i)^{<\w}$-mixing}{\em $x(s^*y_i)^{<\w}$-mixing} if for every $n\in\w$ there exist  pairwise distinct points $x,s,y_1,\dots,y_n\in X$ such that for the points $z_n=s\cdot y_n$ and $z_{k}=(s^*y_k)\cdot z_{k+1}$, $1\le k<n$, the point $x\cdot z_1$ is not contained in the $E$-hull $\langle A\cup\{s\}\rangle_E$ for any proper subset $A$ of $\{x,y_1,\dots,y_n\}$.
    \end{itemize}
\end{definition}

We shall say that the identity $x^*x=y^*y$
(resp.  $xx^*=yy^*$, $x(x^*x)=x$, $(x^*x)(x^*x)=x^*x$) holds in a variety $\K$ of paratopological $\dot E^*$-algebras if this equality holds for any elements $x,y$ of any paratopological $\dot E^*$-algebra $X\in \K$.

\begin{proposition}\label{p:mixing->TA} Let $\K$ be a $\HM$-stable variety of paratopological $\dot E^*$-algebras. For a functionally Hausdorff space $X$ the free paratopological $\dot E^*$-algebra $F_\K(X)$ is a topological algebra of type:
\begin{enumerate}
\item[\textup{1)}] $x(y^*z)$ if the identity $x^*x=y^*y$ holds in $\K$ and $\K$ contains a $x(y^*z)$-mixing paratopological $\dot E^*$-algebra;
\item[\textup{2)}] $(x^*y)(z^*s)$ if the identity $x^*x=y^*y$  holds in $\K$ and $\K$ contains a $(x^*y)(z^*s)$-mixing paratopological $\dot E^*$-algebra;
\item[\textup{3)}]  $x(y(s^*z)y^*)$ if the identities $x(x^*x)=x$, $x^*x=y^*y$ and $xx^*=yy^*$ hold in $\K$ and $\K$ contains a  $x(y(s^*z)y^*)$-mixing paratopological $\dot E^*$-algebra;
\item[\textup{4)}] $(x(s^*z)x^*){\cdot}(y(s^*z)y^*)$ if the identities $x(x^*x)=x$, $x^*x=y^*y$ and $xx^*=yy^*$ hold in $\K$ and $\K$ contains a $(x(s^*z)x^*){\cdot}(y(s^*z)y^*)$-mixing paratopological $\dot E^*$-algebra;
\item[\textup{5)}] $x(s^*y_i)^{<\w}$ if the identities $x(x^*x)=x$ and $(x^*x)(x^*x)=x^*x$ hold in $\K$ and $\K$-contains a $x(s^*y_i)^{<\w}$-mixing paratopological $\dot E^*$-algebra.
\end{enumerate}
\end{proposition}

\begin{proof} The proofs of the five statements are similar. So, we present a detail proof of the first statement only (see also the proof of Lemma~\ref{l:FK-xxy_w}). For a space $X$ and a point $x\in X$ by $\bar x$ we shall denote the image $\delta_X(x)$ of the point $x$ under the canonical map $\delta_X:X\to F_\K(X)$.

 Assume that the identity $x^*x=y^*y$ holds in $\K$ and $\K$ contains a $x(y^*z)$-mixing paratopological $\dot E^*$-algebra $Y\in\K$.
 In this case we shall show that for any functionally Hausdorff space $X$ the ternary $F_\K$-valued operation $$p:X^3\to F_\K(X),\;\;p:(x,y,z)\mapsto \bar x \cdot(\bar y^*\cdot \bar z),$$ witnesses that $F_\K(X)$ is a topological algebra of type $x(y^*z)$.
The identity $y^*y=x^*x$ holding in $\K$ guarantees that $p(x,y,y)=p(x,x,x)$ for all $x\in X$, confirming the condition (1) of Definition~\ref{d:xyz}.

To verify the condition (2) of that definition, we need to check that $$\{x,z\}\subset \supp(p(x,y,z))=\supp(\bar x(\bar y^*\bar z))$$ for any distinct points $x,y,z\in X$.
By Definition~\ref{d:mixing}, the $x(y^*z)$-mixing $\dot E^*$-algebra $Y$ contains  three points $x_0,y_0,z_0\in M$ such that $x_0(y_0^*z_0)\notin\langle \{x_0,z_0\}\rangle_E\cup\{y_0,z_0\}\rangle_E$.
Consider the (discrete) subspace $D=\{x,y,z\}\subset X$ and the map $f:D\to Y$ such that $f(x)=x_0$, $f(y)=y_0$, $f(z)=z_0$. This map generates a unique continuous $E$-homomorphism $\bar f:F_K(D)\to Y$ such that $\bar f\circ\delta_D=f$.

Assuming that $x\notin\supp(p(x,y,z))$ and taking into account that $\supp(x,y,z)\subset \{x,y,z\}$, we conclude that $p(x,y,z)\in F_\K(\{y,z\};X)$. The injectivity of the $E$-homo\-morphism $F_\K i_{D,X}:F_\K(D)\to F_\K(X)$ (proved in Theorem~\ref{t:FKX1}) implies that $\bar x(\bar y^*\bar z)\in \langle\{\bar y,\bar z\}\rangle_E$ in $F_\K(D)$ and hence $$x_0(y_0^*z_0)=F_\K f(\bar x(\bar y^*\bar z))\in F_\K f(\langle\{\bar x,\bar z\}\rangle_E)\subset \langle\{f(x),f(z)\}\rangle_E=\langle\{x_0,z_0\}\rangle_E,$$ which contradicts the choice of the points $x_0,y_0,z_0$. By analogy we can prove that $z\in\supp(p(x,y,z))$.
So, the condition (2) of Definition~\ref{d:xyz} is satisfied too.

To prove the condition (3), fix any point $x\in X$ and a neighborhood $U\subset F_\K(X)$ of the point $p(x,x,x)=\bar x(\bar x^*\bar x)$. By the continuity of the multiplication in $F_\K(X)$, the are open sets $O_{\bar x},O_{\bar x^*\bar x}\subset F_\K(X)$ such that $O_{\bar x}\cdot O_{\bar x^*\bar x}\subset U$. The continuity of the map $\delta_X:X\to F_\K(X)$ at  the point $x$ yields a neighborhood $O_x\subset X$ of $x$ such that $\delta_X(O_x)\subset O_{\bar x}$.
For every $y\in X$ use the equality $\bar y^*\bar y=\bar x^*\bar x\in O_{\bar x^*\bar x}$ and the continuity of the multiplication in $F_\K(X)$ to find a neighborhood $O_{\bar y}\subset F_\K(X)$ of $\bar y$ such that $\bar y^*\cdot O_{\bar y}\subset O_{\bar x^*\bar y}$. By the continuity of $\delta_X$, there exists a neighborhood $O_y\subset X$ such that  $\delta_X(O_y)\subset O_{\bar y}$.
Then the neighborhood $U_\Delta=\bigcup_{y\in X}\{y\}\times O_y\subset X_d\times X$ has the required property: $p(O_x\times U_\Delta)\subset U$.
\end{proof}

\begin{theorem}\label{t:xsy_w-mix} Let $\K$ be a $d{+}\HM{+}k_\w$-stable variety of paratopological $\dot E^*$-algebras such that the identities $x(x^*x)=x$ and $(x^*x)(x^*x)=x^*x$ hold in $\K$ and $\K$ contains an $x(s^*y_i)^{<\w}$-mixing paratopological $\dot E^*$-algebra. Assume that for a $\mu$-complete Tychonoff space $X$ the free paratopological $\dot E^*$-algebra $F_\K(X)$ contains no strong $\Fin^\w$-fan.
\begin{enumerate}
\item[\textup{1)}] If $X$ is a $\bar\aleph_k$-space, then the $k$-modification of $X$ is a topological sum of cosmic $k_\w$-spaces.
\item[\textup{2)}] If $X$ is a $\bar\aleph_k$-space and a $k_\IR$-space, then the space $X$ is a topological sum of cosmic $k_\w$-spaces.
\item[\textup{3)}] If $X$ is an $\aleph_0$-space, then $X$ is hemicompact.
\item[\textup{4)}] If $X$ is cosmic, then $X$ is $\sigma$-compact.
\end{enumerate}
\end{theorem}

\begin{proof} By Theorems~\ref{t:FKX1}(1,3), \ref{t:bar-i-bijective}, \ref{t:d+HM+kw}, the functor $F_\K|\Top_{3\frac12}$ is monomorphic, $\II$-regular, strongly bounded and has finite supports.
By Proposition~\ref{p:mixing->TA}(5), the functor-space $F_\K(X)$ is a topological algebra of type $x(s^*y_i)^{<\w}$. Now it is legal to apply Corollary~\ref{c:xyy_w}  and finish the proof.
\end{proof}

\begin{theorem}\label{t:xsy+xy.zs} Let $\K$ be a $d{+}\HM{+}k_\w$-stable variety of paratopological $\dot E^*$-algebras such that the identities $x(x^*x)=x$ and $(x^*x)(x^*x)=x^*x=y^*y$ hold in $\K$ and $\K$ contains an $x(s^*y_i)^{<\w}$-mixing and an $(x^*y)(z^*s)$-mixing paratopological $\dot E^*$-algebras.
 Assume that for a $\mu$-complete Tychonoff space its free paratopological $\dot E^*$-algebra $F_\K (X)$ contains no strong $\Fin^\w$-fan.
\begin{enumerate}
\item[\textup{1)}] If $X$ is an $\bar \aleph_k$-$k_\IR$-space, then $X$ is a topological sum $C\oplus D$ of a cosmic $k_\w$-space $C$ and a discrete space $D$.
\item[\textup{2)}] If $X$ is an $\aleph_k$-space and $FX$ contains no $\Fin^{\w_1}$-fan, then $X$ is $k$-homeomorphic to a topological sum $K\oplus D$ of a cosmic $k_\w$-space $K$ and a discrete space $D$.
\end{enumerate}
\end{theorem}

\begin{proof} By Theorems~\ref{t:FKX1}(1,3), \ref{t:bar-i-bijective}, \ref{t:d+HM+kw}, the functor $F_\K|\Top_{3\frac12}$ is monomorphic, $\II$-regular, strongly bounded and has finite supports.
By Proposition~\ref{p:mixing->TA}(2,5), the functor-space $F_\K(X)$ is a topological algebra of types $(x^*y)(z^*s)$ and $x(s^*y_i)^{<\w}$. We claim that the functor-space $F_\K(X)$ is a topological algebra of type $x{\cdot}y$. By Lemma~\ref{l:FK-xy} this will follow as soon as we find an $x{\cdot}y$-mixing $\dot E$-algebra in the variety $\K$. By our assumption, the variety $\K$ contains a $(x^*y)(z^*s)$-mixing algebra $Y$. By Definition~\ref{d:mixing}, $Y$ contains pairwise distinct points $x_0,y_0,z_0,s_0$ such that $(x_0^*{\cdot}y_0){\cdot}(z_0^*{\cdot}s_0)\notin\langle \{x_0,z_0,s_0\}\rangle_E\cup\langle \{x_0,y_0,z_0\}\rangle_E$. Then for the elements $x=x_0^*{\cdot}y_0$ and $y=z_0^*{\cdot}s_0$ the product $x{\cdot}y=(x_0^*{\cdot}y_0){\cdot}(z_0^*{\cdot}s_0)$ is not contained in $\langle x\rangle_{E}\cup\langle y\rangle_E\subset\langle\{x_0,y_0\}\rangle_E\cup\langle\{z_0,s_0\}\rangle_E$.
The elements $x,y$ witness that the $\dot E$-algebra $Y$ is $x{\cdot}y$-mixing and $F_\K(X)$ is a topological algebra of type $x{\cdot}y$.

   Now it is legal to apply Corollary~\ref{c:xy.zs+xsy_w} and finish the proof.
\end{proof}

\begin{theorem}\label{t:long-mix->X} Let $\K$ be a $d{+}\HM{+}k_\w$-stable variety of paratopological $\dot E^*$-algebras. Assume that the identities $x(x^*x)=x$, $x^*x=y^*y$ and $xx^*=yy^*$ hold in $\K$ and $\K$ contains an $x(y(s^*z)y^*)$-mixing and an $(x(s^*z)x^*)(y(s^*z)y^*)$-mixing paratopological $\dot E^*$-algebras. Let $X$ be a $\mu$-complete Tychonoff space.
\begin{enumerate}
\item[\textup{1)}] If $F_\K(X)$ contains no strong $\Fin^{\w_1}$-fan and each infinite compact subset of $X$ contains a convergent sequence, then $X$ either $k$-discrete or $\w_1$-bounded.
\item[\textup{2)}] If $F_\K(X)$ contains no $\Fin^{\w_1}$-fan and $X$ is $k^*$-metrizable, then either $X$ is $k$-discrete or $X$ is an $\aleph_0$-space.
\item[\textup{3)}] If $X$ is an $\bar\aleph_k$-space and $F_\K(X)$ contains no strong $\Fin^{\w}$-fans, then $X$ is either $k$-discrete or $X$ is a $k$-sum of hemicompact spaces; moreover, if $X$ is  a $k_\IR$-space, then $X$ is either discrete or a cosmic $k_\w$-space.
\item[\textup{4)}] If $X$ is an $\aleph_k$-space and the space $F_\K(X)$ contains no  $\Fin^{\w}$-fans and no $\Fin^{\w_1}$-fans, then $X$ is either $k$-discrete or  hemicompact.
\end{enumerate}
\end{theorem}

\begin{proof} By Theorems~\ref{t:FKX1}(1,3), \ref{t:bar-i-bijective}, \ref{t:d+HM+kw}, the functor $F_\K|\Top_{3\frac12}$ is monomorphic, $\II$-regular, bounded and has finite supports.
By Proposition~\ref{p:mixing->TA}(3,4), the functor-space $F_\K(X)$ is a topological algebra of types $x(y^*(s^*z)y)$ and $(x^*(s^*z)x)(y^*(s^*z)y)$.
We claim that the functor-space $F_\K(X)$ is a topological algebra of type $x{\cdot}y$. By Lemma~\ref{l:FK-xy} this will follow as soon as we find an $x{\cdot}y$-mixing $\dot E$-algebra in the variety $\K$. By our assumption, the variety $\K$ contains an $x(y(z^*s)y^*)$-mixing algebra $Y$. By Definition~\ref{d:mixing}, $Y$ contains pairwise distinct points $x_0,y_0,z_0,s_0$ such that $x_0{\cdot}((y_0(s_0^*z_0))y_0^*)\notin\langle \{y_0,z_0,s_0\}\rangle_E\cup\langle \{x_0,z_0,s_0\}\rangle_E$. Then for the elements $x=x_0$ and $y=(y_0(s_0^*z_0))y_0^*$ the product $x{\cdot}y=x_0\cdot ((y_0(s_0^*z_0))y_0^*)$ is not contained in $\langle x\rangle_{E}\cup\langle y\rangle_E\subset\langle\{x_0,z_0,s_0\}\rangle_E\cup\langle\{y_0,z_0,s_0\}\rangle_E$.
The elements $x,y$ witness that the $\dot E$-algebra $Y$ is $x{\cdot}y$-mixing and $F_\K(X)$ is a topological algebra of type $x{\cdot}y$.
Now it is legal to apply Theorem~\ref{t:long1}  and finish the proof.
\end{proof}

\begin{theorem}\label{t:gen-eq} Assume that the signature $E$ is a cosmic $k_\w$-space and let  $\K$ be a $d{+}\HM{+}k_\w$-stable variety of topological $\dot E^*$-algebras such that the identities $x(x^*x)=x$, $x^*x=y^*y$, $xx^*=yy^*$ hold in $\K$ and $\K$ contains an $x(y^*(s^*z)y)$-mixing and a $(x^*(s^*z)x)(y^*(s^*z)y)$-mixing topological $\dot E^*$-algebras. For a $\mu$-complete Tychonoff $k_\IR$-space $X$  the following conditions are equivalent:
\begin{enumerate}
\item[\textup{1)}] $X$ is either discrete or a cosmic $k_\w$-space.
\item[\textup{2)}] $F_\K(X)$ is either discrete or a cosmic $k_\w$-space.
\item[\textup{3)}] $X$ is a $\bar\aleph_k$-space and $F_\K(X)$ contains no strong $\Fin^{\w}$-fans.
\item[\textup{4)}] $X$ is an $\aleph_k$-space, $F_\K(X)$ contains no strong $\Fin^\w$-fans and no  $\Fin^{\w_1}$-fans.
\end{enumerate}
\end{theorem}

\begin{proof} $(1)\Ra(2)$ If the space $X$ is discrete, then the free topological $E$-algebra $F_\K(X)$ is discrete by  Proposition~\ref{p:FKX-disc}. If $X$ is a cosmic $k_\w$-space, then $F_\K(X)$ is a cosmic $k_\w$-space by Theorems~\ref{t:bar-i-bijective} and \ref{t:kw}.
The implications $(2)\Ra(3,4)$ follows from Proposition~\ref{k-no-Cld-fan} and the implications $(3)\Ra(1)$ and $(4)\Ra(1)$ were proved in Theorem~\ref{t:long-mix->X}(3,4).
\end{proof}

\section{Free paratopological $\dot E^*$-algebras in $d{+}\HM{+}k_\w$-superstable varieties}

In this section we assume that the signature $E$ is a countable $k_\w$-space and $E$ contains two distinguished symbols $\cdot \in E_2$ and ${}^*\in E_1$ of a binary and unary operations.

A topologized $E$-algebra $X$ is called a \index{paratopological $\dot E^*$-algebra}{\em paratopological $\dot E^*$-algebra} if the binary operation $\cdot\colon X^2\to X$, $\cdot\colon: (x,y)\mapsto xy$, interpreting the symbol $\cdot\in E_2$, is continuous. The unary operation ${}*\colon:X\to X$ is defined but is not necessarily continuous.

A variety $\K$ of paratopological $\dot E^*$-algebras is called \index{variety!$k_\w$-superstable}{\em $k_\w$-superstable} if $\K$ contains any paratopological $\dot E^*$-algebra $X$ which is a $k_\w$-space and admits a continuous bijective $E$-homomorphism $h:X\to Y$ onto a paratopological $\dot E^*$-algebra $Y\in\K$.

We shall say that a variety $\K$ of paratopological $\dot E^*$-algebras is \index{variety!$d{+}\HM{+}k_\w$-superstable}{\em $d{+}\HM{+}k_\w$-superstable} if it is $d$-stable, $\HM$-stable, and $k_\w$-superstable. In this case $\K$ is $k_\w$-stable (which means that $\K$ contains any topological $\dot E^*$-algebra $X$ which is a $k_\w$-space and admits a continuous bijective $E$-homomorphism $h:X\to Y$ onto a topological $\dot E^*$-algebra $Y\in\K$).

\begin{definition}\label{d:supermix} A paratopological $\dot E^*$-algebra $X$ is called \index{paratopological $\dot E^*$-algebra!$(x^*y)(z^*s)$-supermixing}{\em $(x^*y)(z^*s)$-supermixing} if $X$ is $(x^*y)(z^*s)$-mixing and $X$ contains a 4-element set $Q=\{x,y,z,s\}$ such that  $$(x^*y)(z^*s)\notin \langle\{\langle Q\setminus\{x\}\rangle_E\cup Q\rangle_C\cup
\langle\{Q\setminus\{z\}\rangle_E\cup Q\rangle_C$$ where $C=\{\cdot\}\subset E_2\subset E$.
\end{definition}

\begin{proposition}\label{p:supermix-alg} Let $\K$ be a $d{+}\HM{+}k_\w$-superstable variety of paratopological $\dot E^*$-algebras. For a $\mu_s$-complete Tychonoff space $X$ the free paratopological $\dot E^*$-algebra $F_\K(X)$ is a topological superalgebra of type $(x^*y)(z^*s)$ if the identity $x^*x=y^*y$  holds in $\K$ and $\K$ contains a $(x^*y)(z^*s)$-supermixing paratopological $\dot E^*$-algebra.
\end{proposition}

\begin{proof} For a point $x\in X$ by $\bar x$ we denote its image $\delta_X(x)$ in $F_\K(X)$. We claim that the operation
$p_X:X^4\to F_\K(X)$, $p_X:(x,y,z,s)\mapsto (\bar x^*\cdot\bar y)\cdot(\bar z^*\cdot \bar s)$,
turns $F_\K(X)$ into a topological superalgebra of type $(x^*y)(z^*s)$.

We need to check the conditions~(1)--(4) of Definition~\ref{d:xy.zs}.
The condition (1) follows from the identity $x^*x=y^*y$ holding in $\K$, the condition (2) from the existence of a $(x^*y)(z^*s)$-mixing paratopological $\dot E^*$-algebra in the variety $\K$, and the condition (3) from the continuity of the binary operation $\cdot$ in $F_\K(X)$.

It remains to verify the condition (4). Given compact sets $B\subset X$ and $K\subset F_\K(X)$, we need to find a finite subset $A\subset X$ such that for any pairwise distinct points $x,y,z,s\in B$ with $p_X(x,t,z,s)\in K$ we get $\{x,z\}\subset A$.
Let $C=\{\cdot\}\subset E_2\subset E$ be the singleton containing the distinguished symbol for the binary operation. By our assumption, the variety $\K$ contains an $(x^*y)(z^*s)$-supermixing $\dot E^*$-algebra $Y$. By Definition~\ref{d:supermix}, the algebra $Y$ contains a 4-element set $Q_0=\{x_0,y_0,z_0,s_0\}$ such that $$(x_0^*{\cdot}y_0){\cdot}(z^*_0{\cdot}s_0)\notin \big\langle \langle Q_0\setminus\{x_0\}\rangle_E\cup Q_0\big\rangle_C\cup
\big\langle \langle \{Q_0\setminus\{z_0\}\rangle_E\cup Q_0\big\rangle_C.$$

By Theorem~\ref{t:super}, the compact set $K$ is contained in $C^n[D\cup\delta_X(X)]$ for some $n\in\w$ and some finite set $D\subset F_\K(X)$. Find a finite subset $A\subset X$ such that $D\subset \langle A\rangle_E$ and conclude that $K\subset  C^n[\langle A\rangle_E\cup\delta_X(X)]\subset \langle \langle A\rangle_E\cup\delta_X(X)\rangle_C$.

We claim that for any pairwise distinct points $x,y,z,s\in B$ with $p_X(x,y,z,s)\in K$ we get $\{x,z\}\subset A$. To derive a contradiction, assume that $\{x,z\}\not\subset A$. Let $Q=\{x,y,z,s\}$. Taking into account that $p_X(x,y,z,s)\in K\subset \langle\langle A\rangle_E\cup\delta_X(X)\rangle_C$, choose a finite set $H\subset X$ such that $Q\cup A\subset H$ and $p_X(x,y,z,s)\in \langle\langle  A\rangle_E\cup\delta_X(H)\rangle_C$.
By Theorem~\ref{t:FKX1}(4) and Proposition~\ref{p:FembTych}, the $E$-homomorphisms $F_K i_{H,X}:F_\K(H)\to F_\K(X)$ and $F_\K i_{Q,X}:F_\K(Q)\to F_K(X)$ are closed topological embeddings, which allows us to identify the spaces $F_\K(Q)$ and $F_\K(H)$ with closed subspaces in $F_\K(X)$.

Since $\{x,z\}\not\subset A$, either $x\notin A$ or $z\notin A$. If $x\notin A$, then we can choose a retraction $r:H\to Q=\{x,y,z,s\}$ such that $r(A)\subset Q\setminus\{x\}$. Let $f:Q\to Q_0\subset Y$ be the map such that $f(x)=x_0$, $f(y)=y_0$, $f(z)=z_0$ and $f(s)=s_0$. By the definition of a free topologized $E$-algebra, there exists a unique continuous $E$-homomorphism $\bar f:F_\K(Q)\to Y$ such that $\bar f\circ\delta_Q=f$.
Taking into account that $\supp(p_X(x,y,z,s))\subset\{x,y,z,s\}=Q$, we conclude that $$
\begin{aligned}
p_X(x,y,z,s)&=p_Q(x,y,z,s)=F_\K r(p_H(x,y,z,s))\in\\&\in \langle\langle r(A)\rangle_E\cup\delta_H(r(H))\rangle_C\subset \langle \langle Q\setminus\{x\}\rangle_E\cup\delta_H(Q)\}\rangle_C.
\end{aligned}
$$
Applying to this inclusion the $E$-homomorphism $\bar f\colon F_\K(Q)\to Y$, we obtain the inclusion $$(x_0^*{\cdot}y_0){\cdot}(z^*_0{\cdot}s_0)=\bar f(p_Q(x,y,z,s))\in \langle \langle f(Q\setminus\{x\})\rangle_E\cup f(Q)\rangle_C=\langle\langle Q_0\setminus\{x_0\}\rangle_E\cup Q_0\rangle_C,$$ contradicting the choice of the set $Q_0$.

By analogy we can derive a contradiction from the assumption $z\notin A$. So $\{x,z\}\subset A$, which completes the proof of the condition (4) of Definition~\ref{d:xy.zs}.
\end{proof}

\begin{theorem}\label{t:dot*-super} Assume that the signature $E$ is a countable $k_\w$-space and  $\K$ is a $d{+}\HM{+}k_\w$-superstable variety of paratopological $\dot E^*$-algebras such that the identities $x(x^*x)=x$ and $(x^*x)(x^*x)=x^*x=y^*y$ hold in $\K$ and $\K$ contains an $x(s^*y_i)^{<\w}$-mixing and an $(x^*y)(z^*s)$-supermixing paratopological $\dot E^*$-algebras.
 Assume that for a $\mu$-complete Tychonoff $\aleph_k$-space its free paratopological $\dot E^*$-algebra $F_\K (X)$ contains no strong $\Fin^\w$-fans and no $\Fin^{\w_1}$-fans. Then the space $X$ is $k$-homeomorphic to a topological sum $K\oplus D$ of a countable $k_\w$-space $K$ and a discrete space $D$.
\end{theorem}

\begin{proof} By Theorems~\ref{t:FKX1}(1,3), \ref{t:bar-i-bijective}, \ref{t:d+HM+kw}, the functor $F_\K|\Top_{3\frac12}$ is monomorphic, $\II$-regular, strongly bounded and has finite supports.
Since $X$ is an $\aleph_k$-space, each compact subset of $X$ is metrizable and hence sequentially compact. Now the $\mu$-completeness of the space $X$ implies its $\mu_s$-completeness.
By Propositions~\ref{p:mixing->TA}(2,5) and \ref{p:supermix-alg}, the functor-space $F_\K(X)$ is a topological algebra of type $x(s^*y_i)^{<\w}$ and a topological superalgebra of type $(x^*y)(z^*s)$. By Corollary~\ref{c:xy.zs+xsy_w}, the $k$-modification $kX$ of $X$ is a topological sum $K\oplus D$ of a cosmic $k_\w$-space $K$ and a discrete space $D$. Observe that the cosmic space $K\subset kX$ endowed with the original topology of $X$ remains cosmic. By Corollary~\ref{c:cosmic-countable}, $K$ is countable. Therefore $kX$ is a topological sum $K\oplus D$ of the countable $k_\w$-space $K$ and the discrete space $D$.
\end{proof}

\begin{theorem}\label{t:supermix-long}  Assume that the signature $E$ is a countable $k_\w$-space and  $\K$ is a $d{+}\HM{+}k_\w$-superstable variety of paratopological $\dot E^*$-algebras such that the identities $x(x^*x)=x$, $x^*x=y^*y$, $xx^*=yy^*$ hold in $\K$ and $\K$ contains an $x(y^*(s^*z)y)$-mixing, an $(x(s^*z)x^*)(y(s^*z)y^*)$-mixing, and an $(x^*y)(z^*s)$-supermixing paratopological $\dot E^*$-algebras.
Let $X$ be a $\mu$-complete Tychonoff space.
\begin{enumerate}
\item[\textup{1)}] If $X$ is $k^*$-metrizable and $F_\K(X)$ contains no $\Fin^{\w_1}$-fans, then either $X$ is $k$-discrete or $X$ is a countable $\aleph_0$-space.
\item[\textup{2)}] If $X$ is an $\aleph_k$-space and $F_\K(X)$ contains no strong $\Fin^{\w}$-fans and no $\Fin^{\w_1}$-fans, then $X$ is either $k$-discrete or countable and hemicompact.
\end{enumerate}
\end{theorem}

\begin{proof} By Theorems~\ref{t:FKX1}(1,3), \ref{t:bar-i-bijective}, \ref{t:d+HM+kw}, the functor $F_\K|\Top_{3\frac12}$ is monomorphic, $\II$-regular, bounded and has finite supports.
By Propositions~\ref{p:mixing->TA}(3,4) and \ref{p:supermix-alg}, the functor-space $F_\K(X)$ is a topological algebra of types $x(y(s^*z)y^*)$, $(x(s^*z)x^*)(y(s^*z)y^*)$, and a supertopological algebra of type $(x^*y)(z^*s)$. Now it is legal to apply Theorem~\ref{t:long1}  and finish the proof.
\end{proof}

\begin{theorem}\label{t:para-eq} Assume that the signature $E$ is a countable $k_\w$-space and let  $\K$ be a $d{+}\HM{+}k_\w$-superstable variety of paratopological $\dot E^*$-algebras such that the identities $x(x^*x)$, $x^*x=y^*y$, $xx^*=yy^*$ hold in $\K$ and $\K$ contains an $x(y(s^*z)y^*)$-mixing and a $(x(s^*z)x^*)(y(s^*z)y^*)$-mixing paratopological $\dot E^*$-algebras and a $(x^*y)(z^*s)$-supermixing paratopological $\dot E^*$-algebra. For a $\mu$-complete Tychonoff $k_\IR$-space $X$ the following conditions are equivalent:
\begin{enumerate}
\item[\textup{1)}] $X$ is either discrete or a countable $k_\w$-space.
\item[\textup{2)}] $F_\K(X)$ is either discrete or a countable $k_\w$-space.
\item[\textup{3)}] $X$ is a $\bar\aleph_k$-space and $F_\K(X)$ contains no strong $\Fin^{\w}$-fans.
\item[\textup{4)}] $X$ is an $\aleph_k$-space, $F_\K(X)$ contains no strong $\Fin^\w$-fans and no  $\Fin^{\w_1}$-fans.
\end{enumerate}
\end{theorem}

\begin{proof} $(1)\Ra(2)$ If the space $X$ is discrete, then the free topological $E$-algebra $F_\K(X)$ is discrete by  Proposition~\ref{p:FKX-disc}. If $X$ is a countable $k_\w$-space, then $F_\K(X)$ is a countable $k_\w$-space by Lemma~\ref{l:para-kw}.
The implications $(2)\Ra(3,4)$ follows from Proposition~\ref{k-no-Cld-fan} and the implications $(3)\Ra(1)$ and $(4)\Ra(1)$ were proved in Theorem~\ref{t:supermix-long}.
\end{proof}

\chapter{Free objects in some varieties of topologized algebras}\label{ch:magma}

In this section we apply the general results on free topologized $E$-algebras proved in the preceding chapter and study free object in
varieties of topologized magmas and $*$-magmas.

\section{Magmas and topologized magmas}

In this section we discuss the notion of a magma and consider several natural and known classes of magmas.

Following Bourbaki,  by a \index{magma}{\em magma} we understand a set $X$ endowed with a binary operation $\cdot:X\times X\to X$, $\cdot\colon(x,y)\mapsto xy$, called the \index{magma!multiplication of}{\em multiplication}. A subset $A$ of a magma $X$ is called  a \index{submagma}\index{magma!submagma}{\em submagma} if $A{\cdot }A\subset A$ where $A{\cdot}A=\{xy:x,y\in A\}$.

In the language of universal algebras, a magma is a universal algebra of signature $E=E_2=\{\cdot\}$ and submagma is an $E$-subalgebra.
Each magma can be considered as a universal $\dot E$-algebra of signature $E=\{\cdot\}=E_2$ with a distinguished binary operation $\cdot$.

A magma $X$ will be called \index{magma!non-trivial}{\em non-trivial} if $X$ contains more than one point. In the subsequent sections we shall need some results on mixing magmas of various types.

\begin{lemma}\label{l:xy+xsy-mixing} If a magma $X$ contains two points $a,b$ such that $ab=ba=a\ne b=bb$, then the magma $X^2$ is $x{\cdot}y$-mixing and the magma $X^\w$ is $x(sy_i)^{<\w}$-mixing.
\end{lemma}

\begin{proof} To see that the square $X^2$ is an $x{\cdot}y$-mixing algebra, consider the elements $x=(a,b)$ and $y=(b,a)$ and observe that $\langle x\rangle_E=\langle(a,b)\rangle_E\subset X\times\{b\}$ and $\langle y\rangle_E=\langle(b,a)\rangle_E\subset \{b\}\times X$, which implies $x\cdot y=(a,a)\notin \langle x\rangle_E\cup\langle y\rangle_E$.

To see that the countable power $X^\w$ is $x(sy_i)^{<\w}$-mixing, let $s\in\{b\}^\w\subset X^\w$ be the constant function $s:\w\to\{b\}$ and for every $n\in\w$ let $y_n\in\{a,b\}^\w\subset X^\w$ be the unique function such that $y_n^{-1}(a)=\{n\}$. We claim that the elements $x=y_0$, $s$, and $y_1,\dots,y_n$ witness that $X^\w$ is an $x(sy_i)^{<\w}$-mixing magma. Consider the points $z_n=s\cdot y_n$ and $z_{k}=(s\cdot y_k)\cdot z_{k+1}$, $1\le k<n$. We need to show that for any proper subset $A$ of $\{y_0,\dots,y_n\}$ the point $z_0=y_0\cdot z_1$ is not contained in the $E$-hull $\langle A\cup\{s\}\rangle_E$ of the set $A\cup\{s\}$. Let $i\le n$ be a number such that $y_i\notin A$ and observe that $s(i)=y_k(i)=b=bb$ for all $y_k\in A$. This implies that each element $f\in\langle A\cup\{s\}\rangle_E$ has $f(i)=b$. On the other hand, $z_0(i)=a\ne b$ and hence $z_0\notin\langle A\cup\{s\}\rangle_E$.
\end{proof}

Now we define some important classes of magmas. A magma $X$ is called
\begin{itemize}
\item \index{magma!unital}{\em unital} if $X$ has a  two-sided unit  (i.e., a necessarily unique element $1\in X$ such that $x1=x=1x$ for all $x\in X$);
\item \index{magma!commutative}{\em commutative} (or else \index{magma!Abelian}{\em Abelian}) if $xy=yx$ for all $x,y\in X$;
\item \index{magma!idempotent}{\em idempotent} if $xx=x$ for all $x\in X$;
\item \index{magma!a quasigroup} a {\em quasigroup} if for any $a,b\in X$ there are unique elements $x,y\in X$ such that $ax=b$ and $ya=b$;
\item \index{magma!a loop} a {\em loop} if $X$ is a unital quasigroup;
\item a {\em lop} if $X$ is a unital magma and for every $a,b\in X$ there exists a unique $x\in X$ such that $ax=b$;
\item \index{magma!a semigroup}\index{semigroup} a {\em semigroup} if the binary operation of $X$ is {\em associative} in the standard sense: $(xy)z=x(yz)$ for all $x,y,z\in X$;
\item a \index{semigroup!a monoid}{\em monoid} if $X$ is a unital semigroup (i.e., a semigroup possessing a two-sided unit);
\item a \index{magma!a group}\index{semigroup!a group}{\em group} if $X$ is a monoid and for any $x\in X$ there exists an element $x^{-1}\in X$ such that $xx^{-1}=1=x^{-1}x$;
\item a \index{semigroup!regular}{\em regular semigroup} if $X$ is a semigroup and for any $x\in X$ there is an element $x^{-1}\in X$ such that $xx^{-1}x=x$;
\item an \index{semigroup!inverse}{\em inverse semigroup} if $X$ is a semigroup and for every $x\in X$ there exists a unique element $x^{-1}\in X$ such that $xx^{-1}x=x$ and $x^{-1}xx^{-1}$;
\item a \index{semigroup!Clifford}{\em Clifford semigroup} if $X$ is a semigroup such that for every $x\in X$ there exists an element $x^{-1}\in X$ such that $xx^{-1}x=x$, $x^{-1}xx^{-1}=x^{-1}$ and $xx^{-1}=x^{-1}x$;
\item a \index{semigroup!Clifford inverse}{\em Clifford inverse semigroup} if $X$ is both an inverse semigroup and a Clifford semigroup;
\item a \index{semigroup!a band}{\em band} if $X$ is an idempotent semigroup;
\item a \index{semigroup!a semilattice}\index{semilattice}{\em semilattice} if $X$ is a commutative band.
\end{itemize}
These notions relate as follows:
$$
\xymatrix{
\vbox{\hsize57pt\baselineskip8pt \noindent commutative\newline \phantom{iiijii}loop}\ar@{=>}[d]&\vbox{\hsize57pt\baselineskip8pt \noindent commutative\newline \phantom{mm}group}\ar@{=>}[r]\ar@{=>}[d]\ar@{=>}[l]&
\vbox{\hsize78pt\baselineskip8pt\noindent \phantom{M}commutative\newline inverse semigroup}\ar@{=>}[d]&\mbox{semilattice}\ar@{=>}[l]\ar@{=>}[d]\\
\mbox{loop}\ar@{=>}[d]&\mbox{group}\ar@{=>}[r]\ar@{=>}[d]\ar@{=>}[l]&\vbox{\hsize78pt\baselineskip8pt\noindent\phantom{MM}Clifford\newline inverse semigroup}\ar@{=>}[d]\ar@{=>}[rd]&\mbox{band}\ar@{=>}[d]\\
\mbox{lop}\ar@{=>}[d]&\mbox{monoid}\ar@{=>}[ld]\ar@{=>}[rd]&\mbox{inverse semigroup}\ar@{=>}[d]\ar@{=>}[rd]&\vbox{\hsize50pt\baselineskip8pt \noindent\phantom{n}Clifford\newline semigroup}\ar@{=>}[d]\\
\mbox{unital magma}\ar@{=>}[r]&\mbox{magma}&\mbox{semigroup}\ar@{=>}[l]&\vbox{\hsize50pt\baselineskip8pt \noindent\phantom{n}regular\newline semigroup}\ar@{=>}[l]\\
}
$$
\smallskip

\begin{lemma}\label{l:IS-mixing} If $X$ is an inverse semigroup containing more than one element, then the magma $X^2$ is $x{\cdot}y$-mixing and the magma $X^\w$ is $x(sy_i)^{<\w}$-mixing.
\end{lemma}

\begin{proof} It is known \cite[5.1.1]{How95} that a semigroup is inverse if and only if it is regular and all its idempotents commute. If the inverse semigroup $X$ contains more than one idempotent, then we can choose two distinct idempotents $a,b$ and conclude that $ab=ba$ is an idempotent distinct from $a$ or $b$. Without loss of generality we can assume that $ab\ne b$. Replacing the idempotent $a$ by $ab$ we get two idempotents $a,b$ such that $ab=ba=a\ne b=bb$.
If the inverse semigroup $X$ contains a unique idempotent, then $X$ is a group. In this case let $b$ be the unit of $X$ and $a\ne b$ be any element of $X$. In both cases we have found elements $a,b\in X$ such that $ab=ba=a\ne b=bb$. By Lemma~\ref{l:xy+xsy-mixing}, the magma $X^2$ is $x{\cdot}y$-mixing and $X^\w$ is $x(sy_i)^{<\w}$-mixing.
\end{proof}

By a \index{magma!topologized}\index{magma!topological}{\em topologized magma} we understand a magma $X$ endowed with a topology $\tau$. A topologized magma $(X,\tau)$ is called a {\em topological magma} if the multiplication $\cdot\colon X\times X\to X$ is continuous with respect to the topology $\tau$.

A family $\K$ of topologized magmas is called a \index{variety}{\em variety} if $\K$ contains a non-empty topologized magma and $\K$ is closed under taking Tychonoff products, submagmas, and images under topological isomorphisms.

For a variety $\K$ of topologized magmas and a topological space $X$ by $(F_\K(X),\delta_X)$ we denote the free topologized magma over $X$ in the variety $\K$.

\section{Topologized $*$-magmas}

By a \index{$*$-magma}{\em $*$-magma} we understand a magma $X$ endowed with an unary operation ${}^*\colon X\to X$, ${}^*\colon x\mapsto x^*$. A subset $A$ of a $*$-magma $X$ is called a {\em $*$-submagma} if $A^*=\{x^*:x\in A\}\subset A$ and $A\cdot A\subset A$. A map $h:X\to Y$ between two $*$-magmas is called a \index{$*$-homomorphism}{\em $*$-homomorphism} if $h(x\cdot y)=h(x)\cdot h(y)$ and $h(x^*)=h(x)^*$ for every $x,y\in X$.

$*$-Magmas can be thought as universal $E$-algebras of signature $E=E_1\oplus E_2$ where $E_1=\{{}^*\}$ and $E_2=\{\cdot\}$. Then $*$-submagmas are $E$-subalgebras and $*$-homomorphisms are $E$-homomorphisms.

A $*$-magma is called a \index{$*$-semigroup}{\em $*$-semigroup} if the binary operation $\cdot\colon X\times X\to X$ is associative. 

A $*$-semigroup $X$ is called
\begin{itemize}
\item \index{$*$-semigroup!inverse}an {\em inverse semigroup} if $(x^*)^*=x$, $xx^*x=x$ and $(xx^*)(yy^*)=(yy^*)(xx^*)$ for every $x,y\in X$;
\item a \index{$*$-semigroup!Clifford}{\em Clifford semigroup} if $(x^*)^*=x$, $xx^*x=x$ and $x^*x=xx^*$ for all $x\in X$;
\item a \index{$*$-semigroup!band}{\em band} if $x^*=x=xx$ for all $x\in X$;
\item a \index{$*$-semigroup!semilattice}{\em semilattice} if $X$ is a commutative band;
\item a {\em group} if $X$ is a non-empty Clifford semigroup such that $x^*x=y^*y$ for all $x,y\in X$.
\end{itemize}

By \cite[Ch.5]{How95}, inverse semigroups can be equivalently defined as $*$-semigroups such that $(x^*)^*=x$, $xx^*x=x$ and $(xy)^*=y^*x^*$ for all $x,y\in X$.

\begin{lemma}\label{l:mixing*} Let $X$ be a $*$-magma.
\begin{enumerate}
\item[\textup{1)}] The magma $X^2$ is $(x^*y)(z^*s)$-mixing if $X$ contains two points $a,e$ such that $e^*=e=ee\notin\{e(ea),(ea)e\}$.
\item[\textup{2)}] The magma $X^3$ is $x(y(s^*z)y^*)$-mixing if $X$ contains three points $b,c,e$ such that $e^*=e=ee$, $ec=ce=c\ne e$, and $e((bc)b^*)\ne c$.
\item[\textup{3)}] The magma $X^3$ is $(x(s^*z)x^*)(y(s^*z)y^*)$-mixing if $X$ contains three points $b,c,e$ such that $e^*=e=ee$, $ec=ce=c$, and $c((bc)b^*)\ne cc\ne ((bc)b^*)c$.
\item[\textup{4)}] The magma $X^\w$ is $x(s^*y_i)^{<\w}$-mixing if $X$ contains two points $a,e$ such that $ae=ea=a\ne e=ee=e^*$.
\end{enumerate}
\end{lemma}

\begin{proof}
1. Assume that the $*$-magma $X$ contains  two points $a,e$ such that $e^*=e=ee\notin\{e(ea),(ea)e\}$.  We claim that the elements $x=z=(e,e)$, $y=(a,e)$ and $s=(e,a)$ witness that the $*$-magma $X^2$ is $(x^*y)(z^*s)$-mixing.
Indeed, the equality $ee=e=e^*$ implies that $\langle \{x,y,z\}\rangle_E\subset X\times\{e\}$ and $\langle\{x,z,s\}\rangle_E\subset \{e\}\times X$. On the other hand, $(x^*y)(z^*s)=((e^*a)(e^*e),((e^*e)(e^*a))=((ea)e,e(ea))\notin (X\times\{e\})\cup(\{e\}\times X)$.
\smallskip

2. Assume that the magma $X$ contains three points $b,c,e$ such that $e^*=e=ee$, $ec=ce=c\ne e$, and $e((bc)b^*)\ne c$. We claim that the elements $x=(c,e,e)$, $y=(e,e,b)$, $z=(e,e,e)$, $s=(e,c,c)$ witness that the $*$-magma $X^3$ is $x(y(s^*z)y^*)$-mixing. Observe that $\langle\{y,z,s\}\rangle_E\subset \{e\}\times X^2$ and $\langle\{x,z,s\}\rangle_E\subset X\times\Delta$ where $\Delta=\{(u,v)\in X^2:u=v\}$. On the other hand,
$$
\begin{aligned}
x((y(s^*z))y^*)&=\big(c((e(e^*e)e^*),e^*(e(e^*c)e^*),e^*((b(e^*c))b^*\big)=\big(c,c,e((bc)b^*)\big)\notin\\
 &\notin(\{e\}\times X^2)\cup(X\times\Delta).
\end{aligned}
 $$

3. Assume that the magma $X$ contains three points $b,c,e$ such that $e^*=e=ee$, $ec=ce=c\ne e$, and $((bc)b^*)c\ne cc\ne c((bc)b^*)$. We claim that the elements $$x=(b,e,e),\;\;y=(e,e,b),\;\;z=(c,c,c),\;\;s=(e,e,e)$$ witness that the $*$-magma $X^3$ is $(x(s^*z)x^*)(y(s^*z)y^*)$-mixing. Observe that $\langle\{y,z,s\}\rangle_E\subset \Delta\times X$ and $\langle\{x,z,s\}\rangle_E\subset X\times\Delta$ where $\Delta=\{(u,v)\in X^2:u=v\}$. On the other hand,
$$
\begin{aligned}
&((x(s^*z))x^*)((y(s^*z))y^*)=\\
&=\big(((b(e^*c)b^*)((e^*(e^*c))e^*),((e(e^*c))e^*)((e(e^*c))e^*),
((e(e^*c))e^*)((b(e^*c))b^*)\big)=\\
&=\big((bc)b^*)c,cc,c((bc)b^*)\notin(\Delta\times X)\cup (X\times\Delta).
 \end{aligned}
 $$

4. If the $*$-magma $X$ contains two points $a,e$ such that $ae=ea=a\ne e=ee=e^*$, then we by analogy with Lemma~\ref{l:xy+xsy-mixing} we can prove that the $*$-magma $X^\w$ is $x(s^*y_i)^{<\w}$-mixing.
\end{proof}

 A \index{$*$-magma!topologized}{\em topologized $*$-magma} is a $*$-magma endowed with a topology. A topologized $*$-magma is called a \index{$*$-magma!paratopological}{\em paratopological $*$-magma} if its binary operation is continuous. A paratopological $*$-magma is called a \index{$*$-magma!topological}{\em topological $*$-magma} if its unary operation is continuous. In particular, a {\em topological inverse semigroup} is an inverse semigroup $X$ endowed with a topology making the binary and unary operations continuous.

\section{Free topological semigroups}

In this section we characterize $\aleph$-spaces $X$ whose free topological semigroups $\Sem(X)$ are $k$-spaces. A \index{topological semigroup}{\em topological semigroup} is a topological space $X$ endowed with a continuous associative binary operation $\cdot\colon X\times X\to X$.

A \index{topological semigroup!free}\index{free!topological semigroup}{\em free topological semigroup} over a topological space $X$ is a pair $(\Sem(X),\delta_X)$ consisting of a topological semigroup $\Sem(X)$ and a continuous map $\delta_X:X\to\Sem(X)$ such that for every continuous map $f:X\to S$ to a topological semigroup $S$ there is a continuous semigroup homomorphism $\bar f:\Sem(X)\to S$ such that $\bar f\circ\delta_X=f$.

A free topological semigroup $\Sem(X)$ can be identified with the direct sum $X^{<\IN}=\oplus_{n\in\IN}X^n$ of finite powers of $X$, endowed with the associative operation of concatenation of words. The canonical embedding $\delta_X:X\to X^1$ assigns to each point $x\in X$ the constant function $1\to\{x\}\subset X$.

\begin{theorem}\label{t:Sem}  For every Tychonoff space $X$ the following conditions are equivalent:
\begin{enumerate}
\item[\textup{1)}] $X$ either is metrizable or is a topological sum of cosmic $k_\w$-spaces.
\item[\textup{2)}] The free topological semigroup $\Sem(X)$ either is metrizable or is a topological sum of cosmic $k_\w$-spaces.
\item[\textup{3)}] $X$ is $k^*$-metrizable and $\Sem(X)$ is a $k$-space.
\item[\textup{4)}] $X$ is $k^*$-metrizable and $\Sem(X)$ is a Ascoli space.
\item[\textup{5)}] $X$ is an $\aleph$-space and $\Sem(X)$ contains no strong $\Fin^\w$-fans.
\item[\textup{6)}] $X$ is $k^*$-metrizable and $\Sem(X)$ contains no strong $\Fin^\w$-fan and no $\Fin^{\w_1}$-fans.
\end{enumerate}
\end{theorem}

\begin{proof} The implication $(1)\Ra(2)$ follows from the identification of the free topological semigroup $\Sem(X)$ with the topological sum $X^{<\IN}=\oplus_{n\in\IN}X^n$.
The implications $(2)\Ra(5,6)$ follow from Proposition~\ref{k-no-Cld-fan} and Corollary 7.5 of \cite{BBK} (saying that each $\aleph$-space is $k^*$-metrizable).
The implications $(5)\Ra(1)$ and $(6)\Ra(1)$ follows from Theorem~\ref{t:dotE} and Lemma~\ref{l:IS-mixing}. The implications $(1,2)\Ra(3)\Ra(4)\Ra(5)$ are trivial or follow from Proposition~\ref{k-no-Cld-fan} and Corollary~\ref{c:A->noFan}.
\end{proof}

It would be nice to generalize Theorem~\ref{t:Sem} onto some other varieties of topological magmas.

\begin{problem} Let $\K$ be a $d{+}\HM{+}k_\w$-stable variety of topological magmas. Assume that $\K$ contains the additive group $\IN$.  Is the space $F_\K(X)$ metrizable for any metrizable (separable) space $X$?
\end{problem}

\section{Free topological semilattices}

In this section we characterize $\aleph$-spaces $X$ whose free topological semilattices $\Sl(X)$ are $k$-spaces.
We start with a more general result treating free objects in some classes of idempotent topological magmas.

We recall that a magma $X$ is called \index{magma!idempotent}{\em idempotent} if $xx=x$ for all $x\in X$. It is easy to see that idempotent topological magmas form a variety of topological magmas containing all topological semilattices. A \index{topological semilattice}{\em topological semilattice} is an idempotent commutative topological semigroup. A semilattice $X$ is called {\em non-trivial} if $X$ contains more than one point.

Corollary~\ref{c:i-xyy_w} and Lemma~\ref{l:xy+xsy-mixing} imply:

\begin{corollary}\label{c:xyy_w3} Assume that a variety $\K$ of idempotent topological magmas is  $d{+}\HM{+}k_\w$-stable and $\K$ contains a non-trivial semilattice.
Assume that for a $\mu$-complete Tychonoff space $X$ the free topological magma $F_\K(X)$ contains no strong $\Fin^\w$-fans.
\begin{enumerate}
\item[\textup{1)}] If $X$ is a $\bar\aleph_k$-space, then the $k$-modification of $X$ is a topological sum of cosmic $k_\w$-spaces.
\item[\textup{2)}] If $X$ is a $\bar\aleph_k$-$k_\IR$-space, then the space $X$ is a topological sum of cosmic $k_\w$-spaces.
\item[\textup{3)}] If $X$ is an $\aleph_0$-space, then $X$ is hemicompact.
\item[\textup{4)}] If $X$ is cosmic, then $X$ is $\sigma$-compact.
\end{enumerate}
\end{corollary}

\begin{theorem}\label{t:i-magma} Assume that a variety $\K$ of idempotent topological magmas is  $d{+}\HM{+}k_\w$-stable and $\K$ contains a non-trivial semilattice.
 For every $\mu$-complete Tychonoff $k_\IR$-space $X$ the following conditions are equivalent:
\begin{enumerate}
\item[\textup{1)}] $X$ is a topological sum of cosmic $k_\w$-spaces.
\item[\textup{2)}] $F_\K(X)$ is a topological sum of cosmic $k_\w$-spaces.
\item[\textup{3)}] $X$ is an $\bar\aleph_k$-space and $F_\K(X)$ is a $k$-space.
\item[\textup{4)}] $X$ is an $\bar\aleph_k$-space and $F_\K(X)$ is an Ascoli space.
\item[\textup{5)}] $X$ is an $\bar\aleph_k$-space and $F_\K(X)$ contains no strong $\Fin^\w$-fan.
\end{enumerate}
\end{theorem}

\begin{proof} To prove that $(1)\Ra(2)$, assume that $X$ is a topological sum of cosmic $k_\w$-spaces. Since the variety $\K$ is $d$-stable and contains a non-trivial semilattice, it contain the discrete two-element semilattice $\{0,1\}$ endowed with the operation of minimum. To show that $F_\K(X)$ is a topological sum of cosmic $k_\w$-spaces, it suffices to prove that each element $a\in F_\K(X)$ has a closed-and-open neighborhood, which is a cosmic $k_\w$-space. Since the functor $F_\K$ has finite supports and $X$ is a topological sum of cosmic $k_\w$-spaces, there exists a closed-and-open cosmic $k_\w$-space $A\subset X$ containing $\supp(a)$.  Let $f:X\to \{0,1\}$ be the map defined by $f^{-1}(1)=A$. Since $\{0,1\}\in\K$, there exists a continuous magma homomorphism $\bar f:F_\K(X)\to \{0,1\}$ such that $\bar f\circ\delta_X=f$.

Since $A$ is a retract of $X$, the map $F_\K i_{A,X}:F_\K(A)\to F_\K(X)$ is a topological embedding, which allows us to identify the topological magma $F_\K(A)$ with the submagma of $F_\K(X)$. By Theorem~\ref{t:kw}, $F_\K(A)$ is a cosmic $k_\w$-space.

Taking into account that $F_\K(X)=\bigcup_{n\in\w}E^n[\delta_X(X)]$ where $E=\{\cdot\}=E_2$, we can prove that $\bar f^{-1}(1)=\bigcup_{n\in\w}E^n[\delta_X(A)]=F_\K i_{A,X}(F_\K(A))$ is a closed-and-open cosmic $k_\w$-space containing the point $a$.
\smallskip

The implications $(2)\Ra(3)\Ra(4)\Ra(5)\Ra(1)$ follow from Corollary~\ref{c:delta-closed}, Proposition~\ref{k-no-Cld-fan} and Corollaries~\ref{c:A->noFan}, \ref{c:xyy_w3}.
\end{proof}

For a topological space $X$ by $\Sl(X)$ we shall denote its \index{topological semilattice!free}\index{free!topological semilattice}{\em free topological semilattice}. So, $\Sl(X)=F_\K(X)$ where $\K$ is the variety of all topological semilattices.

\begin{theorem}\label{t:SL}
 For every $\mu$-complete Tychonoff $k_\IR$-space $X$ the following conditions are equivalent:
\begin{enumerate}
\item[\textup{1)}] $X$ is a topological sum of cosmic $k_\w$-spaces;
\item[\textup{2)}] $\Sl(X)$ is a topological sum of cosmic $k_\w$-spaces;
\item[\textup{3)}] $X$ is an $\bar\aleph_k$-space and $\Sl(X)$ is a $k$-space;
\item[\textup{4)}] $X$ is an $\bar\aleph_k$-space and $\Sl(X)$ is a Ascoli space;
\item[\textup{5)}] $X$ is an $\bar\aleph_k$-space and $\Sl(X)$ contains no strong $\Fin^\w$-fans.
\end{enumerate}
\end{theorem}

\section{Free Lawson semilattices}

In Theorem~\ref{t:SL} we characterized $\bar\aleph_k$-spaces $X$ whose free topological semilattices $\Sl(X)$ contain no strong $\Fin^\w$-fans. In this section we obtain a similar characterization for free Lawson semilattices.

By a \index{topological semilattice!Lawson}{\em Lawson semilattice} $X$ we understand a topological semilattice which has a base of the topology consisting of open subsemilattices of $X$. It is easy to see that Lawson semilattices form a $d$-stable variety of idempotent topological magmas. However, this variety is neither $\HM$-stable nor $k_\w$-stable. So, most of results proved in  Sections~\ref{ch:free} cannot be applied.

A typical example of a Lawson semilattice is the hyperspace $\Cld(X)$ of closed subsets of a topological space $X$ endowed with the Vietoris topology and the semilattice operation of union. A base of the Vietoris topology on $\Cld(X)$ consists of the subsets
$$\langle U_1,\dots,U_n\rangle=\big\{F\subset X:F\subset\bigcup_{i=1}^nU_i\mbox{ \ and \ }\forall i\le n,\;F\cap U_i\ne\emptyset\big\}$$where $U_1,\dots,U_n$ are open sets in $X$.

For a topological space $X$ its \index{free!Lawson semilattice}\index{topological semilattice!free Lawson}{\em free Lawson semilattice} is a pair $(\Law(X),\delta_X)$ consisting of a Lawson semilattice $\Law(X)$ and a continuous map $\delta_X:X\to \Law(X)$ such that for any continuous map $f:X\to Y$ to a topological semilattice $Y$ there exists a unique continuous semilattice homomorphism $\bar f:\Law(X)\to Y$ such that $\bar f\circ\delta_X=f$.

Since Lawson semilattices form a variety of topological magmas, for any topological space $X$ a free Lawson semilattice $(\Law(X),\delta_X)$ exists and is unique up to a topological isomorphism. To obtain a concrete construction of a free Lawson semilattice space, take the set $\Law(X)$ of non-empty finite subsets of $X$ and endow it with the Vietoris topology. Next, consider the map $\delta_X:X\to\Fin(X)$, $\delta_X:x\mapsto \{x\}$. It can be shown that the pair $(\Fin(X),\delta_X)$ is a free Lawson semilattice over $X$.

The construction of the free topological semilattice determines a functor $\Law:\Top\to\Top$ in the category of topological spaces. Known properties of the Vietoris topology imply the following proposition.

\begin{proposition} The functor of free Lawson semilattice $\Law:\Top\to\Top$ has the following properties:
\begin{enumerate}
\item[\textup{1)}] $\Law$ can be completes to a monad $(\Law,\delta,\mu)$ in the category $\Top$;
\item[\textup{2)}] $\Law$ is monomorphic and preserves closed embeddings of topological spaces;
\item[\textup{3)}] $\Law$ has finite supports;
\item[\textup{4)}] $\Law$ is $\II$-regular;
\item[\textup{5)}] $\Law$ is bounded.
\end{enumerate}
\end{proposition}

The boundedness of the functor $\Law$ is proved in the following (known) lemma.

\begin{lemma}\label{l:comp-un}  For any topological space $X$ and any compact subset $\C\subset \Law(X)$ the set $\bigcup\C$ is compact in $X$. This implies that the functor $\Law:\Top\to\Top$ is bounded.
\end{lemma}

\begin{proof} Given a cover $\U$ of $\bigcup\C$ by open subsets of $X$, we should find a finite subcover of $\V\subset\U$ of $\bigcup\C$. Without loss of generality, the family $\U$ is closed under finite unions. Then any (compact) set $C\in\C$ is contained in some set $U_C\in\U$ and the set $\langle U_C\rangle\cap\Law(X)=\{K\in\Law(X):K\subset U_C\}$ is an open neighborhood of $C$ in $\C$. By the compactness of $\C$ the open cover $\{\langle U_C\rangle:C\in\C\}$ contains a finite subcover $\{\langle U_C\rangle:C\in\F\}$ (here $\F$ is a finite subset of $\C$). Then $\{U_C:C\in\F\}\subset\U$ is a finite subcover of $\bigcup\C$.
\end{proof}


\begin{lemma}\label{l:Law-fan} If a functionally Hausdorff space $X$ contains a convergent sequence $S$, then its free Lawson semilattice $\Law(X)$ contains a strong $D_\w$-cofan. If $X$ contains a (strong) $\bar S^\w$-fan, then $\Law(X)$ contains a (strong) $\Fin^\w$-fan.
\end{lemma}

\begin{proof} We identify the free Lawson semilattice $\Law(X)$ with the hyperspace of non-empty finite subsets of $X$.


Let $x$ be the limit point of the convergent sequence $S$ and $\{x_n\}_{n\in\w}$ be an injective enumeration of the set $S\setminus\{x\}$. Let $i_X:X\to\beta X$ be the (injective) map of $X$ into its Stone-\v Cech compactification. The topology $\tau_{3\frac12}$ on $X$ inherited from $\beta X$ turns $X$ into a Tychonoff space. The injectivity of the map $i_X$ implies that each point $x_n$ has a $\IR$-open neighborhood $U_n\subset X$ such that the family $(\overline{U}_n)_{n\in\w}$ is disjoint.


For every $n,m\in\w$ consider the finite set $u_{n,m}=\{x_{k}:n\le k\le n+m\}\in\Law(X)$ and its open neighborhood $\U_{n,m}=\langle X,U_n,\dots,U_{n+m}\rangle\cap\Law(X)$ in $\Law(X)$. It can be shown that the set $U_{n,m}$ is open in the Tychonoff space $\Law(X,\tau_{3\frac12})$ and hence is $\IR$-open in $\Law(X)$.

We claim that for every $n\in\w$ the set $D_n=\{u_{n,m}\}_{m\in\w}$ is strongly compact-finite in $X$. More precisely, the family $(\U_{n,m})_{m\in\w}$ is compact-finite in $\Law(X)$.
To derive a contradiction, assume that some compact set $K\subset\Law(X)$ meets infinitely many sets $\U_{n,m}$, $m\in\w$.
Then there is an increasing number sequence $(m_i)_{i\in\w}$ such that for every $i\in\w$ the intersection $\U_{n,{m_i}}\cap K$ contains some finite set $b_i\in K$. The definition of the set $\U_{n,m_i}$ implies that for every $i$ the set $b_i$ meets all sets $U_{n,m}$ for $n\le m\le n+m_i$. By the compactness of $K$ the sequence $(b_i)_{i\in\w}$ has an accumulation point $b_\infty\in K$, which is a finite subset $X$. Find $j\ge n$ such that $b_\infty\cap \bar U_{m_i}=\emptyset$ for all $i\ge j$. Since $b_\infty$ is an accumulation point of the sequence $(b_i)$, the neighborhood $\langle X\setminus \overline{U}_{m_j}\rangle$ of $b_\infty$ contains some set $b_i$ for $i\ge j$. Then $b_i\cap \overline{U}_{m_j}=\emptyset$, which contradicts the choice of $b_i$.

Therefore the family $(\U_{n,m})_{m\in\w}$ is compact-finite and the set $D_n$ is strongly compact-finite in $\Law(X)$.
The definition of the Vietoris topology on $\Cld(X)$ ensures that the sequence $(D_n)_{n\in\w}$ converges to the singleton $\{x\}$, which means that $(D_n)_{n\in\w}$ is a strong $D_\w$-cofan in $\Law(X)$.

Now assume that the space $X$ contains a (strong) $\bar S^\w$-fan $(S_n)_{n\in\w}$. Since the set $\{\lim S_n\}_{n\in\w}$ has compact closure in $X$, we can replace the family $(S_n)_{n\in\w}$ by a suitable subfamily and assume that $(S_n)_{n\in\w}$ is disjoint and the union $\bigcup_{n\in\w}S_n$ is disjoint with $\bar S\cup\{\lim S_n\}_{n\in\w}$. For every $n\in\w$ choose a sequence $(y_{n,m})_{m\in\w}$ of pairwise distinct points in the sequence $S_n$.

For every $n,m\in\w$ consider the finite set $w_{n,m}=\{y_{n,m}\}\cup u_{n,m}$. We claim that $(\{w_{n,m}\})_{n,m\in\w}$ is a (strong) $\Fin^\w$-fan in $\Law(X)$.

First we check that the family $(\{w_{n,m}\})_{n,m\in\w}$ is not locally finite in $\Law(X)$. Since the set $\{\lim S_n\}_{n\in\w}$ has compact closure in $X$, there exists a point $y\in X$ such that every neighborhood $O_y\subset X$ of $y$ contains infinitely many points $\lim S_n$, $n\in\w$. We claim that the family $(\{w_{n,m}\})_{n,m\in\w}$ is not locally finite at the set $\{x,y\}\in\Law(X)$. Fix any neighborhood $O_{\{x,y\}}\subset\Law(X)$ and find open neighborhoods $O_x,O_y\subset X$ of the points $x,y$ such that $\langle O_x,O_y\rangle\subset O_{\{x,y\}}$. Find $n\in\w$ such that $\{x_k\}_{k\ge n}\subset O_x$.
Next, find $n\ge k$ such that $\lim S_n\in O_y$ and choose a number $m\in\w$ with $y_{n,m}\in O_y$. Then the set $w_{n,m}=\{y_{n,m}\}\cup \{x_{k}\}^{n+m}_{k=n}\in \langle O_x,O_y\rangle\subset O_{\{x,y\}}$.

Next, we show that the family $(\{w_{n,m}\})_{n,m\in\w}$ is compact-finite in $\Law(X)$. Given any compact subset $K\subset \Law(X)$, apply Lemma~\ref{l:comp-un} and conclude that the union $\bigcup K$ is a compact subset of $X$. Since the family $(S_n)_{n\in\w}$ is compact-finite, there is a number $n_0\in\w$ such that $S_n\cap \bigcup K=\emptyset$ for all $n\ge n_0$. Observe that $w_{n,m}\in \U_{n,m}$ for every $n,m\in\w$. The compact-finiteness of the families $(\U_{n,m})_{m\in\w}$, $n\le n_0$, ensures that for every $n\le n_0$ the set $\{m\in \w:\U_{n,m}\cap K\ne\emptyset\}$ is finite. Consequently, there is a number $m_0\in\w$ such that $w_{n,m}\notin K$ for all $n\le n_0$ and $m\ge m_0$. Now we see that the set $\{(n,m)\in\w\times\w:w_{n,m}\in K\}\subset[0,n_0]\times [0,m_0]$ is finite and hence the family $(\{w_{n,m}\})_{n,m\in\w}$ is compact-finite in $\Law(X)$, so is a $\Fin^\w$-fan in $\Law(X)$.

Finally, assuming that the $\bar S^\w$-fan $(S_n)_{n\in\w}$ is strong, we shall show that the $\Fin^\w$-fan $\big(\{w_{n,m}\}\big)_{n,m\in\w}$ is strong.
Since the fan $(S_n)_{n\in\w}$ is strong, each sequence $S_n$ has an $\IR$-open neighborhood $V_n\subset X$ such that the family $(V_n)_{n\in\w}$ is compact-finite in $X$. It can be shown that for every $n,m\in\w$ the set $\W_{n,m}=\langle V_n,U_n,\dots,U_{n+m}\rangle\cap \Law(X)$ is an $\IR$-open neighborhood of $w_{n,m}$ in $\Law(X)$. Repeating the above argument, we can prove that the family $(\W_{n,m})_{n,m\in\w}$ is compact-finite in $\Law(X)$ and hence the $\Fin^\w$-fan $(\{w_{n,m}\})_{n,m\in\w}$ is strong.
\end{proof}

Now we are able to prove the main result of this subsection.

\begin{theorem}\label{t:Law} For a functionally Hausdorff space $X$ the following conditions are equivalent:
\begin{enumerate}
\item[\textup{1)}] $X$ is metrizable.
\item[\textup{2)}] $\Law(X)$ is metrizable.
\item[\textup{3)}] $X$ is a $k^*$-metrizable $k$-space and $\Law(X)$ contains no $\Fin^\w$-fan.
\item[\textup{4)}] $X$ is a Tychonoff sequential $\aleph$-space and $\Law(X)$ contains no strong $\Fin^\w$-fan.
\item[\textup{5)}] $X$ is $k^*$-metrizable and $\Law(X)$ is a $k$-space.
\item[\textup{6)}] $X$ is $k^*$-metrizable and $\Law(X)$ is Ascoli.
\end{enumerate}
\end{theorem}

\begin{proof} The implication $(1)\Ra(2)$ follows from the metrizability of the Vietoris topology (by the Hausdorff metric) on the hyperspace $\Law(X)$ of finite subsets of a metric space.
The implication $(2)\Ra(3,4)$ follows from Proposition~\ref{k-no-Cld-fan} and the implications $(3)\Ra(1)$ and $(4)\Ra(1)$ follow from Lemma~\ref{l:Law-fan} and Theorem~\ref{t:Sfan}. The implications $(1,2)\Ra(5)\Ra(6)\Ra(3)$ are trivial.
\end{proof}

\section{Free topological inverse semigroups}

This section is devoted to studying free topological inverse (Clifford, Abelian) semigroups.
However we start with a general result related to suitable varieties of paratopological $*$-magmas.

\begin{theorem}\label{t:IS-gen} Let $\K$ be a $d{+}\HM{+}k_\w$-stable variety of paratopological $*$-magmas such that the identities $x(x^*x)=x$ and $(x^*x)(x^*x)=x^*x$ hold in $\K$ and $\K$ contains a non-trivial inverse semigroup. Assume that for a $\mu$-complete Tychonoff space $X$ the free paratopological $*$-magma $F_\K(X)$ contains no strong $\Fin^\w$-fan.
\begin{enumerate}
\item[\textup{1)}] If $X$ is a $\bar\aleph_k$-space, then the $k$-modification of $X$ is a topological sum of cosmic $k_\w$-spaces.
\item[\textup{2)}] If $X$ is a $\bar\aleph_k$-space and a $k_\IR$-space, then the space $X$ is a topological sum of cosmic $k_\w$-spaces.
\item[\textup{3)}] If $X$ is an $\aleph_0$-space, then $X$ is hemicompact.
\item[\textup{4)}] If $X$ is cosmic, then $X$ is $\sigma$-compact.
\end{enumerate}
\end{theorem}

\begin{proof} In order to apply Theorem~\ref{t:xsy_w-mix}, we need to show that the variety $\K$ contains an $x(s^*y_i)^{<\w}$-mixing $*$-magma. By our assumption, $\K$ contains an inverse semigroup $Y$ of cardinality $|Y|>1$. Repeating the proof of Lemma~\ref{l:IS-mixing}, we can find two elements $a,e\in Y$ such that $e^*=e=ee$ and $ea=ae=a\ne e$. By Lemma~\ref{l:mixing*}(4), the inverse semigroup $Y^\w\in\K$ is $x(s^*y_i)^{<\w}$-mixing. So, we can apply Theorem~\ref{t:xsy_w-mix} and finish the proof.
\end{proof}

By analogy with Theorem~\ref{t:i-magma} we can prove the following characterization.

\begin{theorem}\label{t:FK-IS} Let $\K$ be a $d{+}\HM{+}k_\w$-stable variety of topological $*$-magmas such that the identities $x(x^*x)=x$ and $(x^*x)(x^*x)=x^*x$ hold in $\K$ and $\K$ contains a non-trivial semilattice.
For a $\mu$-complete Tychonoff $k_\IR$-space $X$ the following conditions are equivalent:
\begin{enumerate}
\item[\textup{1)}] $X$ is a topological sum of cosmic $k_\w$-spaces.
\item[\textup{2)}] $F_\K(X)$ is a topological sum of cosmic $k_\w$-spaces.
\item[\textup{3)}] $X$ is an $\bar\aleph_k$-space and $F_\K(X)$ is a $k$-space.
\item[\textup{4)}] $X$ is an $\bar\aleph_k$-space and $F_\K(X)$ is an Ascoli space.
\item[\textup{5)}] $X$ is an $\bar\aleph_k$-space and $F_\K(X)$ contains no strong $\Fin^\w$-fan.
\end{enumerate}
\end{theorem}

By $\IS$, $\CS$, $\ICS$, $\IAS$ we shall denote the functor $F_\K:\Top\to\Top$ of free topological $*$-magma for the variety $\K$ of topological inverse semigroups, topological Clifford semigroups, topological inverse Clifford semigroups, topological inverse Abelian semigroups, respectively. The functors $\IS$, $\ICS$, $\IAS$ were studied in \cite{BGG}, where a partial case of the following theorem was proved. \index{free!topological inverse semigroup} \index{free!topological Clifford semigroup}
\index{free!topological inverse Clifford semigroup} \index{free!topological inverse Abelian semigroup}

\begin{theorem}\label{t:IS+} Let $F$ be one of the functors $\IS$, $\CS$, $\ICS$, $\IAS$. For a $\mu$-complete Tychonoff $k_\IR$-space $X$ the following conditions are equivalent:
\begin{enumerate}
\item[\textup{1)}] $X$ is a topological sum of cosmic $k_\w$-spaces.
\item[\textup{2)}] $F(X)$ is a topological sum of cosmic $k_\w$-spaces.
\item[\textup{3)}] $X$ is an $\bar\aleph_k$-space and $F(X)$ is a $k$-space.
\item[\textup{4)}] $X$ is an $\bar\aleph_k$-space and $F(X)$ is an Ascoli space.
\item[\textup{5)}] $X$ is an $\bar\aleph_k$-space and $F(X)$ contains no strong $\Fin^\w$-fan.
\end{enumerate}
\end{theorem}

This theorem follows from Theorem~\ref{t:FK-IS} applied to suitable varieties of topological inverse or Clifford semigroups.

\section{Free topological Abelian groups}

The section is devoted to studying the $k$-space and Ascoli properties in free topological Abelian groups.

\begin{theorem}\label{t:FA-gen} Let $\K$ be a $d{+}\HM{+}k_\w$-stable variety of paratopological $*$-magmas such that the identities $x(x^*x)=x$ and $(x^*x)(x^*x)=x^*x=y^*y$ hold in $\K$ and $\K$ contains a non-trivial group. Assume that for a $\mu$-complete Tychonoff space $X$ the free paratopological $*$-magma $F_\K(X)$ contains no strong $\Fin^\w$-fan.
\begin{enumerate}
\item[\textup{1)}] If $X$ is a $\bar \aleph_k$-$k_\IR$-space, then $X$ is a topological sum $K\oplus D$ of a cosmic $k_\w$-space $K$ and a discrete space $D$.
\item[\textup{2)}] If $X$ is an $\aleph_k$-space and $FX$ contains no $\Fin^{\w_1}$-fan, then $X$ is $k$-homeomorphic to a topological sum $K\oplus D$ of a cosmic $k_\w$-space $K$ and a discrete space $D$.
\end{enumerate}
\end{theorem}

\begin{proof} In order to apply Theorem~\ref{t:xsy+xy.zs}, we need to prove that the variety $\K$ contains an $x(s^*y_i)^{<\w}$-mixing and an $(x^*y)(s^*z)$-mixing paratopological $*$-magmas. By our assumption, the variety $\K$ contains a non-trivial group $Y$. Let $e$ be the unit of $Y$ and $a\in Y\setminus\{e\}$ be any element. Taking into account that $ae=ea=a\ne e=ee=e^*$ and
$e^*=e=ee\notin\{e(ea),(ea)e\}=\{a\}$ and applying Lemma~\ref{l:mixing*}, we conclude that the $*$-magma $Y^\w\in\K$ is $x(s^*y_i)^{<\w}$-mixing and the magma $Y^2\in\K$ is $(x^*y)(z^*s)$-mixing. So, we can apply  Theorem~\ref{t:xsy+xy.zs} and finish the proof.
\end{proof}

\begin{proposition}\label{p:product-abelian} Let $\K$ be a variety of topologized $*$-magmas such that
 the identities $(x^*)^*=x$, $x(x^*x)=x$, $(x^*x)(x^*x)=x^*x=y^*y$, $(xy)^*=y^*x^*$, $xy=yx$, $(xy)(uv)=(xu)(yv)$ hold in $\K$. Then for any non-empty topological spaces $X,Y$ and their topological sum $Z=X\oplus Y$, the map
 $$h:F_\K(X)\times F_\K(Y)\to F_\K(Z),\;\;h:(a,b)\mapsto F_\K i_{X,Z}(a)\cdot F_\K i_{Y,Z}(b),$$
 is a topological isomorphism of topologized $*$-magmas $F_\K(X)\times F_\K(Y)$ and $F_\K(Z)$.
 \end{proposition}

 \begin{proof} By our assumption, the identities $x^*x=y^*y$, $xy=yx$ and $x(x^*x)=x$ hold in $F_\K(X)$, so we can consider the element $1_X=x^*x$ where $x\in F_\K(X)$ is any element of $F_\K(X)$ and conclude that $1_X\cdot x=x\cdot 1_X=x$ for any $x\in F_\K(X)$. This means that $1_X$ is a two-sided unit of the magma $F_\K(X)$. Taking into account the identities $(xy)^*=(y^*x^*)$ and $(x^*)^*=x$, we conclude that $1_X^*=(x^*x)^*=x^*x^{**}=x^*x=1_X$.

Consider the map $r_X:Z\to F_\K(X)$ such that $r_X|X=\delta_X$ and $r_X(Y)=\{1_X\}$. By the definition of the free $*$-magma $F_\K(Z)$, there exists a unique continuous $*$-homomorphism $\bar r_X:F_\K(Z)\to F_\K(X)$ such that $\bar r_X\circ \delta_Z=r_K$. It follows that $\bar r_X\circ F_\K i_{X,Z}$ is the identity map of $F_\K(X)$ and hence the $*$-homomorphism
 $F_\K i_{X,Z}:F_\K(X)\to F_\K(Z)$ is injective, which allows us identify the $*$-magma $F_\K(X)$ with a $*$-submagma of $F_\K(Z)$.

 By analogy we can define the map $r_Y:Z\to F_\K(Y)$ such that $r_Y|Y=\delta_Y$ and $r_Y(X)=\{1_Y\}$ and consider the unique continuous $*$-homomorphism $\bar r_Y:F_\K(Z)\to F_\K(Y)$ such that $\bar r_Y\circ\delta_Z=r_Y$. The composition $\bar r_Y\circ F_\K i_{Y,Z}$ is the identity map of $F_\K(Y)$. This implies that the $*$-homomorphism $F_\K i_{Y,Z}:F_\K(Y)\to F_\K(Z)$ is injective, which allows us identify the $*$-magma $F_\K(Y)$ with a $*$-submagma of $F_\K(Z)$.

The definition of the maps $\bar r_X$ and $\bar r_Y$ guarantee that the map $$\bar r:F_\K(Z)\to F_\K(X)\times F_\K(Y),\;\;\bar r:a\mapsto (\bar r_X(a),\bar r_Y(a)),$$ is a continuous $*$-homomorphism such that $\bar r\circ h$ is the identity map of $F_\K(X)\times F_\K(Y)$.

Next, we show that the map $h:F_\K(X)\times F_\K(Y)\to F_\K(Z)$ is a $*$-homomorphism.
Given any pairs $(a,b),(c,d)\in F_\K(X)\times F_\K(Y)$, take into account the identities $(xy)^*=y^*x^*$, $xy=yx$ and $(xy)(uv)=(xu)(yv)$ holding in $\K$, and conclude that
$$h((a,b)^*)=h(a^*,b^*)=a^*\cdot b^*=(b\cdot a)^*=(a\cdot b)^*=h(a,b)^*$$ and
$$h((a,b)\cdot(c,d))=h(ac,bd)=(ac)\cdot (bd)=(ab)\cdot (cd)=h(a,b)\cdot h(c,d).$$
So, $h$ is a continuous $*$-homomorphism.

It remains to show that $h$ is surjective. Since $h$ is a $*$-homomorphism, its range $H=h(F_\K(X)\times F_\K(Y))$ is a $*$-submagma of $F_\K(Z)$. Observe that $F_\K(Z)=\bigcup_{n\in\w}E^n[\delta_Z(Z)]$ where $E^0[\delta_Z(Z)]=\delta_Z(Z)$ and $$E^{n+1}[\delta_Z(Z)]=E^n[\delta_Z(Z)]\cup E^n[\delta_Z(Z)]^*\cup E^n[\delta_Z(Z)]\cdot E^n[\delta_Z(Z)]$$for $n\in\w$. By induction on $n\in\w$ we shall show that $E^n[\delta_Z(Z)]\subset H$.
For $n=0$ this follows from the fact that $\delta_Z(Z)=(\delta_X(X)\cdot 1_Y)\cup (1_X\cdot\delta_Y(Y))\subset H$. Assume that for some $n\in\w$ the inclusion $E^n[\delta_Z(Z)]\subset h(F_\K(X)\times F_\K(Y))$ has been proved. Then
$$
E^{n+1}[\delta_Z(Z)]=E^n[\delta_Z(Z)]\cup E^n[\delta_Z(Z)]^*\cup E^n[\delta_Z(Z)]\cdot E^n[\delta_Z(Z)]\subset H\cup H^*\cup H\cdot H\subset H,
$$
which completes the inductive step. Therefore, $$F_\K(Z)=\bigcup_{n\in\w}E^n[\delta_Z(Z)]\subset H=h(F_\K(X)\times F_\K(Y)$$ and hence $h:F_\K(X)\times F_\K(Y)\to F_\K(Z)$ is a bijective $*$-homomorphism with inverse $\bar r$. So, $h$ is a topological isomorphism of the topologized $*$-magmas $F_\K(X)\times F_\K(Y)$ and $F_\K(Z)$.
\end{proof}

\begin{theorem}\label{t:FA-gen2} Let $\K$ be a $d{+}\HM{+}k_\w$-stable variety of topological $*$-magmas such that
 the identities $(x^*)^*=x$, $x(x^*x)=x$, $(x^*x)(x^*x)=x^*x=y^*y$, $(xy)^*=y^*x^*$, $xy=yx$, $(xy)(uv)=(xu)(yv)$ hold in $\K$ and $\K$ contains a non-trivial group.
For a $\mu$-complete Tychonoff $k_\IR$-space $X$ following conditions are equivalent:
\begin{enumerate}
\item[\textup{1)}] The space $X$ is a topological sum $K\oplus D$ of a cosmic $k_\w$-space $K$ and a discrete space $D$.
\item[\textup{2)}] $F_\K(X)$ is a product $K\times D$ of a cosmic $k_\w$-space $K$ and a discrete space $D$;
\item[\textup{3)}] $X$ is an $\bar \aleph_k$-space and $F_\K(X)$ contains no strong $\Fin^\w$-fan;
\item[\textup{4)}] $X$ is an $\aleph_k$-space and $F_\K(X)$ contains no strong $\Fin^\w$-fan and no $\Fin^{\w_1}$-fan.
\item[\textup{5)}] $X$ is an $\aleph_k$-space and $F_\K(X)$ is a $k$-space.
\item[\textup{6)}] $X$ is an $\bar\aleph_k$-space and $F_\K(X)$ is an Ascoli space.
\end{enumerate}
\end{theorem}

\begin{proof} The implication $(1)\Ra(2)$ follows from Propositions~\ref{p:product-abelian}, \ref{p:FKX-disc}, and Theorem~\ref{t:kw}.
The implications $(2)\Ra(3,4)$ follow from Corollary~\ref{c:delta-closed} and Proposition~\ref{k-no-Cld-fan}.
The implications $(3)\Ra(1)$ and $(4)\Ra(1)$ follow from Theorem~\ref{t:FA-gen}. The implication $(1,2)\Ra(5,6)$ are trivial and $(5)\Ra(4)$, $(6)\Ra(3)$ follow from Proposition~\ref{k-no-Cld-fan} and Corollary~\ref{c:A->noFan}, respectively.
\end{proof}

Applying Theorem~\ref{t:FA-gen2} to the variety $\K$ of topological Abelian groups and the functor $\FA=F_\K$ of \index{free!topological Abelian group} free topological Abelian group, we get the following characterization.

\begin{theorem}\label{t:FA} For a $\mu$-complete Tychonoff $k_\IR$-space $X$ the following conditions are equivalent:
\begin{enumerate}
\item[\textup{1)}] The space $X$ is a topological sum $K\oplus D$ of a cosmic $k_\w$-space $K$ and a discrete space $D$.
\item[\textup{2)}] The free Abelian topological group $\FA(X)$ is a product $K\times D$ of a cosmic $k_\w$-space $K$ and a discrete space $D$.
\item[\textup{3)}] $X$ is an $\bar \aleph_k$-space and $\FA(X)$ contains no strong $\Fin^\w$-fan.
\item[\textup{4)}] $X$ is an $\aleph_k$-space and $\FA(X)$ contains no strong $\Fin^\w$-fan and no $\Fin^{\w_1}$-fan.
\item[\textup{5)}] $X$ is an $\aleph_k$-space and $\FA(X)$ is a $k$-space.
\item[\textup{6)}] $X$ is an $\bar\aleph_k$-space and $\FA(X)$ is an Ascoli space.
\end{enumerate}
\end{theorem}

\section{Free topological groups}
\begin{theorem}\label{t:FG-gen}
 Let $\K$ be a $d{+}\HM{+}k_\w$-stable variety of paratopological $*$-magmas such that the identities $x(x^*x)=x$, $x^*x=y^*y$ and $xx^*=yy^*$ hold in $\K$ and $\K$ contains a non-Abelian group. Let $X$ be a $\mu$-complete Tychonoff space.
\begin{enumerate}
\item[\textup{1)}] If $F_\K(X)$ contains no strong $\Fin^{\w_1}$-fans and each infinite compact subset of $X$ contains a convergent sequence, then $X$ either $k$-discrete or $\w_1$-bounded.
\item[\textup{2)}] If $F_\K(X)$ contains no $\Fin^{\w_1}$-fans and $X$ is $k^*$-metrizable, then either $X$ is $k$-discrete or $X$ is an $\aleph_0$-space.
\item[\textup{3)}] If $F_\K(X)$ contains no strong $\Fin^{\w}$-fans and $X$ is an $\bar\aleph_k$-space, then $X$ is either $k$-discrete or $X$ is a $\w_1$-bounded $k$-sum of hemicompact spaces; moreover, if $X$ is  a $k_\IR$-space, then $X$ is either discrete or a cosmic $k_\w$-space.
\item[\textup{4)}] If the space $F_\K(X)$ contains no strong $\Fin^{\w}$-fans and no $\Fin^{\w_1}$-fans and $X$ is an $\aleph_k$-space, then $X$ is either $k$-discrete or  hemicompact.
\end{enumerate}
\end{theorem}

\begin{proof} Taking into account that the variety $\K$ contains a non-Abelian group and applying Lemma~\ref{l:mixing*}(2,3), we conclude that $\K$ contains a paratopological $*$-magma which is $x(y(s^*z)y^*)$-mixing and $(x(s^*z)x^*)(y(s^*z)y^*)$-mixing.  Now we can apply Theorem~\ref{t:long-mix->X} and complete the proof.
\end{proof}

\begin{theorem}\label{t:FG-gen2} Let $\K$ be a $d{+}\HM{+}k_\w$-stable variety of topological $*$-magmas such that the identities $x(x^*x)=x$, $x^*x=y^*y$ and $xx^*=yy^*$ hold in $\K$ and $\K$ contains a non-Abelian group. For a $\mu$-complete Tychonoff $k_\IR$-space $X$  the following conditions are equivalent:
\begin{enumerate}
\item[\textup{1)}] $X$ is either discrete or a cosmic $k_\w$-space;
\item[\textup{2)}] $F_\K(X)$ is either discrete or a cosmic $k_\w$-space;
\item[\textup{3)}] $X$ is a $\bar\aleph_k$-space and $F_\K(X)$ contains no strong $\Fin^{\w}$;
\item[\textup{4)}] $X$ is an $\aleph_k$-space, $F_\K(X)$ contains no strong $\Fin^\w$-fan and no  $\Fin^{\w_1}$-fan;
\item[\textup{5)}] $X$ is an $\aleph_k$-space and $F_\K(X)$ is a $k$-space.
\item[\textup{6)}] $X$ is a $\bar \aleph_k$-space and $F_\K(X)$ is an Ascoli space.
\end{enumerate}
\end{theorem}

\begin{proof} Applying Lemmas~\ref{l:mixing*}(2,3), we can show that the variety $\K$ contains a topological $*$-magma which is $x(y(s^*z)y^*)$-mixing and $(x(s^*z)x^*)(y(s^*z)y^*)$-mixing.
This makes possible to apply Theorem~\ref{t:gen-eq} and complete the proof.
 \end{proof}

Applying Theorem~\ref{t:FG-gen2} to the variety $\K$ of all topological groups, for the functor $\PG=F_\K$ of \index{free!topological group} free paratopological group, we get the following characterization.

\begin{theorem}\label{t:FG} For a $\mu$-complete Tychonoff $k_\IR$-space $X$ following conditions are equivalent:
\begin{enumerate}
\item[\textup{1)}] $X$ is either discrete or a cosmic $k_\w$-space.
\item[\textup{2)}] $\FG(X)$ is either discrete or a cosmic $k_\w$-space.
\item[\textup{3)}] $X$ is a $\bar\aleph_k$-space and $\FG(X)$ contains no strong $\Fin^{\w}$.
\item[\textup{4)}] $X$ is an $\aleph_k$-space, $\FG(X)$ contains no strong $\Fin^\w$-fan and no  $\Fin^{\w_1}$-fan.
\item[\textup{5)}] $X$ is an $\aleph_k$-space and $\FG(X)$ is a $k$-space.
\item[\textup{5)}] $X$ is a $\bar\aleph_k$-space and $\FG(X)$ is an Ascoli space.
\end{enumerate}
\end{theorem}

\section{Free paratopological Abelian groups}

In this section we shall study free para\-topological Abelian groups $\PA(X)$, which are partial cases of the free paratopological $*$-magmas $F_\K(X)$ in $d{+}\HM{+}k_\w$-superstable varieties $\K$ of paratopological $*$-magmas. We shall need the following fact on $(x^*y)(z^*s)$-supermixing $*$-magmas (see Definition~\ref{d:supermix}).

\begin{lemma}\label{l:supermix*} For a $*$-magma $X$ its cube $X^3$ is $(x^*y)(z^*s)$-supermixing if $X$ contains two elements $e,a$ such that $e^*=e=ee\notin \{e(ea),(ea)e\}$ and the elements $(a^*e)e$ and $e(a^*e)$ do not belong to the submagma generated by the set $\{e,a\}$.
\end{lemma}

\begin{proof} By Lemma~\ref{l:mixing*}, the square $X^2$ is $(x^*y)(z^*s)$-mixing and so is the cube $X^3$ of $X$. We claim that the elements $x=(a,e,e)$, $y=(e,e,e)$, $z=(e,e,a)$,  $s=(e,a,e)$ witness that $X^3$ is
$(x^*y)(z^*s)$-supermixing. For a subset $A\subset X^3$ by $\lfloor A\rfloor$ we denote the smallest submagma of $X^2$ containing the set $A$. Let $Q=\{x,y,z,s\}$ and observe that
$\langle Q\setminus\{x\}\rangle_E\subset\{e\}\times X^2$ and hence $\lfloor \langle Q\setminus\{x\}\rangle_E\cup Q\rfloor\subset \lfloor\{e,a\}\rfloor\times X^2$.
By analogy, $\langle Q\setminus\{z\}\rangle_E\subset X^2\times\{e\}$ and $\lfloor \langle Q\setminus\{z\}\rangle_E\cup Q\rfloor\subset X^2\times\lfloor\{e,a\}\rfloor$.
On the other hand,
$$
\begin{aligned}
(x^*y)(z^*s)&=\big((a^*e)(e^*e),(e^*e)(e^*a),(e^*e)(a^*e)\big)=\big((a^*e)e,e(ea),e(a^*e)\big)\notin\\
&\notin \big(\lfloor\{e,a\}\rfloor\times X^2\big)\cup \big(X^2\times\lfloor \{e,a\}\rfloor\big).
\end{aligned}
$$
\end{proof}

Combining Lemma~\ref{l:supermix*} with Lemma~\ref{l:mixing*}(4) and Theorem~\ref{t:dot*-super} we get the following result.

\begin{theorem}\label{t:PA-gen} Let $\K$ be a $d{+}\HM{+}k_\w$-superstable variety of paratopological $*$-magmas such that the identities $x(x^*x)=x$ and $(x^*x)(x^*x)=x^*x=y^*y$ hold in $\K$ and $\K$ contains an infinite cyclic group $Y$.
 If for a $\mu_s$-complete Tychonoff $\aleph_k$-space $X$ the free paratopological $*$-magma $F_\K (X)$ contains no strong $\Fin^\w$-fans and no $\Fin^{\w_1}$-fans, then $X$ is $k$-homeomorphic to a topological sum $K\oplus D$ of a countable $k_\w$-space $K$ and a discrete space $D$.
\end{theorem}

\begin{theorem}\label{t:PA-gen2} Let $\K$ is a $d{+}\HM{+}k_\w$-superstable variety of paratopological $*$-magmas such that the identities $(x^*)^*=x$, $x(x^*x)=x$, $(x^*x)(x^*x)=x^*x=y^*y$, $(xy)^*=y^*x^*$, $xy=yx$, $(xy)(uv)=(xu)(yv)$ hold in $\K$ and $\K$ contains a non-trivial group. For a $\mu$-complete Tychonoff $k_\IR$-space $X$ the following conditions are equivalent:
\begin{enumerate}
\item[\textup{1)}] $X$ is a topological sum $K\oplus D$ of a countable $k_\w$-space $K$ and a discrete space $D$.
\item[\textup{2)}] $F_\K(X)$ is a product $K\times D$ of a countable $k_\w$-space $K$ and a discrete space $D$.
\item[\textup{3)}] $X$ is an $\aleph_k$-space and $F_\K(X)$ contains no strong $\Fin^\w$-fan and no $\Fin^{\w_1}$-fan.
\item[\textup{4)}] $X$ is an $\aleph_k$-space and $F_\K(X)$ is a $k$-space.
\end{enumerate}
\end{theorem}

\begin{proof} The implication $(1)\Ra(2)$ follows from Propositions~\ref{p:product-abelian}, \ref{p:FKX-disc}, and Lemma~\ref{l:para-kw}.
The implications $(2)\Ra(3)$ follow from Corollary~\ref{c:delta-closed} and Proposition~\ref{k-no-Cld-fan}.
The implications $(3)\Ra(1)$ is proved in Theorem~\ref{t:PA-gen}. The implication $(1,2)\Ra(4)$ is trivial and $(4)\Ra(3)$ follows from Proposition~\ref{k-no-Cld-fan}.
\end{proof}

Applying Theorem~\ref{t:PA-gen2} to the variety of $\K$ paratopological Abelian groups, for the functor $\PA=F_\K$ \index{free!paratopological Abelian group} we obtain the following characterization.

\begin{theorem}\label{t:PAG} For a $\mu$-complete Tychonoff $k_\IR$-space $X$ the following conditions are equivalent:
\begin{enumerate}
\item[\textup{1)}] $X$ is a topological sum $K\oplus D$ of a countable $k_\w$-space $K$ and a discrete space $D$.
\item[\textup{2)}] $\PA(X)$ is a product $K\times D$ of a countable $k_\w$-space $K$ and a discrete space $D$.
\item[\textup{3)}] $X$ is an $\aleph_k$-space and $\PA(X)$ contains no strong $\Fin^\w$-fan and no $\Fin^{\w_1}$-fan.
\item[\textup{4)}] $X$ is an $\aleph_k$-space and $\PA(X)$ is a $k$-space.
\end{enumerate}
\end{theorem}

\section{Free paratopological groups}

\begin{theorem}\label{t:PG-gen}
 Let $\K$ be a $d{+}\HM{+}k_\w$-superstable variety of paratopological $*$-magmas such that the identities $x(x^*x)=x$, $x^*x=y^*y$ and $xx^*=yy^*$ hold in $\K$ and $\K$ contains a non-Abelian group and an infinite cyclic group. Let $X$ be a $\mu$-complete Tychonoff space.
\begin{enumerate}
\item[\textup{1)}] If $X$ is $k^*$-metrizable and $F_\K(X)$ contains no $\Fin^{\w_1}$-fans, then either $X$ is $k$-discrete or $X$ is a countable $\aleph_0$-space.
\item[\textup{2)}] If $X$ is an $\aleph_k$-space and $F_\K(X)$ contains no strong $\Fin^{\w}$-fans and no $\Fin^{\w_1}$-fans, then $X$ is either $k$-discrete or countable and hemicompact.
\end{enumerate}
\end{theorem}

\begin{proof} Taking into account that the variety $\K$ contains a non-Abelian group and applying Lemma~\ref{l:mixing*}(2,3), we conclude that $\K$ contains a paratopological $*$-magma which is $x(y(s^*z)y^*)$-mixing and $(x(s^*z)x^*)(y(s^*z)y^*)$-mixing. Since $\K$ contains an infinite cyclic group, we can apply Lemma~\ref{l:supermix*} and conclude that $\K$ contains an $(x^*y)(z^*s)$-supermixing paratopological $*$-magma. Now we can apply Theorem~\ref{t:supermix-long} and complete the proof.
\end{proof}

\begin{theorem}\label{t:PG-gen2} Let $\K$ be a $d{+}\HM{+}k_\w$-superstable variety of paratopological $*$-magmas such that the identities $x(x^*x)=x$, $x^*x=y^*y$ and $xx^*=yy^*$ hold in $\K$ and $\K$ contains a non-Abelian group and an infinite cyclic group. For a $\mu$-complete Tychonoff $k_\IR$-space $X$  the following conditions are equivalent:
\begin{enumerate}
\item[\textup{1)}] $X$ is either discrete or a countable $k_\w$-space;
\item[\textup{2)}] $F_\K(X)$ is either discrete or a countable $k_\w$-space;
\item[\textup{3)}] $X$ is an $\aleph_k$-space, $F_\K(X)$ contains no strong $\Fin^\w$-fan and no  $\Fin^{\w_1}$-fan;
\item[\textup{4)}] $X$ is an $\aleph_k$-space and $F_\K(X)$ is a $k$-space.
\end{enumerate}
\end{theorem}

\begin{proof} Applying Lemmas~\ref{l:mixing*}(2,3) and \ref{l:supermix*}, we can show that the variety $\K$ contains a paratopological $*$-magma which is $x(y(s^*z)y^*)$-mixing, $(x(s^*z)x^*)(y(s^*z)y^*)$-mixing, and $(x^*y)(z^*s)$-supermixing. This makes possible to apply Theorem~\ref{t:para-eq} and complete the proof.
 \end{proof}

Applying Theorem~\ref{t:PG-gen2} to the variety $\K$ of all paratopological groups, for the functor $\PG=F_\K$ of \index{free!paratopological group} free paratopological group, we get the following characterization.

\begin{theorem}\label{t:PG} For a $\mu$-complete Tychonoff $k_\IR$-space $X$ following conditions are equivalent:
\begin{enumerate}
\item[\textup{1)}] $X$ is either discrete or a countable $k_\w$-space.
\item[\textup{2)}] $\PG(X)$ is either discrete or a countable $k_\w$-space.
\item[\textup{3)}] $X$ is an $\aleph_k$-space, $\PG(X)$ contains no strong $\Fin^\w$-fan and no  $\Fin^{\w_1}$-fan.
\item[\textup{4)}] $X$ is an $\aleph_k$-space and $\PG(X)$ is a $k$-space.
\end{enumerate}
\end{theorem}

\section{Free locally convex spaces}

This section is devoted to free locally convex spaces.
By a {\em locally convex space} we understand a Hausdorff locally convex linear topological space over the field $\IR$ of real numbers. It is well-known \cite[3.3.11]{AT} that locally convex spaces (more generally, topological groups) are Tychonoff.

For a topological space $X$ its \index{free!locally convex space}{\em free locally convex space} is defined as a pair $(\Lc(X),\delta_X)$ consisting of a locally convex space $\Lc(X)$ and a continuous map $\delta_X:X\to \Lc(X)$ such that for any continuous map $f:X\to Y$ to a locally convex space $Y$ there exists a unique linear continuous operator $\bar f:\Lc(X)\to Y$ such that $\bar f\circ\delta_X=f$.

The class $\K$ of locally convex linear topological spaces is a variety, which implies that for every topological space $X$ the free locally convex space $(\Lc(X),\delta_X)$ exists and is unique up to an isomorphism. A concrete construction of a free locally convex space looks as follows. Given a Tychonoff space $X$ consider the function space $C_k(X)$ and the double function space $C_p(C_k(X))$, which contains the space of all linear continuous functionals on $C_k(X)$ (endowed with the weak-star topology). Next, consider the canonical map $\delta_X:X\to C_p(C_k(X))$ assigning to each point $x\in X$ the Dirac measure $\delta_X(x):C_k(X)\to \IR$, $\delta_X(x):f\mapsto f(x)$, concentrated at $x$. It is easy to see that the map $\delta_X$ is injective and the image $\delta_X(X)$ is linearly independent in $C_p(C_k(X))$. Consider the linear hull $\Lc(X)$ of the set $\delta_X(X)$ in $C_p(C_k(X))$ and endow the linear space $\Lc(X)$ with the strongest locally convex topology making the map $\delta_X:X\to \Lc(X)$ continuous. It can be shown that the pair $(\Lc(X),\delta_X)$ is a free locally convex space and $\delta_X(X)$ is a closed linearly independent subset of $\Lc(X)$.

We shall need the following known fact (see \cite{Fl84}, \cite{Us83}).

\begin{lemma}\label{l:Lc-k} For a $k$-space $X$ the topology of the space $\Lc(X)$ coincides with the topology inherited from the double function space $C_k(C_k(X))$.
\end{lemma}

For a Tychonoff space $X$ elements $\mu\in \Lc(X)$ can be seen as finitely supported sign-measures on $X$, which can be uniquely written as the linear combination $\mu=\sum_{x\in \supp(x))}\mu(x)\delta_x$ with non-zero coefficients $\mu(x)\in\IR\setminus\{0\}$. The real number $\|\mu\|=\sum_{x\in \supp(\mu)}|\mu(x)|$ is the {\em norm} (or {\em complete variation}) of the sign-measure $\mu$.

It is clear that this construction of a free locally convex space determines a functor $\Lc:\Top_{3\frac12}\to\Top_{3\frac12}$ in the category of Tychonoff spaces. To each continuous map $f:X\to Y$ between Tychonoff spaces this functor assigns the restriction $C_pC_k f|\Lc(X)$ of the linear continuous operator  $C_pC_kf:C_p(C_k(X))\to C_p(C_k(Y))$.

\begin{proposition}\label{p:L-prop} The functor $\Lc:\Top_{3\frac12}\to\Top_{3\frac12}$ of the free locally convex space has the following properties:
\begin{enumerate}
\item[\textup{1)}] $\Lc$ can be completed to a monad $(\Lc,\delta,\mu)$ in the category $\Top_{3\frac12}$, which preserves disjoint supports.
\item[\textup{2)}] $\Lc$ has finite supports.
\item[\textup{3)}] $\Lc$ is monomorphic and preserves closed embeddings of metrizable compacta, but $L$ does not preserve preimages.
\item[\textup{4)}] $\Lc$ is quotient-to-open and hence is compact-to-quotient.
\item[\textup{5)}] $\Lc$ is bounded.
\end{enumerate}
\end{proposition}

\begin{proof}
1. For every Tychonoff space $X$, use the definition of the free locally convex space $\Lc^2(X)=\Lc(\Lc(X))$ to find a unique linear continuous operator $\mu_X:\Lc^2(X)\to \Lc(X)$  such that $\mu_X\circ \delta_{\Lc X}=\id_{\Lc X}$. The linear operators $\mu_X$ are components of the natural transformation $\mu:\Lc^2\to \Lc$, which turn the triple $(\Lc,\delta,\mu)$ into a monad in the category $\Top_{3\frac12}$.

In fact, the operator $\mu_X:\Lc^2(X)\to \Lc(X)$ admits a more concrete description. Given an element $\Lambda\in \Lc^2(X)$, write it as the linear combination $\Lambda=\sum_{\lambda\in\supp(\Lambda)}\Lambda(\lambda){\cdot}\delta_{\Lc X}(\lambda)$. Then write each sign-measure $\lambda\in\supp(\Lambda)$ as a unique linear  combination $\lambda=\sum_{x\in\supp(\lambda)}\lambda(x){\cdot}\delta_X(x)$.
Then $$\mu_X(\Lambda)=\sum_{\lambda\in\supp(\Lambda)}
\Lambda(\lambda)\cdot\sum_{x\in\supp(\lambda)}\lambda(x){\cdot}\delta_X(x).$$
This description implies that $\supp^2(\Lambda)=\bigcup_{\lambda\in\Lambda}\supp(\lambda)$ if the family $\big(\supp(\lambda)\big)_{\lambda\in\Lambda}$ is disjoint. This means that the monad $(\Lc,\delta,\mu)$ preserves disjoint supports.
 \smallskip

2. To see that $\Lc$ has finite supports, take any Tychonoff space $X$ and a linear functional $\lambda\in \Lc(X)$. Since $\Lc(X)$ is a linear hull of the set $\delta_X(X)$, the functional $\lambda$ can be written as a finite linear combination  $\sum_{x\in F}\lambda_x\delta_x$ of
the Dirac measures concentrated as some finite set $F\subset X$. Then $a\in \Lc(F;X)$ and hence $a$ has finite support.
\smallskip

3. The functor $\Lc$ is monomorphic since for every Tychonoff space $X$ the linear space $\Lc(X)$ is algebraically generated by the linearly independent set $\delta_X(X)$. Next we prove that the functor $\Lc$ preserves closed embeddings of metrizable compacta. So, fix a continuous injective map $f:X\to Y$ between metrizable compacta and consider the dual linear operator $C_kf:C_k(Y)\to C_k(X)$, $C_kf:\varphi\mapsto \varphi\circ f$.
By Lemma~\ref{l:Lc-k}, the free locally convex spaces $\Lc(X)$ can be identified with linear hull of the sets $\delta_X(X)$ in the double function space $C_kC_k(X)$. The same is true for the space $Y$. By Dugundji's extension Theorem~\cite{Dug51} (see also \cite[\S II.3]{BP75}), there exists a linear continuous operator $E:C_k(X)\to C_k(Y)$ such that $C_kf\circ E=\id_{C_k(X)}$ and the dual linear operator $C_kE:C_kC_k(Y)\to C_kC_k(X)$, $C_kE:\mu\mapsto\mu\circ E$, assigns to each Dirac measure $\delta_y\in \delta_Y(Y)$ a convex combination of Dirac measures on $X$. This implies that $C_kE(L(Y))\subset \Lc(X)$. On the other hand, applying the (contravariant) functor $C_k$ to the equality $C_kf\circ E=\id_{C_k(X)}$ we get $C_kE\circ C_kC_kf=\id_{C_kC_k(X)}$ and $C_k\circ \Lc f=\id_{\Lc X}$, which implies that the map $\Lc f:\Lc(X)\to \Lc(Y)$ is a closed embedding and $C_kE|\Lc Y$ is a retraction of $\Lc (Y)$ on $\Lc (X)$.

To see that the functor $\Lc$ does not preserve preimages, observe that for the retraction $r^3_2:3\to 2$ the element $\mu=\delta_3(0)+\delta_3(1)-\delta_3(2)\in \Lc(3)\setminus \Lc_2(3)$ with support $\supp(\mu)=\{0,1,2\}=3$ is mapped by the map $\Lc \,r^3_2$ to the measure $\Lc r^3_2(\mu)=\delta_2(0)+\delta_2(1)-\delta_2(1)=\delta_2(0)\subset \Lc_1(3)$.
\smallskip

3. To show that the functor $\Lc$ is quotient-to-open, fix any quotient map $f:X\to Y$ between Tychonoff spaces and consider the linear continuous operator $\Lc f:\Lc (X)\to \Lc(Y)$.
Let $Z=(Lf)^{-1}(0)$ be its kernel, $(\Lc X)/Z$ be the quotient locally convex topological space, and $q:\Lc X\to (\Lc X)/Z$ be the quotient linear operator. It follows that $\Lc f=i\circ q$ for some (unique) continuous bijective operator $i:(\Lc X)/Z\to \Lc Y$.
So, we obtain the commutative diagram:
$$\xymatrix{
X\ar_{f}[dd]\ar^{q\circ\delta_X}[rd]\ar^{\delta_X}[rr]&&\Lc X\ar^{\Lc f}[dd]\ar^{q}[ld]\\
&(\Lc X)/Z\ar^{i}[rd]\\
Y\ar_{i^{-1}{\circ}\delta_Y}[ru]\ar_{\delta_Y}[rr]&&LY
}
$$
Taking into account that the map $f$ is quotient and the map
$(i^{-1}\circ\delta_Y)\circ f=q\circ\delta_X:X\to (\Lc X)/Z$ is continuous, we conclude that the map $i^{-1}\circ\delta_Y:Y\to (\Lc X)/Z$ is continuous and then the linear operator $i^{-1}$ is continuous by the definition of the free locally convex topological space $\Lc Y$. Therefore $i:(TX)/Z\to TY$ is a linear topological isomorphism and the map $\Lc f=i\circ q$ is open.
\smallskip

4. The boundedness of the functor $\Lc$ is proved in Lemma~\ref{l:L-b} below.
\end{proof}

\begin{lemma}\label{l:L-b} For any Tychonoff space $X$ and any bounded subset $K\subset \Lc(X)$ the set $\bigcup_{\mu\in K}\supp(\mu)$ is bounded in $X$ and the set $\{\|\mu\|:\mu\in K\}$ is bounded in $\IR$.
\end{lemma}

\begin{proof} We need to prove that the set $\supp(K)=\bigcup_{\mu\in K}\supp(\mu)$ is bounded in $X$. Since each unbounded subset of $X$ contains a countable unbounded set, we can find a countable subset $K_0\subset K$ whose support $\supp(K_0)$ is not bounded in $X$. By Lemma~\ref{l:Runbound}, there exists a continuous function $f:X\to\IR_+$ such that the restriction $f|\supp(K_0)$ is injective and the set $f(\supp(K_0))$ is unbounded in $\IR_+:=[0,\infty)$.

By induction construct a sequence of non-zero functionals $(\mu_n)_{n\in\w}$ in $K_0$ such that $\max f(\supp(\mu_{n+1}))>\max f(\supp(\mu_n))+1$ for all $n\in\w$. For every $n\in\w$ write the functional $\mu_n$ as the linear combination $\mu_n=\sum_{x\in\supp(\mu_n)}\mu_{n}(x)\delta_x$ of the Dirac measures with non-zero coefficients $\mu_n(x)$. Let $y_n=\max f(\supp(\mu_n))$ and $x_n\in\supp(\mu_n)$ be the (unique) point such that $f(x_n)=y_n$. Since the set $f(\supp(\mu_n))$ is finite, we can choose a positive real number $\delta_n<1$ such that $[y_n-\delta_n,y_n]\cap f(\supp(\mu_n))=\{y_n\}$. For every $n\in\w$ the choice of the measures $\mu_n$ guarantees that $y_n-\delta_n>y_n-1>y_{n-1}$ and hence $[y_n-\delta_n,y_n]\cap\bigcup_{k\le n}\supp(\mu_k)=\{y_n\}$.
By induction choose a sequences of positive real numbers $(z_n)_{n\in\w}$ such that $z_n>2+z_{n-1}>0$ and $$z_n\cdot |\mu_n(x_n)|>n+ \sum_{x\in\supp(\mu_n)\setminus\{x_n\}}|\mu_n(x)|\cdot (z_n-1)$$for every $n\in\IN$.

Now take a continuous increasing function $g:\IR_+\to\IR_+$ such that $g(y_n-\delta_n)=z_n-1$ and $g(y_n)=z_n$ for all $n\in\w$.
By the continuity of the identity operator $\Id:L(X)\to C_p(C_k(X))$, the function $E(g\circ f):\Lc X\to \IR$, $E(g\circ f):\mu\mapsto \mu(g\circ f)$, is continuous.

 The choice of the sequence $(z_n)_{n\in\w}$ guarantees that for every $n\in\w$ we get
$$
\begin{aligned}
|E(g\circ f)(\mu_n)|&=|\mu_n(g\circ f)|=\big|\sum_{x\in \supp(\mu_n)}\mu_n(x)\cdot g\circ f(x)\big|>\\
&>|\mu_n(x_n)|\cdot z_n-\sum_{x\in\supp(\mu_n)\setminus\{x_n\}}
|\mu_n(x)|\cdot g\circ f(x)\ge \\
&\ge |\mu_n(x_n)|\cdot z_n-\sum_{x\in\supp(\mu_n)\setminus\{x_n\}}|\mu_n(x)|\cdot (z_n-1)>n,
\end{aligned}$$
which implies that the set $\{\mu_n\}_{n\in\w}\subset K\subset \Lc X$ is not bounded in $\Lc(X)$, which is a desired contradiction showing the boundedness of the set $\supp(K)$.

To show that the set $\{\|\mu\|:\mu\in K\}$ is bounded in $\IR$, consider the Stone-\v Cech compactification $\beta X$ of the Tychonoff space $X$ and observe that the identity inclusion $i:X\to\beta X$ induces an injective linear continuous operator $\Lc\, i:\Lc(X)\to \Lc(\beta X)\subset C_k(C_k(\beta X)$, which preserves the norms of measures $\mu\in \Lc(X)$. It follows that $C_k(\beta X)$ is a Banach space and  $\Lc\, i(K)$ is a bounded subset of the dual Banach space $C_k(\beta X)^*$ endowed with the weak-star topology. By the Banach-Steinhaus Uniform boundedness principle \cite[3.12]{CFA}, the set $\Lc\, i(K)$ is norm-bounded in the dual Banach space $C_k^*(\beta X)$, which means that the set $\{\|\mu\|:\mu\in \Lc\, i(K)\}=\{\|\mu\|:\mu\in K\}$ is bounded in the real line.
\end{proof}

Now we shall prove that each $\mu$-complete Tychonoff space $X$ identified with $\delta_X(X)$ is $C_k$-embedded in $\Lc(X)$. For this consider the linear operator $E_X:C_k(X)\to C_k(\Lc X)$ assigning to each continuous function $f\in C_k(X)$ the function $\bar f=E_X(f):C_k(\Lc X)\to\IR$, $\bar f:\mu\mapsto \mu(f)$. Observe that $\bar f\circ \delta_X(x)=\delta_x(f)=f(x)$ for every $x\in X$. So, $E_X$ can be considered as an operator extending continuous functions from $X$ to $\Lc X$.

\begin{lemma}\label{l:L-Ck-emb} For any $\mu$-complete Tychonoff space $X$ the linear extension operator $E_X:C_k(X)\to C_k(\Lc X)$ is continuous, which implies that $X$ is $C_k$-embedded into $\Lc(X)$.
\end{lemma}

\begin{proof}
 To prove that the operator $E$ is continuous, take any
neighborhood of zero in $C_k(\Lc X)$. We lose no generality assuming that this neighborhood is of basic form
$$[K,\e]=\{f\in C_k(\Lc X):\sup_{x\in K}|f(x)|<\e\}$$for some compact set $K\subset \Lc X$ and some $\e>0$. By Lemma~\ref{l:L-b}, the set $B=\bigcup_{\mu\in K}\supp(\mu)$ is bounded in $X$ and the set $\{\|\mu\|:\mu\in K\}$ is bounded in $\IR$. So, we can find a number $n\in\IN$ such that $\{\|\mu\|:\mu\in K\}\subset [0,n]$. By  the $\mu$-completeness of the space $X$, the  set $B$ has compact closure $\bar B$ in $X$. It can be shown that the set $$[\bar B,\e/n]=\{f\in C_k(X):\sup_{x\in\bar B}|f(x)|<\e/n\}$$ is an open neighborhood of zero in the function space $C_k(X)$ such that $E([\bar B,\e/n])\subset [K,\e]$.

By Proposition~\ref{p:lin->0}, the space $X$ is $C_k$-embedded in $C_k(X)$.
\end{proof}

\begin{lemma}\label{l:Lc-sa} For any cosmic $k_\w$-space $X$ the space $\Lc(X)$ is a stratifiable $\aleph_0$-space.
\end{lemma}

\begin{proof} By Lemma~\ref{l:Lc-k}, the free locally convex space $\Lc(X)$ embeds into the double function space $C_k(C_k(X))$. Since $X$ is a cosmic $k_\w$-space, the space $C_k(X)$ is Polish and $C_k(C_k(X))$ is a stratifiable $\aleph_0$-space by \cite{GR00} and \cite{Mi66}. Then the subspace $Lc(X)$ of $C_k(C_k(X))$ is a stratifiable $\aleph_0$-space, too.
\end{proof}

\begin{lemma}\label{l:L-Fan-in-sequence} For any compact convergent sequence $X$ the free locally convex space $\Lc(X)$ contains a strong $\dot S^\w$-fan, a strong $D_\w$-cofan, and a strong $\Fin^\w$-fan.
\end{lemma}

\begin{proof} Let $x$ be the limit point of the sequence $X$ and $(x_n)_{n\in\w}$ be an injective enumeration of the set $X\setminus \{x\}$.

First we construct a $D_\w$-cofan in $\Lc X$.
For every $n,m\in\IN$ consider the measure $\mu_{n,m}=\frac1n\sum_{k=1}^m\frac1{2^k}\delta_{x_k}\in \Lc (X)\subset C_k(C_k(X))$. It is clear that for every $n\in\IN$ the sequence $(\mu_{n,m})_{m\in\IN}$ converges to the measure $\mu_n=\frac1n\sum_{k=1}^\infty\frac1{2^k}\delta_{x_k}\in C_k(C_k(X))\setminus \Lc X$ and hence the set $D_n=\{\mu_n\}_{n\in\w}$
is closed and discrete in $\Lc(X)$.
Using the compactness of $X$ and the local convexity of $\Lc(X)$, it can be shown that the sequence $(D_n)_{n\in\w}$ converges to zero in $\Lc(X)$. Consequently, $(D_n)_{n\in\w}$ is a $D_\w$-cofan in $\Lc(X)$.

Next, we construct a $\dot S^\w$-fan $(S_n)_{n\in\w}$ in $\Lc(X)$. For every $n,m\in\IN$ consider the measure $\nu_{n,m}=n(\delta_{x_m}-\delta_x)\in \Lc(X)\subset C_k(C_k(X))$. It is clear that for every $n\in\IN$
the set $S_n=\{0\}\cup\{\nu_{n,m}\}_{m\in\w}$ is a sequence convergent to zero. We claim that the family $(S_n)_{n\in\w}$ is compact-finite in $\Lc(X)$. Given a compact set $K\subset \Lc(X)$ we need to prove that $K$ meets only finitely many sets $S_n\setminus\{0\}$. Observe that each measure $\eta_{n,m}$ is a functional of norm $2n$ in the dual Banach space $C_k^*(X)$. So, the set $S_n\setminus\{0\}$ is contained in the sphere of radius $2n$ in the dual Banach space $C_k^*(X)$. On the other hand,  the Banach-Steinhaus uniform boundedness principle guarantees that the compact set $K$ is norm-bounded in $C_k^*(X)$ and hence meets only finitely many sets $S_n\setminus\{0\}$. This means that $(S_n)_{n\in\w}$ is a $\dot S^\w$-fan in $\Lc(X)$.

Therefore, the space $\Lc(X)$ contains a $D_\w$-cofan and a $\dot S^\w$-fan. Since $\Lc(X)$ is a topological group, we can apply Corollary~\ref{c:g->fan} and conclude that $\Lc(X)$ contains a  $\Fin^\w$-fan.

By Lemma~\ref{l:Lc-sa}, $\Lc(X)$ is an $\aleph_0$-space and by Corollary~\ref{c:fan-in-aleph-space}, each countable fan in $\Lc(X)$ is strong.
\end{proof}

Lemmas~\ref{l:Lc-sa}, \ref {l:L-Fan-in-sequence} and Theorem~\ref{t:F-noseq} imply:

\begin{corollary} If a Tychonoff space $X$ contains a convergent sequence, then its free locally convex space $\Lc(X)$ contains a strong $\Fin^\w$-fan.
\end{corollary}

\begin{lemma}\label{l:L-discrete} If a Tychonoff $\mu$-complete space $X$ contains a strongly compact-finite uncountable set $D\subset X$, then $\Lc(X)$ contains a strong $\Fin^{\w_1}$-fan.
\end{lemma}

\begin{proof} In the uncountable set $D$ choose an $\w_1$-sequence $(x_\alpha)_{\alpha\in\w_1}$ of pairwise distinct points. Since $D$ is strongly compact-finite, each point $x_\alpha\in D$ has an $\IR$-open neighborhood $U_\alpha\subset X$ such that the family $(U_\alpha)_{\alpha\in\w_1}$ is compact-finite in $X$.

Fix an almost disjoint family $(A_\alpha)_{\alpha\in\w_1}$ of infinite subsets of $\IN$. Consider the set $\Lambda=\{(\alpha,\beta)\in\w_1\times\w_1:\alpha\ne\beta\}$ of cardinality $\w_1$ and for every pair $(\alpha,\beta)\in\Lambda$ consider the finite set
$$D_{\alpha,\beta}=\Big\{\tfrac1n(\delta_X(x_\alpha)+\delta_X(x_\beta)):n\in A_\alpha\cap A_\beta\Big\}$$in $\Lc(X)$. The lower semicontinuity of the support map $\supp:\Lc(X)\to[X]^{<\w}$ (proved in Lemma~\ref{l:supp-Ropen}) implies that the set
$$W_{\alpha,\beta}=\{\lambda\in \Lc X:\supp(\lambda)\cap U_\alpha\ne\emptyset\ne \supp(\lambda)\cap U_\beta\}$$is an $\IR$-open neighborhood of the set $D_{\alpha,\beta}$ in $\Lc(X)$. Lemma~\ref{l:cont-supp} implies that  the family $(W_{\alpha,\beta})_{(\alpha,\beta)\in\Lambda}$ is compact-finite in $\Lc(X)$.

Finally, we check that the family $(D_{\alpha,\beta})_{(\alpha,\beta)\in\Lambda}$ is not locally finite at zero. Indeed, for any neighborhood $U\subset \Lc (X)$ of zero, we can find another neighborhood $V\subset \Lc(X)$ of zero such that $V+V\subset U$. For every $\alpha\in\w_1$ the convergence of the sequence $\big(\frac1n\delta_X(x_\alpha)\big)_{n\in\IN}$ to zero yields a number $\varphi(\alpha)$ such that $\big(\frac1n\delta_X(x_\alpha)\big)_{n\ge\varphi(\alpha)}$.

By the Pigeonhole Principle, for some $m\in\w$ the set $\Omega=\{\alpha\in\w_1:\varphi(\alpha)=m\}$ is uncountable. By Lemma~\ref{l:ad}, the family $\Lambda'=\{(\alpha,\beta)\in\Lambda\cap(\Omega\times\Omega):A_\alpha\cap A_\beta\not\subset[0,m]\}$ is uncountable. Given any pair $(\alpha,\beta)\in\Lambda'$, choose a number $n\in A_\alpha\cap A_\beta$ with $n\ge m=\varphi(\alpha)=\varphi(\beta)$. Then $\frac1n\delta_X(x_{\alpha}),\frac1n\delta_X(x_{\beta})\in V$ and hence $$D_{\alpha,\beta}\ni \tfrac1n(\delta_X(x_\alpha)+\delta_X(x_\beta))\in V+V\subset U$$ and the set $\{(\alpha,\beta)\in\Lambda:D_{\alpha,\beta}\cap U\ne\emptyset\}\supset\Lambda'$ is uncountable. This means that the family $(D_{\alpha,\beta})_{(\alpha,\beta)\in\Lambda}$ is not locally countable and hence not locally finite at zero.
Being a strongly compact-finite and not locally finite, this family is a strong $\Fin^{\w_1}$-fan in $\Lc(X)$.
\end{proof}

In \cite{Gab14} Gabriyelyan proved that a Tychonoff space $X$ is discrete and countable if and only if its free locally convex space $\Lc(X)$ is a $k$-space and asked if the $k$-space property in this characterization can be generalized to the Ascoli property of $\Lc (X)$. In this section we answer this question affirmatively proving the following theorem.

\begin{theorem}\label{t:Lc} For a topological space $X$ the following conditions are equivalent:
\begin{enumerate}
\item[\textup{1)}] $X$ is countable and discrete;
\item[\textup{2)}] $\Lc(X)$ is a cosmic $k_\w$-space;
\item[\textup{3)}] $\Lc(X)$ is sequential;
\item[\textup{4)}] $\Lc(X)$ is a $k$-space;
\item[\textup{5)}] $\Lc(X)$ and $X$ are Ascoli.
\item[\textup{6)}] $\Lc(X)$ is Ascoli and $X$ is $\mu$-complete;
\item[\textup{7)}] $\Lc(X)$ contains no strict $\Cld$-fans and $X$ is $\mu$-complete;
\item[\textup{8)}] $\Lc(X)$ contains no strong $\Cld^\w$-fans, no strong $\Fin^{\w_1}$-fans and $X$ is $\mu$-complete and contains no $\Clop$-fans;
\item[\textup{9)}] $\Lc(X)$ contains no strong $\Cld^\w$-fans and no strong $\Fin^{\w_1}$-fans, and $X$ contains no strong $\Cld^\w$-fans and no  $\Clop$-fans;
\item[\textup{10)}] $\Lc(X)$ contains no strong $\Fin^{\w_1}$-fans, $X$ contains no strong $\Cld^\w$-fans and no $\Clop$-fans, and each infinite compact subset of $X$ contains a convergent sequence.
\end{enumerate}
\end{theorem}

\begin{proof} The equivalences $(1)\Leftrightarrow(2)\Leftrightarrow(3)\Leftrightarrow(4)$ were proved by Gabriyelyan in \cite{Gab14} and $(1,4)\Ra(6)$ follows from Ascoli's Theorem. The implication $(6)\Ra(7) $ follows from Corollary~\ref{c:A->noFan} and $(7)\Ra(8)\Ra(9)$ follows from Propositions~\ref{l:L-Ck-emb} and \ref{p:Ck-embed2}. The implication $(9)\Ra(1)$ follows from Theorem~\ref{t:discrete2}, Proposition~\ref{p:L-prop}(4) and Lemmas~\ref{l:L-Fan-in-sequence}, \ref{l:L-discrete}.
The implication $(1,9)\Ra(10)$ is trivial and $(10)\Ra(1)$ follows Corollary~\ref{c:F-disc}, Proposition~\ref{p:L-prop} and Lemmas~\ref{l:L-Fan-in-sequence}, \ref{l:L-discrete}.
\end{proof}

\section{Free linear topological spaces}\label{s:Lin}

In this section we shall study free linear topological spaces.
By a {\em linear topological space} we understand a Hausdorff linear topological space over the field $\IR$ of real numbers (endowed with the standard Euclidean topology).

For a topological space $X$ its \index{free!linear topological space} {\em free linear topological space} is a pair $(\Lin(X),\delta_X)$ consisting of a linear topological space $\Lin(X)$ and a continuous map $\delta_X:X\to \Lin(X)$ such that for any continuous map $f:X\to Y$ to a linear topological  space $Y$ there exists a unique linear continuous operator $\bar f:\Lin(X)\to Y$ such that $\bar f\circ\delta_X=f$.

Linear topological spaces can be considered as topological $E$-algebras of signature $E=E_1\oplus E_2$ where $E_1=\IR$ and $E_2=\{+\}$. It is easy to see that the class of all linear topological spaces is a variety. Lemmas~\ref{l:BanHryn} and \ref{l:kw-product} imply that this variety is $\HM{+}k_\w$-stable.

It follows from general theory that for any Tychonoff space $X$ a free linear topological space $(\Lin(X),\delta_X)$  exists and is unique up to a topological isomorphism.
To obtain a concrete construction of a free linear topological  space, take the free locally convex space $(\Lc(X),\delta_X)$ over $X$ and endow it with the strongest topology $\tau_l$ turning $\Lc(X)$ into a topological vector space and making the canonical map $\delta_X:X\to (\Lc (X),\tau_l)$ continuous. The obtained linear topological vector space $(\Lc(X),\tau_l)$ will be denoted by $\Lin X$ and called \index{free!linear topological space}{\em the free linear topological space} over $X$. Since the map $\delta_X:X\to \Lc(X)$ is injective and $\delta_X(X)$ algebraically generates the space $\Lc(X)$ this construction indeed yields a free topological vector space over $X$.

Observe that for every Tychonoff space $X$ we get the chain of canonical continuous (identity) maps
$$\xymatrix{
X\ar^{\delta_X}[r]&\Lin(X)\ar[r]&\Lc(X)\ar[r]&C_pC_k(X).}$$

The construction of the free topological vector space determines a functor $\Lin:\Top_{3\frac12}\to\Top_{3\frac12}$ on the category of Tychonoff spaces.

\begin{proposition} The functor $\Lin:\Top_{3\frac12}\to\Top_{3\frac12}$ has the following properties:
\begin{enumerate}
\item[\textup{1)}] $\Lin$ can be completed to a monad $(\Lin,\delta,\mu)$ that preserves disjoint supports.
\item[\textup{2)}] $\Lin$ is $\HM$-commuting.
\item[\textup{3)}] $\Lin$ is monomorphic and preserves closed embeddings of metrizable compacta (more generally, closed embeddings of stratifiable spaces).
\item[\textup{4)}] $\Lin$ has finite supports.
\item[\textup{5)}] $\Lin$ is $\II$-regular.
\item[\textup{6)}] $\Lin$ is strongly bounded.
\end{enumerate}
\end{proposition}

\begin{proof}
1. For every Tychonoff space $X$, use the definition of the free topological vector space $\Lin^2 X=\Lin(\Lin X)$ to find a unique linear continuous operator $\mu_X:\Lin^2X\to \Lin X$  such that $\mu_X\circ \delta_{\Lin X}=\id_{\Lin X}$. The linear operators $\mu_X$ are components of the natural transformation $\mu:\Lin^2\to \Lin$, which turn the triple $(\Lin,\delta,\mu)$ into a monad in the category $\Tych$. Algebraically, the linear operators $\mu_X$ coincide with the corresponding linear operators for the monas $(\Lc,\delta,\mu)$. Using this coincidence it can be shown that the monad $(\Lin,\delta,\mu)$ preserves disjoint supports.

The statements (1)--(4) follow from Theorem~\ref{t:FKX1} and the $\HM$-stability of the variety of linear topological spaces (which follows from Lemma~\ref{l:BanHryn}).

5. To prove that the functor $\Lin$ is strongly bounded, fix a Tychonoff space $X$ and a compact set $K\subset\Lin X$.  Consider the identity map $\Lin X\to \Lc X$.  Lemma~\ref{l:L-b} implies that $K\subset \Lc(B;X)=\Lin(B;X)$ for some bounded set $B\subset X$.
Next, consider the Stone-\v Cech compactification $\beta X$. The identity embedding $i_{X,\beta X}:X\to\beta X$ induces an injective continuous map $\Lin i_{X,\beta X}:\Lin X\to \Lin(\beta X)$. The $k_\w$-stability of the variety of linear topological spaces allows us to apply Theorem~\ref{t:kw} and Lemma~\ref{l:Ak}, and conclude that $\supp(K)$ is contained $[X]^{\le m}$ for some $m\in\w$.
\end{proof}

The proof of the following lemma literally repeats the proof of Lemma~\ref{l:L-discrete}.

\begin{lemma}\label{l:V-discrete} If a Tychonoff space $X$ contains a strongly compact-finite uncountable subset $D\subset X$, then the free linear topological vector space $\Lin(X)$ contains a strong $\Fin^{\w_1}$-fan.
\end{lemma}

\begin{lemma}\label{l:V-algebra} For every Tychonoff space $X$ the free linear topological space $\Lin(X)$ is a topological algebra of types $x{\cdot}y$, $x(y^*z)$, $(x^*y)(z^*s)$, and $x(s^*y_i)^{<\w}$, .
\end{lemma}

\begin{proof}  The empty set $C=\emptyset$ and the $\Lin$-valued operations $$
\begin{aligned}
&p_2:X^2\to \Lin(X),\quad p_2:(x,y)\mapsto \delta_X(x)+\delta_Y(y),\\
&p_3:X^3\to \Lin(X),\quad p_3:(x,y,z)\mapsto \delta_X(x)-\delta_Y(y)+\delta_X(z),\\
&p_4:X^4\to \Lin(X),\quad p_4:(x,y,z,s)\mapsto-\delta_X(x)+\delta_X(y)-\delta_X(z)+\delta_X(s),
\end{aligned}
$$
witness that the functor-space $\Lin X$ is a topological algebra of types $x{\cdot}y$, $x(y^{*}z)$ and $(x^{*}y)(z^{*}s)$.

To see that $\Lin X$ is a topological algebra of type $x(s^{*}y_i)^{<\w}$ take any point $s\in X$  and for every $n\in\w$ consider the $n$-ary $\Lin$-valued operation $$p_n:X^{1+n}\to \Lin X,\;\;p_n:(y_i)_{i=0}^n\mapsto \sum_{i=0}^n(\delta_X(y_i)-\delta_X(s)).$$
It is easy to check that the operations $p_n$, $n\in\w$, and sets $C_n=\{s\}$, $n\in\w$, witness that the functor-space $\Lin X$ is a topological algebra of type  $x{\cdot}(s^{*}y_i)^{<\w}$.
\end{proof}

Now we are able to prove the main result of this subsection.

\begin{theorem}\label{t:Lin} For a $\mu$-complete Tychonoff $k_\IR$-space $X$ the following conditions are equivalent:
\begin{enumerate}
\item[\textup{1)}] $X$ is a cosmic $k_\w$-space.
\item[\textup{2)}] $\Lin (X)$ is a cosmic $k_\w$-space.
\item[\textup{3)}]  $X$ is an $\aleph_k$-space and $\Lin (X)$ is a $k$-space.
\item[\textup{4)}]  $X$ is a $\bar\aleph_k$-space and $\Lin (X)$ is Ascoli.
\item[\textup{5)}] $X$ is a $\bar\aleph_k$-space and $\Lin (X)$ contains no strong $\Fin^{\w_1}$-fan.
\item[\textup{6)}] $X$ is an $\aleph_k$-space and $\Lin (X)$ contains no $\Fin^{\w_1}$-fan.
\end{enumerate}
\end{theorem}

\begin{proof} The equivalence $(1)\Leftrightarrow(2)$ follows from Theorem~\ref{t:kw} and the known fact that $\delta_X:X\to \Lin(X)$ is a closed embedding. The implications $(1,2)\Ra(3)\Ra(6)$ and $(1,2)\Ra(4)\Ra(5)$ are trivial. It remains to prove that $(5)$ and $(6)$ imply $(1)$.

 Assume that $\Lin(X)$ contains no (strong) $\Fin^{\w_1}$-fan and $X$ is an $\aleph_k$-space (resp. $\bar\aleph_k$-space). By Corollary~\ref{c:xy.zs+xsy_w}, $X$ is a topological sum $X=C\oplus D$ of a cosmic $k_\w$-space $C$ and a discrete space $D$. By Lemma~\ref{l:V-discrete}, the discrete space $D$ is countable, which implies that $X=C\oplus D$ is a cosmic $k_\w$-space.
\end{proof}
\newpage

\section{Summary}

In the preceding sections we have considered 12 free functorial constructions of Topological Algebra: that of free topological group $\PG$, free topological Abelian group $\FA$, free paratopological group $\PG$, free paratopological Abelian group $\PA$, free linear topological space $\Lin$, free locally convex space $\Lc$, free topological semilattice $\Sl$, free Lawson semilattice $\Law$, free topological inverse semigroup $\IS$, free topological inverse Abelian semigroup $\IAS$, free topological Clifford semigroup $\CS$, and free topological semigroup $\Sem$. These functorial constructions, linked by corresponding natural transformations, fit into the following diagram.
$$
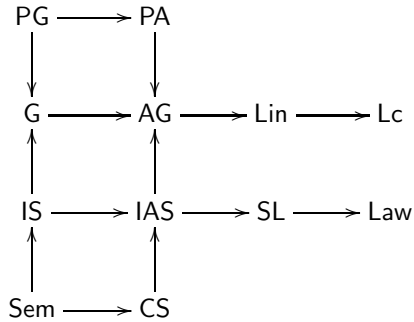
\xymatrix{
\PG\ar[r]\ar[d]&\PA\ar[d]\\
\FG\ar[r]&\FA\ar[r]&\Lin\ar[r]&\Lc\\
\IS\ar[r]\ar[u]&\IAS\ar[r]\ar[u]&\Sl\ar[r]&\Law\\
\Sem\ar[u]\ar[r]&\CS\ar[u]\\
}
$$

In our next diagram we summarize the characterizations proved in  Theorems~\ref{t:Sem}, \ref{t:SL}, \ref{t:Law}, \ref{t:IS+}, \ref{t:FA}, \ref{t:FG}, \ref{t:PAG}, \ref{t:PG}, \ref{t:Lc}, \ref{t:Lin}. In this diagram by $\mathcal D$, $\w \mathcal D$, $\mathcal M$, $k_\w$, and $\w k_\w$ we denote the classes of discrete, countable discrete, metrizable, $k_\w$-spaces, and countable $k_\w$-spaces, respectively. For two classes $\mathcal A,\mathcal B$ of spaces by $\mathcal A\cap \mathcal B$ and $\mathcal A\cup \mathcal B$ we denote their intersection and union, respectively. By $\mathcal A\oplus \mathcal B$ we denote the class of topological sums $A\oplus B$ of spaces $A\in\mathcal A$ and $B\in\mathcal B$, and by $\oplus \A$  the class of topological sums of spaces that belong to the class $\mathcal A$. Given a functor $F$ by $F\mbox{-}k$ we denote the class of Tychonoff $\mu$-complete $\bar\aleph_k$-$k_\IR$-spaces $X$ whose functor-spaces $FX$ are $k$-spaces. In the following diagram for two classes $\A$, $\mathcal B$ of topological spaces an arrow $\A\to\mathcal B$ indicates that $\mathcal A\subset\mathcal B$.
$$\xymatrix{
\mathcal D\cup \w k_\w\ar[r]\ar@{=}[d]&\mathcal D\oplus \w k_\w\ar@{=}[d]\\
\PG\mbox{-}k\ar[r]\ar[d]&\PA\mbox{-}k\ar[d]\\
\FG\mbox{-}k\ar@{=}[d]\ar[r]&\FA\mbox{-}k\ar@{=}[d]&\Lin\mbox{-}k\ar[l]\ar@{=}[d]
&\Lc\mbox{-}k\ar[l]\ar@{=}[d]\\
\mathcal D\cup k_\w\ar[d]\ar[r]&\mathcal \mathcal D\oplus k_\w\ar[d]&k_\w\ar[d]\ar[l]&\w\mathcal D\ar[l]\ar[d]\\
\IS\mbox{-}k\ar[d]\ar@{=}[r]&
\IAS\mbox{-}k\ar@{=}[r]&\Sl\mbox{-}k\ar@{=}[d]&\Law\mbox{-}k\ar@{=}[dd]\\
\Sem\mbox{-}k\ar@{=}[d]&\CS\mbox{-}k\ar@{=}[u]\ar@{=}[r]\ar[l]&\oplus k_\w\\
\mathcal M\cup (\oplus k_\w)&&&\mathcal M\ar[lll]
}
$$

\printindex

\end{document}